\documentclass[11pt]{article}
\usepackage[utf8x]{inputenc}
\usepackage{microtype}
\usepackage[T1]{fontenc}
\usepackage{pinlabel}
\usepackage{booktabs}
\usepackage[table]{xcolor}
\usepackage{times}
\usepackage{float}
\usepackage{verbatim,amsmath,amsthm,amsfonts,amssymb,latexsym, mathtools,color,mathabx}
\usepackage{mathrsfs}
\usepackage{stmaryrd}
\usepackage[all]{xy}
\usepackage{tikz}
\usetikzlibrary{cd}
\usepackage{graphicx}
\usepackage{caption}
\usepackage{subcaption}
\DeclareMathAlphabet{\mathpzc}{OT1}{pzc}{m}{it}
\usepackage[colorlinks,pagebackref,hypertexnames=false]{hyperref} 
\usepackage[alphabetic,backrefs,msc-links]{amsrefs}
\usepackage{latexsym}
\usepackage{amscd}
\usepackage{braket}
\usepackage{comment} 
\usepackage{cleveref} 
\usepackage{accents}

\usepackage{ulem} 
\usepackage{color} 
\usepackage{todonotes}

\usepackage{multicol, color}

\definecolor{gold}{rgb}{0.85,.66,0}
\definecolor{cherry}{rgb}{0.9,.1,.2}
\definecolor{burgundy}{rgb}{0.8,.2,.2}
\definecolor{orangered}{rgb}{0.85,.3,0}
\definecolor{orange}{rgb}{0.85,.4,0}
\definecolor{olive}{rgb}{.45,.4,0}
\definecolor{lime}{rgb}{.6,.9,0}
\definecolor{green}{rgb}{.2,.7,0}
\definecolor{grey}{rgb}{.4,.4,.2}
\definecolor{brown}{rgb}{.4,.3,.1}

\theoremstyle{theorem}
\newtheorem{thm}{\bf Theorem}[section]
\newtheorem{lem}[thm]{\bf Lemma}
\newtheorem{cor}[thm]{\bf Corollary}
\newtheorem{prop}[thm]{\bf Proposition}
\newtheorem{conj}[thm]{\bf Conjecture}
\newtheorem{ques}[thm]{\bf Question}
\theoremstyle{definition}
\newtheorem{defn}[thm]{\bf Definition}
\theoremstyle{remark}
\newtheorem{rem}[thm]{\bf Remark}
\newtheorem{ass}[thm]{\bf Assumption}

\newtheorem{ex}[thm]{\bf Example}

\newcommand{\A}{\mathcal A}
\newcommand{\B}{\mathcal B}

\newcommand{\su}{\mathfrak{su}(2)}

\newcommand{\wh}{\widehat}
\newcommand{\al}{\alpha}

\def\Om{\Omega}

\newcommand{\wt}{\widetilde}
\newcommand{\R}{\mathbb R}
\newcommand{\Z}{\mathbb Z}

\def\Q{{\mathbb Q}}

\newcommand{\newinv}{\mathcal{N}}

\def\ker{\text{Ker }}
\def\im{\text{Im }}
\def\ad{\text{ad}}
\def\rank{\text{rank}}

\def\cs{\text{cs}}

\def\ind{\text{ind}}

\def\wtC{\widetilde{C}}

\def\bPhi{\bar{\Phi}}
\def\bPsi{\bar{\Psi}}

\DeclareMathOperator{\Hom}{Hom}

\DeclareMathOperator{\Tor}{Tor}
\DeclareMathOperator{\Frac}{Frac}
\DeclareMathOperator{\Wh}{Wh}
\tikzset{
contains/.style = {draw=none,"\in" description,sloped}
}
\DeclareMathOperator{\sgn}{sgn}

\newcommand\hrI{{\widehat { I}}}
\newcommand\hrC{{\widehat {\bf C}}}
\newcommand\fhrC{{\widehat { \mathfrak C}}}
\newcommand\hrH{H({\widehat {\bf C}})}
\newcommand\hH{H({\widehat { \mathfrak C}})}
\newcommand\crI{{\widecheck { I}}}
\newcommand\crH{H({\widecheck {\bf C}})}
\newcommand\cH{H({\widecheck {\mathfrak C}})}
\newcommand\crC{{\widecheck {\bf C}}}
\newcommand\fcrC{{\widecheck {\mathfrak C}}}
\newcommand\brI{{\overline { I}}}
\newcommand\brH{H({\overline {\bf C}})}
\newcommand\bH{H({\overline {\mathfrak C}})}
\newcommand\brC{{\overline {\bf C}}}
\newcommand\fbrC{{\overline {\mathfrak C}}}

\newcommand\Sc{\mathcal{S}}

\newcommand{\bF}{{\bf F}}
\newcommand{\fC}{{\mathfrak C}}

\newcommand{\cS}{\mathcal{S}}

\newcommand{\fd}{\mathfrak{d}}
\newcommand{\ffi}{{\mathfrak i}}
\newcommand{\fj}{{\mathfrak j}}
\newcommand{\fp}{{\mathfrak p}}

\newcommand{\locring}{{\Q[\![\Lambda]\!]}}

\newcommand{\lhc}{\widehat{\bf C}}
\newcommand{\lcc}{\widecheck{\bf C}} 
\newcommand{\loc}{\overline {\bf C}}

\newcommand{\shc}{\widehat{\mathfrak C}}
\newcommand{\scc}{\widecheck{\mathfrak C}}
\newcommand{\soc}{\overline {\mathfrak C}}
\newcommand{\oPhi}{\overline \Phi}
\newcommand{\hPhi}{\widehat \Phi}
\newcommand{\cPhi}{\widecheck \Phi}
\newcommand{\oPsi}{\overline \Psi}
\newcommand{\hPsi}{\widehat \Psi}
\newcommand{\cPsi}{\widecheck \Psi}
\newcommand{\li}{{\bf i}}
\newcommand{\lj}{{\bf j}}
\newcommand{\lk}{{\bf p}}
\newcommand{\si}{\mathfrak {i}}
\newcommand{\sj}{\mathfrak {j}}
\newcommand{\sk}{\mathfrak {p}}
\newcommand{\lhl}{\widehat {\boldsymbol\lambda}} 
\newcommand{\lcl}{\widecheck {\boldsymbol\lambda}}
\newcommand{\lol}{\overline {\boldsymbol\lambda}} 

\newcommand{\shl}{\widehat{\mathfrak \lambda}} 
\newcommand{\scl}{\widecheck{\mathfrak \lambda}}
\newcommand{\sol}{\overline {\mathfrak \lambda}} 
\newcommand{\shd}{\widehat{\mathfrak d}} 

\newcommand{\Addresses}{{
  \bigskip
  \footnotesize
  Aliakbar Daemi, \textsc{Washington University in St. Louis, MO}\par\nopagebreak
  \textit{E-mail address}: \texttt{adaemi@wustl.edu}
  \vspace{.2cm}
  
  Hayato Imori, \textsc{Kyoto University, Kyoto, Japan
}\par\nopagebreak
  \textit{E-mail address}: \texttt{imori.hayato.67m@st.kyoto-u.ac.jp}
    \vspace{.2cm}
    
Kouki Sato, \textsc{Meijo University, Nagoya, Japan}\par\nopagebreak
  \textit{E-mail address}: \texttt{satokou@meijo-u.ac.jp}
    \vspace{.2cm}
    
Christopher Scaduto, \textsc{University of Miami, Coral Gables, FL}\par\nopagebreak
  \textit{E-mail address}: \texttt{cscaduto@miami.edu}
    \vspace{.2cm}
    
    Masaki Taniguchi, \textsc{RIKEN iTHEMS 2-1 Hirosawa, Japan} \par\nopagebreak
  \textit{E-mail address}: \texttt{masaki.taniguchi@riken.jp}
}}

\begin{document}

\title{Instantons, special cycles, and knot concordance}

\author{Aliakbar Daemi\thanks{The work of AD was supported by NSF Grants DMS-1812033, DMS-2208181 and NSF FRG Grant DMS-1952762NSF.} \hspace{1cm} Hayato Imori\thanks{The work of HI was supported by JSPS KAKENHI Grant Number 21J20203.} \hspace{.75cm} Kouki Sato \hspace{1cm}\\[5mm]
 Christopher Scaduto\thanks{The work of CS was supported by NSF FRG Grant DMS-1952762.} \hspace{1cm} Masaki Taniguchi\thanks{The work of MT was supported by JSPS KAKENHI Grant Numbers 20K22319, 22K13921 and RIKEN iTHEMS Program. }}

\date{}

\maketitle

\begin{abstract}
We introduce a framework for defining concordance invariants of knots using equivariant singular instanton Floer theory with Chern--Simons filtration. It is demonstrated that many of the concordance invariants defined using instantons in recent years can be recovered from our framework. This relationship allows us to compute Kronheimer and Mrowka's $s^\sharp$-invariant and fractional ideal invariants for two-bridge knots, and more. In particular, we prove a quasi-additivity property of $s^\sharp$, answering a question of Gong. We also introduce invariants that are formally similar to the Heegaard Floer $\tau$-invariant of Oszv\'{a}th and Szab\'{o} and the $\varepsilon$-invariant of Hom. We provide evidence for a precise relationship between these latter two invariants and the $s^\sharp$-invariant.\\

Some new topological applications that follow from our techniques are as follows. First, we produce a wide class of patterns whose induced satellite maps on the concordance group have the property that their images have infinite rank, giving a partial answer to a conjecture of Hedden and Pinz\'on-Caicedo. Second, we produce infinitely many two-bridge knots $K$ which are torsion in the algebraic concordance group and yet have the property that the set of positive $1/n$-surgeries on $K$ is a linearly independent set in the homology cobordism group. Finally, for a knot which is quasi-positive and not slice, we prove that any concordance from the knot admits an irreducible $SU(2)$-representation on the fundamental group of the concordance complement.\\

While much of the paper focuses on constructions using singular instanton theory with the traceless meridional holonomy condition, we also develop an analogous framework for concordance invariants in the case of arbitrary holonomy parameters, and some applications are given in this setting.
\end{abstract}

\newpage

\parskip = 0.95 mm
\setcounter{tocdepth}{2}
{
  \hypersetup{linkcolor=black}
  \tableofcontents
}
\parskip = 2.5 mm


\section{Introduction}

Homological knot invariants provide many useful tools to study 4-dimensional aspects of knots. For instance, there are several knot homology theories that can be used to prove a conjecture due to Milnor \cite{Mil68}, which says that the slice genus of any torus knot $T_{p,q}$ is given by
\[
  g_4(T_{p,q})= \frac{(p-1)(q-1)}{2}.
\]
The original proof due to Kronheimer and Mrowka used instanton gauge theory with respect to connections on 4-manifolds that are singular along embedded surfaces \cite{KM93i}. Later, alternative proofs were given by Ozsv\'{a}th and Szab\'{o} \cite{OS03} using Heegaard knot Floer homology, and Rasmussen \cite{Ras10} using Khovanov homology. Motivated by Rasmussen's work, Kronheimer and Mrowka introduced a concordance invariant $s^\sharp$ in \cite{KM13}, further studied by Gong \cite{Gong21}. Kronheimer and Mrowka's proof of the Milnor conjecture can be recast in terms of $s^\sharp$.

Equivariant singular instanton Floer theory is a package of invariants, studied in \cite{DS19,DS20}, which is roughly the $S^1$-equivariant Morse--Floer theory of a Chern--Simons functional defined on a space of connections which are singular along a knot. As is explained in more detail below, Kronheimer and Mrowka's $s^\sharp$-invariant, and in fact all of the recent concordance invariants they construct in \cite{KM19b}, can be recovered from equivariant singular instanton theory in a systematic way. This new perspective allows us to compute all of these concordance invariants for two-bridge knots, prove a quasi-additivity property of $s^\sharp$ answering a question of Gong \cite{Gong21}, and more.

In this setting, we also introduce concordance invariants $\wt s$ and $\wt \varepsilon$, which can be recovered from $s^\sharp$. We provide evidence that $\wt s$ and $\wt \varepsilon$ are equal to the Heegaard Floer $\tau$-invariant \cite{OS03} and $\varepsilon$-invariant \cite{Hom14cable}, respectively. Further incorporating the Chern--Simons filtration structure into the equivariant singular instanton package, we construct a suite of numerical concordance invariants generalizing the $\Gamma$-invariants studied in \cite{D18, DS19} and the $r_s$-invariants of \cite{NST19}. The combination of these techniques leads to new applications involving satellite operations in concordance, the homology cobordism group, and the existence of irreducible $SU(2)$-representations for concordance complements.

\subsection*{Concordance invariants from equivariant singular instanton theory}

In \cite{DS19}, to a knot in the $3$-sphere, there is associated a certain algebraic object, called an {\it{$\mathcal{S}$-complex}}, up to chain homotopy. Let $\mathcal{C}$ be the smooth knot concordance group. Following a general strategy such as the one in \cite{Sto20}, equivariant singular instanton Floer theory gives rise to a homomorphism
\begin{equation}\label{eq:localequivalencehomintro}
	\mathcal{C}\longrightarrow \Theta_R^\mathcal{S}
\end{equation}
where $\Theta_R^\mathcal{S}$, an algebraic object, is the {\it{local equivalence group}} of $\mathcal{S}$-complexes over the coefficient ring $R$. To define homology concordance invariants, one may use \eqref{eq:localequivalencehomintro} and then work to algebraically extract information from $\Theta_R^\mathcal{S}$. Some progress in this direction was given in \cite{DS19,DS20}.

In search of a relationship between these types of invariants and Kronheimer and Mrowka's $s^\sharp$-invariant, we consider \eqref{eq:localequivalencehomintro} in the case that $R$ is the power series ring $\Q[\![\Lambda]\!]$ which appears in the construction of $s^\sharp$ \cite{KM13}. We define several local equivalence invariants in this setting, two of which when composed with \eqref{eq:localequivalencehomintro} give rise to concordance maps denoted
\[
  \widetilde s:\mathcal{C} \to \mathbb Z,\hspace{1cm}\wt{\varepsilon}:\mathcal{C} \to \{-1,0,1\}.
\]

\begin{thm}\label{s-tilde-slice-torus}
	The invariant $ \widetilde s$ defines a homomorphism from the smooth knot concordance group $\mathcal{C}$ to $\Z$. For any knot $K$ in the $3$-sphere, we have an inequality
	\begin{equation}\label{slice-genus-bound}
	  \widetilde s(K)\leq g_4(K).
	\end{equation}
	This inequality is sharp for any given positive torus knot $T_{p,q}$, in that we have 
	\begin{equation}\label{torus-knot-bound}
	  \widetilde s(T_{p,q})=\frac{(p-1)(q-1)}{2}.
	\end{equation}
	Moreover, Kronheimer and Mrowka's $s^\sharp$-invariant is determined by $\wt s$ and $\wt \varepsilon$ as follows:
	\begin{equation}\label{eq:ssharpis2stildeminusepsilon}
	  s^\sharp(K)=2\widetilde s(K)-\wt{\varepsilon}(K). 
	\end{equation}
\end{thm}

As an immediate consequence of Theorem \ref{s-tilde-slice-torus}, we see that $s^\sharp(K)$ factors through the local equivalence construction of \eqref{eq:localequivalencehomintro}. This gives an affirmative answer to \cite[Question 8.44]{DS19}. Theorem \ref{s-tilde-slice-torus} also implies that $2\widetilde s(K)$ is a slice-torus invariant in the sense of \cite{Li04, Le14}. The results of \cite[Proposition 3.3]{Li04} and \cite[Corollary 5.9.]{Le14} allow us to compute $\widetilde s$ for several families of knots.
\begin{cor}\label{s-tilde-values}
	If  $K$ is a quasi-positive knot, then \eqref{slice-genus-bound} is an equality:
	\[
	  \widetilde s(K)=g_4(K).
	\] 
	If $K$ is an alternating knot, then we have 
	\[
	  \widetilde s(K)=-\frac{\sigma(K)}{2},
	\]
	where $\sigma(K)$ denotes the knot signature.
\end{cor}

In \cite{Gong21}, Gong posed the question of whether there exists a constant $C$ such that 
\begin{equation}\label{eq:quasiadditivity}
	|s^\sharp(K\# K') - s^\#(K) - s^\#(K')|\leq C
\end{equation}
for any knots $K$ and $K'$ in the $3$-sphere. Relation \eqref{eq:ssharpis2stildeminusepsilon} and additivity of $\wt s$ can be used to show that \eqref{eq:quasiadditivity} holds with $C=3$. Using our techniques, we obtain the optimal version of quasi-additivity for $s^\sharp$:

\begin{thm}\label{quasi-additivity}
	 For any pair of knots $K$ and $K'$ in the $3$-sphere, we have
	\[
	  \vert s^\sharp(K\#K')-s^\sharp(K)-s^\sharp(K')\vert\leq 1.
	\]
\end{thm}

In \cite{Gong21}, Gong studies a pair of concordance invariants $s^\sharp_+(K)$ and $s^\sharp_-(K)$ which satisfy
\[
	s^\sharp(K) = s^\sharp_+(K)+ s^\sharp_-(K).
\]
We also exhibit these as local equivalence invariants, prove connected sum inequalities, and derive slice genus bounds which improve on the ones in \cite{Gong21}. See \Cref{section: Kronheimer-Mrowkainvariant}
 and \Cref{more genus bound} for details.
 
 The above properties of the instanton concordance invariants $s^\sharp$, $\wt s$ and $\wt \varepsilon$ are reminiscent of similar properties for certain invariants defined in Heegaard Floer theory. Recall the concordance invariants $\tau$, $\nu$ defined by Osv\'{a}th and Szab\'{o} \cite{OS03, OS11}, and the $\varepsilon$-invariant of Hom \cite{Hom14cable}.

\begin{conj}\label{equiv-HF-inv}
	For any knot $K\subset S^3$, we have $s^\sharp_+(K) = \nu(K)$. In particular, we have the following:
	\[
	 \widetilde s(K)=\tau(K),\hspace{1cm}s^\sharp(K)=\nu(K)-\nu(K^\ast), \hspace{1cm}\widetilde{\varepsilon}(K)=\varepsilon(K).
	\]
\end{conj}

\noindent Here we write $K^\ast$ for the mirror of a knot $K$. As further evidence towards this conjecture, we have the following result. For an integer homology $3$-sphere $Y$, we write $h(Y)\in \Z$ for the instanton Fr\o yshov invariant of $Y$ defined in \cite{Fr02}. 
\begin{thm}\label{s-tilde-froy}
	For any knot $K\subset S^3$ with $\wt s(K)>0$, we have $h(S^3_{1}(K))<0$.
\end{thm}
\noindent A folklore conjecture asserts that the Fr\o yshov instanton invariant $h$ is equal to half the $d$-invariant of Osv\'{a}th and Szab\'{o} \cite{ozabs} for integer homology $3$-spheres. Given this, \Cref{s-tilde-froy} corresponds to an analogous relation between $\tau$ and $d$. This latter relation can be understood via the $\nu^+$-invariant \cite{HW16}; for more details, see for example \cite[\S 2.2]{Sa18}.

Recently, Baldwin and Sivek \cite{BS21} introduced concordance invariants $\tau^\sharp$, $\nu^\sharp$, $\varepsilon^\sharp$ derived from the behavior of the framed instanton homology groups ${I}^\sharp(S^3_r(K))$ with respect to the surgery coefficient $r\in \Q$. Perhaps a more accessible variation of Conjecture \ref{equiv-HF-inv} is the following.
\begin{conj}\label{equiv-I-inv}
	For any knot $K\subset S^3$, we have
	\[
	 \widetilde s(K)=\tau^\sharp(K),\hspace{1cm}s^\sharp(K)=\nu^\sharp(K), \hspace{1cm}\wt{\varepsilon} (K) = \varepsilon^\sharp (K).
	\]
\end{conj}

\noindent Further evidence is an analogue of \Cref{s-tilde-froy} for $\tau^\sharp$ and $h$ which is given in \cite[\S 9]{BS22}. Note that the Baldwin--Sivek invariants are defined in the context of non-singular, non-equivariant instanton homology for $3$-manifolds, and as such, the manner in which they are constructed is entirely different from the methods of the current paper. We remark that an analogue of $\tau$ was also defined using sutured instanton theory by Li \cite{liknothom}, and was shown to agree with $\tau^\sharp$ in \cite{glw}.

Kronheimer and Mrowka recently introduced in \cite{KM19b} new concordance invariants that are constructed using certain versions of singular instanton homology for knots with local coefficients. An example is the fractional ideal invariant $z^\natural(K)$, which is an $\mathscr{S}$-submodule
\[
z^\natural(K)\subset \text{Frac}\left(\mathscr{S}\right), \qquad \mathscr{S}:= \mathbf{F}[T_1^{\pm 1}, T_2^{\pm 1}, T_{3}^{\pm 1}].
\]
They also define numerical invariants $f_\sigma(K)$ derived from such ideal invariants, for a choice of homomorphism $\sigma$ from  $\mathscr{S}$ to a valuation ring. These invariants fit into our framework as well:

\begin{thm}\label{thm:introkminvtsarelocequiv}
	All of the concordance invariants of Kronheimer and Mrowka from \cite{KM19b}, such as $z^\natural(K)$ and $f_\sigma(K)$, factor through the map \eqref{eq:localequivalencehomintro} for an appropriate choice of coefficient ring $R$. In other words, they are determined by the local equivalence class of the equivariant singular instanton $\mathcal{S}$-complex.
\end{thm}

\noindent For more details, see \Cref{Section: Concordance invariants from filtered special cycles}. It should be noted that while Kronheimer and Mrowka's invariants are defined using cobordism maps, our characterization of the invariants is entirely in terms of the equivariant singular instanton $\mathcal{S}$-complex of the knot. The same remark holds for $s^\sharp$.

As the equivariant singular instanton theory of two-bridge knots is partially understood (see \cite{DS20}), we may use our new perspective on these invariants to carry out computations for this class of knots.

\begin{thm}\label{thm:intro2bridge}
	For a two-bridge knot $K$, the invariant $s^\sharp(K)$, as well as the concordance invariants of \cite{KM19b}, are determined by the knot signature $\sigma(K)$. 
\end{thm}
 
\noindent See \Cref{sec:twobridge} for more detailed statements. Note that Kronheimer and Mrowka previously computed the concordance invariants of \cite{KM19b} for special families of two-bridge knots. \Cref{thm:intro2bridge} implies:

\begin{cor}
\Cref{equiv-HF-inv} is true for two-bridge knots.
\end{cor}

\vspace{-0.2cm}

\noindent We also verify \Cref{equiv-HF-inv} for positive knots; see \Cref{precise comp of s}.

All of the invariants discussed thus far, and in fact all constructions in this paper, work more generally for knots in integer homology $3$-spheres. In particular, we extend the definitions of $s^\sharp(K)$ and the concordance invariants of \cite{KM19b} to knots in integer homology $3$-spheres. However, not all of the properties of these invariants given above are established in this more general setting.

One of the main technical (algebraic) tools used in our constructions is that of {\it{special cycles}}. The notion of a special cycle can be traced back to Fr\o yshov's definition of the $h$-invariant in the setting of integer homology $3$-spheres \cite{Fr02}; see in particular \eqref{eq:froyshovdefpos}--\eqref{eq:froyshovdefneg}. In this paper we systematically develop special cycles. All of our concordance invariants are obtained by utilizing special cycles in various incarnations of the equivariant singular instanton Floer complex of a knot.

\subsection*{Topological Applications}

There is an additional layer of structure on the equivariant singular instanton Floer $\mathcal{S}$-complex of a knot which comes in the form of a real-valued filtration, which is roughly obtained by keeping track of the Chern--Simons values of flat singular connections. Incorporating this into the above framework leads to a plethora of numerical concordance invariants which we describe below. However, for the benefit of the reader, we first describe some of the topological applications that come out of this story.

\subsubsection*{Satellites and concordance}

Consider a knot $P\subset S^1\times D^2$, called a {\it{pattern}}. Given a knot $K$ in the $3$-sphere, cutting out a solid torus neighborhood of $K$ and gluing back in $(S^1\times D^2, P)$ gives the {\it{satellite}} knot $P(K)$. The assignment $K\mapsto P(K)$ descends to define a map on the smooth concordance group:
\[
	P:\mathcal{C}\longrightarrow \mathcal{C}
\]
In some cases this is a constant map. For example, if the pattern $P$ is contained in a 3-ball embedded inside $S^1\times D^2$, then for all knots $K$, we have $P(K)=P(U_1)$ where $U_1$ is the unknot. At the other extreme, Hedden and Pinz\'{o}n-Caicedo make the following conjecture.

\begin{conj}[ {\cite[Conjecture 2]{HC21}}]\label{conj:hpc}
If for a given pattern $P$ the satellite operator $P:\mathcal{C}\to \mathcal{C}$ is non-constant, then its image generates an infinite rank subgroup of $\mathcal{C}$. 
\end{conj}

\noindent The following result, proved by our methods, verifies \Cref{conj:hpc} for a large class of patterns.  

\begin{figure}[t]
    \centering
    \labellist
	\Large\hair 2pt
	\pinlabel ${\color{blue}b_1}$ at 490 682
	\pinlabel ${\color{red}a_1}$ at 551 618
	\pinlabel ${\color{blue}b_2}$ at 431 560
	\pinlabel ${\color{red}a_2}$ at 493 495
	\pinlabel ${\color{blue}b_3}$ at 349 436
	\pinlabel ${\color{blue}b_n}$ at 213 314
	\pinlabel ${\color{red}a_n}$ at 274 248
	\pinlabel ${\color{red}a_m}$ at 55 56
	\pinlabel ${\color{red}a_{n+1}}$ at 213 177
	\endlabellist
	\vspace{-0.3cm}
    \includegraphics[scale=0.42]{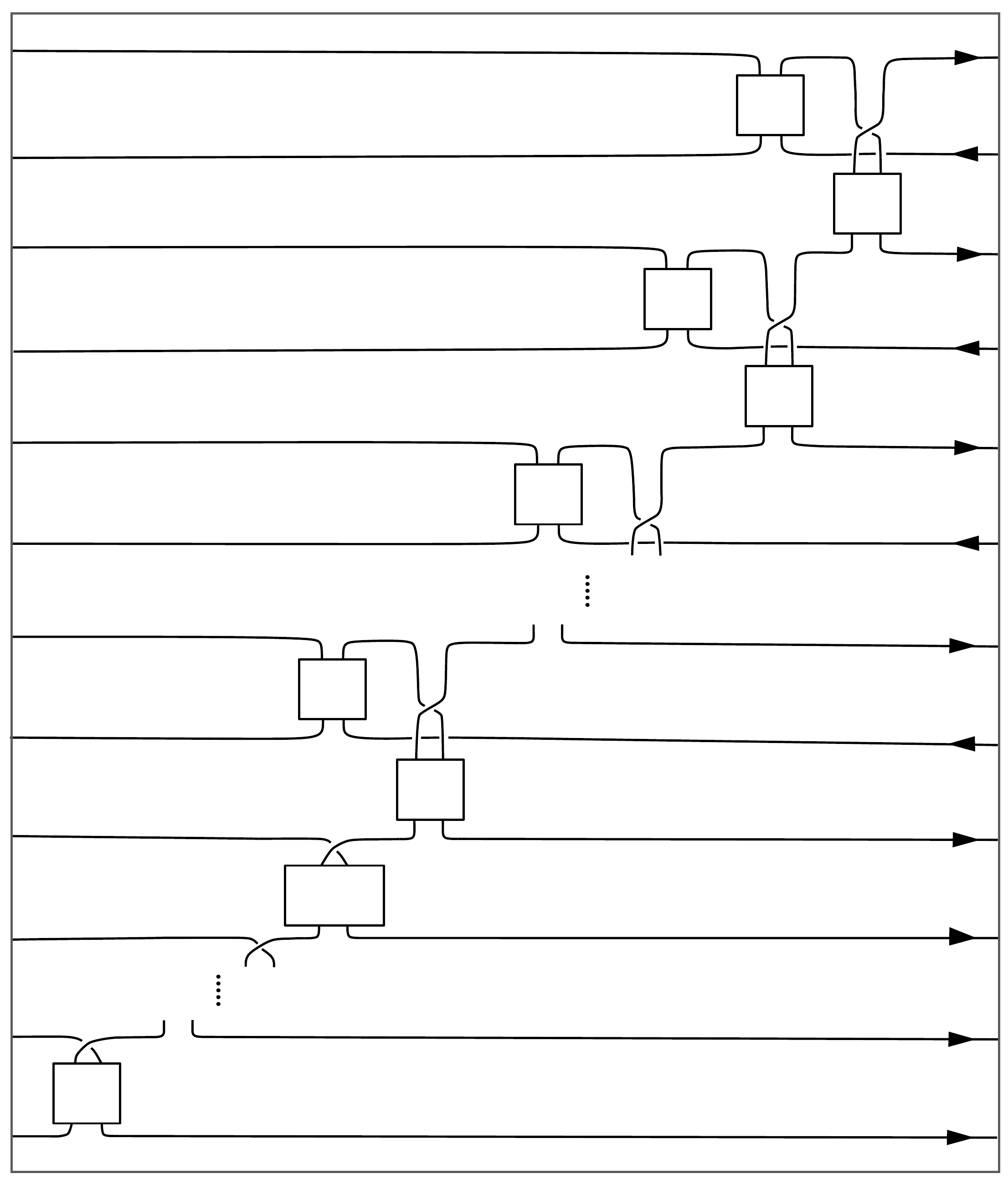}
    \caption{Tangle corresponding to a pattern $P$ determined by $\{a_i\}_{i=1}^m$ and $\{b_i\}_{i=1}^n$. }
    \label{fig:satellite}
\end{figure}
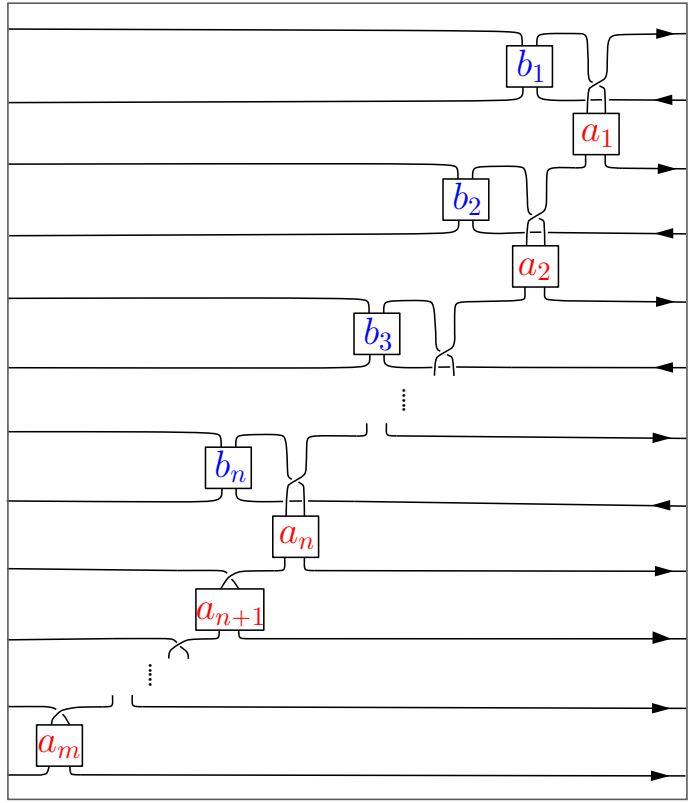

\begin{thm}\label{main thm satellite}
Suppose a pattern $P$ satisfies the following:
\begin{itemize}
    \item[(i)] There exists a knot $K$ that can be unknotted by a sequence of positive-to-negative crossings changes, and such that $P(K)$ is concordant to a non-slice quasi-positive knot. 
    \item[(ii)] $P(U_1)$ is the unknot.
\end{itemize}
Then, the image of the induced map 
$P \colon \mathcal{C} \to \mathcal{C}$
generates an infinite rank subgroup of $\mathcal{C}$.
\end{thm}

\noindent We prove a more general result; see \Cref{knot concordance main thm2}. 

As concrete examples, we may consider a pattern $P$ associated to two sequences of negative integers
$\{a_i\}_{i=1}^m$ and $\{b_i\}_{i=1}^n$ where $m\geq n-1$ and $\max\{m,n\}>0$, obtained from the tangle diagram in \Cref{fig:satellite}
by identifying the vertical edges of the rectangle. Included in this family of patterns are Whitehead doubles, $(m+1,1)$-cables, and  
Yasui's pattern $P_{0,k}$ from \cite[Figure 10]{Yas15}. We show in \Cref{section: knot conc group} that this general class of patterns satisfies the assumptions of \cref{main thm satellite}.

In fact, any pattern $P$ as in \Cref{fig:satellite} satisfies a stronger version of (i) in \Cref{main thm satellite}: $P$ preserves strong quasi-positivity, in the sense that if $K$ is strongly quasi-positive, then so too is $P(K)$. Thus if $P_1, \ldots, P_l$ are constructed as in \Cref{fig:satellite}, then the composition $P_l \circ \cdots  \circ P_1$ also satisifies the assumptions of \cref{main thm satellite}. Note that the patterns of \Cref{fig:satellite} have winding number zero when $m=n-1$. 

The above remarks imply that our result is independent of the ones in \cite[Theorem 3]{HC21}. There, \Cref{conj:hpc} is proved for non-zero winding number patterns, and also for certain patterns satisfying an assumption involving non-triviality of rational linking numbers. By an argument similar to one in \cite[\S 5.3]{HC21}, if at least one (resp. two) of $P_1,  \ldots, P_l$ from \Cref{fig:satellite} has winding number zero, then the composition has zero winding number (resp. vanishing rational linking numbers).

For a pattern $P$ as described in \Cref{fig:satellite} (or a composition thereof), we specifically show that
\[
	\{P(T_{p,q+np})\}_{n=0}^\infty
\]
is a linearly independent set in $\mathcal{C}$. Here $p,q$ are coprime positive integers. More generally, this is true for any pattern $P$ satisfying \Cref{main thm satellite} where in (i) the knot $K$ can be taken as $T_{p,q}$. In particular, taking $P$ to be an iterated Whitehead double, we obtain the following.

\begin{cor}\label{Cor satellite}
Let $p$ and $q$ be coprime integers with $p,q >1$.
Then, for any integer $r>0$, 
the $r$-th iterated Whitehead doubles
$\{\Wh^r(T_{p,q+np})\}_{n=0}^\infty$ are
linearly independent in $\mathcal{C}$. 
\end{cor}

\noindent This result generalizes \cite[Theorem 1.12]{NST19} (the case $r=1$), which in turn is a generalization of a result in \cite{HK12}. See also \cite{PC19,Park18} for related results.

\subsubsection*{Two-bridge knots and homology cobordism}

The study of the homology cobordism group $\Theta^3_\Z$ of integer homology 3-spheres is a central topic in the 4-dimensional topology. See \cite{Ma18} for a survey. The techniques of this paper lead to the following.

\begin{thm}\label{1028}
There exist infinitely many two-bridge knots $K$ which are torsion in the algebraic concordance group and for which the set $\{S_{1/n}^3(K)\}_{n=1}^\infty$ is linearly independent in $\Theta^3_\Z$. 
\end{thm}

Recall that $K$ is torsion in the algebraic concordance group if and only if the Tristram--Levine signature function of $K$ is identically zero. Write $K(p,q)$ for the two-bridge knot of type $(p,q)$. Our convention is such that $K(3,1)$ is the right-handed trefoil. Define a family of a two-bridge knots as follows:
\begin{equation}\label{eq:twobridgefamily}
K_{m,n}:=  K(212mn-68n+53,106m-34 ).
\end{equation}
We show, more specifically, for integers $m\geq 7$ and $n\geq 0$, that $K_{m,n}$ has vanishing Tristram--Levine signature function, and the set of homology spheres $\{S_{1/k}^3(K^\ast)\}_{k=1}^\infty$ is linearly independent in $\Theta_\Z^3$.

The simplest example is $K^\ast_{m,0}$ (independent of $m$), which is the knot $10^\ast_{28}$. \Cref{1028} also applies to two-bridge knots outside of the family $K^\ast_{m,n}$. For example, we show that it is true for
\[
	K(65,51) = 11a^\ast_{333}, \qquad K(81,52)=12a^\ast_{596}.
\]
These two examples have the additional feature that the Rohlin invariants of their $1/n$-surgeries all vanish.

While the methods of this paper, together with previous work, can produce many variations of \Cref{1028}, this particular result is of interest because it seems difficult to prove using other flavors of Floer theory. Moreover, in the context of instanton theory, a result proved in \cite{NST19} says that 
\begin{equation}\label{eq:nstthm}
	h(S_1^3(K))<0 \qquad \Longrightarrow \qquad \{S_{1/n}^3(K)\}_{n=1}^\infty  \text{ is linearly independent in } \Theta^3_\Z
\end{equation}
where $h$ is the instanton Fr\o yshov invariant \cite{Fr02}. However, we expect that all of our examples satisfying \Cref{1028} have $h(S_1^3(K))=0$. We remark that $h$ is conjecturally equal to half the $d$-invariant of \cite{ozabs}. Our results, to the best of our knowledge, give the first examples of knots for which the $1/n$-surgeries, for all positive integers $n$, are linearly independent and their $d$-invariants vanish.

In the other direction, our methods give new information about the instanton Fr\o yshov invariant $h$, and combined with the result \eqref{eq:nstthm} from \cite{NST19} we can prove the following.

\begin{thm}\label{hom cob main}
	Let $K$ belong to one of the following two classes of knots:
	\begin{enumerate}
	\item[(i)] Alternating knots with negative signature;
	\item[(ii)] Quasi-positive knots which are not slice.
\end{enumerate}
Then $\{S_{1/n}^3(K)\}_{n=1}^\infty$ is a linearly independent set in the homology cobordism group.
\end{thm}

Recently, Baldwin and Sivek \cite{BS22} independently gave an alternative proof of \cref{hom cob main} using a relationship between $\tau^\sharp$ and the Fr\o yshov invariant $h$, and also passing through \eqref{eq:nstthm} from \cite{NST19}. Our method uses $\wt{s}$ instead of $\tau^\sharp$. Also, note that when $K$ is a positive torus knot, \cref{hom cob main} recovers linear independence results proven in \cite{Fu90, FS92}.

In the course of proving the above results, various concordance properties of knots are established. Here is one such result. A knot $K\subset S^3$ is called {\it{H-slice}} in a given closed 4-manifold $X$ if there is a properly and smoothly embedded null-homologous disk in $X \setminus \text{int} D^4$ bounded by $K$. 

\begin{thm}\label{two bridge H-slice}
Consider the two-bridge knots $K_{m,n}$ defined in \eqref{eq:twobridgefamily}. Let $m$ be a positive integer, and let $K$ be a knot whose concordance class is represented as follows:
\[
[K] = a_1[ K_{m, 0}]+  a_2[ K_{m, 1}]+  \cdots +a_N [K_{m, N}],
\]
where $a_i$ are integers and $a_N>0$. Then $K$ is not smoothly H-slice in any positive definite smooth closed 4-manifold with $b_1=0$. For $m\geq 7$, the knot $K$ has vanishing Tristram--Levine signature function.
\end{thm}

We remark that Heegaard Floer theory \cite{OS03} and Seiberg-Witten theory \cite{Ba22} provide several obstructions to H-sliceness in definite 4-manifolds. 
However, it is known that these obstructions reduce to the classical knot signature for two-bridge knots.

\subsubsection*{Existence of irreducible $SU(2)$-representations}

There is a strong relationship between signatures of knots and $SU(2)$-representations of knot complements. For example, suppose the Tristram--Levine signature $\sigma_\omega(K)$ of a knot $K$ is non-zero for some $\omega\in (0,1/2)$, where $e^{4\pi i \omega}$ is not a root of the Alexander polynomial. It then follows from \cite{herald-casson} that there exists an irreducible $SU(2)$ representation of $\pi_1(S^3\setminus K)$ which sends a meridian to the conjugacy class of
 \begin{equation}\label{omega-conj-2}
\left[\begin{array}{cc}
e^{2\pi i \omega} &0\\
0&e^{-2\pi i \omega} \end{array}
\right] \in SU(2).
\end{equation}
See also \cite{herald,heusener-kroll}. The following is a $4$-dimensional extension of this result, applying to complements of concordances.

\begin{thm}\label{alg conc irrep}Let $K$ be a knot with non-vanishing Tristram--Levine signature $\sigma_\omega(K)\neq 0$ for some $\omega\in (0,1/2)$, where $e^{4\pi i\omega}$ is not a root of the Alexander polynomial. Then for any smooth knot concordance $S\subset [0,1]\times S^3$ from $K$ to a knot $K'$, there exists a non-abelian $SU(2)$-representation on the fundamental group of the concordance complement which sends meridians to the conjugacy class of \eqref{omega-conj-2}.
\end{thm}

This generalizes an analogous result involving the ordinary knot signature $\sigma(K)=\sigma_{1/4}(K)$, proved in \cite[Theorem 5]{DS20}. Versions of \Cref{alg conc irrep} for the $\omega=1/4$ case hold with the Tristram--Levine signature replaced by many of the concordance invariants considered in this paper. For example:

\begin{thm}\label{thm:stilderepresentationsintro}
If any of $\wt s(K)$, $s^\sharp(K)$, $\wt \varepsilon(K)$ is non-zero, then, for any knot concordance $S \subset [0,1]\times S^3$ from $K$ to $K'$, there exists a non-abelian traceless $SU(2)$-representation on the concordance complement. 
\end{thm}

In particular, \Cref{thm:stilderepresentationsintro} and \Cref{s-tilde-values} imply the following.

\begin{cor}\label{irrep for quasi-positive}
For a non-slice quasi-positive knot $K$ and any knot concordance $S \subset [0,1]\times S^3$ from $K$ to $K'$, there exists a non-abelian traceless $SU(2)$-representation on the concordance complement. 
\end{cor}

\subsection*{Concordance invariants and the Chern--Simons filtration}

We now give a brief overview of the invariants that are used to prove the above applications. As already alluded to above, these invariants are obtained by incorporating the Chern--Simons filtration into the strategy, following \eqref{eq:localequivalencehomintro}, of extracting concordance invariants by probing the local equivalence group of $\mathcal{S}$-complexes. First, we broaden our topological viewpoint.

As already mentioned, all of the constructions in this paper are carried out for knots in integer homology $3$-spheres. A (smooth) cobordism of pairs $(W,S):(Y,K)\to (Y',K')$ between knots in integer homology $3$-spheres is a {\it{homology concordance}} if $W$ is an integer homology cobordism, and $S$ is a properly embedded annulus. The collection of knots in integer homology $3$-spheres modulo homology concordance gives rise to an abelian group $\Theta^{3,1}_\Z$, called the {\it{homology concordance group}}. There is a natural homomorphism from the smooth concordance group $\mathcal{C}$ to the group $\Theta^{3,1}_\Z$.

In \cite{DS19,DS20}, a homology concordance invariant $\Gamma_{(Y,K)}:\Z\to [0,\infty]$ is defined and studied, modelled after the homology cobordism invariant $\Gamma_Y$ from \cite{D18}. Homology cobordism invariants $r_s(Y)\in [0,\infty]$, where $s\in [-\infty,0]$, were defined in \cite{NST19}. All of these invariants utilize the Chern--Simons filtration on instanton homology. In this paper, we adapt the construction of $r_s(Y)$ to knots and define homology concordance invariants $r_s(Y,K)$. In fact, we define a homology concordance invariant
\begin{equation}\label{eq:newinvdefintro}
	\newinv_{(Y,K)}: \Z\times [-\infty, 0) \longrightarrow [0,\infty]
\end{equation}
which simultaneously generalizes $\Gamma_{(Y,K)}$ and $r_s(Y,K)$. That is, for certain choices of coefficient rings, 
\begin{align*}
	\Gamma_{(Y,K)}(k) &=\newinv_{(Y,K)}(k,-\infty), \\[2mm]
	r_s(Y,K) &= - \min\{ \inf\{ r\in [-\infty,0) \mid \newinv_{(Y,K)}(0,r)\leq -s\}, 0\}.
\end{align*}
The construction of $\newinv_{(Y,K)}$ roughly uses the equivariant singular instanton $\mathcal{S}$-complex of $(Y,K)$, together with its Chern--Simons filtration structure, and the algebraic machinary of {\it{filtered special cycles}}. 

All of the properties proved for $r_s(Y)$ in \cite{NST19} have analogues for $r_s(Y,K)$. More generally, we show that the invariant $\newinv_{(Y,K)}$ satisfies certain inequalities with respect to cobordisms. See \Cref{Fr ineq} and \Cref{newinvinqtranspose}. These inequalities generalize the inequalities of $\Gamma_{(Y,K)}$ given in \cite{DS20}, and the inequalities for $r_s(Y,K)$ which are analogous to the ones for $r_s(Y)$ in \cite{NST19}. 

The behavior of these invariants with respect to connected sums is also studied.

\begin{thm}\label{thm:connsumineqintro}
Given knots in integer homology $3$-spheres $(Y, K)$ and $(Y', K')$, suppose $s^\otimes<0$, where $s^\otimes := \max \{ \newinv_{(Y,K)} (k,s) +s', \newinv_{(Y',K')}(k',s') +s\}$. Then
\[
\newinv_{(Y \# Y' , K \# K') }(k+k', s^\otimes) \leq \newinv_{(Y, K)} (k,s) + \newinv_{(Y', K')} (k',s').
\]
Consequently, the invariants $\Gamma$ and $r_s$ satisfy the following inequalities:
	\begin{align*}
		&\Gamma_{ (Y\# Y', K \# K') }(k+k') \leq \Gamma_{(Y,K)}(k) + \Gamma_{(Y', K')} (k'),  \\[2mm]
		&r_{s+s'}(Y \# Y' , K \# K')- s-s' \geq \min \{ r_{s}(Y, K)-s , \; r_{s'}(Y', K') -s'\}. 
	\end{align*}
\end{thm}

To make contact with $3$-manifold invariants, we have the following result. It is proved using a natural cobordism from $(Y_1(K),U_1)$, which is $1$-surgery with an unknot, to the pair $(Y,K)$.

\begin{thm}\label{intro non-triviality of r0 for surgery}
	Let $K$ be a knot in an integer homology 3-sphere $Y$ satisfying $\sigma  (Y,K) \leq 0$. Suppose 
	$\frac{1}{8} < \Gamma_{(Y,K)} \left( -\frac{1}{2} \sigma (Y,K) \right)$.
	Then, $\Gamma_{Y_1(K)} (0) >0$ and $r_0 (Y_{1}(K))< \infty $.
\end{thm}

\noindent \Cref{1028} is an application of this result.

The topological applications described earlier in this introduction all make use of $\newinv_{(Y,K)}$, to varying degrees. We note that $\newinv_{(Y,K)}$, and thus $\Gamma_{(Y,K)}$ and $r_s(Y,K)$, are defined in the setting of singular instanton theory with holonomy around shrinking meridians of $K$ having trace limiting to zero. The same is true for the invariants $\wt s$, $s^\sharp$, $\wt \varepsilon$, and those of \cite{KM19b}, discussed above. 

More generally, we consider the case in which the meridional holonomies have limiting trace $2\cos(2\pi \omega)$, where $\omega\in (0,1/2)$, the traceless case corresponding to $\omega=1/4$. The analytical framework here was established by Kronheimer and Mrowka \cite{KM93i,KM11}. Knot invariants defined using singular instanton theory with general holonomy parameter $\omega$ were studied in \cite{Ech19, Ha21}.

The cases in which $\omega \neq 1/4$ are distinguished from the traceless case in that the coefficient ring used must take into account the possibility of infinitely many isolated instantons interpolating between two singular flat connections. In \Cref{Theory for general holonomy parameters}, we develop the local equivalence story, with Chern--Simons filtration, for arbitrary holonomy parameter $\omega\in (0,1/2)$. For example, generalizing \eqref{eq:newinvdefintro}, we define
\[
\newinv^\omega_{(Y,K)}: \Z\times [-\infty, 0) \longrightarrow [0,\infty]
\]
for each $\omega\in (0,1/2)$ such that $e^{4\pi i \omega}$ is not a root of the Alexander polynomial of the knot. This invariant is constant with respect to homology concordances in its domain of definition. \Cref{alg conc irrep} is an application of this more general general framework.

\subsection*{Organization}

In \Cref{sec:singinstthy}, we review aspects of equivariant singular instanton theory, following \cite{DS19,DS20}. Furthermore, certain deformed versions of the instanton $\mathcal{S}$-complex, called {\it{(un)reduced framed}} homology theories, are studied. In \Cref{section: special cycles}, {\it{special cycles}} are introduced and studied. In \Cref{section:invs from froyshov cycles}, we use the machinary of special cycles to define concordance invariants. In particular, $\wt s$ and $\wt \varepsilon$ are defined, and we prove \Cref{s-tilde-slice-torus} (recovering $s^\sharp$), \Cref{s-tilde-values} (computations of $\wt s$), and \Cref{quasi-additivity} (quasi-additivity of $s^\sharp$). Here we also prove \Cref{thm:introkminvtsarelocequiv}, and \Cref{thm:intro2bridge} (two-bridge knot computations). In \Cref{Section: Concordance invariants from filtered special cycles} we incorporate the Chern--Simons filtration into our constructions, and study {\it{filtered special cycles}}. We introduce $\newinv_{(Y,K)}$, and prove Theorems \ref{s-tilde-froy} and \ref{thm:connsumineqintro}. In \Cref{sec:applications}, all of the topological applications are proved, except for \Cref{alg conc irrep}. In this section, we also prove \Cref{intro non-triviality of r0 for surgery}. Finally, in \Cref{Theory for general holonomy parameters}, the theory for general holonomy parameters is developed, and \Cref{alg conc irrep} is proved.

\subsection*{Acknowledgements}
The authors would like to thank Nikolai Saveliev, Daniel Ruberman and Juanita Pinz{\'o}n-Caicedo for valuable comments. AD, HI and CS thank Nikolai Saveliev for co-organizing the 2022 workshop ``Gauge Theory'' in Miami, where some of the work was carried out. HI thanks RIKEN iTHEMS for their hospitality during a visit where aspects of this work with MT were discussed. 


\section{Singular instanton theory}\label{sec:singinstthy}

This section provides background for the rest of the paper. We first review relevant constructions and results from \cite{DS19,DS20} on $\mathcal{S}$-complexes for knots derived from equivariant singular instanton Floer theory, with meridional holonomy parameter $\omega=1/4$. (More general holonomy parameters are considered in \Cref{Theory for general holonomy parameters}
.) Next, equivariant complexes from \cite[\S 4]{DS19} are reviewed. Finally, the {\it{framed}} complexes that are central to the concordance invariants defined in \Cref{section:invs from froyshov cycles} are introduced and studied.

\subsection{$\mathcal{S}$-complexes from singular instantons} \label{review-S-comp}

We begin by reviewing material from \cite{DS19, DS20}. 

\begin{defn}\label{S-comp}
An {\it $\mathcal{S}$-complex} over a ring $R$ consists of a finitely generated free graded $R$-module $\widetilde C_\ast$ and endomorphisms $\widetilde d$ and $\chi$ of $\widetilde C_\ast$ such that the following hold:
\begin{itemize}	
\item[(i)] $\widetilde d$ decreases the grading by $1$ and $\chi$ increases the grading by $1$;
\item[(ii)] $\widetilde d^2=0$, $\chi^2=0$ and $\chi\circ \widetilde d+\widetilde d \circ \chi=0$; and
\item[(iii)] $H(\widetilde C_\ast,\chi)=\ker(\chi)/\im(\chi)$
is isomorphic to $R_{(0)}$, a copy of $R$ in grading 0.
\end{itemize}
\end{defn}

\noindent The {\it{differential}} of $\widetilde C$ is the endomorphism $\widetilde d$. Our $\cS$-complexes will typically be graded by $\Z/4$, at least up until \Cref{Section: Concordance invariants from filtered special cycles}, where certain types of filtered $\mathcal{S}$-complexes are studied. For a chain complex $C_\ast$ we typically omit the grading from the subscript and simply write $C$. 

The {\it{trivial $\mathcal{S}$-complex}} is given by $\wt C=R_{(0)}$ with $\wt d=\chi=0$. Given two $\mathcal{S}$-complexes $(\wt C,\wt d, \chi)$ and $(\wt C', \wt d', \chi')$, the tensor product complex is naturally an $\mathcal{S}$-complex:
\[
	(\wt C^\otimes,\wt d^\otimes, \chi^\otimes)=(\wt C\otimes \wt C', d\otimes 1 + \epsilon \otimes d' , \chi\otimes 1 + \epsilon \otimes \chi').
\]
Here $\epsilon$ is the sign map that multiplies an element in grading $i$ by $(-1)^i$. The dual $\mathcal{S}$-complex $(\wt C^\dagger, \wt d^\dagger , \chi^\dagger)$ is defined by $\wt C_i^\dagger = \text{Hom}(\wt C_{-i},R)$, with $\wt d^\dagger(f)=-\epsilon(f)\circ \wt d$ and $ \chi^\dagger(f)=-\epsilon(f)\circ \chi$ for $f\in C^\dagger$.

A typical $\mathcal{S}$-complex has the form $\wt C_\ast = C_\ast \oplus C_{\ast-1} \oplus R_{(0)}$ with differential
\begin{equation}\label{eq:dtildedef}
\wt{d} = 
\left[
\begin{array}{ccc}
d & 0 & 0\\
v & -d  & \delta_2\\
\delta_1 & 0 &0 
\end{array}
\right]
\end{equation}
and where $\chi$ maps $C_\ast$ isomorphically to the summand $C_{\ast-1}$, and is otherwise zero. Every $\mathcal{S}$-complex is isomorphic to one of this form, and most $\mathcal{S}$-complexes that we encounter come with such a decomposition. We refer to a choice of such a decomposition as a {\it{splitting}} of the the $\mathcal{S}$-complex. The summand $R_{(0)}\subset \widetilde C$ will be referred to as the {\it{reducible summand}} of the $\mathcal{S}$-complex.

\begin{defn} Given $\mathcal{S}$-complexes $(\widetilde C, \widetilde d, \chi )$ and $(\widetilde C', \widetilde d', \chi' )$, a graded $R$-module map $\wt \lambda : \widetilde C \to \widetilde C'$ is an {\it $S$-morphism} (or simply {\it{morphism}}) if $\wt \lambda \wt d = \wt d' \wt \lambda$ and $\wt \lambda \chi = \chi' \wt \lambda$. An {\it{$\mathcal{S}$-chain homotopy}} $\wt K$ between morphisms $\wt\lambda$ and $\wt\lambda'$ is an $R$-module map $\wt K:\widetilde C\to \widetilde C'$ satisfying $\wt\lambda - \wt\lambda' = \wt K \wt d + \wt d' \wt K$ and $\chi' \wt K + \wt K \chi = 0$. 
\end{defn}

With respect to decompositions $\widetilde C=C_\ast\oplus C_{\ast-1}\oplus R_{(0)}$ and $\widetilde C'=C'_\ast\oplus C'_{\ast-1}\oplus R_{(0)}$ of our $\mathcal{S}$-complexes as described above, an $\mathcal{S}$-morphism $\wt\lambda:\wt C\to \wt C'$ has the form
\begin{align}\label{S-morphism description}
\wt{\lambda} = 
\left[
\begin{array}{ccc}
\lambda & 0 & 0\\
\mu & \lambda  & \Delta_2\\
\Delta_1 & 0 & \eta 
\end{array}
\right].
\end{align}
An $\mathcal{S}$-chain homotopy has a similar shape, but with a sign appearing in the middle entry.

\vspace{1mm}

\begin{rem}
The morphisms defined here are more general than those of \cite{DS19,DS20}, which require either $\eta=1$ or that $\eta$ is invertible. Morphisms of the latter type are {\it{strong height $0$}} morphisms, in terminology introduced below.
\end{rem}

Let $Y$ be an integer homology $3$-sphere, and $K\subset Y$ a knot. In \cite{DS19}, singular instanton gauge theory is used to construct a $\Z/4$-graded $\mathcal{S}$-complex associated to $(Y,K)$, denoted
\begin{equation}\label{eq:scomplexofknotdef}
	\widetilde C(Y,K) = C_\ast(Y,K)\oplus C_{\ast-1}(Y,K)\oplus R_{(0)}.
\end{equation}
We proceed to summarize the construction in what follows, ignoring various technicalities.

Choose an orbifold metric on $Y$ which is singular along $K$ with cone-angle $\pi$. Associated to $(Y,K)$, we have the space $\A(Y, K)$ of singular $SU(2)$ connections on $Y$ that are singular along $K$. Roughly speaking, any such singular connection is a connection on $Y\setminus K$ which is compatible with a preferred reduction near $K$ and around shrinking meridians of $K$ has asymptotic traceless holonomy. There is a Chern--Simons functional
\begin{equation}\label{eq:csmap}
\cs : \A(Y, K) \to \R 
\end{equation}
which descends to $\cs:\B(Y,K)\to \R/\Z$, where $\B(Y,K)$ is the quotient of $\A(Y,K)$ by a group of gauge transformations. The critical set of $\cs$ in $\B(Y,K)$ consists of flat singular connections with traceless holonomy around meridians $\mu$ of $K$. As such, $\text{Crit}(\cs)\subset \B(Y,K)$ can be identified with 
\[
	\left\{ \rho\in \text{Hom}(\pi_1(Y\setminus K), SU(2) ) \mid \text{tr}\rho(\mu) = 0\right\}/SU(2)
\]
where the $SU(2)$-action is by conjugation. There is a unique reducible representation class in this set which factors through $\pi_1(Y\setminus K)\to H_1(Y\setminus K;\Z)$ and sends the generator of $H_1(Y\setminus K;\Z)$ to a traceless element. The corresponding reducible connection class in $\B(Y,K)$ is typically denoted $\theta$.

The $\Z/4$-graded chain complex $(C(Y,K),d)$ is roughly the Morse--Floer complex with respect to the functional $\cs$, having discarded $\theta$. In general, perturbation data $\pi$ is chosen to ensure that transversality holds. The underlying $R$-module of $C(Y,K)$ is defined by
\[
	C(Y,K)  = \bigoplus_{\alpha\in \mathfrak{C}^{*}_{\pi}(Y,K)} R\cdot \alpha,
\]
where $\mathfrak{C}_{\pi}(Y,K) = \mathfrak{C}^{*}_{\pi}(Y,K)\sqcup \{\theta\}$ is the critical set of the $\cs$ perturbed by $\pi$. The differential $d$ counts (perturbed) instantons on $\R\times Y$ that have the prescribed singularity type along $\R\times K$. The underlying technical work here relies heavily on work of Kronheimer and Mrowka \cite{KM11}.

Now fix a basepoint $p\in K$ and an orientation of $K$. Consider the framed configuration space $\wt\B(Y,K)$ consisting of singular connections that around a family of meridians of $K$ shrinking around $p$, oriented compatibly with $K$, have limiting holonomy exactly equal (and not just conjugate) to the element
\begin{equation}\label{i-conj}
\left[\begin{array}{cc}
i &0\\
0&-i \end{array}
\right] \in SU(2)
\end{equation}
Then $\wt\B(Y,K)$ is a principal $S^1$-bundle over $\B(Y,K)$. The (perturbed) Chern--Simons functional gives an $S^1$-invariant function on $\wt\B(Y,K)$, and the $\mathcal{S}$-complex \eqref{eq:scomplexofknotdef} is a model for its Morse--Floer complex with some additional structure induced by the $S^1$-action. The summand $R_{(0)}$ is generated by the reducible $\theta$. The maps $\delta_1,\delta_2$ in \eqref{eq:dtildedef} involve counting singular instantons on $\R\times Y$ with one reducible limit, and the map $v$ involves cutting down moduli spaces by holonomy along the path $\R\times \{p\}$. For more details, see \cite{DS19}, where the following is proved.

\begin{thm}(\cite[Theorem 3.34]{DS19})
	The $\mathcal{S}$-chain homotopy type of the $\Z/4$-graded $\mathcal{S}$-complex $\widetilde C(Y,K)$ is an invariant of $(Y,K)$. In particular, it is independent of metric, perturbation and basepoint.
\end{thm}

This invariant is an enhancement of Kronheimer and Mrowka's $I^\natural(Y,K)$ from \cite{KM11a}:

\begin{thm}(\cite[Theorem 8.9]{DS19})
	The total homology of $\widetilde C(Y,K)$, over $\Z$-coefficients, is naturally isomorphic to the $\Z/4$-graded abelian group $I^\natural(Y,K)$, by an isomorphism of degree $\sigma(Y,K)\pmod{4}$.
\end{thm}

The $\mathcal{S}$-complex \eqref{eq:scomplexofknotdef} can also be defined using a local coefficient system which is modelled on a construction from \cite{KM11}. In its simplest form, this is a $\Z/4$-graded $\mathcal{S}$-complex over the ring $\Z[T^{\pm 1}]$:
\begin{equation*}\label{eq:scomplexofknotdeflocal}
	\widetilde C(Y,K;\Delta) = C_\ast(Y,K;\Delta)\oplus C_{\ast-1}(Y,K;\Delta)\oplus \Z[T^{\pm 1}]_{(0)}.
\end{equation*}
If $R$ is an algebra over $\Z[T^{\pm 1}]$, we write $\widetilde C(Y,K;\Delta_R)=\widetilde C(Y,K;\Delta)\otimes R$. 

We next turn to morphisms induced by cobordisms. It will be important in the sequel to classify the types of morphisms we obtain by how they behave with respect to the reducible summands of the $\mathcal{S}$-complexes. To this end, we introduce the following algebraic definition, which is a minor variation of a definition that appears in \cite{DS20}.

\begin{defn} \label{height i morphism}
Let $i\in \Z_{\geq 0}$. Let $\wt \lambda : \widetilde C\to \widetilde{C}'$ be a morphism as in \eqref{i-conj} and set $c_0=\eta$ and
\begin{equation}\label{eq:cjdefn}
c_{j}:= \delta_1' (v')^{j-1} \Delta_2(1) + \Delta_1 v^{j-1} \delta_2(1) + \sum_{l=0}^{j-2} \delta_1' (v')^l \mu v^{j-2-l} \delta_2(1) 
\end{equation}
for $j\in \Z_{>0}$. Then $\wt \lambda$ is a {\it{height $i$ morphism}} if it has homological degree $2i$, and satisfies $c_j=0$ for $j<i$. It is a {\it{strong height $i$ morphism}} if, in addition, the element $c_i$ is invertible. A strong height $0$ morphism is also called a {\it{local morphism}}.
\end{defn}

Now consider a cobordism of pairs $(W,S):(Y,K)\to (Y',K')$ where $K\subset Y$ and $K'\subset Y'$ are knots in integer homology $3$-spheres. Here $W$ is a cobordism from $Y$ to $Y'$ and $S$ is an orientable surface cobordism from $K$ to $K'$, embedded in $W$. Let $E$ be a $U(2)$-bundle over $W$. Then we consider connections on $W$ which are singular along $S$, with traceless holonomy around shrinking meridians of $S$. The {\it{topological energy}} of such a connection $A$ is given by
\[
	\kappa(A) = \frac{1}{8 \pi^2}\int_{W\setminus S} \text{tr}(F_{\text{ad}(A)}\wedge F_{\text{ad}(A)})
\]
where $F_{\text{ad}(A)}$ is the traceless part of the curvature. We note that the Chern--Simons functional in \eqref{eq:csmap}, for $A_0\in \mathcal{A}(Y,K)$, can be defined by choosing a connection on $[0,1]\times Y$ restricting to $A_0$ at $0$ and $A_1=\theta$ at $1$ and then setting $\cs(A_0)=2\kappa(A)$.

A cobordism map can be constructed from $(W,S,E)$, under some topological assumptions. One first fixes an orbifold metric on $W$ with con-angle $\pi$ along the surface $S$. Roughly, the map counts singular instantons on $E$, or more precisely, $E$ with cylindrical ends attached. If $b_1(W)=b^+(W)=0$, reducible singular instantons on $E$ (for any metric) are in bijection with elements of $H^2(W;\Z)$; an instanton $A_L$ compatible with a splitting $E=L\oplus L^{\ast}\otimes \det(E)$  is sent to $c_1(L)\in H^2(W;\Z)$. Its topological energy is
\begin{equation}\label{eq:kappareducible}
	\kappa(A_L) = - \left( c_1(L) + \frac{1}{4} S - \frac{1}{2}c_1(E) \right)^2.
\end{equation}
The {\it{index}}, $\text{ind}(A)$, of a singular connection $A$ on a cobordism with cylindrical ends refers to the index of the linearized ASD operator defined with weighted Sobolev spaces that give exponential decay at the ends. For $A_L$, this formula for the index is as follows:
\[{\rm ind}({A_{L}})=8\kappa(A_{L})+\frac{3}{2}(\sigma(W)+\chi(W))+\frac{1}{2}S\cdot S+ \chi(S)+\sigma(Y,K)-\sigma(Y',K')-1.\]
Any reducible which has minimal index among all reducibles is called a {\it{minimal reducible}}. The minimal topological energy over all reducibles is defined as follows, where $c:=c_1(E)$:
\[
\kappa_{\rm min}(W, S, c):=\min \left\{\kappa(A_L)   \mid c_1(L)\in H^2(W;\Z) \right\}.
\]
For the following definition, let $R$ be any algebra over the ring $\Z[T^{\pm 1}]$.

\begin{defn}\label{defn:negdefcob} For a non-negative integer $i$, the data $(W,S,c)$ where $(W,S):(Y,K)\to (Y',K')$ is a cobordism of pairs and $c\in H^2(W;\Z)$ is {\it negative definite of height $i$} if the following hold:
\begin{itemize}
\item[(i)]$b_{1}(W)=b^{+}(W)=0$;
\item[(ii)] The index of one (and hence all) minimal reducibles is $2i-1$.
\end{itemize} 
Furthermore, define the following element of $R$:
\begin{equation}\label{eta}
	\eta(W,S,c):=\sum_{A_{L} \; {\rm minimal} }(-1)^{c_{1}(L)^{2}}T^{(2c_1(L)-c)\cdot S}.
\end{equation}
If in addition to (i) and (ii), $\eta(W,S,c)$ is invertible in $R$, then $(W,S,c)$ is of {\it{strong}} height $i$ over $R$.
\end{defn}

\vspace{.1cm}

\begin{prop}(\cite[Propositions 2.24, 4.17]{DS20})
\label{cobordism map}
	Suppose $(W,S,c)$ as above is negative definite of (strong) height $i\geq 0$. Then there is an associated (strong) height $i$ morphism $\wt C(Y,K;\Delta_{R})\to \wt C(Y',K';\Delta_{R})$. The term $c_i$ as defined in \eqref{eq:cjdefn} is equal to $\eta(W,S,c)$.
\end{prop}

The construction of the above morphism depends on some auxiliary choices including a metric and perturbation, but different choices lead to homotopy equivalent morphisms. The construction also depends, a priori, on a choice of path between the basepoints of $K$ and $K'$ used to define the $\mathcal{S}$-complexes, but this is not important in the sequel.

There are also morphisms induced by cobordisms where the surface is immersed, with transverse double points. For simplicity, suppose in what follows that $W:Y\to Y'$ is a homology cobordism, and $S:K\to K'$ is a connected orientable surface with $s_\pm$ many $\pm$-double points, and genus $g(S)$. Then there is an induced morphism of $\mathcal{S}$-complexes $\wt C(Y,K;\Delta_{R})\to \wt C(Y',K';\Delta_{R})$ with height
\begin{equation}\label{eq:immersedheight}
	 i:=  -g(S)+\frac{\sigma(Y,K)}{2}-\frac{\sigma(Y',K')}{2},
\end{equation}
assuming $i\geq 0$. Furthermore, for this morphism, the term $c_i$ as given in \eqref{eq:cjdefn} is, up to a unit, equal to
\[
	c_i = (T^2-T^{-2})^{s_+}
\]
See \cite[\S 4.3]{DS20} for the construction. This type of morphism is used at several points in the sequel.

If $(W,S):(Y,K)\to (Y',K')$ is a homology concordance, there is a unique minimal reducible of index $-1$, and we obtain a local morphism $\wt C(Y,K;\Delta_{R})\to \wt C(Y',K';\Delta_{R})$ for any $R$. By reversing the direction and orientation of this cobordism, we obtain a local morphism $\wt C(Y',K';\Delta_{R})\to \wt C(Y,K;\Delta_{R})$. These observations motivate the following construction.

Let $\widetilde C$ and $\widetilde C'$ be two $\Z/4$-graded $\mathcal{S}$-complexes over some ring $R$. We say that $\widetilde C$ and $\widetilde C'$ are {\it{locally equivalent}}, and write $\widetilde C\sim \widetilde C'$, if there are local morphisms $\widetilde C\to \widetilde C'$ and $\widetilde C' \to \widetilde C$. Write
\[
	\Theta^\mathcal{S}_R  = \{ \Z/4\text{-graded }\; \mathcal{S}\text{-complexes}\}/\sim
\]
Then $\Theta^\mathcal{S}_R$ is an abelian group, where the identity element is the trivial $\mathcal{S}$-complex $R_{(0)}$, the group operation is tensor product of $\mathcal{S}$-complexes, and inverses are dual $\mathcal{S}$-complexes. For any algebra $R$ over $\Z[T^{\pm 1}]$, the assignment described above, $(Y,K)\mapsto \widetilde C(Y,K;\Delta_{R})$, induces a group homomorphism
\begin{equation}\label{eq:basicloceqhom}
	\Theta^{3,1}_\Z \to \Theta^\mathcal{S}_R
\end{equation}
where $\Theta^{3,1}_\Z$ is the homology concordance group of knots in integer homology 3-spheres.

In light of \eqref{eq:basicloceqhom}, homology concordance invariants with values in a set $R'$ can be defined, in a purely algebraic manner, by constructing some map $\Theta^\mathcal{S}_R\to R'$ (not necessarily a homomorphism), from the local equivalence group to $R'$. The simplest example, for $R$ an integral domain, is a homomorphism
\begin{equation}\label{eq:froyshovinvtaslocmap}
	h:\Theta^\mathcal{S}_R\to \Z
\end{equation}
called the {\it{Fr\o yshov invariant}}, essentially introduced in \cite{Fr02}. The Fr\o yshov invariant of an $\mathcal{S}$-complex $\wt C$ is uniquely determined by the following conditions, where $k$ is a non-negative integer:
\begin{align}
	h(\wt C) > \phantom{+} k  \quad &\Longleftrightarrow \quad \begin{array}{c}\text{ there exists }\alpha\in C_\ast\text{ satisfying } d\alpha=0, \\[\smallskipamount] \delta_1v^i(\alpha)=0 \; \text{ for } 0\leq i < k, \text{ and } \delta_1 v^{k}(\alpha)\neq 0\end{array} \label{eq:froyshovdefpos} \\[\medskipamount]
	h(\wt C) \geq  -k  \quad &\Longleftrightarrow \quad \qquad \begin{array}{c}\text{ there exist }a_0, \cdots, a_{-k} \in R \text{ satisfying }  \\[\smallskipamount]  d \alpha = \sum_{i=0}^{-k} v^i \delta_2 (a_i) \text{ and } a_{-k}\neq 0 \end{array} \label{eq:froyshovdefneg}
\end{align}
Moreover, if there is a local map $\wt C\to \wt C'$, then $h(\wt C) \leq h(\wt C')$. Finally, \eqref{eq:froyshovinvtaslocmap} is an isomorphism if $R$ is a field. Thus the Fr\o yhov invariant for a general integral domain $R$ factors as $\Theta^\mathcal{S}_R\to \Theta^\mathcal{S}_{\text{Frac}(R)}\cong \Z$.

The following result computes the Fr\o yshov invariants for most of the $\mathcal{S}$-complexes used in this paper.

\begin{thm}[{\cite[Theorem 7]{DS20}}]\label{thm:froyshovinvofknot}
Let $K$ be a null-homotopic knot in an integer homology $3$-sphere $Y$. 	If $T^4\neq 1$ in the $\Z[T^{\pm 1}]$-algebra $R$, then the Fr\o yshov invariant of $\wt C(Y,K;\Delta_R)$ is as follows:
	\begin{equation}\label{eq:froyshovinvofknot}
		h\left(\wt C(Y,K;\Delta_R)\right) = -\frac{1}{2}\sigma(Y,K) + 4 h(Y)
	\end{equation}
	where $\sigma(Y,K)$ is the knot signature and $h(Y)$ is Fr\o yshov's instanton invariant defined in \cite{Fr02}.
\end{thm}
In the sequel, we often write $h(Y,K)$ for the quantity appearing in \eqref{eq:froyshovinvofknot}.


\subsection{Equivariant complexes} \label{subsection: equivariant}

We next review several algebraic constructions that can be applied to $\cS$-complexes. We start with {\it equivariant theories} that one can associate to an $\cS$-complex $(\wt C,\wt d,\chi)$ over a ring $R$. There are two models for these equivariant theories: {\it large model} and {\it small model} \cite{DS19}. Each of these models has some advantages over the other one. In the following, $R [\![x^{-1},x]$ denotes  the ring of Laurent power series in the variable $x^{-1}$ and coefficients in $R$. This ring is an algebra over the polynomial ring $R[x]$ in the obvious way,  and the quotient algebra is denoted by $R [\![x^{-1},x] / R[x]$.

The large equivariant complexes associated to $(\wt C,\wt d,\chi)$ are given by $(\hrC,\widehat d)$, $(\crC,\widecheck d)$ and  $(\brC, \overline d)$, where 
\[
 \hrC=\wt C \otimes_{R}R[x],\hspace{1cm} \crC=\wt C \otimes_{R} R [\![x^{-1},x] / R[x],\hspace{1cm} \brC=\wt C \otimes_{R} R [\![x^{-1},x] ,
\]
are $R[x]$-modules and the corresponding differentials are given as
\[
  \widehat{d} =-\wt d  \otimes 1 +  \chi \otimes   x,\hspace{1cm}
  \widecheck{d} =\wt d  \otimes 1 - \chi \otimes   x, \hspace{1cm}
  \overline{d}=-\wt d  \otimes 1 +\chi \otimes  x.
\] 
If we equip $\hrC$, $\crC$ and $\brC$ with the $\Z/4$-grading induced by that of $\wt C$ and requiring that  $x^{2j}$ has degree $-2j$, then the differentials of these complexes decrease the grading by $1$. The subspace of $\hrC$ given by elements with $\Z/4$-grading is denoted by $\hrC_i$, and a similar convention is used for $\crC$ and $\brC$. In particular, $\hrC_i$ is a module over $R[x^2]$ and the same comment applies to $\crC$ and $\brC$.

We write $\hrH$, $\crH$ and $\brH$ for the homology groups of the large equivariant complexes. These homology groups fit into an exact triangle of the form
 \begin{equation}\label{ltriangle}
	\xymatrix{
	\crH \ar[rr]^{ {\bf j}_*}& &
	\hrH  \ar[dl]^{{\bf i}_*}\\
	& \brH \ar[ul]^{{\bf p}_*} &
	} 
\end{equation}
where the module homomorphisms are induced by the inclusion map ${\bf i}:\hrC\to \brC$, the map $\bf j$ given as 
\[
  {\bf j} (\sum_{i=-\infty}^{-1} \zeta_i x^i)=- \chi (\zeta_{-1}),
\]
and the map ${\bf p}:\hrC\to \brC$ given as the composition of the projection map and the sign map $\epsilon$. Note that ${\bf i}$ and ${\bf p}$ are $R[x]$-module homomorphisms while $\bf j$ is only an $R$-module homomorphism. However, we have 
\[ x{\bf j}-{\bf j} x=\widehat{d} K+K\widecheck{d} \]
for the homomorphism $K:\crC\to \hrC$ that sends an element $\sum_{i=-\infty}^{-1} \zeta_i x^i$ to $\zeta_{-1}$. In particular, the induced map ${\bf j}_*$ is an $R[x]$-module homomorphism.

\begin{prop}\label{prop:tosion}
	Let $R$ be an integral domain and $n:=\rank_R (C_{k-2})$ for $k \in \Z/4$. 
	Then for any element $\xi \in H_k(\crC)$, there exists a non-zero polynomial $f(x) \in R[x]$ with $\deg_x(f(x)) \leq n$ such that 
	\[f(x^2)\cdot {\bf j}_*(\xi)=0.\]
	In particular, $\text{\rm{Im }}{\bf j}_*$ is a torsion submodule of $\hrH$.
\end{prop}
\begin{proof}
	Let $\zeta=\sum_{i=-\infty}^{-1}\zeta_ix^{i}\in  \crC_k$ with $\zeta_i=(\alpha_i,\beta_i,a_i) \in \wt{C}_{k+2i}$ be a representative for $\xi$.  Then we have
	\[
	  {\bf j}(R\cdot \langle \zeta, x^2\zeta, \ldots, x^{2n}\zeta \rangle)
	  =\chi (R\cdot \langle \zeta_{-1}, \zeta_{-3}, \ldots, \zeta_{-2n-1}\rangle)
	  =R\cdot \langle \alpha_{-1}, \alpha_{-3}, \ldots, \alpha_{-2n-1} \rangle \subset C_{k-2}.
	\]
	Since $n=\rank_{R}C_{k-2}$, we have a non-trivial linear relation $\sum_{i=0}^{n}b_i\alpha_{-2i-1}=0$ with $b_i \in R$. Therefore, for the non-zero polynomial $f(x):= \sum_{i=0}^{n}b_{i}x^{i}$, we have
	\[
	  f(x^2)\cdot {\bf j}_*(\xi)
	  =[{\bf j}(f(x^2)\cdot \zeta)]
	  =[-\chi(\sum_{i=0}^{n}b_i\zeta_{-2i-1})]
	  =-[\sum_{i=0}^{n}b_i\alpha_{-2i-1}]=0. \qedhere
	\]
\end{proof}

Small equivariant complexes provide smaller models for the equivariant homology groups of the $\cS$-complex $\wt C$. Pick a splitting $\wt C=C_*\oplus C_{*-1}\oplus R_{(0)}$ and let $d$, $v$, $\delta_1$ and $\delta_2$ be the homomorphisms associated to $\wt d$ with respect to this splitting. The three versions of small equivariant complexes are denoted by $(\fhrC,\widehat \fd)$, $(\fcrC,\widecheck \fd)$ and  $(\fbrC, \overline \fd)$ where 
\[
 \fhrC={C}_{*-1} \oplus R[x],\hspace{1cm} \fcrC={C}_* \oplus R[\![x^{-1}, x]/R[x],\hspace{1cm} \fbrC=R [\![x^{-1},x] ,
\]
with differentials defined as follows:
\[
  \widehat{\fd}(\al, \sum_{i=0}^{N} a_i x^i ) =( {d}\al- \sum_{i=0}^{N} v^i \delta_2 (a_i), 0 ),\hspace{1cm}
  \widecheck{\fd}({\al}, \sum_{i=-\infty}^{-1} a_i x^i )=({d}\al, \sum_{i= -\infty}^{-1} \delta_1 v^{-i-1} (\al) x^i  ), \hspace{1cm}
  \overline{\fd}=0.
\] 
The $R[x]$-module structures on $\fhrC$ and $\fcrC$ are respectively given by the chain maps
\[
  x \cdot (\al, \sum_{i=0}^{N} a_i x^i )= (v\al, \delta_1(\al) + \sum_{i=0}^{N} a_i x^{i+1}),\hspace{1cm}
  x \cdot (\al, \sum_{i=-\infty}^{-1} a_i x^i )= (v\al + \delta_2(a_{-1}), \sum_{i=-\infty}^{-2} a_i x^{i+1}),
\]
whereas the module structure on $\fbrC$ is the obvious one. We equip each of these complexes with a $\Z/4$-grading using the gradings of the summands where again the degree of $x^{2j}$ is equal to $-2j$.

The analogue of the exact triangle in \eqref{ltriangle} for the homology groups $\hH$, $\cH$ and $\bH$ is given by
\begin{equation}\label{striangle}
	\xymatrix{
	\cH \ar[rr]^{\mathfrak{j}_*}& &
	 \hH \ar[dl]^{\mathfrak{i}_*}\\
	& \bH \ar[ul]^{\mathfrak{p}_*} &
	}
\end{equation}
where the maps are defined at the chain level as follows:
\begin{align*}
  	\mathfrak{i} (  \alpha, \sum_{i=0}^N a_i x^i) := \sum_{i= -\infty}^{-1} &\delta_1v^{-i-1} (\alpha ) x^i + \sum_{i=0}^N a_i x^i,\hspace{1cm}\mathfrak{j} (  \alpha, \sum_{i=-\infty}^{-1} a_i x^i) := (-\alpha, 0),\\
	&\mathfrak{p} (  \sum_{i=-\infty}^N a_i x^i) :=( \sum_{i= 0}^{N} v^{i}\delta_2 (a_i )  , \sum_{i=-\infty}^{-1} \alpha_i x^i).
\end{align*}

The following relation between large and small complexes is proved in \cite[Lemma 4.11]{DS19}.

\begin{prop} \label{large and small}
	There are $R$-module homomorphisms  
	\[
	  \hPhi : \hrC \to \fhrC,\hspace{1cm}\cPhi : \crC \to \fcrC,\hspace{1cm}\bPhi : \brC \to \fbrC,
	\]
	that are chain maps and commute with the action of $x$ up to chain homotopy. Moreover, these homomorphisms are chain homotopy equivalences with the chain homotopy inverses 
	\[
	  \hPsi : \fhrC \to \hrC,\hspace{1cm}\cPsi : \fcrC \to \crC,\hspace{1cm}\bPsi : \fbrC \to \brC.
	\]	
	In particular, $\hPhi_*:\hrH\to \hH$ is an isomorphism of $R[x]$-modules with the inverse $\hPsi_*$, and similar claims hold for $(\cPhi_*, \cPsi_*)$ and  $(\bPhi_*,\bPsi_*)$. Moreover, these maps define a homomorphism 
	between \eqref{ltriangle} and \eqref{striangle} commuting with the homomorphisms in the exact triangle.
\end{prop}

\begin{proof}[Sketch of the proof]
	The splitting $\wt C=C_*\oplus C_{*-1}\oplus R_{(0)}$ induces a splitting of large equivariant complexes.
	With respect to these splittings, the maps from the large complexes to the 
	small complexes are as follows:
	\begin{align*}
		\hPhi(\sum_{i=0}^N \alpha_i x^i,  \sum_{i=0}^N \beta_i x^i,  \sum_{i=0}^N a_i x^i )&:=( \sum_{i=0}^N v^i (\beta_i), \sum_{i=0}^N a_i x^i + \sum_{i=1}^{N} \sum_{j=0}^{i-1} \delta_1 v^j (\beta_i ) x^{i-j-1} ),\\
		\cPhi(\sum_{i=-\infty}^{-1} \alpha_i x^i,  \sum_{i=-\infty}^{-1} \beta_i x^i,  \sum_{i=-\infty}^{-1} a_i x^i )&:=(\alpha_{-1}, \sum_{i=-\infty}^{-1} a_ix^i + \sum_{i=-\infty}^{-1} \sum_{j=0}^\infty \delta_1 v^j (\beta_i ) x^{i-j-1} ),\\
		\bPhi(\sum_{i=-\infty}^N \alpha_i x^i,  \sum_{i=-\infty}^N \beta_i x^i,  \sum_{i=-\infty}^N a_i x^i )&:= \sum_{i=-\infty}^N a_i x^i  + \sum_{i=-\infty}^N \sum_{j=0}^\infty \delta_1 v^j (\beta_i) x^{i-j-1},
	\end{align*}
	The maps in the reverse direction are given by:
	\begin{align*}
		\hPsi( \alpha, \sum_{i=0}^N a_i x^i)&:=( \sum_{i=1}^N \sum_{j=0}^{i-1} v^j \delta_2 (a_i ) x^{i-j-1} , \alpha, \sum_{i=0}^N a_i x^i ),\\ 
		\cPsi( \alpha, \sum_{i=-\infty}^{-1}  a_i x^i)&:=( \sum_{i= -\infty}^{-1} v^{-i-1} (\alpha ) x^i + \sum_{i=-\infty}^{-1} \sum_{j=0}^\infty v^j \delta_2 ( a_i) x^{i-j-1}  , 0, \sum_{i=-\infty}^{-1} a_i x^i ),\\
		\bPsi(\sum_{i=-\infty}^{N}  a_i x^i)&:=(  \sum_{i=-\infty}^{N} \sum_{j=0}^\infty v^j \delta_2 ( a_i) x^{i-j-1}  , 0, \sum_{i=-\infty}^{N} a_i x^i ). \qedhere
	\end{align*}
\end{proof}

\begin{rem}\label{phi-psi-further-props}
	The maps in \Cref{large and small} satisfy a few additional useful properties.
	The equivariant complexes $\fbrC$ and $\brC$ are $R[\![x^{-1},x]$-modules in the obvious way, and the 
	maps $\bPhi$ and $\bPsi$ are $R[\![x^{-1},x]$-module homomorphisms.
	The maps $\hPhi$, $\cPhi$ and $\bPhi$ are in fact left inverses to $\hPsi$, $\cPsi$ and $\bPsi$:
	\[
	  \hPhi\circ \hPsi=1_{\fhrC},\hspace{1cm}\cPhi\circ \cPsi=1_{\fcrC},\hspace{1cm}\bPhi\circ \bPsi=1_{\fbrC}.
	\]
	We also have the relations
	\[
	  \bPhi \circ {\bf i}=\mathfrak{i}\circ \hPhi,\hspace{1cm} \hPhi \circ {\bf j}=\mathfrak{j}\circ \cPhi,\hspace{1cm} \hPsi \circ \mathfrak{j} ={\bf j}\circ \cPsi,\hspace{1cm} \cPsi \circ \mathfrak{p} ={\bf p}\circ \bPsi.
	\]
\end{rem}

\begin{cor} \label{hat and check}
	The equivariant homology groups $\hH $ and $\cH $ fit into the exact sequences
	\begin{equation} \label{hrH}
	\xymatrixcolsep{.5cm}
		\xymatrix{
		0 \ar[r] &  {\text{\rm Im }} {\fj}_* \ar@{^{(}->}[rr] &&\hH  \ar^{\ffi_*\hspace{.4cm}}[rr] &&{\text{\rm Im }} \ffi_* \ar[r] &0,
		}
	\end{equation}
	\begin{equation}\label{crH}
	\xymatrixcolsep{.5cm}
		\xymatrix{
		0 \ar[r] & R[\![ x^{-1},x ]/{\text{\rm Im }}\ffi_*   \ar^{\hspace{1cm}\fp_*}[rr] &&\cH \ar^{\fj_*\hspace{0.2cm}}[rr] &&{\text{\rm Im }}\fj_* \ar[r] &0,
		}
	\end{equation}
	where 
	${\text{\rm Im }}{\fj}_*$ is a torsion $R[x]$-submodule 
	of $\hH$. 
	If $R$ is an integral domain, then ${\text{\rm Im }}{\fj}_*$ agrees with the  torsion submodule of $\hH $. For a field $R$, the  $R[x]$-module ${\text{\rm Im }}\ffi_*$ is isomorphic to $R[x]$. 
\end{cor}
\begin{proof}
	The short exact sequences \eqref{hrH} and \eqref{crH} immediately follow from \eqref{striangle}.
	\Cref{prop:tosion} implies that ${\text{\rm Im }}{\fj}_*$ is a torsion $R[x]$-submodule of $\hH$. In the other direction and if $R$ is an integral domain, then a torsion element of $\hH$ is in the kernel of the map ${\ffi}_*$
	because $R[\![x^{-1},x]$ has trivial torsion. 
	
	Now, let $R$ be a field. 
	The submodule $\im \ffi_*$ contains elements of the form
	$\sum_{i=0}^N a_i x^i$ for which
	\[a_0 \delta_2 (1)+a_1v \delta_2 (1)+\dots+a_Nv^i \delta_2 (a_N)\]
	is trivial. In particular, $\im \ffi_*$ is not trivial. Since the $\cS$-complex $\wt C$ is finitely generated over $R$, the submodule $\im \ffi_*$ of $R[\![x^{-1},x]$ has an element $Q(x)$ 
	with minimal $x$-degree. It is straightforward to check that 
	any element of $\im \ffi_*$ can be written uniquely as a multiple of $Q(x)$ by some element of $R[x]$.
\end{proof}

Any morphism of $\cS$-complexes $\wt \lambda: \wt{C} \to \wt{C}'$ induces a morphism of equivariant theories. In the case of large equivalent complexes, we have the $R[x]$-module homomorphisms 
\[
  \lhl:\hrC\to \hrC',\hspace{1cm}\lcl:\crC\to \crC',\hspace{1cm}\lol:\brC\to \brC',
\]
each given by $\wt \lambda\otimes 1$. These are compatible with \eqref{ltriangle} in that we have the following commutative diagram:
\begin{align} \label{large morphism}
	\begin{CD}
		\cdots @>{\lj}>> \lhc  @>{\li}>> 
		\loc @>{\lk}>> \lcc @>{\lj}>> \cdots \\
	@. @V{\lhl}VV @V{\lol}VV @V{\lcl}VV\\
	\cdots @>{\lj'}>> \lhc'  @>{\li'}>> 
	\loc @>{\lk'}>> \lcc' @>{\lj'}>> \cdots 
	\end{CD}
\end{align}
We also remark that chain homotopic morphisms of $\cS$-complexes induce chain homotopic homorphisms of large equivariant complexes. 

\Cref{large and small} can be used to obtain homomorphisms of small equivalent complexes
\begin{align*}
 \shl:\fhrC\to \fhrC', \phantom{A} && \scl:\fcrC\to \fcrC', \phantom{A}&& \sol:\fbrC\to \fbrC', \phantom{A}\\
\shl := \hPhi' \circ \lhl \circ \hPsi, &&
\scl := \cPhi' \circ \lcl \circ \cPsi, &&
\sol := \oPhi' \circ \lol \circ \oPsi. 
\end{align*}
We may describe the induced map $\overline \lambda:  \fbrC \to  \fbrC'$ more explicitly as an endomorphism of $R[\![x^{-1},x]$ given as multiplication by a Laurent power series of the form
\begin{equation}\label{bar-hor}
  c_0+c_{1}x^{-1}+c_2x^{-2}+c_3x^{-3}+\dots,
\end{equation}
where $c_j$ is introduced in Definition \ref{height i morphism}. (In the case $c_0=1$, this computation is done in \cite[\S 4.2]{DS19}, and the verification in the more general case is similar.) If $\wt \lambda$ has height $i$, then \eqref{bar-hor} has degree at most $-i$, and the equality holds if $\wt \lambda$ is a strong height $i$ morphism. The commutative diagram in \eqref{large morphism} together with \Cref{large and small} gives rise to the following diagram for the small equivariant theories that is commutative up to chain homotopy:
\begin{align} \label{small morphism}
	\begin{CD}
		\cdots @>{\sj}>> \shc_*  @>{\si}>> 
		\soc_* @>{\sk}>> \scc_* @>{\sj}>> \cdots \\
		@. @V{\shl}VV @V{\sol}VV @V{\scl}VV\\
		\cdots @>{\sj}>> \shc'_*  @>{\si}>> 
		\soc_* @>{\sk}>> \scc'_* @>{\sj}>> \cdots 
	\end{CD}
\end{align}

Next, let $(\lhc^{\otimes},\loc\vphantom{}^{\otimes},\lcc^{\otimes})$ be the large equivariant complexes of the $\cS$-complex $(\wt{C}^{\otimes}, \wt{d}^{\otimes}, \chi^{\otimes})$ obtained by taking the tensor product of $\cS$-complexes $(\wt{C}, \wt{d}, \chi)$ and $(\wt{C}', \wt{d}', \chi')$. The isomorphisms
\[
  R[x]\otimes_{R[x]} R[x] \cong R[x],\hspace{1cm} R[\![x^{-1},x]\otimes_{R[\![x^{-1},x]} R[\![x^{-1},x] \cong R[\![x^{-1},x]
\]
induce the following chain maps, which are module isomorphisms: 
\[
  \wh{T} \colon \lhc \otimes_{R[x]} \lhc' \to \lhc^{\otimes},\hspace{1cm}
  \overline{T} \colon \loc \otimes_{R[\![x^{-1},x]} \loc' \to \loc^{\otimes}.
\]
The following claim is essentially proved in the proof of \cite[Lemma 4.27]{DS19}.
\begin{lem} \label{tensor product lem}
	The composition 
	\[
	  \oPhi^{\otimes} \circ \overline{T}\circ \left(\oPsi \otimes_{R[\![x^{-1},x]} \oPsi'\right): \soc \otimes_{R[\![x^{-1},x]} \soc' 
	  \to \soc^\otimes 
	\]
	is equal to the multiplication map $R[\![x^{-1},x]\otimes_{R[\![x^{-1},x]} R[\![x^{-1},x]\to R[\![x^{-1},x]$.
\end{lem}

\begin{proof}
	Noting that the maps $\oPhi^{\otimes}$, $\oPsi \otimes_{R[[x^{-1},x]} \oPsi'$ and $\overline{T}$ are all
	$R[\![x^{-1},x]$-module homomorphisms, it suffices to verify the claim for $1\otimes 1$:
	\begin{align*}
		\oPhi^{\otimes} \circ \overline{T}\circ [\oPsi \otimes_{R[\![x^{-1},x]} \oPsi'] (1\otimes 1)
		&= \oPhi^{\otimes} \circ \overline{T}(
		\left(  \sum_{j=0}^\infty v^j \delta_2 (1) x^{-j-1}  , 0, 1\right)\otimes
		\left(  \sum_{j=0}^\infty v^j \delta_2 (1) x^{-j-1}  , 0, 1 \right))\\[2mm]
		&=\oPhi^{\otimes}(A,0,1) = 1
	\end{align*}
	for some chain $A \in C^{\otimes}_* \otimes R[\![x^{-1},x]$. This completes the proof.
\end{proof}

The following proposition provides a naturality result for the maps $\widehat{T}$ and $\overline{T}$.
\begin{prop} \label{tensor product}
	The following diagram of $R[x]$-modules
	\[
	  \begin{tikzcd} 
	  H(\lhc) \otimes_{R[x]} H(\lhc')  
          \ar[r, "\li_* \otimes \li'_*"] 
	  \ar[d]
	  &
	  H(\loc) \otimes_{R[\![x^{-1},x]} H(\loc')
        \ar[rrr,  
        "\oPhi_* \otimes_{R[\![x^{-1},x]} \oPhi'_*"]
        \ar[d]
        & & &
        \soc \otimes_{R[\![x^{-1},x]} \soc'
        \ar[d]
        \\
        H(\lhc^{\otimes})
        \ar[r, "\li^{\otimes}_*"]
        &
        H(\loc^{\otimes})
        \ar[rrr, "\oPhi^{\otimes}_*"]
        & & &
        \soc^{\otimes}
	\end{tikzcd}
	\]
	is commutative, where the vertical maps are respectively induced by $\wh{T}$, $\overline{T}$ and multiplication.
\end{prop}

\begin{proof}
	The commutativity of the left square follows readily from the following relation:
	\begin{equation}\label{i-T-commute}
	  \overline T\circ \left(\li\otimes \li'\right)= \li^\otimes \circ \wh T.
	\end{equation}
	The commutativity of the right square follows from \Cref{tensor product lem} and the fact that 
	$\oPsi_*$ and $\oPsi_*'$ are the inverses of $\oPhi_*$ and $\oPhi_*'$, respectively.
\end{proof}

Let $S$ be an algebra over $R$.  Define the $\cS$-complex $\widetilde C^S$ over $S$ as the base change $\widetilde C\otimes_R S$. Associated to $\widetilde C^S$ are equivariant complexes $(\lhc^{S},\loc\vphantom{\lhc}^{S},\lcc^{S})$ and $(\shc^{S},\soc\vphantom{\shc}^{S},\scc^{S})$. We have natural isomorphisms 
\[
  \lhc^{S}\cong \lhc \otimes_R S,
\]
and similar isomorphisms hold for the other versions of equivariant theories. The chain maps in the exact triangles \eqref{ltriangle}, \eqref{striangle} and the homomorphisms in the proof of \Cref{large and small} with these isomorphisms. A morphism of $\cS$-complexes  $\wt \lambda: \wt C \to \wt C'$ induces a morphism $\wt \lambda^S: \wt C^S \to \wt C'^S$ of the base changes. This in turn induces morphisms of equivariant theories. For instance, we have $\lhl\vphantom{}^S:\hrC^S\to \hrC'^S$ that is equal to $\wt \lambda^S\otimes_{S} 1_{R[x]}=\lhl \otimes_{R} 1_{S}$.

The constructions of equivariant theories from above can be applied to the $\Z/4$-graded $\cS$-complex $\wt C(Y,K;\Delta_R)$ of a knot $K$ in an integer homology sphere $Y$, where $R$ is any algebra over $\Z[T^{\pm 1}]$. In particular, after using the package of either large or small equivariant complexes and then passing to homology, we obtain equivariant instanton knot homology groups that fit into an exact triangle:
 \begin{equation}\label{ktriangle}
	\xymatrix{
	\crI(Y,K;\Delta_R) \ar[rr]^{ j_*}& &
	\hrI(Y,K;\Delta_R)  \ar[dl]^{i_*}\\
	& \brI(Y,K;\Delta_R) \ar[ul]^{p_*} &
	} 
\end{equation}
Moreover, we have $ \brI(Y,K) \cong R[\![x^{-1},x]$. The exact triangle in \eqref{ktriangle} is functorial with respect to negative definite cobordisms of pairs. That is to say, if the cobordism of pairs $(W,S):(Y,K)\to (Y',K')$ together with $c\in H^2(W;\Z)$ is negative definite of height $i\geq 0$, then there is a cobordism map
\begin{align*}
  \widehat \lambda_{(W,S,c)}: \hrI(Y,K;\Delta_R) \to \hrI(Y',K';\Delta_R),
\end{align*}
and similarly $\widecheck \lambda_{(W,S,c)}$, $\overline \lambda_{(W,S,c)}$. These maps commute with the maps in \eqref{ktriangle}. Furthermore, after identification of $ \brI(Y,K;\Delta_R) $ and $ \brI(Y',K';\Delta_R) $ with $R[\![x^{-1},x]$ using small model, the map $\overline \lambda_{(W,S,c)}$ is multiplication by an expression of the form \eqref{bar-hor}, where $c_j=0$ for $j<i$, and $c_i$ is given in \eqref{eta}.

\subsection{Deformed complexes and framed singular instanton homology}\label{subsec:deformedcomplexes}

The equivariant complex $\hrC$ of an $\cS$-complex $\widetilde C$ can be regarded as a deformation of the chain complex $(\widetilde C,-\widetilde d)$ after applying the base change of $R[x]$ to $R$, by evaluation of $x$ at $1$. This definition of equivariant complexes from this viewpoint can be generalized in the following way. Suppose $\varphi:R[x]\to {\bf S}$ is a ring homomorphism. In particular, ${\bf S}$ can be regarded as an algebra over $R$. We define the {\it $\varphi$-deformed complex} associated to $\wt C$ as follows:
\[
  \wt C^\varphi:=\wt C\otimes_{R} {\bf S},\hspace{1cm} \wt d^{\varphi}:=\widetilde d\otimes 1+\chi\otimes \varphi (x).
\]
In the case that $\widetilde C$ is the $\cS$-complex of $(Y,K)$, the homology of the $\varphi$-deformed complex $(\wt C^\varphi(Y,K),\wt d^\varphi)$ is denoted by ${\wt I}^\varphi(Y,K)$. Previously several knot invariants under the general name of {\it singular instanton Floer homology} were constructed by Kronheimer and Mrowka \cite{KM11a,KM13,km-barnatan,KM19b}. It is shown in \cite[\S 8]{DS19} that any of these knot invariants can be characterized as ${\wt I}^\varphi(Y,K)$ for an appropriate choice of $\varphi$.

Suppose $R$ is an algebra over $\Q[T^{\pm 1}]$ and $\Lambda := T-T^{-1} \in R$. For an $\mathcal{S}$-complex $\wt{C}$ over $R$, we can associate the {\it unreduced framed complex} $C^\sharp$ defined as follows:
\[
 C^\sharp :=  \wt{C}_{*}  \oplus \wt{C}_{*+2}, \hspace{.5cm} 
 d^\sharp := \left[
\begin{array}{cc}
\wt{d}  & 2\Lambda^2\chi  \\
2 \chi &\wt{d}  \\
\end{array} \right].
 \]
We write $(C^\sharp_*(Y,K;\Delta_R), d^\sharp )$ for the unreduced framed complex of the $\cS$-complex $\wt{C}(Y,K;\Delta_R)$. The motivation to consider unreduced framed complexes comes from the following result.

\begin{thm}(\cite[Theorem 8.20]{DS19})\label{thm:recoverisharp}
	The total homology of $(C^\sharp_*(Y,K;\Delta_R), d^\sharp )$, defined with coefficients $R=\Q[T^{\pm 1}]$, is naturally isomorphic to Kronheimer--Mrokwa's singular instanton homology $I^\sharp(Y,K)$ of \cite{KM13}, 
	by an isomorphism of degree $\sigma(Y,K)+1\pmod{4}$.
\end{thm}

When comparing the notations in this paper and  \cite{KM13}, the reader should note that $u$ and $\lambda$ in \cite{KM13} are equal to $T^2$ and $\Lambda$ in our notation. The degree of the isomorphism in the theorem comes from comparing the grading conventions in \cite{DS19} and \cite{KM13}. We may also let $R=\Q$ be the algebra over $\Q[T^{\pm 1}]$ where we set $T=1$. Then it is shown in \cite[Theorem 8.13]{DS19} that the homology of $(C^\sharp_*(Y,K;\Delta_R), d^\sharp )$ agrees with the flavor of $I^\sharp(Y,K)$ defined in \cite{KM11a}.

The unreduced framed complex can be identified as a deformed complex. Suppose $R$ is an algebra over $\Q[T^{\pm 1}]$, $\Lambda := T-T^{-1} \in R$ and ${\bf S}$ is the quotient of $R[x]$ by the ideal generated by $x^2-4\Lambda^2$. Let $\varphi:R[x]\to {\bf S}$ be the the quotient map. Then $C^\sharp$ is the $\varphi$-deformed complex associated to $\wt C$. More concretely, we identify $(\zeta_1, \zeta_2)\in C^\sharp$ with $\zeta_1+\tfrac{1}{2}x\zeta_2\in \wt{C}^{\varphi}$. 
We may also identify $(C^\sharp,d^\sharp)$ with $(\lhc\otimes_{R[x]}{\bf S},-\widehat d\otimes 1_{\bf S})$ using the isomorphism that sends $(\zeta_1, \zeta_2)\in C^\sharp$ to $\zeta_1-\tfrac{1}{2}x\zeta_2$. Composing the inverse of this isomorphism and the sign map $\epsilon$ gives an isomorphism ${\bf f}^\sharp:(\lhc\otimes_{R[x]}{\bf S},\widehat d\otimes 1_{\bf S})\to (C^\sharp,d^\sharp)$ of chain complexes. If we define the $R[x]$-module structure on $(C^\sharp,d^\sharp)$ by 
\begin{equation}\label{sharp action}
x \cdot (\zeta_1,\zeta_2) = (-2\Lambda^2 \zeta_2, -2 \zeta_1),
\end{equation}
then ${\bf f}^\sharp$ is an $R[x]$-module homomorphism. We have the exact sequence of $R[x]$-modules 
\begin{equation} \label{hat and sharp}
	\begin{tikzcd}
	\cdots 
	\ar[r]
	&
	H(\lhc) 
\ar[rr,"x^2-4\Lambda^2"]
& &
H(\lhc)
\ar[r,"{\bf q}^\sharp_*"]
&
H(C^\sharp) 
\ar[r]
& 
H(\lhc)
\ar[rr,"x^2-4\Lambda^2"]
& &
\cdots
\end{tikzcd}
\end{equation}
where ${\bf q}^\sharp_\ast$ is induced by composing the quotient map with the chain isomorphism $\bf f^\sharp$.\\

\begin{rem} \label{rem:global1}
	For any knot $K$ in $S^3$, the singular instanton Floer homology $I^\sharp(S^3,K)$ has rank $1$ over the ring
	$\Q[T^{\pm 1}]$ in degrees $0$ and $2$ and rank $0$ in degrees $1$ and $3$ \cite{KM13}.
	This property is special to the unreduced framed homology of $\cS$-complexes that are given by classical knots, and it does not hold for an arbitrary $\cS$-complex.
\end{rem}   

A morphism $\wt{\lambda} \colon \wt{C} \to \wt{C}'$ with homological degree $2i$ induces a chain map $\lambda^\sharp:C^\sharp\to C'^\sharp$
\[\lambda^\sharp := \left[
\begin{array}{cc}
\wt{\lambda}  &0 \\
0 &\wt{\lambda}  \\
\end{array} \right],
\]
that has the same homological degree as $2i$.
\begin{prop} \label{sharp height i}
	For any morphism of $\mathcal{S}$-complexes $\wt{\lambda} \colon \wt{C} \to \wt{C}'$ as above, we have the following commutative diagram of $R[x]$-modules:
	\[
	\begin{tikzcd}
		H_*(\lhc)
		\ar[rr, "{\bf q}^\sharp"] 
		\ar[d, "\lhl_*"]
		& &
		H_*(C^\sharp)
		\ar[d, "\lambda^\sharp_*"]
		\\
		H_{*+2i}(\lhc')
		\ar[rr, "{\bf q}'^\sharp"] 
		& &
		H_{*+2i}(C'^\sharp)
\end{tikzcd}
\]
\end{prop}

Next, we discuss unreduced framed complexes for tensor products. Let $(\wt{C}^{\otimes}, \wt{d}^{\otimes}, \chi^{\otimes})$ be the tensor product of $\mathcal{S}$-complexes $(\wt{C}, \wt{d}, \chi)$ and $(\wt{C}', \wt{d}', \chi')$. Let $(\lhc^{\otimes},\loc\vphantom{\lhc}^{\otimes},\lcc^{\otimes})$ and $C^{\otimes\sharp}$ be the large equivariant complexes and the unreduced framed complex of $\wt{C}^{\otimes}$, respectively. Then we have canonical chain isomorphisms
\begin{align*}
	C^\sharp \otimes_{R[x]} C'^\sharp
	&\cong 
	\left(\lhc 
	\otimes_{R[x]}
	{\bf S}
	\right)
	\otimes_{R[x]}
	\left(\lhc' 
	\otimes_{R[x]}
	{\bf S}
	\right)
	\\[2mm]
	&\cong
	\lhc^{\otimes}
	\otimes_{R[x]}
	{\bf S}
	\cong C^{\otimes\sharp}
\end{align*}
We denote by $T^\sharp \colon C^\sharp \otimes_{R[x]} C'^\sharp\to C^{\otimes\sharp}$ the resulting isomorphism. By the definition of $T^\sharp$
and arguments in \Cref{subsection: equivariant}, the following proposition holds.
\begin{prop} \label{sharp tensor}
We have the following commutative diagram of $R[x]$-modules:
\[
	\begin{tikzcd}
		H(\lhc)
		\otimes_{R[x]}
		H(\lhc')
		\ar[d, "{\bf q}^\sharp_* \otimes_{R[x]} {\bf q}'^\sharp_*"] 
		\ar[r]
		&
		H(\lhc
		\otimes_{R[x]}
		\lhc')
		\ar[d, "({\bf q}^\sharp \otimes_{R[x]} {\bf q}'^\sharp)_*"] 
		\ar[rr, "\wh{T}_*"]
		&&
		H(\lhc^{\otimes})
		\ar[d, "{\bf q}^{\otimes\sharp}_*"] \\
		H(C^\sharp)\otimes_{R[x]}H(C'^\sharp)\ar[r] & H(C^\sharp\otimes_{R[x]}C'^\sharp)\ar[rr, "T^\sharp_*"]
		& &
		H(C^{\otimes \sharp})
	\end{tikzcd}
\]
\end{prop}

Finally, we discuss the duality of $C^\sharp$. We begin with a few remarks about dual complexes. Let $(V_*, d)$ be a $\Z/4$-graded chain complex freely generated over $R$, and $(V^\dagger_*,d^\dagger)$ be its dual complex. This is defined as follows: for any $i$, we have $V^\dagger_{i}=\text{Hom}(V_{-i},R)$, and for any $f\in V^\dagger$, we have $d^\dagger(f):=-\epsilon(f)\circ d$ where $\epsilon$ is the sign map. (Note that the dual of an $\cS$-complex from \Cref{review-S-comp} is defined using a similar convention.) We have a natural pairing $\langle\cdot ,\cdot \rangle: V^\dagger_{-i} \otimes V_{i} \to R$ given as
\begin{equation}\label{pairing-hom-co}
	\langle f,v \rangle \mapsto f(v).
\end{equation}
The sign convention above is chosen such that if we equip $V^\dagger_{*} \otimes V_{*}$ with the differential $d^\dagger \otimes 1+\epsilon \otimes d$, then \eqref{pairing-hom-co} determines a chain map. In particular, if $R$ is an integral domain, we obtain the non-degenerate {\it evaluation} bilinear form $\langle\cdot ,\cdot \rangle:\big(H(V^\dagger_*)/\Tor_R\big) \times  \big(H(V_*)/\Tor_R\big) \to R$.

Let $(C^{\dagger\sharp}, d^{\dagger\sharp})$ denote the unreduced framed complex of $\wt{C}^{\dagger}$. Note that
\begin{equation}\label{framed-dual}
	C^{\dagger\sharp}_i = \wt{C}^{\dagger}_i \oplus \wt{C}^{\dagger}_{i+2}= \Hom_R(\wt{C}_{-i},R)\oplus \Hom_R(\wt{C}_{-i-2},R).
\end{equation}
Similarly, denote by $(C^{\sharp \dagger}, d^{\sharp\dagger})$ the dual complex of $C^\sharp$. Then
\begin{equation}\label{dual-framed}
	C^{\sharp\dagger}_i = \Hom_R(C^{\sharp}_{-i},R)= 
	\Hom_R(\wt{C}_{-i},R)\oplus \Hom_R(\wt{C}_{-i+2},R).
\end{equation}
\begin{prop}\label{duality for sharp}
	The map $C^{\dagger\sharp}_* \to C^{\sharp\dagger}_{*+2}$, defined with respect to the splittings in \eqref{framed-dual} and
	\eqref{dual-framed} by
	\[
	  (\varphi,\psi) \mapsto (\psi,\varphi),
	\]
	is a chain isomorphism. In particular, for any integral domain $R$, we have the isomorphism
	\[
	  H_*(C^{\dagger\sharp})/\Tor_R \cong \Hom_R(H_{-*-2}(C^\sharp),R).
	\]
\end{prop}

The {\it reduced framed complexes} of the $\mathcal{S}$-complex $\wt C$ are given by the $\Z/2$-graded deformed complexes 
\[(\wt C^{\pm},\wt d^{\pm}):=(\wt{C}, \wt{d} \pm 2\Lambda \chi).\] 
Since $\wt C$ is $\Z/4$-graded, and $\wt d$ and $\chi$ have respectively degrees $-1$ and $1$, the homology groups $H(\wt C^{+},\wt d^{+})$ and $H(\wt C^{-},\wt d^{-})$ are isomorphic. The reduced theories are related to the unreduced framed complex $C^\sharp$ in the following way. First, define $R$-module homomorphisms
\begin{align} 
	\iota_{\pm} &\colon \wt{C}^\pm \to C^{\sharp}_*, \quad \iota_{\pm}:= 
	\left[
	\begin{array}{c}
		\Lambda  \\
		\pm 1   \\
	\end{array}
	\right],\label{iota}
	\\
	\pi_{\pm} &\colon C^\sharp \to \wt{C}^\pm, \quad \pi_{\pm} :=
	\left[
	\begin{array}{cc}
		1  & \pm \Lambda   \\
	\end{array}
	\right].
	\label{pi}
\end{align} 
These are chain maps, and we have a short exact sequence
\begin{equation}\label{sharp and twisted}
		0\longrightarrow \wt{C}^-\overset{\iota_-}{\longrightarrow}C^\sharp\overset{\pi_+}{\longrightarrow} \wt{C}^+\longrightarrow0\\
\end{equation}
of $\Z/2$-graded chain complexes over $R$. There is a similar exact sequence with $+$ and $-$ interchanged. Note that the maps $\iota_{\pm}$ and $\pi_{\pm}$ are $R[x]$-module homomorphisms, where the action of $x$ on $\wt C^{\pm}$ is set to be multiplication by $\mp 2\Lambda$. Futhermore, $\pi_+ \circ \iota_+ = 2\Lambda$, and hence the exact sequence \eqref{sharp and twisted} splits if the element $\Lambda$ is invertible in $R$.

Similar to $C^\sharp$, we may identify $\wt C^{\pm}$ with appropriate base changes of the equivariant complex $(\lhc,\widehat {d})$. Suppose ${\bf S}^\pm$ is the ring $R$ regarded as an algebra over $R[x]$ where the action of $x$ is given by multiplication by $\mp 2\Lambda$. Then $(\wt C^{\pm},\wt d^{\pm})$ is naturally isomorphic to $(\lhc\otimes_{R[x]}{\bf S^\pm},-\widehat d\otimes 1_{\bf S^\pm })$. Composing the inverse of this isomorphism and the sign map determines a chain isomorphism ${\bf f}^\pm:(\lhc\otimes_{R[x]}{\bf S^\pm },\widehat d\otimes 1_{\bf S^\pm })\to (\wt C^{\pm},\wt d^{\pm})$. From this description, we obtain the following long exact sequence of $R[x]$-modules:
\begin{equation} \label{hat and twisted}
	\begin{tikzcd}
	\cdots \ar[r] &H(\lhc)  \ar[rr,"x \pm 2\Lambda"]
	& &
	H(\lhc) \ar[r,"\wt{\bf q}^{\pm}_*"] & H(\wt{C}_*^\pm) \ar[r]& H(\lhc)\ar[rr,"x \pm 2\Lambda"]
	& &
	\cdots
	\end{tikzcd}
\end{equation}
Here $\wt {\bf q}^\pm_\ast$ is induced by composing the quotient map with the chain isomorphism ${\bf f}^\pm$.\\

\begin{rem} \label{thm:global2}
	For any knot $K$ in $S^3$, the singular instanton Floer homology $\wt I^\pm(S^3,K) := H(\wt C^\pm)$ associated to $\wt C(S^3,K;\Delta_{\Q[T^{\pm 1}]})$, defined over the ring 
	$\Q[T^{\pm 1}]$, has rank $1$ in degree $0$ and rank $0$ in degree $1$.
	This follows from  Remark \ref{rem:global1} and the exact sequence induced by \eqref{sharp and twisted}.
\end{rem}

For the remainder of this section, we focus on basic properties of $\wt{C}^+$. Similar arguments hold for $\wt{C}^-$.
We start with the counterpart of \Cref{sharp height i}.
\begin{prop} \label{twisted height i}
	 For any morphism of $\mathcal{S}$-complexes $\wt{\lambda} \colon \wt{C} \to \wt{C}'$ of even degree, 
	 we have the following commutative diagram of $\Z/2$-graded $R[x]$-modules:
	\[
	\begin{tikzcd}
		H_*(\lhc)\ar[rr, "\wt{\bf q}^+_*"] \ar[d, "\lhl_*"]
		& &
		H(\wt{C}^+)\ar[d, "\wt{\lambda}^+_*"]\\H(\lhc')\ar[rr, "\wt{\bf q}'^+_{*}"] 
		& &
		H(\wt{C}'^+)
	\end{tikzcd}
	\]
	where $\wt{\lambda}^+ \colon \wt{C}^+ \to \wt{C}'^+$ is equal to $\wt{\lambda}$, and $\wt{\lambda}^+_*$ is induced 
	by $\wt{\lambda}^+$.
\end{prop}

Next, we discuss the tensor product of reduced framed complexes. Let $(\wt{C}^{\otimes}, \wt{d}^{\otimes}, \chi^{\otimes})$ be the tensor product of $\mathcal{S}$-complexs $(\wt{C}, \wt{d}, \chi)$ and $(\wt{C}', \wt{d}', \chi')$, and $(\lhc^{\otimes},\loc\vphantom{\lhc}^{\otimes},\lcc^{\otimes})$ be the large equivariant complexes of $\wt{C}^{\otimes}$. Analogous to the unreduced case, we have canonical chain isomorphisms
\begin{align*}
	\wt{C}^+ \otimes_{R[x]} \wt{C}'^+&\cong \left(\lhc \otimes_{R[x]}{\bf S}^+\right)\otimes_{R[x]}
	\left(\lhc' \otimes_{R[x]}{\bf S}^+\right)\\[2mm]
	&\cong \lhc^{\otimes} \otimes_{R[x]} {\bf S}^+\cong \wt{C}^{\otimes+},
\end{align*}
where $\wt{C}^{\otimes+}$ is the reduced framed complex of $\wt{C}^{\otimes}$. Denote by $\wt{T}^+ \colon \wt{C}^+ \otimes_{R[x]} \wt{C}'^+\to \wt{C}^{\otimes+}$ the resulting isomorphism over $R[x]$. As the $x$-action on $\wt{C}$ is multiplication by $-2\Lambda$, we have a canonical isomorphism
\[
  \wt{C}^+\otimes_{R[x]} \wt{C}'^+\cong\wt{C}^+\otimes_{R} \wt{C}'^+
\]
Under this identification, the map $\wt{T}^+$ coincides with the identity. The following result is the counterpart of \Cref{sharp tensor}.
\begin{prop}\label{twisted tensor}
	We have the following commutative diagram of $R[x]$-modules:
	\[
	\begin{tikzcd}
		H(\lhc)\otimes_{R[x]}H(\lhc')\ar[d, "\wt{\bf q}^+_* \otimes_{R[x]} \wt{\bf q}'^+_*"'] \ar[r]&H(\lhc\otimes_{R[x]}\lhc')\ar[d, "(\wt{\bf q}^+ \otimes_{R[x]} \wt{\bf q}'^+)_*"'] \ar[rr, "\wh{T}_*"]
		&&
		H(\lhc^{\otimes})\ar[dll, "\wt{\bf q}^{\otimes+}_*"] \\[5mm]H(\wt{C}^+)\otimes_{R}H(\wt{C}'^+)\ar[r]&H(\wt{C}^{\otimes+})
	\end{tikzcd}
\]
\end{prop}


\section{Special cycles}
\label{section: special cycles}

The characterization of the Fr\o yshov invariant of an $\mathcal{S}$-complex given in \eqref{eq:froyshovdefpos}--\eqref{eq:froyshovdefneg} highlights the importance of cycles satisfying certain conditions when constructing local equivalence invariants. In this section we systematically study such {\it{special cycles}} in the (large) equivariant complexes associated to the $\mathcal{S}$-complex. We describe the behavior of these cycles in the context of the associated unreduced and reduced framed complexes. These algebraic constructions will be applied in later sections to defined homology concordance invariants of knots.

\subsection{Special cycles of $\mathcal{S}$-complexes}

Fix a $\Z/4$-graded $\mathcal{S}$-complex $\wt C$ over a ring $R$, and use the same notation for its associated equivariant complexes as was used in \Cref{subsection: equivariant}. Let $\mathfrak{z}=(\al, \sum_{i=0}^{N} a_i x^i )$ be an element of the small equivariant complex $\shc$. If $\mathfrak{z}$ is a cycle, i.e. $\widehat{\fd}(\mathfrak{z})=0$, then $\mathfrak{i}(\mathfrak{z})\in R[\![x^{-1}, x]$ defines an element in the image of the homology map $\mathfrak{i}_\ast$. Noting the expressions
\begin{equation}\label{eq:dhatandimap}
 \widehat{\fd}(\mathfrak{z} ) =( {d}\al- \sum_{i=0}^{N} v^i \delta_2 (a_i), 0 ), \hspace{.5cm}
 \mathfrak{i}(\mathfrak{z})
= \sum_{i=0}^N a_i x^i + \sum_{i=-\infty}^{-1}\delta_1v^{-i-1}(\alpha)x^i,
\end{equation}
we see that there is some $\mathfrak{z}\in \shc$ that produces an element in $\im\mathfrak{i}_\ast$ of $x$-degree equal to $-k-1<0$ if and only if condition \eqref{eq:froyshovdefpos} holds; and an element in $\im\mathfrak{i}_\ast$ of $x$-degree $-k\geq 0$ if and only if condition \eqref{eq:froyshovdefneg} holds. Thus we have the following alternative characterization of the Fr\o yshov invariant: 
\begin{equation}\label{eq:hdefusingdeg}
h (\wt C  ) =- \min \left\{ \deg_x (Q(x) ) \mid 0\neq   Q(x) \in \im  \mathfrak{i}_*\subset  R[\![x^{-1}, x]   \right\} \in \Z.
\end{equation}
To explore this structure further, following \cite[\S 4.7]{DS19} we associate to each $k\in \Z$ the ideal
\[
J_k (\wt{C})  := \{ c_{-k}\in R \mid c_{-k} x^{-k} + c_{-k-1} x^{-k-1}+ \cdots \in \im \mathfrak{i}_* \subset R [\![x^{-1}, x] \}. 
\]
The Fr\o yshov invariant is the support of these ideals in the sense that $J_k(\wt C)\neq 0$ if and only if $h(\wt C)\geq k$. Since $\im \mathfrak{i}_*$ is an $R[x]$-module, the non-trivial instances of $J_k(\wt C)$ form a nested sequence of ideals in $R$:
\[
J_{h(\wt{C})}(\wt{C}) \subseteq J_{h(\wt{C})-1}(\wt{C}) \subseteq \cdots  \subseteq R
\]
Moreover, the collection $\{J_k(\wt{C})\}_{k \in \Z}$ is a local equivalence invariant of the $\mathcal{S}$-complex $\wt C$.

The cycles of interest will be those that realize the elements of the ideals $J_k(\wt C)$. We would like to view these as in $\lhc$, as many of the constructions in subsequent sections are tied more directly to the large equivariant complexes. To this end, we introduce the following terminology. Recall that $\wh{\Psi}:\shc\to \lhc$ is a chain homotopy equivalence. 

\begin{defn} \label{def:special cycle}
For $k \in \Z$ and $f \in R$, a chain $z \in  \lhc$ is a {\it special $(k, f)$-cycle} if there exists a cycle $\mathfrak{z} \in \shc$ such that $\wh{\Psi}(\mathfrak{z})=z$ and $\mathfrak{i}(\mathfrak{z}) = f x^{-k} + \sum_{i=-\infty}^{-k-1} b_i x^{i}$.
\end{defn} 

We often call $z$ simply a {\it{special cycle}}, even though the data $(k,f)$ is essential. Note that every cycle in the image of $\wh{\Psi}$ is a special cycle for some $(k,f)$: take $f=0$ and $k\ll 0$. Interesting special cycles have $f\neq 0$. In fact, for any $k\in \Z$ and $f\in R$, there exists a special $(k,f)$-cycle in $\lhc$ if and only if $f\in J_k(\wt C)$.

We remark that the cycle $\mathfrak{z}$ in the above definition is uniquely determined by the special cycle $z \in \lhc$, since $\wh \Phi:\lhc\to \shc$ satisfies $\mathfrak{z}=\wh{\Phi}\circ \wh{\Psi}(\mathfrak{z})=\wh{\Phi}(z)$.

The following construction plays a central role in subsequent sections.

\begin{prop}\label{uniqueness theorem}
Given a $\Z/4$-graded $\mathcal{S}$-complex $\wt C$ over $R$, let $h:= h(\wt{C})$ and $f \in J_h(\wt{C})$.
Then there is a special $(h,f)$-cycle $z \in \lhc_{2h}$. Moreover, the homology class $[z]$ is uniquely determined by $f$ up to $R[x^2]$-torsion. Consequently, we have an injective $R$-homomorphism
\begin{align} \label{hat xi}
\wh{\xi} \colon J_h(\wt{C}) \to H_{2h}(\lhc)/\Tor_{R[x^2]},
\quad f \mapsto \wh{\xi}(f):=[z]
\end{align}

\end{prop}

\begin{proof}
	That there exists a special $(h,f)$-cycle $z$ follows from the remarks preceding the proposition. To prove the second statement of the proposition, let $z$ and $z'$ be two distinct special $(h,f)$-cycles, whose associated cycles are $\mathfrak{z}$ and $\mathfrak{z}'$ respectively. Then $\mathfrak{i} (\mathfrak{z}) = \mathfrak{i} (\mathfrak{z}')$. Indeed, if not, then $\mathfrak{i} (\mathfrak{z}-\mathfrak{z}')$ would have $x$-degree strictly smaller than $-h$, contradicting \eqref{eq:hdefusingdeg}. Noting that
\[
	{\li}_\ast [z] = {\li}_\ast \wh{\Psi}_\ast [\mathfrak{z}] = \overline{\Psi}_\ast \si_\ast [\mathfrak{z}],
\]
we have $\li_\ast[z] - \li_\ast[z']=0$, and hence $[z] - [z']=0 \in \im \lj_\ast$. As $\im \lj_\ast $ is exactly the $R[x^2]$-torsion in $H(\lhc)$, we conclude that $[z]$ is uniquely determined up to $R[x^2]$-torsion.

To prove injectivity, note that by construction, the composition of maps
\[
	\xymatrix{
	 J_h(\wt C) \ar[r]^{\wh\xi\hspace{.8cm}} & H_{2h}(\lhc)/\Tor_{R[x^2]} \ar[r]^{\hspace{.7cm}\wh \Phi_\ast \circ \mathfrak{i}_\ast} & \im \mathfrak{i}_\ast \ar[r]^{\mathfrak{c}_h} & J_h(\wt C) 
	}
\]
is the identity, where $\mathfrak{c}_h$ is the map which sends $\sum_{-\infty}^N b_jx^{j}$ to $b_{-h}$.
\end{proof}

We now describe the behavior of special cycles under various operations involving $\mathcal{S}$-complexes. 

\begin{lem}\label{special cycle: height i}
Let $\wt{\lambda}:\wt C\to \wt C'$ be a height $i$ morphism. 
Then, for a special $(k, f)$-cycle $z \in \lhc$, 
the chain 
\[
\hPsi' \circ \hPhi' \circ \lhl (z) \in \lhc'
\]
is a special $(k+i, c_if)$-cycle, where $c_i$ is defined by the expression in \eqref{eq:cjdefn}.
\end{lem}
\begin{proof}
Let $z = \hPsi(\mathfrak{z})$
and $\mathfrak{z}'=\hPhi' \circ \lhl (z) = \shl(\mathfrak{z})$.
Recall that $\sol$ is multiplication by $ \sum_{j=-\infty}^{-i} c_j x^j$. Then
\[
\si'(\mathfrak{z}') = \si' \circ \shl (\mathfrak{z})
= \sol \circ \si (\mathfrak{z})
= \left(\sum_{j=-\infty}^{-i}c_{-j}x^j\right)\cdot
\left(fx^{-k} + \sum_{j=-\infty}^{-k-1} b_jx^j \right).
\]
Thus $\mathfrak{z}'$ is a cycle such that $\mathfrak{i}'(\mathfrak{z}')$ has leading term $c_if x^{-k-i}$.
\end{proof}

Let $(\wt{C}^{\otimes}, \wt{d}^{\otimes}, \chi^{\otimes})$
be the tensor product
of $(\wt{C}, \wt{d}, \chi)$
and $(\wt{C}', \wt{d}', \chi')$,
and write $(\lhc^{\otimes},\loc\vphantom{\lhc}^{\otimes},\lcc^{\otimes})$ for the large equivariant complexes associated to $\wt{C}^{\otimes}$.

\begin{lem}\label{special cycle tensor}
Let $z \in \lhc$  (resp. $z' \in \lhc'$)
be a special $(k,f)$-cycle (resp. $(k',f')$-cycle).
Then, the chain
\[
\hPsi^{\otimes} \circ \hPhi^{\otimes} \circ \wh{T}
(z \otimes_{R[x]} z') \in \lhc^{\otimes}
\]
is a special $(k+k', ff')$-cycle.
\end{lem}

\begin{proof}
Recall the commutative diagram in \Cref{tensor product}. Since $\soc\vphantom{\shc}^{\otimes}_*$ has trivial differential,
we have
\begin{align*}
\si^{\otimes} \circ  \hPhi^{\otimes}
\circ \wh{T}
(z \otimes_{R[x]} z')
&=
\oPhi\vphantom{\hPhi}^{\otimes} \circ \li^{\otimes} 
\circ \wh{T}
(z \otimes_{R[x]} z')
=
\oPhi\vphantom{\hPhi}^{\otimes}_* \circ \li^{\otimes}_* 
\circ \wh{T}_*
([z] \otimes_{R[x]} [z'])\\[2mm]
&=
(\oPhi_* \circ \li_* [z] )
\cdot
(\oPhi'_* \circ \li'_* [z'] )
=
(\oPhi \circ \li (z) )
\cdot
(\oPhi' \circ \li' (z') )\\[2mm]
&=
\left( 
fx^{-k} + \sum_{i= -\infty}^{-k-1} b_ix^i
\right) \cdot
\left( 
f'x^{-k'} + \sum_{i= -\infty}^{-k'-1} b'_ix^i
\right).
\end{align*}
The leading term is $ff'x^{-k-k'}$, and this completes the proof.
\end{proof}

Let $S$ be an algebra over $R$. For an $\mathcal{S}$-complex $\wt{C}$
over $R$,
let $(\lhc^{S},\loc\vphantom{\lhc}^{S},\lcc^{S})$ denote the large equivariant complexes of the $\mathcal{S}$-complex $\wt{C} \otimes_R S$ over the coefficient ring $S$. The following is immediate.

\begin{lem}\label{change coeff cycle}
Let $z \in \lhc$ 
be a special $(k,f)$-cycle.
Then, the chain
\[
z \otimes_{R} 1 \in \lhc^{S} \cong \lhc \otimes_{R} S
\]
is a special $(k, f)$-cycle, where $f = f\cdot 1 \in S$.
\end{lem}

\subsection{Behavior in unreduced framed complexes}

Throughout, we suppose that $R$ is an integral domain algebra over $\Q[T^{\pm 1}]$.
Let $\wt{C}$ be an $\mathcal{S}$-complex over $R$ and $h:=h(\wt{C})$. Recall from \eqref{hat and sharp} that we have a map
\[
\begin{tikzcd}
	H(\lhc)
\ar[r,"{\bf q}^\sharp_*"]
&
H(C^\sharp) 
\end{tikzcd}
\]
In what follows, it will be convenient to quotient $H(C^\sharp)$ by the  
image of $R[x^2]$-torsion under ${\bf q}^\sharp_\ast$. For brevity, we use the following notation for this quotient:
\[
	H(C^\sharp)^{\bf q} := H(C^\sharp) / {\bf q}^\sharp_*\left( \text{Tor}_{R[x^2]} \right) 
\]
In particular, we have a well-defined induced map
\[
  \begin{tikzcd}
	H_i(\lhc)/\Tor_{R[x^2]} \ar[r,"{\bf q}^\sharp_*"]
	& H_i(C^\sharp)^{\bf q}
  \end{tikzcd}
\]
We now apply this map to the special cycle construction of \Cref{uniqueness theorem}.

\begin{prop}\label{prop:sharpspecialcyclemap}
	Let $h:= h(\wt{C})$. Then the following maps are injective $R$-module homomorphisms:
	\begin{align}
	{\xi}^\sharp_+ \colon J_h(\wt{C}) \to H_{2h}(C^\sharp)^{\bf q}, & \qquad f \mapsto {\xi}^\sharp_+(f):={\bf q}^\sharp_* \wh\xi(f) \label{sharp xi 1},\\[2mm]
	{\xi}^\sharp_- \colon J_h(\wt{C}) \to H_{2h-2}(C^\sharp)^{\bf q}, &\qquad f \mapsto {\xi}^\sharp_-(f):=x\cdot {\bf q}^\sharp_* \wh\xi(f)  \label{sharp xi 2}.
\end{align}
In particular, ${\xi}^\sharp_+$ and ${\xi}^\sharp_-$ map any non-trivial element in $J_h(\wt{C})$ into a non-torsion element.
\end{prop}

\begin{proof}
	Suppose $\xi^\sharp_+(f)\in H(C^\sharp)^{\bf q}$ is zero. Then there is a representative $[y]\in H(\lhc)$ of $\wh \xi(f)\in H(\lhc)/\text{Tor}_{R[x]}$ such that ${\bf q}^\sharp_* [y]=0$. By the exact sequence \eqref{hat and sharp}, we have $[y]=(x^2-4\Lambda^2)[z]$ for some $[z]\in H(\lhc)$. Now if $f\neq 0$, by definition of $\wh\xi(f)$, we have
	\[
	\overline\Psi_\ast\mathbf{i}_\ast((x^2-4\Lambda^2)[z])= f x^{-h} + \sum_{j=-\infty}^{-h-1}b_jx^j.
\]
 On the other hand, $\overline\Psi_\ast\mathbf{i}_\ast [z]$ then has leading term of $x$-degree equal to $-2-h$, contradicting \eqref{eq:hdefusingdeg}. Thus $f=0$ and $\xi^\sharp_+$ is injective. For $\xi^\sharp_-$ the argument is similar.
 The last claim follows from injectivity and the fact that $J_h(\wt{C})$ is an ideal in an integral domain $R$.

	An alternative proof follows by applying the pairing result given below as \Cref{sharp xi dual}.
\end{proof}

\begin{rem}
We can make the following observation about the module ${\bf q}^\sharp_*\left( \text{Tor}_{R[x^2]} \right)$ used in the definition of $H_i(C^\sharp)^{\bf q}$.
Suppose $[z]\in H_{i}(\lhc)$ satisfies $f(x^2)[z]=0$ for some $f(x)\in R[x]$. As $C^\sharp$ is isomorphic to the quotient of $\lhc$ by the relation $x^2-4\Lambda^2$, we have $0={\bf q}^\sharp_* (f(x^2) \cdot [z]) = f(4\Lambda^2) \cdot {\bf q}^\sharp_* [z]$. Thus in the case that $f(4\Lambda^2)\neq 0$, the element $ [z]$ is $R$-torsion in $H(C^\sharp)$. 
\end{rem}

To make the above construction more explicit, if $\xi^{\sharp}_+(f) = [(\zeta_1,\zeta_2)]$
for chains $\zeta_1 \in \wt{C}_{2h}$ and $\zeta_2 \in \wt{C}_{2h+2}$,
then the description of the $x$-module structure on $C^\sharp$ in \eqref{sharp action} gives the expression
\begin{equation} \label{sharp xi minus}
\xi^{\sharp}_-(f) = 
x \cdot [(\zeta_1,\zeta_2)]
=[(-2\Lambda^2\zeta_2, -2 \zeta_1)].
\end{equation}
We now describe the behavior of $\xi_\pm^\sharp$ under some basic operations.

\begin{lem} \label{sharp xi height i}
Let $\wt{C}$ and $\wt{C}'$ be
$\mathcal{S}$-complexes over $R$
with $h:=h(\wt{C})$
and $h':=h(\wt{C}')$.
If $h' \geq h$,
then for any height $(h'-h)$ morphism
$\wt{\lambda} \colon \wt{C}_* \to \wt{C}'$ the induced map $\lambda^\sharp_\ast:H(C^\sharp)^{\bf q}\to H(C'^\sharp)^{\bf q}$ satisfies
\[
	\lambda^\sharp(\xi_{\pm}^\sharp(f) ) = \xi_{\pm}^\sharp(c_{h'-h}f) 
\]
Here $c_{h'-h}$ is defined by the expression in \eqref{eq:cjdefn}.
\end{lem}
\begin{proof}
Let $z \in \lhc_{2h}$ be a special $(h,f)$-cycle.
Then, it follows from
\Cref{special cycle: height i}
that 
$\hPsi' \circ \hPhi' \circ \lhl (z) \in \lhc_{2h'}$
is a special $(h',c_{h'-h}f)$-cycle.
Then we compute
\[
\lambda^\sharp_*(\xi^{\sharp}_+(f)) =
\lambda^\sharp_* \circ {\bf q}^\sharp_*([z]) =
{\bf q}'^\sharp_* \circ \lhl_*([z]) =
{\bf q}'^\sharp_* ([\hPsi' \circ \hPhi' \circ \lhl (z)])
=\xi^\sharp_+(c_{h'-h}f).
\]
The second identity follows from the commutative diagram in \Cref{sharp height i}. Since $\lambda^\sharp$ is an $R[x]$-module homomorphism,
this completes the proof.
\end{proof}

Let $\wt{C}^{\otimes}$
be the tensor product
of $\mathcal{S}$-complexes $\wt{C}$
and $\wt{C}'$ over $R$,
and
$(\lhc^{\otimes},\loc\vphantom{\lhc}^{\otimes},\lcc^{\otimes})$ (resp.\ $C^{\otimes\sharp}$) be the large equivariant complexes
(resp. unreduced framed complex) of $\wt{C}^{\otimes}$.
Let $h:=h(\wt{C})$ and $h':=h(\wt{C}')$.
Recall 
the chain isomorphism
$T^\sharp \colon C^\sharp \otimes_{R[x]} C'^\sharp
\to C^{\otimes\sharp}$ from \Cref{sharp tensor}. We have an induced map
\[
t^\sharp \colon
H(C^\sharp)^{\bf q} \otimes_{R[x]} 
H(C'^\sharp)^{\bf q}
\to H(C^\sharp \otimes_{R[x]} C'^\sharp)^{\bf q}.
\]

\begin{lem} \label{sharp xi tensor}
For any $f \in J_h(\wt{C})$
and $f' \in J_{h'}(\wt{C}')$,
the composition of $t^\sharp$ with $T^\sharp_*$
maps the tensor products of $\xi^\sharp_{\pm}(f) \in H(C^\sharp)^{\bf q}$ and
$\xi^\sharp_{\pm}(f') \in H(C'^\sharp)^{\bf q}$ as follows:
\begin{align*}
\xi_+^\sharp(f)\otimes \xi_+^\sharp(f') &\mapsto    
\xi^\sharp_+(ff')
&
\xi^\sharp_+(f)\otimes \xi^\sharp_-(f') &\mapsto    
\xi^\sharp_-(ff')\\[2mm]
\xi^\sharp_-(f)\otimes \xi^\sharp_+(f') &\mapsto    
\xi^\sharp_-(ff')
&
\xi^\sharp_-(f)\otimes \xi^\sharp_-(f') &\mapsto    
\xi^\sharp_+(4\Lambda^2ff')
\end{align*}
\end{lem}
\begin{proof}
Since $x \cdot \xi^\sharp_+(ff')= \xi^\sharp_-(ff')$
and 
$x^2 \cdot \xi^\sharp_+(ff') = 4 \Lambda^2 \cdot \xi^\sharp_+(ff')
= \xi^\sharp_+(4\Lambda^2 ff')$,
we only need to check 
$\xi_+^\sharp(f)\otimes \xi_+^\sharp(f') \mapsto    
\xi^\sharp_+(ff')$.
Let $z \in \lhc_{2h}$ (resp. $z' \in \lhc'_{2h'}$) be a special $(h,f)$-cycle (resp.\ $(h',f')$-cycle).
Then, from \Cref{special cycle tensor}, $\hPsi^{\otimes} \circ \hPhi^{\otimes} \circ \wh{T}
(z \otimes z') \in \lhc^{\otimes}_{2h+2h'}$
is a special $(h+h', ff')$-cycle.
The result now follows from
\begin{align*}
T^\sharp_* \circ t^\sharp(\xi^{\sharp}_+(f) \otimes \xi^{\sharp}_+(f'))
&=
T^\sharp_* \circ t^\sharp
\circ({\bf q}^\sharp_* \otimes_{R[x]} {\bf q}'^\sharp_*)([z]\otimes[z'])
=
{\bf q}^{\otimes\sharp}_*\circ \wh{T}_* ([z\otimes z'])
\\[2mm]
&= {\bf q}^{\otimes\#}_*([
\hPsi^{\otimes} \circ \hPhi^{\otimes} \circ \wh{T}
(z \otimes z') 
])
= \xi^\sharp_+(ff'),
\end{align*}
where we have used the commutative diagram of \Cref{sharp tensor}. 
\end{proof}

Finally, we discuss a duality result involving the two maps $\xi^\sharp_+$ and $\xi^\sharp_-$.
We start with the following lemma. In the formulas below, we are using the splitting $\wt C_\ast = C_\ast \oplus C_{\ast -1} \oplus R_{(0)}$.
\begin{lem} \label{sharp xi dual lem}
Let $\wt C$ be an $\mathcal{S}$-complex over $R$, and $h:=h(\wt C)$. The following assertions hold:
\begin{itemize}
\item[(i)]
If $h>0$, then any representative of the homology class of $\xi^\sharp_+(f)$ is 
a cycle
$((0,\alpha,0),0) \in \wt{C}_{2h}\oplus \wt{C}_{2h+2}$
for some $\alpha\in C_{2h-1}$ satisfying
\[
\delta_1 v^i (\alpha) = 0 \quad (0 \leq i < h-1 ) 
\quad \text{and} \quad \delta_1 v^{h-1} (\alpha ) = f.
\]
\item[(ii)] If $h=0$,
any representative of $\xi^\sharp_+(f)$ has the form 
$\left( (0, \alpha, f), 0\right) \in \wt{C}_{2h}\oplus \wt{C}_{2h+2}$ for some $\alpha\in C_{-1}$.
\item[(iii)] If $h< 0$,
then any representative of $\xi^\sharp_+(f)$ is  given by a cycle in 
$\wt{C}_{2h}\oplus \wt{C}_{2h+2}$ in the form of
\[
\left( (v^{-h-1}\delta_2(f) + \sum_{i=0}^{-h-2} v^i\delta_2(b_i), \alpha, b
), \ \ 
(\sum_{i=0}^{-h-2} v^i \delta_2 (b'_i ), 0, b' 
)\right).
\]
\end{itemize}
\end{lem}
\begin{proof}
Here we consider the case $h<0$. 
Let $z = \hPsi( \alpha, \sum_{i=0}^{-h} a_i x^i) \in \lhc_{2h}$
be a special $(h,f)$-cycle. In particular, $a_{-h}=f$. 
Then 
\[
\xi^\sharp_+(f)={\bf q}^\sharp_*([z]) = 
[{\bf q}^\sharp \circ \hPsi (\alpha, \sum_{i=0}^{-h}a_i x^i)].
\]
Here, by the definitions of $\hPsi$ and ${\bf q}^\sharp$, we have
\begin{align*}
{\bf q}^\sharp \circ  \hPsi ( \alpha, \sum_{i=0}^{-h} a_i x^i) 
&=
{\bf q}^\sharp  \left( \sum_{i=1}^{-h} \sum_{j=0}^{i-1} v^j \delta_2 (a_i ) x^{i-j-1} , \alpha, \sum_{i=0}^{-h} a_i x^i \right) \\[2mm]
&=
\left( \sum_{i=1}^{-h} \sum_{j=0}^{i-1} v^j \delta_2 (a_i ) 
(4\Lambda^2)^{\lfloor \tfrac{i-j-1}{2} \rfloor} 
x^{ (i-j-1)- 2\lfloor \tfrac{i-j-1}{2} \rfloor}, \ \alpha, \  \sum_{i=0}^{-h} a_i (4\Lambda^2)^{\lfloor \tfrac{i}{2} \rfloor}
x^{i-2\lfloor \tfrac{i}{2}\rfloor} \right)\\[2mm]
&=
\left( v^{-h-1}(a_{-h}) + \sum_{i=0}^{-h-2} v^i\delta_2(b_i) 
+\sum_{i=0}^{-h-2} v^i\delta_2(b'_i)x,\  \alpha,\  b+b'x \right)
\end{align*}
for some elements $b_i, b'_i, b, b' \in R
\ (0 \leq i \leq -k-2)$.
The remaining parts are proved similarly.
\end{proof}

Let $\wt{C}^\dagger$ denote the dual complex of an $\mathcal{S}$-complex $\wt{C}$ over $R$, and $C^{\sharp\dagger}$ denote the framed complex of $\wt{C}^\dagger$. Set $h:= h(\wt{C}) = - h(\wt{C}^\dagger)$.
Then, the duality pairing for $\xi^\sharp_{\pm}$ is stated as follows.
\begin{prop} \label{sharp xi dual}
	For any non-zero elements $f \in J_h(\wt{C})$ and $f' \in J_{-h}(\wt{C}^\dagger)$, we have
	\[
	  \langle \xi^\sharp_+(f'), \xi^\sharp_-(f)\rangle = 
	  \langle \xi^\sharp_-(f'), \xi^\sharp_+(f)\rangle = -2ff',
	\]
	where $\xi^\sharp_{\pm}(f) \in H_*(C^\sharp)^{\bf q}$ and $\xi^\sharp_{\pm}(f') \in H_*(C^{\dagger\sharp})^{\bf q}$, and the above pairing is induced by the identification of \Cref{duality for sharp}. In particular, the pairing 
	is independent of the choice of representatives for $\xi^\sharp_{\pm}(f)$.
\end{prop}

The proof of this proposition uses \Cref{sharp xi dual lem}. First we give some remarks on dual $\cS$-complexes following \cite[\S 4.4]{DS19}. For any $\cS$-complex $\wt C_\ast = C_\ast \oplus C_{\ast -1} \oplus R_{(0)}$, the dual complex $\wt C_\ast^\dagger$ has a splitting of the form $C_\ast^\dagger\oplus C_{\ast-1}^\dagger\oplus R_{(0)}$ where $C_\ast^\dagger=\Hom_R(C_{-*-1},R)$. With respect to this splitting, the components of the differential $\wt d^\dagger$ are given by $d^\dagger$, $v^\dagger$, $\delta_1^\dagger$ and $\delta_2^\dagger$, where for any $f\in C_i^\dagger$, we have
\[
  d^\dagger (f)=(-1)^if\circ d,\hspace{1cm} v^\dagger (f)=f\circ v,\hspace{1cm} \delta_1^\dagger (f)=-f\circ \delta_2(1),\hspace{1cm} \delta_2^\dagger (1)=\delta_1.
\]
Furthermore, the pairing in \eqref{pairing-hom-co} for $\zeta = (\alpha, \beta, a) \in \wt{C}_{i}$ and $\zeta^\dagger = (f, g, b) \in \wt{C}^\dagger_{-i}$ is given by
\begin{equation}\label{duality for tilde}
  \langle \zeta^\dagger, \zeta \rangle=  (-1)^{i}  g(\alpha) + f(\beta)+ ab.
\end{equation}

\begin{proof}[Proof of Proposition \ref{sharp xi dual}]
	Here we prove $\langle \xi^\sharp_+(f'), \xi^\sharp_-(f)\rangle = -2ff'$ for the case $h>0$. By \Cref{sharp xi dual lem}, the homology class $\xi^\sharp_+(f) \in H_{2h}(C^\sharp)^{\bf q}$ is represented by
	a cycle $((0,\alpha,0),0) \in \wt{C}_{2h}\oplus \wt{C}_{2h+2}$ such that
	\[
	  \delta_1 v^i (\alpha) = 0 \quad (0 \leq i < h-1 ) 
	  \quad \text{and} \quad \delta_1 v^{h-1} (\alpha ) = f,
	\]
	while $\xi^\sharp_+(f') \in H_{-2h}(C^{\dagger\sharp})^{\bf q}$ is represented by a cycle in $\wt{C}^\dagger_{-2h}\oplus \wt{C}^\dagger_{-2h+2}$ that has the following form
	\[
	  \big( \left((v^\dagger)^{h-1}\delta^\dagger_2(f') + \sum_{i=0}^{h-2} (v^\dagger)^i\delta^\dagger_2(b_i), \varphi, b\right), \ \ 
	  \left(\sum_{i=0}^{h-2} (v^\dagger)^i \delta^\dagger_2 (b'_i ), 0, b'\right) \big).
	\]
	Moreover, the formula \eqref{sharp xi minus} shows that $\xi^\sharp_-(f) \in H_{2h}(C^\sharp)$ is represented by
	$(0,(0,-2\alpha,0)) \in \wt{C}_{2h-2}\oplus \wt{C}_{2h}$.
Therefore, from \Cref{duality for sharp} and \eqref{duality for tilde}, 
we have
\begin{align*}
\langle \xi^\sharp_+(f'), \xi^\sharp_-(f)\rangle &=
 \left\langle \left(
 (v^\dagger)^{h-1}\delta^\dagger_2(f') +
 \sum_{i=0}^{h-2} (v^\dagger)^i\delta^\dagger_2(b_i), \varphi, b
\right), (0,-2\alpha,0) \right\rangle\\[2mm]
&=
\left(
f'\delta_1 v^{h-1}+
 \sum_{i=0}^{h-2} b_i\delta_1v^i
\right)(-2 \alpha)
= - 2f f'.
\end{align*}
The remaining parts are proved similarly.
\end{proof}

\subsection{Behavior in reduced framed complexes}

Let $R$ be an integral domain algebra over $\Q[T^{\pm 1}]$,
$\wt{C}$ an $\mathcal{S}$-complex over $R$ and $h:=h(\wt{C}_*)$. Parallel to what was done for the unreduced complexes, we now consider special cycles in the reduced framed complexes $\wt C^\pm$ associated to $\wt C$. First, recall from \eqref{hat and twisted} that we have maps
\begin{equation*} 
	\begin{tikzcd}
	H(\lhc) \ar[r,"\wt{{\bf q}}^{\pm}_*"] & H(\wt{C}^\pm).
	\end{tikzcd}
\end{equation*}
Similar to our convention in the unreduced case, we write $H(\wt C^\pm)^{\bf q}$ for the quotient of $H(\wt C^\pm)$ by the submodule generated by the image of $R[x^2]$-torsion under $\wt{{\bf q}}^{\pm}_*$. Thus we have induced maps
\[
\begin{tikzcd}
	H_i(\lhc)/\Tor_{R[x^2]} 
\ar[r,"\wt{\bf q}^\pm_*"]
&
H_i(\wt C^\pm)^{\bf q}
\end{tikzcd}
\]
As $\wt C^+$ and $\wt C^-$ are isomorphic complexes, for the remainder of this section we focus on the case of $\pm =+$, and drop the symbol $+$ from most of the notation. The proof of the following is analogous to that of \Cref{prop:sharpspecialcyclemap}.

\begin{prop}
	Let $h:= h(\wt{C})$. Then the following map is an injective $R$-module homomorphism:
	\begin{align}
	\wt \xi \colon J_h(\wt{C}) \to H_{2h}(\wt C^+)^{\bf q}, &
\qquad f \mapsto {\wt \xi}(f):={\wt {\bf q}}_* \wh\xi(f) \label{tilde xi}
\end{align}
In particular, $\wt \xi$ maps any non-trivial element in $J_h(\wt{C})$ into a non-torsion element.
\end{prop}

In this context, the behavior of special cycles under non-negative height morphisms is described as follows. The proof is entirely analogous to that of \Cref{sharp xi height i}.

\begin{lem} \label{tilde xi height i}
Let $\wt{C}$ and $\wt{C}'$ be
$\mathcal{S}$-complexes over $R$
with $h:=h(\wt{C})$
and $h':=h(\wt{C}')$.
If $h' \geq h$,
then for any height $(h'-h)$ morphism
$\wt{\lambda} \colon \wt{C} \to \wt{C}'$, the induced map $\wt\lambda^+_\ast:H(\wt{C}^+)^{\bf q} \to H(\wt{C}'^+)^{\bf q}$ satisfies
\[
	\wt\lambda^+_\ast(\wt\xi(f)) = \wt\xi(c_{h-h'} f)
\]
Here $c_{h'-h}$ is defined by the expression in \eqref{eq:cjdefn}.
\end{lem}

For tensor products, we have the following. Let $\wt{C}^{\otimes}$
be the tensor product
of $\mathcal{S}$-complexes $\wt{C}$
and $\wt{C}'$ over $R$.
Let $h:=h(\wt{C})$ and $h':=h(\wt{C}')$. We have an induced map
\[
\wt{t} \colon
H(\wt{C}^+)^{\bf q} \otimes_{R}  
H(\wt{C}'^+)^{\bf q}
\to H(\wt{C}^{\otimes +})^{\bf q}
\]

\begin{lem} \label{tilde xi tensor}
For any $f \in J_h(\wt{C})$
and $f' \in J_{h'}(\wt{C}')$, we have
\[
\wt{t}\left(\wt{\xi}(f)\otimes \wt{\xi}(f')
\right)=\wt{\xi}(ff').
\]
\end{lem}
\begin{proof}
Let $z \in \lhc_{2h}$ (resp. $z' \in \lhc'_{2h'}$) be a special $(h,f)$-cycle (resp.\ $(h',f')$-cycle).
Then, it follows from
\Cref{special cycle tensor}
that 
$\hPsi^{\otimes} \circ \hPhi^{\otimes} \circ \wh{T}
(z \otimes z') \in \lhc^{\otimes}_{2h+2h'}$
is a special $(h+h', ff')$-cycle.
Now we see, using the commutative diagram from \Cref{twisted tensor},
\begin{align*}
\wt{t}(\wt{\xi}(f) \otimes \wt{\xi}(f'))
&=
\wt{t}
\circ(\wt{{\bf q}}_* \otimes_{R[x]} \wt{\lk}'_*)([z]\otimes[z'])
=
\wt{{\bf q}}^{\otimes}_*\circ \wh{T}_* ([z\otimes z'])
\\[2mm]
&= \wt{{\bf q}}_*([
\hPsi^{\otimes} \circ \hPhi^{\otimes} \circ \wh{T}
(z \otimes z') 
])
= \wt{\xi}(ff').\qedhere
\end{align*}
\end{proof}


\section{Concordance invariants from special cycles} 
\label{section:invs from froyshov cycles}

In this section, we construct several local equivalence invariants of $\mathcal{S}$-complexes using the machinary of special cycles. As in the previous sections, the element $\Lambda=T-T^{-1}$ of $\Q[T^{\pm 1}]$ plays an important role for us. The formal power series ring of $\Lambda$ may be identified with the completion of the localization of $\Q[T^{\pm 1}]$ at $T=1$. An explicit identification comes from noting that $\Lambda=a(T-1)$ where $a=1+T^{-1}$ is a unit in $\Q[T^{\pm 1}]$. Thus as a power series in $T-1$, we can write
\[
	\Lambda = 2(T-1) - (T-1)^2 + (T-1)^3 - (T-1)^4 + \cdots
\]
With this identification, $\locring$ is a $\Q[T^{\pm 1}]$-algebra. Recall that $\Theta^\mathcal{S}_{\locring}$ denotes the local equivalence group of $\Z/4$-graded $\mathcal{S}$-complexes over $\locring$. The first set of invariants are of the form
 \[
 \wt{s}, \;\; s^\sharp_\pm,\;\;  s^\sharp,\;\;  \wt \varepsilon \;\; \colon \Theta^\mathcal{S}_{\locring} \to \Z. 
 \]
Applying these constructions to the $\mathcal{S}$-complex $\wt C(Y,K;\Delta_{\locring})$ of a knot in an integer homology $3$-sphere gives corresponding homology concordance invariants. For knots in $S^3$ we show: $\wt s$ is a homomorphism and half a slice-torus invariant in the sense of \cite{Le14}; $s^\sharp$ recovers Kronheimer and Mrowka's invariant from \cite{KM13}; and $s^\sharp_\pm$ recover Gong's refinements of $s^\sharp$ from \cite{Gong21}. In particular, we prove Theorems \ref{s-tilde-slice-torus}, \ref{quasi-additivity} and \cref{s-tilde-values} from the introduction. 

We go on to define, for any $\Z/4$-graded $\mathcal{S}$-complex over $\Z[T^{\pm 1}]$, a $\Z[T^{\pm 1}]$-submodule
\[
	\wh z(\wt C) \subset \text{Frac}\left(\Z[T^{\pm 1}]\right)
\]
which is a local equivalence invariant. Applying this to knots, we show that this invariant recovers all of the concordance invariants defined by Kronheimer and Mrowka from \cite{KM19b}, proving \Cref{thm:introkminvtsarelocequiv} from the introduction. Finally, we compute all of the invariants introduced in this section for two-bridge knots, and in particular prove \Cref{thm:intro2bridge}.

\subsection{The invariant $\wt s$}\label{section: A new slice torus invariant from singular instanton theory}

 Let $\wt{C}$ be an $\mathcal{S}$-complex over $\locring$ and $h:=h(\wt C)$. 
Since any ideal of $\locring$ is the form of $(\Lambda^n)$,
we have $J_{h}(\wt{C})=(\Lambda^{n_0})$ for some $n_0 \in \Z_{\geq 0}$. Now, $\wt{s}(\wt{C})$ is defined as follows. 
\begin{defn}\label{defn:stilde}
	We define
	\begin{align}\label{def-s-tilde}
		\wt{s}(\wt{C}) :=  \min \left\{ n-m  \mid  \Lambda^n \in J_h(\wt{C}) ,\, \wt{\xi}(\Lambda^n)=\Lambda^m  \cdot  y\,\text{ for some }\, y \in H(\wt{C}^+)^{\bf q} \right\},
\end{align}
where 
$\wt{\xi} \colon J_h(\wt{C}) \to H(\wt{C}^+)^{\bf q}
$ is the $R$-homomorphism in \eqref{tilde xi}. 
\end{defn}
An equivalent definition of $\wt s$ uses the induced map $\wt \xi:J_h(\wt C)\to H(\wt C^+)^{\bf{q}}/\text{Tor}_\locring$. In this version, we can fix any non-negative integer $n$ such that $\Lambda^n \in J_h(\wt{C})$ and take the minimum over all $m$ for which the condition in \eqref{def-s-tilde} holds.

The following proposition implies the invariance of $\wt{s}$ under local equivalence.
\begin{prop}\label{definite ineq for stilde}
Let $\wt{C}$ and $\wt{C}'$ be $\mathcal{S}$-complexes over
$\locring$ with $i:= h(\wt{C}')-h(\wt{C}) \geq 0$.
If there exists a height $i$ morphism
$\wt{\lambda} \colon \wt{C} \to \wt{C}'$ with $c_i= a \Lambda^k$ for some unit $a \in \locring$, we have
 \[
 \wt{s} (\wt{C}') \leq  \wt{s} (\wt{C}) + k. 
 \]
 In particular, if $\wt{\lambda}$ is a local map
 (or equivalently, $i=0$ and $k=0$), 
 then  $ \wt{s} (\wt{C}') \leq  \wt{s} (\wt{C})$.
 \end{prop}
\begin{proof}
	Pick $n\in \Z_{\geq 0}$ and $y\in H(\wt{C}^+)^{\bf q}$ such that $ \wt{\xi}(\Lambda^n)=\Lambda^{n-\wt s(\wt C)}  \cdot  y$. By applying $\wt \lambda_*^+$ to this identity and then using \Cref{tilde xi height i}, we have
	\[
	  \wt{\xi}(a \Lambda^{n+k})=\Lambda^{n-\wt s(\wt C)}  \cdot  \wt \lambda_*^+(y).
	\]
	Since $a$ is a unit, this implies that $\wt{s}(\wt{C}') \leq (n+k) - (n-\wt{s}(\wt{C}))= \wt{s}(\wt{C})+k$.
\end{proof} 
Next, we prove sub-additivity of $\wt{s}$.

 \begin{prop}\label{stilde tensor}
	For any two $\mathcal{S}$-complexes $\wt{C}$ and $\wt{C}'$ over $\locring$, we have 
	\begin{equation}\label{sub-additivity}
	\wt{s} (\wt{C}\otimes \wt{C}' ) \leq  
	\wt{s} (\wt{C}) + \wt{s} ( \wt{C}' ). 
	\end{equation}
	In particular, for any $\mathcal{S}$-complex $\wt C$ and its dual $\mathcal{S}$-complex $\wt C^\dagger$ we have 
	\begin{equation}\label{dual-ineq}
		\wt s(\wt C^\dagger ) \geq -\wt s(\wt C).
	\end{equation}
\end{prop}
\begin{proof}
	Let $\wt{s}:=\wt{s}(\wt{C})$, $\wt{s}':=\wt{s}(\wt{C}')$. Pick $n,n'\in \Z_{\geq 0}$, $[y]\in H(\wt{C}^+)$, $[y']\in H(\wt{C}^+)$ and representatives $[z]\in H(\wt{C}^+)$ and 
	$[z']\in H(\wt{C}'^+)$ for $\wt{\xi}(\Lambda^n)\in H(\wt{C}^{+})^{\bf q}$, $\wt{\xi}(\Lambda^{n'}) \in H(\wt{C}'^{+})^{\bf q}$ such that
	\[
	  [z]=\Lambda^{n-\wt s}  \cdot  [y],\hspace{1cm}[z']=\Lambda^{n'-\wt s'}  \cdot  [y'].
	\]
	By \Cref{tilde xi tensor}, the image $[z^\otimes ]$ of $[z]\otimes [z']$ with respect to the map 
	\begin{equation}\label{tensor-hom-map}
	  H(\wt{C}^+)\otimes_{R}H(\wt{C}'^+)\to H(\wt{C}^{\otimes+})
	\end{equation}
	is a representative for $\wt{\xi}(\Lambda^{n+n'})\in H(\wt{C}^{\otimes +})^{\bf q}$, and we have
	\begin{equation}\label{eq:proofofsubadditivityofstilde}
	  [z^\otimes]=\Lambda^{n+n'-\wt s-\wt s'}[y^\otimes] 
	\end{equation}
	where $[y^\otimes]$ is the image of $[y]\otimes [y']$ with respect to \eqref{tensor-hom-map}. This gives \eqref{sub-additivity}. The inequality in \eqref{dual-ineq} follows from \eqref{sub-additivity} and the fact that 
	the tensor $\cS$-complex $\wt{C}\otimes \wt{C}^\dagger$ is locally equivalent to the trivial $\cS$-complex.
\end{proof}

Additivity of $\wt s$ will be established under the following condition, which is a condition which holds for $\mathcal{S}$-complexes of knots in the $3$-sphere, see \Cref{thm:global2}.\\

\begin{ass}
\label{rank1}
$H_k(\wt C^+)$ has rank $1$ for $k\equiv 0\pmod{2}$, and rank $0$ for $k\equiv1\pmod{2}$.\\
\end{ass}

\begin{rem}\label{rem:assumprank1underproducts}
It is straightforward to verify, using the K\"{u}nneth formula, that if the $\mathcal{S}$-complexes $\wt C$ and $\wt C'$ satisfy \Cref{rank1}, then so too does the tensor product $\mathcal{S}$-complex $\wt C\otimes \wt C'$. Similarly, if $\wt C$ satisfies \Cref{rank1}, then so does the dual $\wt C^\dagger$.
\end{rem}

Note that under \Cref{rank1}, the natural quotient map $H(\wt C^+)\to H(\wt C^+)^{\bf q}$ is an isomorphism mod $\locring$-torsion. This follows from the injectivity of the map $\wt \xi$. Thus $H(\wt C^+)/\text{Tor}_\locring$ can be used as the homology group appearing in the definition of $\wt s$.

\begin{prop}\label{prop:additivityofstilde}
	If $\wt C$ and $\wt C'$ satisfy \Cref{rank1}, then we have
	\begin{equation}\label{eq:additivityofstilde}
		\wt s (\wt C\otimes \wt C') = \wt s(\wt C) + \wt s (\wt C').
	\end{equation}
	In particular, if $\wt C$ satisfies \Cref{rank1}, then $\wt s(\wt C^\dagger)=-\wt s(\wt C)$.
\end{prop}

\begin{proof}
	With the remark preceding the statement of the proposition in mind, in the proof of \Cref{stilde tensor} we replace the instances of $H(\wt C^+)^\mathbf{q}$, $H(\wt C'^+)^\mathbf{q}$  and $H(\wt C^{\otimes +})^\mathbf{q}$ with the respective $\locring$-modules  $H(\wt C^+)$, $H(\wt C'^+)$  and $H(\wt C^{\otimes +})$ modulo $\locring$-torsion. Now $[y]$, $[y']$ are generators of the free modules $H(\wt C^+)/\text{Tor}_{\locring}$, $H(\wt C'^+)/\text{Tor}_{\locring}$ respectively, and $[y^\otimes ]$ is a generator of $H(\wt C^{\otimes +})/\text{Tor}_{\locring}$. This last point, together with \eqref{eq:proofofsubadditivityofstilde}, computes \eqref{eq:additivityofstilde}.
\end{proof}

Now, if $K$ is a knot in an integer homology $3$-sphere $Y$, we define
\begin{equation}\label{eq:stildedefn}
	\wt s(Y,K) := \wt s\left( \wt C(Y,K;\Delta_{\locring})\right).
\end{equation}
which by construction induces a map from the homology concordance group to $\Z$. In particular, restricting to knots in the $3$-sphere induces a concordance invariant
\begin{equation}\label{eq:stildeconcordancehom}
	\wt s: \mathcal{C}\to \Z
\end{equation}
which is a homomorphism by \Cref{thm:global2} and \Cref{prop:additivityofstilde}.

\vspace{1mm}

\begin{rem}
	The construction of $\wt s$ does not require the completion of a ring: we could have simply used $\Q[T^{\pm 1}]$ localized at $T-1$. We take the completion as a matter of convenience, to more easily align our constructions with those of \cite{KM13}, as is done in the proof of \Cref{coincide s}. Similar remarks hold for the invariants defined in the next section. See also the discussion of coefficients \Cref{subsec:recoveringssharp}.
\end{rem}

\vspace{1mm}
 
 \begin{rem}
The reader may find the choice of coefficient ring $\locring$ and the element $\Lambda$ used in this section mysterious.
	The definition of $\wt s$, together with \Cref{definite ineq for stilde} and \Cref{stilde tensor}, in fact work more generally for $\mathcal{S}$-complexes over an integral domain $R$ localized at a prime principal ideal, such as $(\Lambda)$ above. The particular choice of $R=\locring$ with its prime ideal $(\Lambda)$ is motivated by \Cref{thm:global2}, which says that the corresponding instanton homology for knots in $S^3$ satisfies \Cref{rank1}.
\end{rem}
 
\subsection{Kronheimer and Mrowka's $s^{\sharp}$-invariant}\label{section: Kronheimer-Mrowkainvariant}

We now define local equivalence invariants $s^\sharp_+$ and $s^\sharp_-$ that induce maps
\begin{equation*}
s^{\sharp}_\pm : \Theta^{\mathcal{S}}_{\locring} \to \Z
\end{equation*}
Similar to the case of $\wt s$, these are defined using special cycles. We write the sum as $s^\sharp := s^\sharp_++s^\sharp_-$. Applying these constructions to the $\mathcal{S}$-complex of a knot defines $\Z$-valued homology concordance invariants. For knots $K$ in the $3$-sphere, we show that our $s^\sharp$ recovers Kronheimer and Mrowka's invariant from \cite{KM13}, and that $s^\sharp_+$ and $s_-^\sharp$ recover Gong's refinements of $s^\#(K)$ from \cite{Gong21}. 

\subsubsection{$s^{\sharp}$-invariants of $\mathcal{S}$-complexes}\label{subsec:ssharps}

Let $\wt{C}$ be an $\mathcal{S}$-complex over $\locring$
and $h:=h(\wt{C})$.
Recall that $\Lambda^n \in J_h(\wt{C})$ for $n\gg 0$.
 \begin{defn}\label{Def of s for us}
We define $s^{\sharp}_{\pm}(\wt{C})$ as follows:
 \begin{align}
s^{\sharp}_+(\wt{C}) 
&:=  \min 
\left\{ n-m  \ \middle|  \Lambda^n \in J_h(\wt{C}) ,\, \ \xi^{\sharp}_+(\Lambda^n)=\Lambda^m y_+  \,\text{ for some }\, y_+\in H_{2h}(C^{\sharp})^{\bf q} \right\}
\\[2mm]
s^{\sharp}_-(\wt{C}) 
&:=  \min 
\left\{ n-m  \ \middle|  \Lambda^n \in J_h(\wt{C}) ,\, \ \xi^{\sharp}_-(\Lambda^n)=\Lambda^m   y_-  \,\text{ for some }\, y_-\in H_{2h-2}(C^{\sharp})^{\bf q} \right\}
\end{align}
Here 
$\xi^{\sharp}_{\pm} \colon J_h(\wt{C}) \to 
H_*(C^{\sharp})^{\bf q}
$ are the homomorphisms in \eqref{sharp xi 1}--\eqref{sharp xi 2}. 
The {\it $s^{\sharp}$-invariant} of
$\wt{C}$ is defined by
\[
s^{\sharp}(\wt{C}):= s^{\sharp}_+(\wt{C}) + s^{\sharp}_-(\wt{C}).
\]
 \end{defn}
 
 \begin{rem}Our description in \cref{Def of s for us} of the $s^\sharp$-invariant, using the divisibility of certain canonical classes, is reminiscent of an alternative description of the Rasmussen invariant given in \cite{Sa20}.
\end{rem}

The equalities 
$x \cdot \xi^{\sharp}_+(\Lambda^n)=\xi^{\sharp}_-(\Lambda^n)$
and
$\tfrac{1}{4}x \cdot \xi^{\sharp}_-(\Lambda^n)=  \xi^{\sharp}_+(\Lambda^{n+2})$
imply the inequality
\begin{equation} \label{ssharp plus and minus}
0
\leq s^{\sharp}_+(\wt{C})-s^{\sharp}_-(\wt{C})
\leq 2.
\end{equation}

\begin{prop}\label{definite ineq for ssharp}
Let $\wt{C}$ and $\wt{C}'$ be $\mathcal{S}$-complexes over
$\locring$ with $i:= h(\wt{C}')-h(\wt{C}) \geq 0$.
If there exists a height $i$ morphism
$\wt{\lambda} \colon \wt{C} \to \wt{C}'$ with $c_i= a \Lambda^k$ for some unit $a \in \locring$, we have
 \[
 s^{\sharp}_{+} (\wt{C}') \leq  s^{\sharp}_{+} (\wt{C}) + k
\quad \text{and} \quad 
s^{\sharp}_{-} (\wt{C}') \leq  s^{\sharp}_{-} (\wt{C}) + k.
 \]
 In particular, if $\wt{\lambda}$ is a local map
 (or equivalently, $i=0$ and $k=0$), 
 then  $ s^{\sharp}_{\pm} (\wt{C}') \leq  s^{\sharp}_{\pm} (\wt{C})$
 for each sign.
 \end{prop}
\begin{proof}
The proof is the same as \Cref{definite ineq for stilde},
except we use \Cref{sharp xi height i}
instead of \Cref{tilde xi height i}.
\end{proof} 

\begin{prop}\label{ssharp tensor general}
For any two $\mathcal{S}$-complexes $\wt{C}$ and 
$\wt{C}'$
over $\locring$, we have the inequalities
\[
s^{\sharp}_+ (\wt{C}\otimes \wt{C}' ) \leq 
\min
\left\{
\begin{aligned}
&s^{\sharp}_+(\wt{C}) + s^{\sharp}_+(\wt{C}'), \\
&s^{\sharp}_-(\wt{C}) + s^{\sharp}_-(\wt{C}') +2 
\end{aligned}
\right\}
\quad \text{and} \quad 
s^{\sharp}_- (\wt{C}\otimes \wt{C}' ) \leq 
\min
\left\{
\begin{aligned}
s^{\sharp}_+(\wt{C}) + s^{\sharp}_-(\wt{C}'), \\
s^{\sharp}_-(\wt{C}) + s^{\sharp}_+(\wt{C}') 
\end{aligned}
\right\}.
\]
\end{prop}
\begin{proof}
Let $s_{\pm}:=s^{\sharp}_{\pm}(\wt{C})$,
$s'_{\pm}:=s^{\sharp}_{\pm}(\wt{C}')$,
and $C^{\otimes\sharp}_*$
denote the framed complex of the tensor product of $\wt{C}$ and $\wt{C}'$.
Then there exist elements
\[
y_{\pm} \in H_{2h(\wt{C}) -1 \pm 1}(C^{\sharp})^{\bf q}
\quad \text{and} \quad
y'_{\pm} \in H_{2h(\wt{C}') -1 \pm 1}(C'^{\sharp}_*)^{\bf q}
\]
such that $\Lambda^{n_{\pm}-s_{\pm}} y_{\pm} = \xi^{\sharp}_{\pm}(\Lambda^{n_{\pm}})$
and $\Lambda^{n'_{\pm}-s'_{\pm}} y'_{\pm} =
\xi^{\sharp}_{\pm}(\Lambda^{n'_{\pm}})$.
By \Cref{sharp xi tensor},
we have:
\begin{align*}
\Lambda^{n_+ + n_+'-s_+ - s'_+} T^{\sharp}_* \circ t^{\sharp}( y_+ \otimes y'_+)
&=
\xi^{\sharp}_+(\Lambda^{n_++n'_+})\\[2mm]
\Lambda^{n_+ + n_-'-s_+ - s'_-} T^{\sharp}_* \circ t^{\sharp}( y_+ \otimes y'_-)
&=
\xi^{\sharp}_-(\Lambda^{n_++n'_-})
\\[2mm]
\Lambda^{n_- + n_+'-s_- - s'_+} T^{\sharp}_* \circ t^{\sharp}( y_- \otimes y'_+)
&=
\xi^{\sharp}_-(\Lambda^{n_-+n'_+})\\[2mm]
\tfrac{1}{4}\Lambda^{n_- + n_-'-s_- - s'_-} T^{\sharp}_* \circ t^{\sharp}( y_- \otimes y'_-)
&=
\xi^{\sharp}_+(\Lambda^{n_-+n'_-+2}).
\end{align*}
These complete the proof.
\end{proof}

To proceed, we impose \Cref{rank1}, which in the current setting has the following characterization.

\begin{lem}\label{lem:eqofassumpt}
	\Cref{rank1} for an $\mathcal{S}$-complex $\wt C$ is equivalent to the condition that $H_k(C^{\sharp})$ has rank 1 if $k\pmod{4}$ is even, 
and is zero if $k\pmod{4}$ is odd.
\end{lem}

\begin{proof}
	The equivalence follows from the short exact sequence of $R$-chain complexes in \eqref{sharp and twisted}, which splits over the field of fractions of $R=\locring$.
\end{proof}

It follows from \Cref{lem:eqofassumpt} and the injectivity of the maps $\xi^\sharp_\pm$ from \Cref{prop:sharpspecialcyclemap} that under \Cref{rank1}, the natural quotient map $H(C^\sharp)\to H(C^\sharp)^{\bf q}$ is an isomorphism mod $\locring$-torsion.  In particular, $H(C^\sharp)/\text{Tor}_\locring$ can be used as the homology group appearing in the definitions of $s^\sharp_\pm$. This fact is used in the proof of the following.

\begin{prop} \label{ssharp dual}
Let $\wt{C}^\dagger$ be the dual of an $\mathcal{S}$-complex $\wt{C}$ over $\locring$.
If $\wt{C}$ satisfies \Cref{rank1},
then
 \[
 s^{\sharp}_{+} (\wt{C}^\dagger) =  -s^{\sharp}_{-} (\wt{C})
\quad \text{and} \quad 
s^{\sharp}_{-} (\wt{C}^\dagger) =  -s^{\sharp}_{+} (\wt{C}).
 \]
\end{prop}

\begin{proof}
Here we prove
$s^{\sharp}_{+} (\wt{C}^\dagger) =  -s^{\sharp}_{-} (\wt{C})$.
Let $s:=s^{\sharp}_{-}(\wt{C})$ and
$s^{\dagger}:=s^{\sharp}_{+}(\wt{C}^\dagger)$, 
and let $C^{\dagger\#}$ denote the unreduced framed complex of $\wt{C}^\dagger$.
By the assumption,
there exist generators 
\[y_- \in H_{2h-2}(C^{\sharp})/\Tor_{\locring}
\quad \text{and} \quad
y^\dagger_+ \in H_{-2h}(C^{\dagger\#}_*)/\Tor_{\locring}\]
such that $\Lambda^{n-s} y_- = \xi_-^{\sharp}(\Lambda^n)$
and $\Lambda^{n'-s^\dagger} y^\dagger_+ =
\xi_+^{\sharp}(\Lambda^{n'})$.
By \Cref{duality for sharp}, $H_{-2h}(C^{\dagger\sharp})/\Tor_{\locring}$ is isomorphic to $\Hom(H_{2h-2}(C^{\sharp}),\locring)$.
Under this identification, we see that 
$\langle y^\dagger_+, y_- \rangle = a$ for some unit $a \in \locring$,
and we have
\[
\langle \xi_+^{\sharp}(\Lambda^{n'}),\ \  \xi_-^{\sharp}(\Lambda^n) \rangle
= \langle \Lambda^{n'-s^\dagger} y^\dagger_+,\ \  
\Lambda^{n-s} y_- \rangle
= a \Lambda^{n+n'-s -s^\dagger}.
\]
On the other hand, \Cref{sharp xi dual} gives
\[
\langle \xi_+^{\sharp}(\Lambda^{n'}),\ \  \xi_-^{\sharp}(\Lambda^n) \rangle
= -2\Lambda^{n+n'}.
\]
These imply $s^{\sharp}_-(\wt{C})+s^{\sharp}_+(\wt{C}^\dagger)
=s+s^\dagger=0$.
Similarly, we can prove 
$s^{\sharp}_+(\wt{C})+s^{\sharp}_-(\wt{C}^\dagger)=0$.
\end{proof}

 \begin{prop}\label{ssharp tensor}
For $\mathcal{S}$-complexes $\wt{C}$ and 
$\wt{C}'$
over $\locring$
satisfying \Cref{rank1}, we have:
\begin{align}
\label{ssharpplus tensor}
\max
\left\{
\begin{aligned}
s^{\sharp}_+(\wt{C}) + s^{\sharp}_-(\wt{C}'), \\
s^{\sharp}_-(\wt{C}) + s^{\sharp}_+(\wt{C}') 
\end{aligned}
\right\}
&\leq s^{\sharp}_+ (\wt{C}\otimes \wt{C}' ) \leq 
\min
\left\{
\begin{aligned}
&s^{\sharp}_+(\wt{C}) + s^{\sharp}_+(\wt{C}'), \\
&s^{\sharp}_-(\wt{C}) + s^{\sharp}_-(\wt{C}') +2 
\end{aligned}
\right\}
\\[3mm]
\label{ssharpminus tensor}
\max
\left\{
\begin{aligned}
&s^{\sharp}_-(\wt{C}) + s^{\sharp}_-(\wt{C}'), \\
&s^{\sharp}_+(\wt{C}) + s^{\sharp}_+(\wt{C}') -2 
\end{aligned}
\right\}
&\leq s^{\sharp}_- (\wt{C}\otimes \wt{C}' ) \leq 
\min
\left\{
\begin{aligned}
s^{\sharp}_+(\wt{C}) + s^{\sharp}_-(\wt{C}'), \\
s^{\sharp}_-(\wt{C}) + s^{\sharp}_+(\wt{C}') 
\end{aligned}
\right\}
\end{align}
\end{prop}

\begin{proof}
Combining \Cref{ssharp dual} with \Cref{rem:assumprank1underproducts}, we have 
\[
s^{\sharp}_+(\wt{C}\otimes \wt{C}' )
= - s^{\sharp}_-(\wt{C}^\dagger\otimes \wt{C}'^\dagger)
\geq 
- \min
\left\{
\begin{aligned}
s^{\sharp}_+(\wt{C}^\dagger) + s^{\sharp}_-(\wt{C}'^\dagger), \\
s^{\sharp}_-(\wt{C}^\dagger) + s^{\sharp}_+(\wt{C}'^\dagger) 
\end{aligned}
\right\}
=
\max
\left\{
\begin{aligned}
s^{\sharp}_+(\wt{C}) + s^{\sharp}_-(\wt{C}'), \\
s^{\sharp}_-(\wt{C}) + s^{\sharp}_+(\wt{C}') 
\end{aligned}
\right\}.
\]
The remaining part is proved similarly.
\end{proof}

 Now we prove 
 the $\mathcal{S}$-complex version of
 \Cref{quasi-additivity}.
 (Note that for proving \Cref{quasi-additivity},
 we need  the agreement of 
 $s^{\sharp}\big(\wt{C}(S^3, K;\locring)\big)$
 with Kronheimer and Mrowka's $s^{\sharp}(K)$, discussed below.)
 \begin{thm}\label{thm:quasiaddscomplexes}
For two $\mathcal{S}$-complexes $\wtC$ and $\wt{C}'$ 
over $\locring$ satisfying \Cref{rank1}, we have 
\[
| s^{\sharp} (\wt{C}\otimes \wt{C}') - s^{\sharp} (\wt{C}) - s^{\sharp} (\wt{C}' )| \leq 1.
\]
\end{thm}
\begin{proof}
As a consequence of \Cref{ssharp tensor}, we have the inequalities:
\begin{align*}
s^{\sharp}(\wt{C})+ s^{\sharp}(\wt{C}')
&\leq 2s^{\sharp}_+(\wt{C} \otimes \wt{C}')
\leq s^{\sharp}(\wt{C})+ s^{\sharp}(\wt{C}') + 2\\[2mm]
-2+s^{\sharp}(\wt{C})+ s^{\sharp}(\wt{C}')
&\leq 2s^{\sharp}_-(\wt{C} \otimes \wt{C}')
\leq s^{\sharp}(\wt{C})+ s^{\sharp}(\wt{C}')
\end{align*}
These gives the desired inequality.
\end{proof}

\subsubsection{Recovering Kronheimer and Mrowka's $s^{\sharp}$}\label{subsec:recoveringssharp}
 
Write $s^\sharp(K)$ for Kronheimer and Mrowka's concordance invariant from \cite{KM13}, and $s^\sharp_\pm(K)$ for the refinements satisfying $s^\sharp(K)=s_+^\sharp(K)+s_-^\sharp(K)$, studied by Gong in \cite{Gong21}. Here we prove:
 
 \begin{thm}\label{coincide s}
 For any knot $K$ in $S^3$, we have 
 \[
 s^{\sharp}_\pm (K) =  
 s^{\sharp}_{\pm}\big(\wt{C}(S^3, K;\Delta_{\locring})\big).
 \]
 Consequently, $s^{\sharp}(K)$ agrees with $ s^{\sharp}\big(\wt{C}(S^3, K;\Delta_{\locring})\big)$.
 \end{thm}

Before giving the proof, we review the definition of the concordance invariants $ s^{\sharp}_\pm (K)$. The main ingredient in the definition is the framed version $I^{\sharp}(K)$ of singular instanton Floer homology defined in \cite{KM13}. First, we give some remarks on coefficient rings. 

In \cite{KM13}, $I^\sharp(K)$ is first defined over the coefficient ring $\Q[u^{\pm 1}]$, and then over the ring $\Q[\![\lambda]\!]$. This latter ring is identified with the completion of the ring $\Q[u^{\pm 1}]$ localized at the prime ideal $(u-1)$. This completion is taken so that there is a square-root for the expression
\begin{equation}
\label{eq:kmrelu}u+u^{-1}-2
\end{equation} 
in the ring, one of which corresponds to $\lambda$. Our variable $T$ is related to their $u$ by $u=T^2$, and our $\Lambda=T-T^{-1}$, which is already a square-root of \eqref{eq:kmrelu} (so that a completion is in fact unnecessary), replaces the role of $\lambda$. More precisely, there is a natural map from $\Q[u^{\pm 1}]_{(u-1)} = \Q[T^{\pm 2}]_{(T^2-1)}$ to $\Q[T^{\pm 1}]_{(T-1)}$, which induces an isomorphism on completions, and which in turn is used to identify $\Q[\![\lambda]\!]$ with $\locring$. All homology groups $I^\sharp(K)$ below will be taken with the coefficient ring $\locring$. We also write $I^\sharp(K)'=I^\sharp(K)/\text{Tor}_{\locring}$.

In \cite{KM13}, it is proved that for any knot $K$ in the $3$-sphere, $I^{\sharp}(K)$ has rank $2$ over $\locring$ in degrees $1$ and $-1$ $\pmod{4}$. 
In the case of the unknot $U_1$, these generators are respectively called $u_+$ and $u_-$. Next, take a normally immersed oriented surface cobordism $S$ from $U_1$ to $K$ in $I \times S^3$ with genus $g$ and $s_+ $ positive double points. Then, $S$ induces a map 
\begin{equation}\label{cob-map-sharp}
  m^{\sharp} (S): \locring\langle u_+\rangle  \oplus \locring \langle u_+\rangle = I^{\sharp}(U_1 )'\to I^{\sharp}(K )'.
\end{equation}
Then the invariants $s^{\sharp}_\pm (K)$ defined in \cite{Gong21} are given by 
\begin{align*}
&s^{\sharp}_+ (K) :=g +  s_+ - m_+(S), \hspace{.4cm} s^{\sharp}_- (K) := g +s_+ - m_-(S), 
\end{align*}
where $m_\pm (S)$ are the maximal non-negative integers satisfying, for some $y_\pm\in I^\sharp (K)'$ in degree $\pm 1$:
\begin{align*}
m^{\sharp}(S)u_ +  & =  \begin{cases}\Lambda^{m_+(S)} y_+  \text{ if $g$ is even } \\ 
\Lambda^{m_+(S)} y_-  \text{ if $g$ is odd }
\end{cases} \\ 
m^{\sharp}(S)u_-  & =  \begin{cases}\Lambda^{m_-(S)} y_-  \text{ if $g$ is even } \\ 
\Lambda^{m_-(S)} y_+  \text{ if $g$ is odd }
\end{cases}
\end{align*}
The invariant $s^\sharp(K)$ from \cite{KM13} is defined to be the sum $s^\sharp_+(K)+s^\sharp_-(K)$.

\begin{prop}[\cite{KM13}]\label{mirror s}
For any knot $K$ in $S^3$, we have 
\[
s^{\sharp}_\pm (K^* ) = - s^{\sharp}_\mp (K ). 
\]
\end{prop}

\begin{proof}
We first take an embedded cobordism $S$ from $U$ to $K$ appearing in the definition of $s^{\sharp}_\pm(K)$. We obtain a cobordism $S^*$ from $K^*$ to $U$ which is the same as $S$, but the incoming and the outgoing ends are switched. This gives a cobordism map $m^\sharp(S^*)$ from $I^{\sharp}(K^*)'$ to $I^{\sharp}(U_1 )'$. As is explained in \cite[Lemma 3.2]{KM13}, this map is the dual of the map \eqref{cob-map-sharp} where the generators  $y_{\pm}^*$ of $I^{\sharp}(K^*)'$ in degrees $\pm 1$ are identified with the dual basis using the relation $\langle y_{\pm}^*,y_{\mp}\rangle=1$. Also, we take an embedded cobordism $T$ from $U$ to $K^*$, and for simplicity, we assume that $g(S)$ and $ g(T)$ are even.
Then, since the composition $S^* \circ T$ is a cobordism from the unknot to itself, 
\[
0=s^{\sharp}_\pm (U )  = g (S^* \circ T)  - m_\pm (S^* \circ T). 
\]
Now, $m_\pm  (S^* \circ T)= m_\pm (T) + m_\mp (S) $. We then compute:
\[
    0=s^{\sharp}_{\pm} (U )  = g(T) + g(S) - m_{\pm} (T) -  m_{\mp} (S) = s^{\sharp}_{\mp} (K) + s^{\sharp}_{\pm} (K^*). \qedhere
\]
 \end{proof} 
For a knot $K\subset S^3$, write $C^\sharp(K)$ for the unreduced framed complex associated to the $\mathcal{S}$-complex $\wt C(S^3,K;\Delta_\locring)$. Analogous to $I^\sharp(K)'$, the quotient of $H(C^\sharp(K))$ by $\text{Tor}_{\locring}$ is denoted by $H(C^\sharp(K))'$.

 \begin{lem}\label{C sharp map}
 Let $S:U\rightarrow K$ be an orientable surface cobordism in $I\times S^3$ with genus zero and $s_{+}$ positive double points. Suppose $\sigma(K)\leq 0$. Then $S$ induces a homomorphism 
 \begin{equation}\label{eq:morphismunreducedframedssharp}
  \lambda^\sharp_S : H(C^{\sharp}(U) )' \rightarrow H(C^{\sharp}(K) )'
  \end{equation}
  This homomorphism sends the generators $u_\pm$ to $\xi^\sharp_\pm(\Lambda^{s^+})$.

 \end{lem}
 \begin{proof}
 This follows from \cref{sharp xi height i}, and the fact that such a cobordism induces a morphism of height $h(S^3,K)=-\sigma(K)/2\geq 0$, as explained in \Cref{review-S-comp}; see \eqref{eq:immersedheight}. Note that since \Cref{rank1} is satisfied for knots in the $3$-sphere, in this setting $H(C^\sharp)^{\bf q}/\text{Tor}_\locring$ is isomorphic to $H(C^\sharp)'$.
 \end{proof}
 Now, we give a proof of \cref{coincide s}: 
 \begin{proof}[Proof of \cref{coincide s}]
From \cref{mirror s},  one can assume $h(S^3, K)\geq 0$.
Let $S: U_1\rightarrow K$ be a negative definite cobordism with genus $0$.
Since $S$ also induces the following commutative diagram: 
\[
  \begin{CD}
     I^\sharp(U )' @>{m^{\sharp} (S)}>> I^\sharp(K )' \\
  @V{\cong}VV    @V{\cong}VV \\
      H(C^{\sharp} (U) )'  @>{\lambda^{\sharp}_S}>>  H(C^{\sharp} (K ))'
  \end{CD}
\]
The vertical isomorphisms are from \Cref{thm:recoverisharp}. By definition, we have $s^{\sharp}_\pm (K) =  s_+ - m_\pm$,
where $m_\pm$ are the maximal non-negative integers such that $m^{\sharp} (S) u_\pm = \Lambda^m y_\pm$ for some elements $y_\pm$ in $I^\sharp(K)'$. 
Lemma \ref{C sharp map} implies 
\[
 \Lambda^{m_\pm} y_\pm= m^{\sharp} (S)u_{\pm}= \lambda^\sharp_S u_{\pm} = \xi^\sharp_{\pm} (\Lambda^{s^+}),  
 \]
 and hence we obtain
\[
{s}^\sharp_{\pm}(  \wt{C}(S^3,K; \Delta_{\locring}))= s^+ -m_\pm=   s^{\sharp}_{\pm}(K). \qedhere
\]
 \end{proof}

The following is a restatement of \Cref{quasi-additivity} from the introduction.

\begin{thm}
	 For any pair of knots $K$ and $K'$ in the $3$-sphere, we have
	\[
	  \vert s^\sharp(K\#K')-s^\sharp(K)-s^\sharp(K')\vert\leq 1.
	\]
\end{thm}

\begin{proof}
	This follows from \Cref{coincide s} and \Cref{thm:quasiaddscomplexes}.
\end{proof}

\noindent Similarly, the inequalities of \Cref{subsec:ssharps} give rise to connected sum inequalities for $s^\sharp_\pm$.

\subsection{Relationship between $\wt{s}$ and $s^{\sharp}$}\label{Alternative description of s}

We next study the relationship between the concordance homomorphism $\wt s$ and the invariants $s^\sharp_\pm$. In the course of doing so, we establish that $2\wt s$ is a slice-torus invariant, in the sense of \cite{Le14}. We begin by studying the relationship between these invariants at the level of $\mathcal{S}$-complexes.

\begin{prop}\label{ssharp and stilde}
For any $\mathcal{S}$-complex $\wt{C}$ over $\locring$ satisfying \Cref{rank1},
we have the inequalities
\begin{equation} \label{stilde inequality}
\max \{ s^{\sharp} _+ ( \wt{C}) -1 , s^{\sharp}_- ( \wt{C})\}
\leq
\wt{s}(\wt{C}) \leq \min \{ s^{\sharp} _+ ( \wt{C})  , s^{\sharp}_- ( \wt{C})+1\} . 
\end{equation}
In particular, we have
\[
|s^{\sharp}(\wt{C})-2\wt{s}(\wt{C})|\leq 1.
\]
 \end{prop}
\begin{proof}
By \Cref{stilde tensor} and \Cref{ssharp dual}, we only need to prove the right-hand inequality in \eqref{stilde inequality}.

Recall the map $\pi_{+} \colon C^{\sharp} \to \wt C^+$ in \eqref{pi}, defined over $\locring$ by the formula $\left[1 \;\; \Lambda \right]$ with respect to the decomposition of  $\wt C^+$. Regard a special $(h(\wt{C}), \Lambda^n)$-cycle in $\lhc$ as a polynomial $Q(x)$
whose coefficents are chains of $\wt{C}$, and 
consider the decomposition $Q(x) = Q'(x^2)+xQ''(x^2)$.
Using \eqref{sharp action}, we have:
\begin{align*}
(\pi_+)_*(\xi^{\sharp}_+(\Lambda^n))
&=
(\pi_+)_* \circ \lk^{\sharp}_*([Q'(x^2)+xQ''(x^2)])
= (\pi_+)_*
\left(\left[(Q'(4\Lambda^2),-2 Q''(4\Lambda^2))\right]\right)
\\[2mm]
&=
\left[Q'(4\Lambda^2)
-2\Lambda Q''(4\Lambda^2)
\right]
= \wt{\lk}_* \left( \left[Q'(x^2)
+ x Q''(x^2)
\right]
\right)
=\wt{\xi}(\Lambda^n)\\[6mm]
(\pi_+)_*(\xi^{\sharp}_-(\Lambda^n))
&=
(\pi_+)_* \circ \lk^{\sharp}_*([xQ'(x^2)+x^2Q''(x^2)])
= (\pi_+)_*
\left(\left[(4\Lambda^2Q''(4\Lambda^2),-2 Q'(4\Lambda^2))\right]\right)
\\[2mm]
&=
\left[4\Lambda^2Q''(4\Lambda^2)
-2\Lambda Q'(4\Lambda^2)
\right]
= 
-2\Lambda \wt{\lk}_* \left( \left[Q'(x^2)
+ x Q''(x^2)
\right]
\right)
=
-2 \wt{\xi}(\Lambda^{n+1})
\end{align*}
These complete the proof.
\end{proof}
We give two applications of \Cref{ssharp and stilde}.
The first is a description of 
$\wt{s}$ as a limit.
\begin{prop}\label{ssharp limit}
For any $\mathcal{S}$-complex $\wt{C}$
over $\locring$ satisfying \Cref{rank1}, 
we have 
\[
\wt{s}(\wt{C})
= 
\dfrac{1}{2}
\left(
\lim_{n \to \infty } 
\dfrac{s^{\sharp} (\wt{C}^{\otimes n} )}{n} 
\right)
\]
where $\wt{C}^{\otimes n}$ denotes the
tensor product of $n$ copies of $\wt{C}$.
\end{prop}
\begin{proof}
This immediately follows from the inequality
\[
|s^{\sharp}(\wt{C}^{\otimes n}) - n \cdot 2\wt{s}(\wt{C})| \leq 1
\]
given by 
\Cref{prop:additivityofstilde} and \Cref{ssharp and stilde}.
\end{proof}

As the $\mathcal{S}$-complexes of knots in the $3$-sphere satisfy \Cref{rank1}, we obtain the following.

\begin{cor}\label{cor:stildefromsharpknots}
	For a knot $K$ in the $3$-sphere, we have
	\[
		\wt s(K) = \frac{1}{2} \left( \lim_{n\to \infty } \frac{s^\sharp(\#_n K)}{n}\right).
	\]	
\end{cor}

This description allows us to
prove the slice-torus property
\cite{Le14}
of $2\wt{s}$,
where 
$\wt{s} \colon \mathcal{C} \to \Z$ is the concordance homomorphism of \eqref{eq:stildeconcordancehom}. Write $g_4(K)$ for the smooth $4$-ball genus of $K$.
\begin{thm}\label{sslice-torus}
The map 
$\wt{s} \colon \mathcal{C} \to \Z$ is half a slice-torus invariant, i.e. $\wt{s}$ satisfies the following properties: 
\begin{itemize}
    \item[(i)] $\wt{s}(K)$ is a homomorphism;  
    \item[(ii)] $\wt{s}(K) \leq g_4(K)$; 
    \item[(iii)] $\wt{s}(T_{p,q}) = g_4(T_{p,q})$ where $T_{p,q}$ denotes the $(p,q)$ torus knot. 
\end{itemize}
The equality $\wt{s}(K) = g_4(K)$ holds more generally for any quasi-positive knot. 
\end{thm}
\begin{proof}
Property (i) immediately follows from 
\Cref{Omega} and
\cref{stilde tensor}. 
Property (ii) follows from \Cref{cor:stildefromsharpknots} and
the genus bound for $s^{\sharp}$ given by
\[
s^{\sharp} (K) \leq 2g_4(K),
\]
which is a direct consequence of Kronheimer and Mrowka's construction of $s^{\sharp}(K)$ in \cite{KM13}. Indeed,
\[
g_4(K) \geq
\frac{g_4(\#_n K)}{n}  \geq
\frac{1}{2}
\left(
\frac{s^{\sharp} (\#_n K)}{n} 
\right)
\to 
\wt{s}(K)
\quad
(n \to \infty).
\]
Property (iii), and the last statement regarding quasi-positive knots, follow from 
\cref{cor:stildefromsharpknots} and \cite[Proposition 1.8]{Gong21}: for any quasi-positive knot, 
\begin{equation} \label{Gong's inequalities}
2g_4(K)-1 \leq s^{\sharp}(K) \leq 2g_4(K).
\end{equation}
Indeed, since the connected sum of two quasi-positive knots is also quasi-positive,
and $g_4$ is additive among quasi-positive knots \cite{Rud93},
the inequalities \eqref{Gong's inequalities} imply 
\[
2g_4(K)-\frac{1}{n} \leq 
\frac{s^{\sharp}(\#_n K)}{n} \leq 
2g_4(K)
\]
for any $n \in \Z_{>0}$.
Thus, we have
$\wt{s}(K)=
\frac{1}{2} \left( \displaystyle \lim_{n\to\infty} \frac{s^{\sharp}(\#_n K)}{n} \right) = g_4(K)$.
\end{proof}
Using \cite[Corollary 5.9]{Le14}, we also obtain the following. Note that this result together with \Cref{sslice-torus} proves \Cref{s-tilde-values} from the introduction.
\begin{cor}\label{alternating stilde}
For an alternating knot $K$, we have 
\[
\wt{s}(K) = -\frac{1}{2}\sigma (K). 
\]
In particular, we have 
\[
| s^{\sharp}(K) + \sigma(K)|\leq 1. 
\]
\end{cor}

As the second application of
\cref{ssharp and stilde},
we give a new $\{-1, 0 , 1 \}$-valued concordance invariant $\wt{\varepsilon}$,
which behaves similar to Hom's $\varepsilon$-invariant \cite[Definition 3.4]{Hom14cable} in Heegaard Floer theory.
\begin{defn}
For an $\Sc$-complex $\wt{C}$ over $\locring$,
we define
\[
\wt{\varepsilon}(\wt{C}):= 2\wt{s}(\wt{C})- 
s^{\sharp}(\wt{C}).
\]
By \Cref{ssharp and stilde}, if the $\mathcal{S}$-complex $\wt C$ satisfies \Cref{rank1}, then $\wt{\varepsilon}(\wt C)\in \{-1,0,1\}$. For any knot $K$ in an integer homology 3-sphere $Y$, we define
\[
\wt{\varepsilon}(Y,K):= \wt{\varepsilon}\big(\wt{C}(Y, K;\Delta_{\locring})\big).
\]
In the case that $Y=S^3$, we abbreviate
$\wt{\varepsilon}(S^3,K)$ to $\wt{\varepsilon}(K)$. In particular, we have
\[
	\wt\varepsilon(K)\in \{-1,0,1\}.
\]

\begin{proof}[Proof of \Cref{s-tilde-slice-torus}]
	The result follows from \Cref{sslice-torus} and the above properties of $\wt \varepsilon$.
\end{proof}

\end{defn}
We show the following properties of $\wt{\varepsilon}$
 analogous to \cite[Proposition 3.6]{Hom14cable}.
\begin{prop} \label{epsilon analogy}
The invariant $\wt{\varepsilon}$ satisfies the following properties for knots in $S^3$:
\begin{itemize}
\item[(i)] if $K$ is smoothly slice, then $\wt{\varepsilon}(K)=0$;
\item[(ii)] $\wt{\varepsilon}(-K)=-\wt{\varepsilon}(K)$;
\item[(iii)] 
\begin{itemize}
\item[(a)]
if $\wt{\varepsilon}(K) = \wt{\varepsilon}(K')$, then $\wt{\varepsilon}(K\#K') = \wt{\varepsilon}(K) = \wt{\varepsilon}(K')$;
\item[(b)]
if $\wt{\varepsilon}(K) = 0$, then $\wt{\varepsilon}(K\#K') = \wt{\varepsilon}(K')$.
\end{itemize}
\end{itemize}
\end{prop}
To prove the proposition, we introduce three types of $\Sc$-complexes over $\locring$
according to the inequalities \eqref{ssharp plus and minus}.
See also the discussion in \cite[Proposition 1.7]{Gong21}.
\begin{defn}
Let $\wt{C}$ be an $\Sc$-complex over $\locring$.
We call $\wt{C}$ {\it Type O}, {\it Type I} and
{\it Type II}
if the value of $s^{\sharp}_+(\wt{C})-s^{\sharp}_-(\wt{C})$ is
equal to 0, 1 and 2
respectively.
\end{defn}
\begin{lem} \label{types lem1}
For an $\Sc$-complex $\wt{C}$ over $\locring$ satisfying \Cref{rank1},
we have the following:
\begin{itemize}
\item[(i)]
$\wt{C}$ is Type O if and only if 
$\wt{s}(\wt{C})=s^{\sharp}_+(\wt{C})=s^{\sharp}_-(\wt{C})$.
\item[(ii)]
$\wt{C}$ is Type I if and only if 
$\wt{s}(\wt{C})=s^{\sharp}_+(\wt{C})=s^{\sharp}_-(\wt{C})+1$
\ or \ 
$\wt{s}(\wt{C})=s^{\sharp}_+(\wt{C})-1=s^{\sharp}_-(\wt{C})$.
\item[(iii)]
$\wt{C}$ is Type II if and only if 
$\wt{s}(\wt{C})=s^{\sharp}_+(\wt{C})-1=s^{\sharp}_-(\wt{C})+1$.
\end{itemize}
In particular, $\wt{\varepsilon}(\wt{C})=0$
if and only if $\wt{C}$ is Type O or Type II.
\end{lem}
\begin{proof}
All assertions directly follow from the inequalities
\eqref{stilde inequality}.
\end{proof}

\begin{lem} \label{types lem2}
Let $\wt{C}$ and $\wt{C}'$ be 
$\Sc$-complexes over $\locring$ satisfying \Cref{rank1}.
\begin{itemize}
\item[(i)] If $\wt{C}$ is Type O, then $s^{\sharp}_+(\wt{C} \otimes \wt{C}') =
s^{\sharp}_+(\wt{C}) + s^{\sharp}_+(\wt{C}')$ and $s^{\sharp}_-(\wt{C} \otimes \wt{C}') =
s^{\sharp}_-(\wt{C}) + s^{\sharp}_-(\wt{C}')$.
\item[(ii)] If $\wt{C}$ is Type II, then
$
s^{\sharp}_+(\wt{C} \otimes \wt{C}') =
s^{\sharp}_+(\wt{C}) + s^{\sharp}_-(\wt{C}')
$ and $
s^{\sharp}_-(\wt{C} \otimes \wt{C}') =
s^{\sharp}_-(\wt{C}) + s^{\sharp}_+(\wt{C}')$.
\end{itemize}
In particular, if $\wt{\varepsilon}(\wt{C})=0$,
then 
$s^{\sharp}(\wt{C} \otimes \wt{C}')=
s^{\sharp}(\wt{C}) + s^{\sharp}(\wt{C}')$.
\end{lem}
\begin{proof}
Suppose that $\wt{C}$ is Type O. Then
\cref{ssharp tensor} gives the following:
\[
s^{\sharp}_+(\wt{C}) + s^{\sharp}_+(\wt{C}') 
=
s^{\sharp}_-(\wt{C}) + s^{\sharp}_+(\wt{C}') 
\leq 
s^{\sharp}_+ (\wt{C}\otimes \wt{C}' ) 
\leq 
s^{\sharp}_+(\wt{C}) + s^{\sharp}_+(\wt{C}'),
\]
\[
s^{\sharp}_-(\wt{C}) + s^{\sharp}_-(\wt{C}')
\leq s^{\sharp}_- (\wt{C}\otimes \wt{C}' ) \leq 
s^{\sharp}_+(\wt{C}) + s^{\sharp}_-(\wt{C}') 
=
s^{\sharp}_-(\wt{C}) + s^{\sharp}_-(\wt{C}').
\]
Similarly, if $\wt{C}$ is Type II, then
\[
s^{\sharp}_+(\wt{C}) + s^{\sharp}_-(\wt{C}') 
\leq 
s^{\sharp}_+ (\wt{C}\otimes \wt{C}' ) 
\leq 
s^{\sharp}_-(\wt{C}) + s^{\sharp}_-(\wt{C}') + 2
=
s^{\sharp}_+(\wt{C}) + s^{\sharp}_-(\wt{C}'),
\]
\[
s^{\sharp}_-(\wt{C}) + s^{\sharp}_+(\wt{C}')
=
s^{\sharp}_+(\wt{C}) + s^{\sharp}_+(\wt{C}') -2
\leq s^{\sharp}_- (\wt{C}\otimes \wt{C}' ) \leq 
s^{\sharp}_-(\wt{C}) + s^{\sharp}_+(\wt{C}').
\]
Finally, \cref{types lem1} completes the proof.
\end{proof}

Now we prove \cref{epsilon analogy}.

\begin{proof}[Proof of \cref{epsilon analogy}]
If $K$ is smoothly slice, then
$\wt{\varepsilon}(K)=2\wt{s}(K)-s^{\sharp}(K)=0$ follows from
the slice genus bounds for $\wt{s}$ and $s^{\sharp}$
(see \cref{s-tilde-slice-torus}).
By \cref{prop:additivityofstilde} and \cref{ssharp dual},
we have property (ii):
\[
\wt{\varepsilon}(-K)
= 2\wt{s}(-K) - s^{\sharp}(-K)
= -(2\wt{s}(K) - s^{\sharp}(K))
= -\wt{\varepsilon}(K).
\]
Let us consider property (iii).
By \cref{types lem2} and property (ii), it suffices to prove the assertion (a) for the case $\wt{\varepsilon}(K)=\wt\varepsilon(K')=1$.
Indeed, \cref{prop:additivityofstilde} and \cref{thm:quasiaddscomplexes} imply
\[
\wt{\varepsilon}(K\#K')
= 2\wt{s}(K \# K') - s^{\sharp}(K\# K')
\geq (2\wt{s}(K)- s^{\sharp}(K)) + (2\wt{s}(K')-s^{\sharp}(K')) -1
=1. \qedhere
\]
\end{proof}

Motivated by \cite[Proposition 3.6]{Hom14cable}, we expect a positive answer to the following question.
\begin{ques}
Does $\wt{\varepsilon}$ satisfy the following properties?
\begin{itemize}
\item[(i)] If $\wt{\varepsilon}(K) = 0$, then $\wt{s}(K) = 0$.
\item[(ii)] If $|\wt{s}(K)| = g(K)$, where $g(K)$ is the Seifert genus of $K$, then 
$\wt{\varepsilon}(K) = \sgn\wt{s}(K)$.
\item[(iii)] If $K$ is homologically thin in the sense of Heegaard Floer theory, then $\wt{\varepsilon}(K) = \sgn\wt{s} (K)$.
\end{itemize}
\end{ques}
We expect that one can define a corresponding homologically thinness for instanton theory in terms of $\wt{C}$. Item (iii) is proved below for the special case of two-bridge knots in \Cref{prop:twobridgessharp}, and the proof gives some indication of how one might define homological thinness in terms of $\mathcal{S}$-complexes.

\subsection{Fractional ideal invariants from instantons}\label{subsec:ideals}

Several concordance invariants for knots $K\subset S^3$ were introduced recently by Kronheimer and Mrowka in \cite{KM19b}, using a version of singular instanton homology defined for webs. Building on the work in \cite{DS19}, we define an $R[x]$-submodule $\wh{z}(K)$ inside the field of fractions of $R[x]$ which is constructed from the $\mathcal{S}$-complex of the knot, using special cycles. We show that $\wh{z}(K)$ specializes, upon changing coefficients, to Kronheimer and Mrowka's fractional ideal $z^\natural_\sigma(K)$ from \cite{KM19b}. As a consequence of this relationship, all of the concordance invariants defined in \cite{KM19b} depend only on the local equivalence class of the $\mathcal{S}$-complex $\wt C(K;\Delta)$.

Let $R$ be an integral domain and $M$ be an arbitrary $R$-module. 
For two $R$-submodules $N_{1}$ and $N_{2}$ in $M$, recall that the ideal $[N_{1}: N_{2}]$ is the $R$-submodule of ${\rm Frac}(R)$ given by
\[
[N_{1}:N_{2}]=
\left\{\frac{a}{b}\ \middle|\ aN_{2}\subset bN_{1}\right\}.\]
For $m\in M$, we write $[N_1:m] :=[N_1:N_2]$ where $N_2$ is the submodule generated by $m$. Note that if $m$ is altered by a torsion element, then $[N_1:m]$ is unchanged. For the remainder of this section, $R$ will denote the ring $\Z[T^{\pm 1}]$.

We now review the construction of Kronheimer and Mrowka's ideals $z_\sigma^\natural(K)$ for any knot $K$ in the $3$-sphere. Let $I^\natural(K;\mathscr{S})$ be the instanton homology as defined in \cite{KM19b}, using local coefficients over
\[
	\mathscr{S} = \mathbf{F}[T_1^{\pm 1}, T_2^{\pm 1}, T_3^{\pm 1}]
\]
where $\mathbf{F}$ is the field with two elements. In what follows, we will often ignore torsion. Thus we write
\[
	I^\natural(K;\mathscr{S})' = I^\natural(K;\mathscr{S})/\text{Tor}_\mathscr{S}
\]
Kronheimer and Mrowka show that $I^\natural (K;\mathscr{S})'\cong \mathscr{S}$ as an $\mathscr{S}$-module. Let $S:U_1\to K$ be an orientable surface cobordism in $[0,1]\times S^3$, possibly immersed, with transverse double points. Let $g$ be its genus, and $s_+$ be the number of positive double points. There is an induced $\mathscr{S}$-module homomorphism
\[
	I^\natural(S):I^\natural(U_1;\mathscr{S})'\to I^\natural(K;\mathscr{S})'
\]
With these choices, define $z^\natural(K)\subset \text{Frac}(\mathscr{S})$ as follows:
\begin{equation}\label{eq:znaturaldef}
	z^\natural(K) := P^g (T^2_1-T_1^{-2})^{s_+} [I^\natural(K;\mathscr{S})' : I^\natural(S)(1)]
\end{equation}
Here, the element $P$ is the symmetric Laurent polynomial defined as follows:
\[
	P =T_1T_2T_3+T_1T_2^{-1}T_3^{-1}+T_1^{-1}T_2T_3^{-1}+T_1^{-1}T_2^{-1}T_3
\]
If $\sigma:\mathscr{S}\to \mathscr{S}'$ is a ring homomorphism, we obtain an ideal $z_\sigma^\natural(K)\subset \text{Frac}(\mathscr{S}')$ by base change. Kronheimer and Mrowka prove that $z^\natural_\sigma(K)$ depends only on the knot $K$, and is a concordance invariant.

\vspace{1mm}

\begin{rem}\label{rem:immersedconvention}
For defining cobordism maps of immersed surfaces, we are following the conventions of \cite[\S 2.5]{DS20}. The definition of maps for immersed surfaces in \cite{KM19b} is slightly different and is given as the sum of maps of the form $I^\natural (\overline W,\overline S,c)$ where $W$ is the blow up of $[0,1]\times S^3$ at the double points, $\overline S$ is the proper transform of $S$, and $c$ runs over a set of degree two cohomology classes of exceptional spheres. The definition of $z^{\natural}(K)$ given above, at least when $s_+\neq 0$, is slightly different from the one in \cite{KM19b} because of this difference in convention. However, the above definition gives the same invariant $z^{\natural}(K)$ as in \cite{KM19b} because we can 
use a surface cobordism $S$ with $s_+=0$, and then use the invariance of both definitions from the choice of $S$. 
\end{rem}

The theory $I(K;\mathscr{S})'$ used in the above construction is related to the $\mathcal{S}$-complex of a knot as follows.

\begin{thm}[Corollary 8.41 of \cite{DS19}] \label{thm:inaturalwithbncoeffs} There is a natural isomorphism of $\mathscr{S}$-modules
\begin{equation*}
	  I^{\natural}(K)\cong \wh{I}(K;\Delta)\otimes_{\Z[T^{\pm 1},x]}\mathscr{S}
\end{equation*}
where on the right, the $\Z[T^{\pm 1},x]$-algebra structure of $\mathscr{S}$ is given by reducing coefficients mod $2$, and 
\begin{equation}\label{eq:bnbasechange}
  T\mapsto T_1,\hspace{1cm} x\mapsto P.
\end{equation}
\end{thm}

We now introduce a local equivalence invariant associated to any $\mathcal{S}$-complex, which we will show below can be used to recover Kronheimer and Mrowka's invariants $z_\sigma^\natural(K)$.

\begin{defn}\label{defn:zfrom}
Let $R$ be an integral domain,
$\wt{C}$ an $\Sc$-complex over $R$,
$h:=h(\wt{C})$,  $f \in J_{h}(\wt{C})\setminus \{0\}$, and 
$\wh{\xi} \colon J_h(\wt{C}) \to H_{2h}(\lhc)/\Tor_{R[x^2]}
$ the $R$-homomorphism in \eqref{hat xi}.
Define 
\[\wh{z}(\wt{C}):=f[H(\lhc) : \wh{\xi}(f)]
\subset {\rm Frac}(R[x]).
\]

\end{defn}
\begin{prop}
$\wh{z}(\wt{C})$ is independent of the choice of $f \in J_h(\wt{C})$.
\end{prop}
\begin{proof}
This follows from the general property $[N_1:N_2]=r[N_1:rN_2]$ for any $r\in R$.
\end{proof}
\begin{prop}\label{definite ineq for ztilde}
Let $\wt{C}$ and $\wt{C}'$ be $\mathcal{S}$-complexes over
$R$ with $i:= h(\wt{C}')-h(\wt{C}) \geq 0$.
If there exists a height $i$ morphism
$\wt{\lambda} \colon \wt{C} \to \wt{C}'$, then, with $c_i$ as defined in \eqref{eq:cjdefn}, we have
 \[
c_i  \cdot \wh{z} (\wt{C}) \subset  \wh{z} (\wt{C}'). 
 \]
 In particular, if $\wt{\lambda}$ is a local map
 (or equivalently, $i=0$ and $c_i$ is a unit of $R$), 
 then  $  \wh{z} (\wt{C}) \subset  \wh{z} (\wt{C}')$.
 \end{prop}
\begin{proof}
\Cref{special cycle: height i} implies 
$\lhl_*(\wh{\xi}(f))=\wh{\xi}(c_if) \in H_*(\lhc'_*)$.
Hence we have
\begin{align*}
c_i  \wh{z} (\wt{C})
&= c_if \left\{\frac{a}{b}\ \middle|\ 
a \wh{\xi}(f) \in bH_*(\lhc_*)\right\}
\subset
c_if \left\{\frac{a}{b}\ \middle|\ 
a \lhl_*(\wh{\xi}(f)) \in b
\lhl_*(H_*(\lhc_*))\right\}\\[2mm]
&\subset
c_if \left\{\frac{a}{b}\ \middle|\ 
a \wh{\xi}(c_if) \in b
H_*(\lhc'_*)\right\}
= \wh{z}(\wt{C}').\qedhere
\end{align*}
\end{proof}

It follows that the $R[x]$-submodule $\wh{z}(\wt C)\subset {\rm Frac}(R[x])$ is a local equivalence invariant of $\wt C$.

\begin{defn}
Let $K$ be a knot in an integer homology $3$-sphere $Y$. With $R=\Z[T^{\pm 1}]$, define 
\[
	\wh z(Y,K) := \wh z(\wt C(Y,K;\Delta))
\]
This is an $R[x]$-submodule $\wh z(Y,K)\subset {\rm{Frac}}(R[x])$. If $Y=S^3$, we simply write $\wh z(K)$.
\end{defn}
From above, we see that $\wh z(Y,K)$ is a homology concordance invariant of $(Y,K)$. For the remainder of this section we focus on the case of knots in the $3$-sphere.

To relate $\wh{z}(K)$ to $z_\sigma^\natural(K)$, we give a cobordism interpretation of $\wh{z}(\wt C)$, similar to \eqref{eq:znaturaldef}. Take an immersed cobordism $S:U_1\to K$ as above, with $g=0$. If $h:=-\sigma(K)/2\geq 0$, we obtain the induced morphism $\wt\lambda_S:\wt C(U_1;\Delta)\to \wt C(K;\Delta)$ of height $h$, with $c_{h}$ equal to $(T^2-T^{-2})^{s_+}$, up to a unit. Thus 
\begin{equation}\label{eq:idealcobint0}
	\wh I(S)(1)= \wh I(S)(\wh \xi(1)) = \wh \xi((T^2-T^{-2})^{s_+})
\end{equation}
where $\wh I(S)=(\wh\lambda_S)_\ast$. From this we immediately obtain
\begin{equation}\label{eq:idealcobint}
	\wh z(K) = (T^2-T^{-2})^{s_+}[\wh I(K;\Delta) : \wh I(S)(1)]
\end{equation}
Using \Cref{thm:inaturalwithbncoeffs} and its naturality with respect to cobordism maps, we see that $z^\natural(K)$ is obtained from $\wh z(K)$ by tensoring with $\mathscr{S}$, using the base change described in \eqref{eq:bnbasechange}.

The above argument leading to \eqref{eq:idealcobint} only works for knots with non-positive signature. To work around this, we extend our discussion to cobordisms $S:K'\to K$ inside $[0,1]\times S^3$ which as above have genus $g$ and some double points, but for which $K'$ is not necessarily the unknot. From \eqref{eq:immersedheight}, we have that if
\[
	h:=\frac{1}{2}\sigma(K') - \frac{1}{2}\sigma(K) - g \geq 0,
\]
then such a cobordism induces a morphism of $\mathcal{S}$-complexes of height $h$, with $c_h$ equal to $(T^2-T^{-2})^{s_+}$ up to a unit. Now, if $-\sigma(K)/2 < 0$, $S$ has genus zero, and $K'= \#_l T_{2,3}^\ast$ where $l\geq \sigma(K)/2$, then $h\geq 0$, and so $S$ induces such a morphism. Having replaced $U_1$ by a connected sum of left-handed trefoils, we require the following.

\begin{lem}\label{lem:lefttrefoilmap}
	Let $l\geq 0$. Then $J_{-l}(\#_lT_{2,3}^\ast)=R$, and $\wh\xi(1)$ is a generator of $\wh I( \#_l T_{2,3}^\ast;\Delta)'\cong  R[x]$.
\end{lem}

\begin{proof}
	The $\mathcal{S}$-complex for $T_{2,3}^\ast$ is $\wt C = C_\ast \oplus C_{\ast -1}\oplus R_{(0)}$ where $C_\ast$ has one generator $\alpha$ in degree $2\pmod{4}$; further, $d=\delta_1=v=0$ and $\delta_2(1)=(T^2-T^{-2})\alpha$. For details, see \cite[\S 9]{DS19}. Thus $\fhrC = \langle \alpha\rangle \oplus R[x]$ with differential
	\[
		\wh{\fd}(a \alpha, \sum_{i=0}^N a_i x^i)=-(a_0(T^2-T^{-2})\alpha,0).
	\]
	The $R[x]$-module structure on $\fhrC$ is the obvious one on $ \langle \alpha\rangle \oplus R[x]$, where $x\cdot \alpha=0$. 
Thus $\wh I(T_{2,3}^\ast;\Delta)$ is isomorphic to $R[x]\oplus R/(T^2-T^{-2})R$ via mapping the cycle $(0,x^i)$ for $i\geq 1$ to $x^{i-1}\in R[x]$, and mapping $(\alpha,0)$ to $R/(T^2-T^{-2})R$. In particular, $\wh I(T_{2,3}^\ast;\Delta)'\cong R[x]$. The map $\mathfrak{i}$ sends $(0,x^i)$ to $x^i$. Thus $\im \mathfrak{i}_\ast =  xR[x]\subset R[\![x^{-1},x]$. From this we obtain $h(T_{2,3}^\ast)=-1$ and $J_{-1}(\wt C)=R$. A special $(-1,1)$-cycle is given by $\wh \Psi(\mathfrak{z})$ where $\mathfrak{z}=(0,x)$, and so $\wh \xi(1)$ maps to the generator $x\in xR[x]\cong \wh I(T_{2,3}^\ast;\Delta)'$.
		
	The result for $l>1$ follows from the $l=1$ case using the connected sum theorem.
\end{proof}

We will also make use of the following.

\begin{lem}\label{lem:crossingchangemaptrefoil}
	Let $S:T_{2,3}^\ast\to U_1$ be the immersed cobordism induced by the movie of a crossing change, from negative to positive. Then there is an identification $\wh I(T_{2,3}^\ast)'\cong R[x]$ such that the map $\wh I(S):\wh I(T_{2,3}^\ast)'\to \wh I(U_1)'=R[x]$ is multiplication by $T^2-T^{-2}$.
\end{lem}

\begin{proof}
The cobordism $S$ induces a height 1 morphism $\wt \lambda_S: \wt C(T_{2,3}^\ast;\Delta)\to \wt C(U_1;\Delta)$, which is determined by the component $\Delta_1$. Appealing to the arguments in \cite{DS19}, we obtain that $\Delta_1 \delta_2(1)$ is equal to a count of two reducibles, the latter of which is $(T^2-T^{-2})^{s_+}=T^2-T^{-2}$. Thus $\wt \lambda_S$ is a height 1 morphism with $c_1=\Delta_1 \delta_2(1)$ equal to $T^2-T^{-2}$ up to a unit. The result follows.
\end{proof}

With these preliminaries we now prove the main result of this section.

\begin{thm}\label{z-BN-z-hat}
	Kronheimer and Mrowka's invariant $z_\sigma^{\natural}(K)\subset {\rm{Frac}}(\mathscr{S})$ can be recovered from $\widehat z(K)$ as 
	\begin{equation*}\label{Znatzfrom}
	  z^{\natural}_\sigma(K)=\widehat z(K) \otimes_{\Z[T^{\pm 1},x]}\mathscr{S'}
	\end{equation*}
	where the module structure defining the tensor product uses \eqref{eq:bnbasechange} and the base change $\sigma:\mathscr{S}\to \mathscr{S}'$.
\end{thm}

\begin{proof}
We prove the result for $\sigma=\text{id}_\mathscr{S}$. The more general case follows easily. 

As already indicated, if $-\sigma(K)/2\geq 0$, the result follows from \eqref{eq:idealcobint} and the isomorphisms in \Cref{thm:inaturalwithbncoeffs}, which are natural with respect to cobordism maps (when they are defined).

Suppose $-\sigma(K)/2<0$. Choose a cobordism $S:\#_lT_{2,3}^\ast\to K$ in the cylinder with genus $0$ and $s_+$ double points, and $l\geq \sigma(K)/2$. Specifically choose $S$ so that it is a composition of $S': \#_l T_{2,3}^\ast\to U_1$ and $S'':U_1\to K$, where $S'$ is a connected sum of cobordisms $T_{2,3}^\ast\to U_1$ each induced by a negative to positive crossing change. 

By \Cref{lem:lefttrefoilmap}, we still have \eqref{eq:idealcobint0}, and thus expression \eqref{eq:idealcobint} for $\wh z(K)$ is true for the cobordism $S$. By \Cref{lem:crossingchangemaptrefoil} and \Cref{thm:inaturalwithbncoeffs}, the map $I^\natural(S'):I^\natural(\#_l T_{2,3}^\ast)'\to I^\natural(U_1)'$ is up to a unit equal to multiplication by $(T_1^2-T_1^{-2})^l$. Using this input, we now compute
\begin{align*}
	\widehat z(K) \otimes_{\Z[T^{\pm 1},x]}\mathscr{S} & = (T_1^{2}-T_1^{-2})^{s_+} [I^\natural(K;\mathscr{S})': I^\natural(S)(1)] \\[2mm]
	&= (T_1^{2}-T_1^{-2})^{s_+} [I^\natural(K;\mathscr{S})': (T_1^2-T_1^{-2})^{l} I^\natural(S'')(1)]\\[2mm]
	&= (T_1^{2}-T_1^{-2})^{s_+-l} [I^\natural(K;\mathscr{S})':  I^\natural(S'')(1)]
\end{align*}
As $s_+-l$ is the number of positive double points of $S'':U_1\to K$, this last expression is $z^\natural(K)$.
\end{proof}

The cobordism interpretation for $\wh z(K)$ as in \eqref{eq:idealcobint} extends to the non-zero genus case. First we need to study further the invariants of the connected sums of $T_{2,3}$ and $T_{2,3}^*$.

\begin{lem}\label{lem:genus1cobmap}
	Let $S:T_{2,3}^\ast\to U_1$ be a genus 1 embedded cobordism. Then $\wh I(S)$ is multiplication by $x$.
\end{lem}

\begin{proof}
Such a cobordism induces a morphism $\wt C(T_{2,3}^\ast;\Delta)\to \wt C(U_1;\Delta)$ of height $\sigma(T^\ast_{2,3})/2-g=0$, with $c_0=1$. This morphism is determined by mapping  the reducible summand $R_{(0)}\subset \wt C(T_{2,3};\Delta)$ isomorphically to $R_{(0)}=\wt C(U_1;\Delta)$. From this and the isomorphism of $\wh I(T_{2,3}^\ast)'\cong R[x]$ given in the proof of \Cref{lem:lefttrefoilmap}, we see that the induced map on $\wh I$ is multiplication by $x$. 
\end{proof}

\begin{lem}\label{lem:Ixirhtrefoils}
 	For $l >0$, we have $\wh I( \#_l T_{2,3};\Delta)\cong \mathcal I^l$ where
 	\begin{equation}\label{I-from-k-trefoil}
		\mathcal I^l=\left(x^l,x^{l-1}(T^2-T^{-2}),x^{l-2}(T^2-T^{-2})^2,\dots,(T^2-T^{-2})^l\right).
	\end{equation}
	Moreover, $J_{l}(\#_lT_{2,3})=((T^2-T^{-2})^l)$ and $\wh\xi((T^2-T^{-2})^l)=(T^2-T^{-2})^l\in \mathcal I^l$.
\end{lem}

\begin{proof}
	The $\mathcal{S}$-complex for $T_{2,3}$ is given by $\wt C = C_\ast \oplus C_{\ast -1}\oplus R_{(0)}$ where $C_\ast$ has a single generator $\alpha$ in degree $1\pmod{4}$; 
	further, $d=\delta_2=v=0$ and $\delta_1(\alpha)=T^2-T^{-2}$. See \cite[\S 9]{DS19}. Thus $\fhrC(T_{2,3}) = \langle \alpha\rangle \oplus R[x]$ with zero differential, and so 
	$\wh I(T_{2,3};\Delta)=\fhrC(T_{2,3})  $. The $R[x]$-module structure is
	\[
		x\cdot(a\alpha,\sum_{i=0}^N a_i x^i) = (0, (T^2-T^{-2})a + \sum_{i=0}^N a_i x^i).
	\]
	That means that $\wh I(T_{2,3};\Delta)$ is isomorphic as an $R[x]$-module to $\mathcal I^1 = (x,T^2-T^{-2})$ by sending $(\alpha,0)$ to $T^2-T^{-2}$ and $(0,x^i)$ to $x^{i+1}$. Furthermore, we have $\mathfrak{i}(a\alpha, \sum_{i=0}^N a_i x^i) = a(T^2-T^{-2})x^{-1}+ \sum_{i=0}^N a_i x^i$. Thus $h(T_{2,3})=1$ and $J_1(\wt C)= (T^2-T^{-2})\subset R$. A special $(1,T^2-T^{-2})$-cycle is given by $\wh \Psi (\mathfrak{z})$ where $\mathfrak{z}=(\alpha,0)$, and so $\wh\xi(T^2-T^{-2})$ corresponds under the isomorphism $\wh I(T_{2,3};\Delta)\cong \mathcal I^1$ to the element $T^2-T^{-2}$.
	The completes the proof in the case $l=1$. 
	The more general case follows in a similar fashion, making use of the connected sum theorem. 
\end{proof}
The following lemma can be proved in the same way as \Cref{lem:crossingchangemaptrefoil} and \Cref{lem:genus1cobmap}.
\begin{lem}\label{lem:cobmaprhtrefoils}
	Let $S:U_1\to T_{2,3}$ be the immersed cobordism induced by the movie of a crossing change, from negative to positive. With respect to the identification of $\wh I(T_{2,3})$ with 
	$\mathcal I$ in Lemma \ref{lem:Ixirhtrefoils}, the cobordism map $\wh I(S):\wh I(U_1)=R[x]\to \wh I(T_{2,3})$ is given by multiplication by $T^2-T^{-2}$.
	If $S':U_1\to T_{2,3}$ is an embedded cobordism of genus $1$, then the cobordism map $\wh I(S):\wh I(U_1)=R[x]\to \wh I(T_{2,3})$ is given by multiplication by $x$.
\end{lem}

Now suppose $S:\#_lT_{2,3}^\ast\to K$ is an immersed cobordism with genus $g$ and $s_+$ positive double points, and for which $l\geq \sigma(K)/2+g$. Then we have the following:
\begin{equation}\label{z-from-i-neg}
	\widehat z(K) =x^g(T^2-T^{-2})^{s_+}[\hrI(K;\Delta)':\hrI(S)(1)].
\end{equation}
To see this, we first assume that $l\geq g$, and that $S$ is the composition of $S':\#_l T_{2,3}^\ast \to \#_{l-g} T_{2,3}^\ast$ and $S'':\#_{l-g} T_{2,3}^\ast\to K$ where $S'$ is embedded and has genus $g$, and $S''$ has genus zero; further assume that $S'$ is the connected sum of $g$ genus 1 cobordisms $T_{2,3}^*\to U_1$ and the product cobordism $ \#_{l-g} T_{2,3}^\ast \to \#_{l-g} T_{2,3}^\ast$. The cobordisms $S'$ and $S''$ respectively induce morphisms of $\cS$-complexes with levels $0$ and $l-g-\sigma(K)/2$. In particular, $\wh I(S')$ and $\wh I(S'')$ are both well-defined and  $\wh I(S) = \wh I(S'')\circ \wh I(S')$. By \Cref{lem:genus1cobmap} and naturality of the connected sum theorem, $\wh I(S')$ is multiplication by $x^g$. Then \eqref{z-from-i-neg} follows from $\wh I(S) = \wh I(S'')\circ \wh I(S') = x^g \wh I(S'')$ and the fact it already has been established for the cobordism $S''$. In the case $g>l$, we may still obtain cobordisms $S'$ and $S''$ as above except that $\#_{l-g} T_{2,3}^\ast$ should be replaced with $\#_{g-l} T_{2,3}$. \Cref{lem:genus1cobmap} and \Cref{lem:cobmaprhtrefoils} imply that $\wh I(S')$ is still multiplication by $x^g$. Therefore, we have
\begin{align*}
	x^g(T^2-T^{-2})^{s_+}[\hrI(K;\Delta)':\hrI(S)(1)]=&x^g(T^2-T^{-2})^{s_+}[\hrI(K;\Delta)':\hrI(S'')(x^g)]\\[2mm]
	=&x^g(T^2-T^{-2})^{s_++g-l}[\hrI(K;\Delta)':\hrI(S'')((T^2-T^{-2})^{g-l}x^g)]\\[2mm]
	=&(T^2-T^{-2})^{s_++g-l}[\hrI(K;\Delta)':\hrI(S'')(\widehat \xi((T^2-T^{-2})^{g-l}))]\\[2mm]
	=&(T^2-T^{-2})^{s_++g-l}[\hrI(K;\Delta)':\widehat \xi((T^2-T^{-2})^{s_++g-l})]\\[2mm]
	=&\widehat z(K)
\end{align*}
Note that we used \Cref{lem:Ixirhtrefoils} for the third identity. More generally, for any $S:\#_lT_{2,3}^\ast\to K$ with $l\geq \sigma(K)/2+g$, we can homotope $S$ to a surface of the type just considered through a sequence of twists and finger moves, and appeal to \cite[Proposition 2.30]{DS19}, to see that \eqref{z-from-i-neg} still holds.

\begin{prop}\label{prop:zfromgenus}
	Let $S:U_1\to K$ be an orientable immersed cobordism with genus $g$ and $s_+$ positive double points. Let $i=\max\{0,g+\sigma(K)/2\}$. Then for any $j,k\in \Z_{\geq 0}$ satisfying $j+k=i$, we have
	\begin{equation}\label{genus-positive-dble-point}
	  x^{g+j}(T^2-T^{-2})^{s_++k}\in \widehat z(K).
	\end{equation}
\end{prop}

\begin{proof}
	If $g+\sigma(K)/2\leq 0$, so that $i=j=k=0$, this follows directly from \eqref{z-from-i-neg}. If $i= g+\sigma(K)/2>0$, form a cobordism $S_1:\#_lT_{2,3}^\ast\to K$ by pre-composing $S$ with a cobordism $S_2:\#_l T_{2,3}^\ast \to U_1$, where $S_2$ is formed by a connected sum of $j$ genus 1 cobordisms and $k$ standard crossing change cobordisms. Applying \eqref{z-from-i-neg} to the cobordism $S_1$ yields the result.
\end{proof}

\begin{rem}
For an immersed orientable cobordism $S:U_1\to K$ with genus $g$ and $s_+$ positive double points, it follows from Kronheimer and Mrowka's definition of $z^\natural(K)$ that we have
\[
	P^{g}(T_1^2-T_1^{-2})^{s_+}\in  z^\natural(K).
\]
This is recovered by \eqref{genus-positive-dble-point} when $g+\sigma(K)/2\leq 0$. One might expect that $ x^g(T^2-T^{-2})^{s_+}\in \widehat z(K)$ holds with no condition on $g+\sigma(K)/2$. This shortcoming is likely a deficiency of the above trick involving trefoils. We expect the stronger claim to follow once cobordism maps are defined in more generality. 
\end{rem}

We now turn to some computations.

\begin{lem}\label{lem:zfromlhtrefoils}
	For $l\geq 0$, we have $\wh z(\#_l T^\ast_{2,3})= R[x] \subset {\rm{Frac}}(R[x])$.
\end{lem}

\begin{proof}
	As explained above, the equality \eqref{z-from-i-neg} is valid for the identity cobordism $S:\#_l T_{2,3}^\ast\to \#_l T_{2,3}^\ast$. This has $g=s_+=0$, and of course $\wh I(S)(1)$ is a generator of $\wh I(\#_l T_{2,3}^\ast)'$. The result follows.
\end{proof}

For connected sums of the right-handed trefoil, we have the following.

\begin{lem}\label{lem:zfromrhtrefoils}
	 For $l >0$, we have $\wh z(\#_l T_{2,3})= \mathcal I^l \subset {\rm{Frac}}(R[x])$.
\end{lem}

\begin{proof}
	From \Cref{defn:zfrom} and \Cref{lem:Ixirhtrefoils}, we compute
	\begin{align*}
		\wh z(\#_lT_{2,3}) &  = (T^2-T^{-2})^l[ \wh I(\#_lT_{2,3};\Delta): \wh{\xi}((T^2-T^{-2})^l)]\\[2mm]
		 &  =  (T^2-T^{-2})^l [ \mathcal{I}^l: (T^2-T^{-2})^l] =  \mathcal I^l\qedhere
	\end{align*}
\end{proof}

\begin{rem} For simplicity, we have restricted our discussion above to the invariant $z^\natural_\sigma(K)$. Kronheimer and Mrowka also define in \cite{KM19b} a concordance invariant $z^{\sharp}(K)$ which is a fractional ideal for the ring $\mathbf{F}[T_0^{\pm 1}, T_1^{\pm 1},T_2^{\pm 1},T_3^{\pm 1}]$. Following the discussion in \cite[\S 8.8]{DS19}, we have
\begin{equation*}
	  z^{\sharp}(K)=\widehat z(K)\otimes_{R[x]}\bF[T_{0}^{\pm 1},T_1^{\pm1},T_2^{\pm1},T_3^{\pm1}]
\end{equation*}
where $R[x]$ acts on the right-hand ring by reducing modulo $2$, $T\mapsto T_0$ and $x\mapsto P$.
\end{rem}

In \cite{KM19b}, Kronheimer and Mrowka define several concordance invariants using constructions in singular instanton homology. For example, for any homomorphism $\sigma:\mathscr{S}\to \mathscr{S}'$ where $\mathscr{S}'$ is a valuation ring with valuation group $G$, they define a homomorphism from the knot concordance group to $G$,
\[
	f_\sigma:\mathcal{C} \to G.
\]
Specific choices of homomorphisms $\sigma$ are given in \cite[\S 5]{KM19b}. Any each such invariant $f_\sigma$ is determined by $z^\natural_\sigma(K)$. Since the latter is determined by the local equivalence invariant $\wh z(K)$ by \Cref{z-BN-z-hat}, we arrive at the following, which proves \Cref{thm:introkminvtsarelocequiv}.

\begin{thm}\label{thm:kminvtsarelocequiv}
	All of the concordance invariants defined by Kronheimer and Mrowka in \cite{KM19b}, such as $z_\sigma^\#(K)$, $z^\natural_\sigma(K)$, $f_\sigma(K)$, ... depend only on the local equivalence class of the $\mathcal{S}$-complex $\wt C(K;\Delta)$.
\end{thm}

\subsection{Two-bridge knots}\label{sec:twobridge}

In the case of two-bridge knots, many of the invariants considered thus far are determined by the signature. The results in this section imply \Cref{thm:intro2bridge} as stated in the introduction.

\begin{prop}\label{prop:twobridgescomplexes}
	Let $K$ be a two-bridge knot and $F$ any field. Then the local equivalence class of 
	\begin{equation}\label{eq:twobridgescomplexresult}
		\wt C(K;\Delta\otimes_{R} F[T^{\pm 1}])
	\end{equation}
	is determined by $\sigma(K)$. In particular, if $h:=-\sigma(K)/2 \geq 0$, it is locally equivalent to the $\mathcal{S}$-complex over $F[T^{\pm 1}]$ of both the knots $T_{2,2h+1}$ and $\#_h T_{2,3}$.
\end{prop}

\begin{proof}
Assume $h=h(\wt C_0)\geq 0$; the result for $h<0$ follows by duality.	As explained in \cite[\S 3.1]{DS20}, the $\mathcal{S}$-complex $\wt C$ in \eqref{eq:twobridgescomplexresult} for a two-bridge knot $K$ has a free basis over $F[T^{\pm 1}]$ such that its differential $\wt d$ can be written, with respect to this basis, as $T^2-T^{-2}$ times a matrix with integer entries. As a consequence, 
	\[
		\wt C = \wt C_0 \otimes_F F[T^{\pm 1}], \qquad \wt d =\wt d_0 \otimes (T^2-T^{-2})
	\]
	where $(\wt C_0 ,\wt d_0)$ is an $\mathcal{S}$-complex over the field $F$. The local equivalence class of any $\mathcal{S}$-complex over a field is determined by its Fr\o yshov invariant, see \cite[Proposition 4.30]{DS19}. Concretely, having assumed $h(\wt C)=h(\wt C_0)\geq 0$, the $\Z/4$-graded $\mathcal{S}$-complex $\wt C_0$ over $F$ is locally equivalent to the $\mathcal{S}$-complex
	\[
		\wt C' = C'_\ast \oplus C'_{\ast- 1} \oplus F_{(0)}
	\]
	with $C'$ freely generated over $F$ by elements $\alpha_1,\ldots ,\alpha_h$ with the $\Z/4$-grading of $\alpha_i$ given by $2i-1\pmod{4}$; the only non-trivial differentials are $v(\alpha_i)=\alpha_{i-1}$ for $i\geq 2$ and $\delta_1(\alpha_1)=1$.
	
	Let $\wt \lambda:\wt C_0\to \wt C'$ and $\wt \lambda': \wt C'\to \wt C_0$ be local morphisms realizing the local equivalence between $\wt C_0$ and $\wt C'$. Define an $\mathcal{S}$-complex over $F[T^{\pm 1}]$ as follows:
	\[
		\wt C'' = \wt C'\otimes_{F} F[T^{\pm 1}], \qquad \wt d'' = \wt d'\otimes (T^2-T^{-2}) 
	\]
	Then $\wt \lambda$ and $\wt\lambda'$ extend in the obvious way to local morphisms $\wt C\to \wt C''$ and $\wt C''\to \wt C$, respectively. Finally, $h(\wt C)=-\sigma(K)/2$ by \cite[Theorem 7]{DS20}.
\end{proof}

\begin{rem}
	The above argument does not work with $F=\Z$, because $\mathcal{S}$-complexes over $\Z$ are not classified by Fr\o yshov invariants. On the other hand, the authors have no counterexample to the general statement of \Cref{prop:twobridgescomplexes} with $F$ replaced by $\Z$.
\end{rem} 

\vspace{1mm}

\begin{rem}
An additional layer of structure on $\mathcal{S}$-complexes for knots comes from the Chern--Simons filtration, which is studied in the next section. The analogue of \Cref{prop:twobridgescomplexes} is not true in the filtered setting. This is evidendent from our applications of the theory to two-bridge knots; see also \cite{DS19,DS20}.
\end{rem}

Using \Cref{prop:twobridgescomplexes} we can compute the concordance invariants $s^\sharp_\pm, s^\sharp, \wt\varepsilon$ for two-bridge knots.

\begin{prop}\label{prop:twobridgessharp}
	The invariants $s^\sharp_\pm$ for a two-bridge knot $K$ are as follows:
	\begin{align*}
		s^\sharp_+(K) = \begin{cases} -\sigma(K)/2 & {\rm{ if }}\; \sigma(K)<0 \\  0 & {\rm{ if }}\; \sigma(K)=0 \\ -\sigma(K)/2+1 & {\rm{ if }}\; \sigma(K)>0  \end{cases} \hspace{1.5cm} 		s^\sharp_-(K) = \begin{cases} -\sigma(K)/2-1 & {\rm{ if }}\; \sigma(K)<0 \\  0 & {\rm{ if }}\; \sigma(K)=0 \\ -\sigma(K)/2 & {\rm{ if }}\; \sigma(K)>0  \end{cases}
	\end{align*}
	Thus Kronheimer and Mrowka invariant $s^\sharp$ and the invariant $\wt \varepsilon$ are as follows (where ${\rm{sign}}(0)=0$):
	\[
	s^\sharp (K) = -\sigma(K) +{\rm{sign}}(\sigma(K)), \hspace{1.5cm} 	\wt\varepsilon	(K) = {\rm{sign}}(\sigma(K)).
\]
\end{prop}

\begin{proof}
	By \Cref{mirror s}, it suffices to treat the case $h:= -\sigma(K)/2\geq 0$. If $h=0$, by \Cref{prop:twobridgescomplexes}, the $\mathcal{S}$-complex for $K$ with coefficients in $\locring$ is locally trivial, and thus $s^\sharp_\pm(K)=\wt \varepsilon(K)=0$. So henceforth we assume $h>0$.
	
	By \Cref{prop:twobridgescomplexes}, and the fact that $s_\pm^\sharp$ are local equivalence invariants, we may assume that the $\mathcal{S}$-complex of $K$ over $\locring$ is of the form $\wt C_\ast = C_\ast \oplus C_{\ast -1}\oplus \locring$ where $C_\ast$ is freely generated by $\alpha_1,\ldots ,\alpha_h$ with non-trivial components of the differential given by $v(\alpha_i)=\Lambda\alpha_{i-1}$ for $i\geq 2$ and $\delta_1(\alpha_1)=\Lambda$. To choose such a basis we use that $T^2-T^{-2}$ is equal to $\Lambda$ times a unit in the ring $\locring$.
	
	The complex $\fhrC=C_{\ast-1} \oplus \Q[\![ \Lambda ]\!][x]$ has zero differential and is therefore equal to $H(\hrC)$. For simplicity write $\alpha_i$ for the generator in $H(\hrC)$ which is $(\alpha_i,0)$, and $x^i$ for $(0,x^i)$. The $x$-action on $H(\hrC)$ is depicted as follows:
		\[
	\begin{tikzcd}
		\alpha_h \ar[dr, "\Lambda"]  & &\; \ldots \; \ar[dr, "\Lambda"]  & & \alpha_{1} \ar[dr, "\Lambda"] &   &  x \ar[dr] &  \\
		& \alpha_{h-1} \arrow[ur, "\Lambda"] &  &  \alpha_{2} \arrow[ur, "\Lambda"] & & 1\arrow[ur] & & x^2\; \cdots 
	\end{tikzcd}
\]
More precisely, $x\cdot \alpha_i = \Lambda \alpha_{i-1}$ for $i\geq 2$, $x\cdot \alpha_1 = \Lambda$, and $x\cdot x^i  = x^{i+1}$. In the diagram, each vertex generates a copy of $\Q[\![ \Lambda ]\!]$. The elements in the top row are in grading $2\pmod{4}$, and in the bottom row, $0\pmod{4}$. In the diagram we have assumed that $h$ is odd, the even case being similar. We have an isomorphism of $\Q[\![ \Lambda ]\!][x]$-modules:
\[
	H(\hrC) \cong \left(x^h, x^{h-1}\Lambda, \ldots, x\Lambda^{h-1}, \Lambda^h\right) = \mathcal I^h \otimes_{R[x]} \Q[\![ \Lambda ]\!][x]
\]
The isomorphism sends $\alpha_i$ to $\Lambda^i x^{h-i}$, and $x^i$ to $x^{i+h}$. The map $\mathfrak{i}:\fhrC=H(\hrC)\to \Q[\![ \Lambda ]\!][\![ x^{-1},x]$ sends $\alpha_i\in \fhrC$ to $\Lambda^i x^{-i}$ and sends $x^i$ to $x^{i}$. From this we obtain $J_h(\wt C)= (\Lambda^h)\subset \Q[\![ \Lambda ]\!]$. We also have
\[
	\wh \xi: J_h(\wt C) =  (\Lambda^h) \to H(\hrC), \qquad \wh \xi(\Lambda^h ) = \alpha_h
\]
\begin{figure}
\centering
\[\begin{tikzcd}[column sep=5mm]
	\vdots &&&&&& \vdots \\
	{\Lambda^3} &&&&&& {\Lambda^2 x} \\
	{\Lambda^2} &&&&&& {\Lambda x} \\
	\Lambda & {\Lambda \alpha_2} & {\Lambda \alpha_4} & \ldots & {\Lambda \alpha_{h-1}} && x & {\Lambda \alpha_1} & {\Lambda \alpha_3} & \ldots & {\Lambda \alpha_h} \\
	1 & {\alpha_2} & {\alpha_4} && {\alpha_{h-1}} &&& {\alpha_1} & {\alpha_3} && {\alpha_h}
	\arrow[from=5-1, to=4-1]
	\arrow[from=4-1, to=3-1]
	\arrow[from=3-1, to=2-1]
	\arrow[from=5-2, to=4-2]
	\arrow[from=5-3, to=4-3]
	\arrow[from=5-5, to=4-5]
	\arrow[from=2-1, to=1-1]
	\arrow[from=4-7, to=3-7]
	\arrow[from=3-7, to=2-7]
	\arrow[from=2-7, to=1-7]
	\arrow[from=5-8, to=4-8]
	\arrow[from=5-9, to=4-9]
	\arrow[from=5-11, to=4-11]
	\arrow[from=4-2, to=3-1]
	\arrow[from=4-3, to=3-1]
	\arrow[from=4-5, to=3-1]
	\arrow[from=4-8, to=3-7]
	\arrow[from=4-9, to=3-7]
	\arrow[from=4-11, to=3-7]
\end{tikzcd}\]
     \caption{The $\locring$-module $H(C^\sharp)$ in the proof of \Cref{prop:twobridgessharp}.}
    \label{fig:csharp2bridgediagram}
  \end{figure}
Next we turn to $H(C^\sharp)$. The action of $x^2- 4\Lambda^2$ on $H(\hrC)$ is injective, and so $H(C^\sharp)$ may be identified with the quotient of $H(\hrC)$ by $(x^2-4\Lambda^2)H(\hrC)$. To describe $H(C^\sharp)$ as this quotient, we have the diagram in \Cref{fig:csharp2bridgediagram}, assuming $h$ is odd. As a $\Q$-vector space, $H(C^\sharp)$ is the direct sum of the $\Q$-spans of the vertices. (We abuse notation and write $\alpha_i$ also for the equivalence class of $\alpha_i$ in $H(C^\sharp)$, and so on.) The $\Lambda$-action is described by the arrows, each of which is multiplication by $\Lambda$ up to a unit in $\Q$. More precisely, 
\[
	\Lambda\cdot \alpha_i =\Lambda\alpha_i , \qquad \Lambda\cdot \Lambda\alpha_i=\begin{cases} \frac{1}{4}\Lambda^2, & i\text{ even} \\[2mm] \frac{1}{4}\Lambda x, & i\text{ odd}\end{cases}
\]
and the obvious $\Lambda$-action holds for the two $\Q[\![ \Lambda ]\!]$ towers generated by $1$ and $x$. (The $\Z/4$-gradings are induced from $H(\hrC)$ and not represented in the diagram.) We obtain
\begin{equation}\label{eq:twobridgexisharpplusminus}
	\xi^\sharp_+(\Lambda^h)  = \alpha_h, \qquad \xi_-^{\sharp}(\Lambda^{h}) = \Lambda\alpha_{h-1}
	\end{equation}
	where, if $h=1$, we set the convention $\alpha_0:=1$. From the description of $H(C^\sharp)$ it is easy to see that the minima appearing in the definitions for $s^\sharp_\pm(\wt C)$ are realized by the expressions \eqref{eq:twobridgexisharpplusminus}. Thus $s^\sharp_+(\wt C)=h = -\sigma(K)/2$ and $s^\sharp_-(\wt C)=h-1=-\sigma(K)/2-1$. This completes the proof.
\end{proof}

\begin{cor}
	Let $K$ be a two-bridge knot, $F$ a field. Then $\wh z(K)$ with coefficients $F[T^{\pm 1}]$ is as follows:
	\[
		\wh z(K)\otimes_{R[x]} F[T^{\pm 1},x] = \begin{cases} F[T^{\pm 1},x] & \sigma(K) \geq 0\\[2mm]
		\mathcal I^{-\sigma(K)/2}\otimes_{R[x]} F[T^{\pm 1},x] & \sigma(K) <0  \end{cases}
	\]
	where $\mathcal I^l$ is defined in \eqref{I-from-k-trefoil}. As a consequence, for a two-bridge knot, Kronheimer and Mrowka's $z_\sigma^\natural(K)$, and in fact all concordance invariants defined in \cite{KM19b}, are determined by $\sigma(K)$ via these formulas.
\end{cor}

\begin{proof}
	This follows from \Cref{prop:twobridgescomplexes}, \Cref{lem:zfromlhtrefoils}, \Cref{lem:zfromrhtrefoils}, and \Cref{thm:kminvtsarelocequiv}.
\end{proof}

\begin{ques}
	Is \Cref{prop:twobridgescomplexes} true for alternating knots?
\end{ques}


\section{Concordance invariants from filtered special cycles} \label{Section: Concordance invariants from filtered special cycles}

The $\mathcal{S}$-complexes that arise from singular instanton Floer theory posess additional structure which roughly comes in the form of an $\R$-valued filtration, defined using values of the Chern--Simons functional. In this section we use this additional structure to define more homology concordance invariants, generalizing both the invariant $\Gamma_{(Y,K)}$ in the singular setting from \cite{DS19,DS20}, as well as the invariants $\Gamma_Y$ and $r_s(Y)$ in the setting of integer homology 3-spheres $Y$ from \cite{D18} and \cite{NST19}, respectively.

We first introduce the notion of an enriched complex, following \cite{DS19} with minor variations. This is the algebraic enhancement of an $\mathcal{S}$-complex that singular instanton theory outputs when keeping track of the Chern--Simons filtration. We then adapt the material of \Cref{section: special cycles} to this setting, and define {\it{filtered}} special cycles. These are then used to define various homology concordance invariants. In this section, Theorems \ref{s-tilde-froy} and \ref{thm:connsumineqintro} from the introduction are proved.

\subsection{Enriched $\cS$-complexes}\label{enriched-s-comp} 

Let $R$ be an integral domain over $\Z[T^{\pm 1}]$.
Let $\delta$ be a non-negative real number. 
An {\it I-graded $\mathcal{S}$-complex} over $R[U^{\pm 1}]$ is an $\mathcal{S}$-complex ${\wt C}$ over $R[U^{\pm 1}]$ with $\Z\times \R $-bigrading: 
\[
{\wt C} = \bigoplus_{(i,j) \in \Z \times \R}  {\wt C}_{i,j}
\] 
which satisfies the following properties: 
\[
 U {\wt C}_{i,j} \subset  {\wt C}_{i+4,j+1} ,\hspace{.3cm}  \wt{d} {\wt C}_{i,j} \subset \bigcup_{k< j} C_{i-1, k}, \hspace{.3cm}\text{ and
}   \hspace{.2cm} \chi {\wt C}_{i,j} \subset {\wt C}_{i+1,j}.
\]
We also impose that ${\wt C}$ is freely, finitely generated as an $R[U^{\pm 1}]$-module by homogeneously bigraded elements, and the canonical element $1$ in the distinguished summand $R[U^{\pm 1}] \subset {\wt C}$ is in $\wt C_{0,0}$. Note that the ring $R[U^{\pm 1}]$ can be thought of as the trivial I-graded $\mathcal{S}$-complex, with $U^j$ havings $\Z\times \R$-grading $(4j , j)$. For any non-zero element $\zeta\in \wt C_{i,j}$ we denote the $\R$-grading (also called the {\it{I-grading}}) by $\text{deg}_I(\zeta) =j$. This is extended to non-homogeneous elements as follows: if $\zeta= \sum \zeta_i$ is a finite sum of homogeneous elements in distinct bigradings, then
\begin{equation}\label{deg-I-non-hg}
	\text{deg}_I(\zeta ) = \sup_i \left\{ \text{deg}_I(\zeta_i)\right\} \in \R \cup \{-\infty\}
\end{equation}
with the convention that $\text{deg}_I(0)=-\infty$. For an $I$-graded $\mathcal{S}$-complex $\wt C$, the dual I-graded $\mathcal{S}$-complex is formed by taking the underlying dual $\mathcal{S}$-complex $\wt C^\dagger$ and further defining
\[
	\deg_I(\zeta^\dagger ) = \sup\{ \text{deg}_I(\zeta^\dagger(z))-\deg_I(z) \mid  0 \neq z \in \wt C\}
\] 
for any $\zeta^\dagger \in C^\dagger = \text{Hom}_{R[U^{\pm 1}]}(\wt C, R[U^{\pm 1}])$.

{\it A height $l$ morphism $\wt \lambda:{\wt C}\to {\wt C}'$ of level $\delta $} between I-graded S-complexes is a
height $l$ morphism of $\mathcal{S}$-complexes such that the following holds:
\[
\wt \lambda \;{\wt C}_{i,j} \subset \bigcup_{k \leqslant  j + \delta  } {\wt C}'_{i+ 2l,k}. 
\]
 A $\mathcal{S}$-chain homotopy $\wt K$ of level $\delta$ between height $l$ morphisms $\wt \lambda$ and $\wt \lambda'$ (of any levels) is an $R$-module homomorphism and an $\mathcal{S}$-chain homotopy between $\wt \lambda$ and $\wt \lambda'$ such that 
\[
\wt K \;{\wt C}_{i,j} \subset \bigcup_{k \leqslant  j + \delta} {\wt C}'_{i+2l +1,k}. 
\]
We also have the analogous notions of strong height $l$ morphisms and local morphisms in this context.

In the fortunate circumstance that no perturbations are needed in the construction of equivariant singular instanton Floer homology, the output is simply an I-graded $\mathcal{S}$-complex. In general, one must take a sequence of perturbations approaching zero. The type of algebraic structure that arises in this more general case is as follows.  

\begin{defn}\label{def of enriched cpx}
An {\it enriched $\mathcal{S}$-complex} $\mathfrak{E}$
over $R[U^{\pm 1}]$ is a sequence $\{({\wt C}^i, \tilde d^i,\chi^i )\}_{i\geq 1}$ of I-graded $\mathcal{S}$-complexes over $R[U^{\pm 1}]$, local morphisms $\wt \phi^j_i:{\wt C}^i\to {\wt C}^j$ of levels $\delta_{i,j}$, and a discrete subset $\mathfrak{K}\subset \R$ with no accumulation point satisfying the following:
\begin{itemize}
\item[(i)] $\wt \phi^i_i = \text{id}$ and $\wt \phi^k_j\circ \wt \phi_i^j$ is $\mathcal{S}$-chain homotopy equivalent to $\wt \phi_i^k$ by an $\mathcal{S}$-chain homotopy of level $\delta_{i,j,k}$.
\item[(ii)] For each $\delta >0$ there exists an $N$ such that $i,j,k>N$ implies $\delta_{i,j}\leqslant \delta$ and $\delta_{i,k,j} \leqslant \delta$.
\item[(iii)]  
 For every $\delta>0$, there exists $N \in \Z_{>0}$ such that for any $n>N$, and any non-zero $\zeta\in \wt C^n$, we have
\begin{equation}\label{cond-k}
\deg_I (\zeta ) \in B_\delta (\mathfrak{K}), 
\end{equation}
where $B_\delta (\mathfrak{K}):= \{ r \in \R \mid  | r-r' | < \delta \text{ for some } r' \in \mathfrak{K} \}$. 
	\end{itemize}
\end{defn}
\begin{rem} The I-graded $\mathcal{S}$-complexes defined above are I-graded $\mathcal{S}$-complexes of level $0$ in the terminology of \cite{DS19}.
Furthermore, in \cite{DS19}, condition (iii) is not imposed in the definition of enriched $\mathcal{S}$-complexes. Condition (iii) is used in the course of defining the $r_s$-type invariants, following \cite{NST19}. 
\end{rem}

\begin{prop}\label{S-chain homotopy type welldef}
Let $\mathfrak{E}$ be an enriched $\mathcal{S}$-complex over $R[U^{\pm 1}]$ consisting of $\{{\wt C}^i\}, \{\wt \phi_i^j\}, \mathfrak{K}$. Define
\[
\wt{C}^{n,\leq s} := \{ \zeta \in \wt{C}^n | \deg_I (\zeta ) \leq s\}.
\]
Then $\wt{C}^{n,\leq s}$ is a chain complex over $R$. For $s \notin \mathfrak{K}$ and sufficiently large $n,m\in \Z_{>0}$, the chain complexes $\wt{C}^{n,\leq s}$ and $\wt{C}^{m,\leq s}$ are canonically chain homotopy equivalent.
\end{prop}
\begin{proof}
By \cref{def of enriched cpx} (iii), for sufficiently large $n$ and $ m$, $\wt \phi_n^m$ induces a well-defined local morphism from $\wt{C}^{n,\leq s} $ to $\wt{C}^{m,\leq s}$. Moreover, conditions (i) and (ii) in \cref{def of enriched cpx} imply that these maps provide canonical $\mathcal{S}$-chain homotopy equivalences between these $\mathcal{S}$-complexes.
\end{proof}

\begin{rem}
	The chain complexes $\wt C^{n,\leq s}$ over $R$ in \Cref{S-chain homotopy type welldef} have $\chi$-actions and the result is compatible with these actions. However, these complexes are not $\mathcal{S}$-complexes over $R$, because they are not finitely generated, and the reducible summands do not have rank $1$.
\end{rem}

\begin{defn}\label{enriched morphism definition}
Fix a non-negative integer $l$ and non-negative real number $\kappa$. A {\it morphism $\mathfrak{L}:\mathfrak{E}(1) \to \mathfrak{E}(2)$ of height $l$ and level $\kappa$} between two enriched $\mathcal{S}$-complexes
\[
\mathfrak{E}(r)=(\{\widetilde C^i(r)\},\{\wt \phi_i^j(r)\}, \mathfrak{K}) \qquad r\in \{1,2\}
\]
is a collection of height $l$, level $(\delta_{i,j} + \kappa)$ morphisms $\widetilde \lambda_i^j:\widetilde C^i(1)\to \widetilde C^j(2)$ of I-graded $\mathcal{S}$-complexes satisfying the following conditions:
\begin{itemize}
\item[(i)] For each $\delta>0$, there exists an $N$ such that $i,j>N$ implies that $\delta_{i,j}<\delta$.
\item[(ii)] The maps $\widetilde \lambda_j^k\circ \wt \phi_i^j(1)$ and $\wt \phi_j^k(2)\circ \widetilde \lambda_i^j$ are $\mathcal{S}$-chain homotopy equivalent to $\widetilde \lambda_i^k$ via an $\mathcal{S}$-chain homotopy of some level $\delta_{i,j,k}$. For every $\delta>0$, there exists an $N$ such that $\delta_{i,j,k} < \delta$ for $i,j, k>N$.
\end{itemize}
The morphism is {\it local} if each $\wt \lambda_i^j$ is a local morphism of I-graded $\mathcal{S}$-complexes. A morphism between enriched complexes is a {\it chain homotopy equivalence} if each $\widetilde \lambda_i^j$ is an $\mathcal{S}$-chain homotopy equivalence, where the involved $\mathcal{S}$-chain homotopy inverses and $\mathcal{S}$-chain homotopies have levels which converge to zero. A {\it{weak}} morphism is the above data, but without condition (ii) necessarily holding.
\end{defn}

Let $K$ be a knot in an integer homology $3$-sphere $Y$. For a perturbation $\pi$ of the Chern--Simons functional (along with other auxiliary choices) with nondegenerate critical set and unobstructed moduli spaces, one can associate an I-graded $\mathcal{S}$-complex $\wt C(Y,K,\pi)$ over $R[U^{\pm 1}]$. The generators of $C(Y,K,\pi)$ roughly correspond to homotopy classes $[A]$ of singular connections on $\R\times (Y,K)$, mod gauge, with limit at $-\infty$ some irreducible critical point $\alpha$ of the perturbed Chern--Simons functional, and limit at $+\infty$ the reducible $\theta$. The connection $[A]$ determines an $\R$-valued lift of the perturbed Chern--Simons invariant of $\alpha$, which determines the I-grading, and a $\Z$-valued lift of the $\Z/4$-grading of $\alpha$. Setting $U=1$ and forgetting the I-grading structure gives the $\Z/4$-graded $\mathcal{S}$-complex of $(Y,K)$ considered in \Cref{review-S-comp}.

To define the enriched $\mathcal{S}$-complex $\mathfrak{C}(Y,K)$, we take a sequence of perturbations $\pi_i$  as above (with the other necessary auxiliary choices), such that the norms $\| \pi_i\|$ converge to zero as $i\to \infty$. Then set 
\[
\mathfrak{E} ( Y, K) := (\{\wt C^i=\wt{C} (Y, K, \pi_i)\}, \;\{\wt \phi^j_i\}, \;\mathfrak{K} ).
\]
The map $\wt \phi^j_i$ is induced by a 1-parameter family of perturbations (and other auxiliary data). Moreover, $\mathfrak{K}\subset \R$ is the set of critical values of the (unperturbed) Chern--Simons functional. By taking suitable 2-parameter families of perturbations and using the associated parametrized singular instanton moduli spaces, one can verify that $\wt \phi^k_j\circ \wt \phi_i^j$ is $\mathcal{S}$-chain homotopy equivalent to $\wt \phi_i^k$ via an $\mathcal{S}$-chain homotopy of some level $\delta_{i,k,j}$ as required. The enriched complex $\mathfrak{C}(Y,K)$ is an invariant of $(Y,K)$ up to homotopy equivalence. See \cite[\S 7]{DS19} and \cite[\S 2]{DS20}, although note that there the unperturbed Chern--Simons functional is used to define I-gradings.

For cobordism maps, we have the following. Recall the terminology of \Cref{defn:negdefcob}.

\begin{prop}[{\cite[Theorem 7.18]{DS19}}] \label{nagative definte cobordism induces}
Suppose $(W, S,c)$, with $c\in H^2(W;\Z)$, is a cobordism $(Y,K)\to (Y',K')$ which is negative definite of height $l\geq 0$ over $R$. Then there is an induced height $l$ morphism of level $2\kappa_{\rm{min}}(W, S,c)$ from $\mathfrak{E} (Y, K)$ to $\mathfrak{E} (Y', K')$ such that for each morphism $\wt \lambda_i^j$ in the sequence, $c_l=\eta(W,S,c)$. 
\end{prop}

Two enriched $\mathcal{S}$-complexes $\mathfrak{E}$ and $\mathfrak{E}'$ are {\it locally equivalent}, written $\mathfrak{E}\sim \mathfrak{E}'$, if there are local morphisms $\mathfrak{E}\to \mathfrak{E}'$ and $\mathfrak{E}' \to \mathfrak{E}$. Define the local equivalence group of enriched complexes:
\[
\Theta^{\mathfrak{E}} _R = \left\{ \text{ enriched $\mathcal{S}$-complexes over $R[U^{\pm 1}]$ } \right\} / \sim
\]
Just as in the case for $\mathcal{S}$-complexes, this is an abelian group. The zero element is determined by $\{({\wt C}^i, \tilde d^i,\chi^i) := (R[U^{\pm 1}],0,0)\}_{i \geq 1}$ with
$\{\phi^j_i:= \mathrm{id}_{R[U^{\pm 1}]}\}_{i,j \geq 1}$ and $\mathfrak{K}=\Z\subset \R$. Addition is defined by taking tensor products of enriched $\mathcal{S}$-complexes.

It is sometimes convenient to use a slightly weaker definition of local equivalence by replacing local morphisms with weak local morphisms. We call the resulting relation {\it weak local equivalence}.

\vspace{1mm}

\begin{rem} \label{example-enriched}
The group $\Theta^\mathfrak{E}_R$ is uncountable. To see this, we construct for each $r\in \R_{>0}$ an enriched $\mathcal{S}$-complex as follows. First, define an I-graded $\mathcal{S}$-complex as follows:
\[
	\wt{C}\{r\} = R[U^{\pm 1}]\langle \zeta \rangle \oplus R[U^{\pm 1}]\langle \chi(\zeta) \rangle \oplus R[U^{\pm 1}]_{(0)}
\]
The generator $\zeta$ has $\text{deg}_I=r$ and $\Z$-grading $1$. Further, $\wt{d}(\zeta) = \delta_1 (\zeta)=1$. By taking the constant sequence of this I-graded $\mathcal{S}$-complex, we obtain an enriched $\mathcal{S}$-complex $\mathfrak{E}\{r\}$. For distinct positive real numbers $r$ and $r'$, one can show that $\mathfrak{E}\{r\}$ and $\mathfrak{E}\{r'\}$ are not locally equivalent. 
\end{rem}

Recall that $\Theta^{3,1}_\Z$ is the homology concordance group of knots in integer homology 3-spheres. We define a map $\Omega$ from $\Theta^{3,1}_\Z$ to the local equivalence group $\Theta_R^\mathfrak{E}$ by assigning to $(Y,K)$ the class of the enriched complex $\mathfrak{E} (Y, K)$.

\begin{thm} [{\cite[Theorem 7.18]{DS19}}] \label{Omega}
The map $
\Om : \Theta^{3,1}_\Z \to \Theta^{\mathfrak{E}}_R  $ is a homomorphism.  
\end{thm}

\subsection{Filtered special cycles}\label{filt-spec-cyc}

We next adapt some of the constructions involving special cycles from \Cref{section: special cycles} to the setting of I-graded $\mathcal{S}$-complexes and enriched $\mathcal{S}$-complexes. As before, $R$ is an integral domain algebra over $\Z[T^{\pm 1}]$.

We begin by discussing equivariant complexes in this context. Let $\wt C$ be an I-graded $\mathcal{S}$-complex over $R[U^{\pm 1}]$. Recall that in \Cref{subsection: equivariant} we associated to the underlying $\mathcal{S}$-complex of $\wt C$ the large equivariant complexes $\lhc, \loc, \lcc$ and small equivariant complexes $\shc, \soc, \scc$. These are $\Z$-graded complexes over $R[U^{\pm 1}][x]$. We extend $\text{deg}_I$ in an obvious manner to each of these complexes. For example, given $\zeta = \sum_{i=0}^N \zeta_i x^i \in \lhc$ where $\zeta_i\in \wt C$, we have
\[
	\text{deg}_I ( \zeta ) = \sup_i \left\{ \text{deg}_I(\zeta_i) \right\}
\]
Given this, it is immediate from the I-graded structure of $\wt C$ that we have $\deg_I ( \wh{d}(\zeta)) \leq \deg_I ( \zeta)$ for each $\zeta\in \lhc$. The case for the other equivariant complexes is similar. Note that since $\loc$ and $\lcc$ contain elements having infinite Laurent power series in $x^{-1}$, the possible values of $\deg_I$ in these cases may include $\infty$.

The following result concerns the maps studied in the context of equivariant theories in
 \Cref{subsection: equivariant}.
 
\begin{lem}
The chain maps $\hPhi, \oPhi, \cPhi$ and $\hPsi, \oPsi, \cPsi$ are filtered maps (level $0$): we have $\deg_I (\hPhi(\zeta) ) \leq \deg_I (\zeta)$ for all $\zeta \in \lhc$, and similarly for the others. Also, the maps $\mathfrak{i}, \mathfrak{j},\mathfrak{k}$ and ${\bf i},{\bf j},{\bf k}$ are filtered.
\end{lem}
\begin{proof}
We prove $\oPsi$ is filtered. Recall that $\oPsi : \soc  \to  \loc$ is defined by
 \begin{align*}
\oPsi ( \sum_{j=-\infty}^{N}  a_j x^i  )=
(  \sum_{j=-\infty}^{N} \sum_{l=0}^\infty v^l \delta_2 ( a_j) x^{j-l-1}  , 0, \sum_{j=-\infty}^{N} a_j x^j ).
\end{align*}
Note that $x$ does not affect the I-grading and both of $\delta_2$ and $v^j$ are decreasing in I-grading. This implies $\oPsi$ is also filtered. The proofs for the other maps are the same. 
\end{proof}

To an I-graded $\mathcal{S}$-complex $\wt C$ over $R[U^{\pm 1}]$, each $k\in \Z$ and $s\in \R$, define
\begin{equation}\label{ideal-filtered}
J_k ^{\leq s} (\wt{C})  := \left\{ c_{-k} \in R \;\; \Big|\;\; \begin{array}{c}c_{-k} x^{-k} + c_{-k-1} x^{-k-1}+ \cdots = \overline \Phi{\bf{i}}(z) \in R[U^{\pm 1}] [\![x^{-1}, x],\\[2mm] \;\; z\in \lhc_{2k},\;\; \deg_I(\wh d (z)) \leq s \end{array} \right\} 
\end{equation}
This is an ideal of $R$. If $s\leq s'$ then we have an inclusion $J_k^{\leq s} (\wt{C}) \subset J_k^{\leq s'} (\wt{C})$. We also have $J_k^{\leq -\infty} (\wt{C}) = J_k (\wt{C})$, where $J_k(\wt C)$ is the ideal defined in \Cref{section: special cycles}, regarding $\wt C$ as an $\mathcal{S}$-complex over $R$, having set $U=1$. We will see below that for $s<0$, these ideals give rise to local equivalence invariants.

Now we arrive at our definition of special cycles in this setting.
 
\begin{defn}\label{filtered special cycles}
Let $f\in R$, $s \in \R\cup \{-\infty\}$, $k \in \Z$ and $\wt{C}$ be an I-graded $\mathcal{S}$-complex. 
We say $z\in  \lhc$ is a {\it filtered special $(k, f, s)$-cycle} if there exists $\mathfrak{z} \in \shc_{2k}$ 
satisfying $\wh{\Psi}(\mathfrak{z})=z$ and also:
\[
	  \mathfrak{i}(\mathfrak{z}) = f x^{-k} + \sum_{i=-\infty}^{-k-1} b_i x^{i},  \hspace{.5cm } \deg_I (\shd \mathfrak{z}) \leq s.
\]
\end{defn}

\begin{rem}\label{remarks special cycles}
 As in the case of unfiltered special cycles, the chain $\mathfrak{z}$ 
in Definition \ref{filtered special cycles} is uniquely determined by 
the special cycle $z \in \lhc$.
Note that if $s \leq  s'$ and $z$ is a filtered special $(k,f,s)$-cycle, then $z$ is also a filtered special $(k,f,s')$-cycle. Furthermore, $z$ is a filtered special $(k,f,s)$-cycle, then $x\cdot z$ is a filtered special $(k-1,f,s)$-cycle.
Consequently, we have
\[
  J_k ^{\leq s} (\wt{C})\subset J_{k'} ^{\leq s'} (\wt{C})
\]
if $k'\leq k$ and $s'\leq s$.
 Filtered special $(k,f,-\infty)$-cycles coincide with unfiltered special $(k,f)$-cycles. 
In general, filtered special cycles are {\it not} cycles in $\lhc$: they are cycles in $\lhc / \lhc^{\leq s} $, where $\lhc^{\leq s}$ consists of elements in $\lhc$ with $\text{deg}_I\leq s$. For any $k\in \Z$, $f \in R$ and $s\in \R$, there exists a special $(k, f, s)$-cycle
in $\lhc$ if and only if
$f\in J_{k}^{\leq s} (\wt{C})$. 
\end{rem}

\vspace{1mm}

\begin{rem}\label{rem:degioffscyc}
Let $z$ be a special $(k,f,s)$-cycle with $f\neq 0$, where $z=\wh \Psi(\mathfrak{z})$ and $\mathfrak{z} = (\alpha,\sum_{i=0}^N a_i x^i)\in \shc_{2k}$. Since $\wh \Phi(z)=\mathfrak{z}$ and $\wh \Psi$, $\wh\Phi$ are filtered, we have $\text{deg}_I(z)=\text{deg}_I(\mathfrak{z})$. We further obtain
\[
	\text{deg}_I(z) = 	\max\{ \text{deg}_I(\alpha), 0 \}.
\]
This follows from two observations: $\mathfrak{i}(\mathfrak{z})$ is equal to $fx^{-k}$ plus lower order terms; the $\Z$-grading of $\mathfrak{z}$ equals $2k$, implying that $a_i = f_i U^{(i+k)/2}$ where $f_i\in R$. In particular, $\deg_I(a_i)\leq 0$. Note that if $k>0$, all $a_i$'s are zero and $\deg_I(\alpha)$ is automatically positive.
\end{rem}

The following describes the behavior of filtered special cycles under morphisms.
 
\begin{lem}\label{filtered Fr under morphism}
Let $\wt{\lambda}:\wt C\to \wt C'$ be a height $i$ morphism of level $\kappa\geq 0$ between I-graded $\mathcal{S}$-complexes. Suppose $s\in \R$ satisfies $s<-\kappa$. Then for a special $(k, f, s)$-cycle $z \in \lhc_{2k}$, the chain defined by
\[
\hPsi' \circ \hPhi' \circ \lhl (z) \in \lhc_{2k+2i}'
\]
is a special $(k+i, c_if, s+ \kappa)$-cycle, where $c_i$ is defined as in \eqref{eq:cjdefn}. Moreover, we have
\begin{align}\label{functoriality of filtered cycle}
\deg_I(\hPsi' \circ \hPhi' \circ \lhl (z)) \leq \deg_I (z) + \kappa. 
\end{align}
\end{lem}
\begin{proof}
This is a filtered version of \cref{special cycle: height i}. Note that $\hPsi'$ and $\hPhi'$ are filtered maps and $\lhl$ is a map of level $\kappa$.
Thus, we have \eqref{functoriality of filtered cycle}. Let $z = \hPsi(\mathfrak{z})$
and $\mathfrak{z}' =\hPhi' \circ \lhl (z) = \shl(\mathfrak{z})$. First note that $\mathfrak{d}'(\mathfrak{z}')=\mathfrak{d}'\wh \lambda(\mathfrak{z}) = \wh \lambda \mathfrak{d}(\mathfrak{z})$. Then $\text{deg}_I(\mathfrak{d} \mathfrak{z} )\leq s$ and the assumption that $\wt \lambda$ has level $\kappa$ imply $\text{deg}_I(\mathfrak{d}' \mathfrak{z}')\leqslant s+ \kappa$. Further, 
\begin{align}
	\mathfrak{i}'(\mathfrak{z}')=\mathfrak{i}'\circ \wh \lambda(\mathfrak{z}) & = \mathfrak{i}'\circ \wh \Phi' \circ {\wh{\boldsymbol{\lambda}}}\circ \wh \Psi (\mathfrak{z}) =\overline{\Phi}'\circ \overline{\bf \lambda} \circ \mathbf{i} \circ \wh \Psi (\mathfrak{z}) \nonumber \\[2mm]
	&  = 
	 \overline{\Phi}'\circ \overline{\boldsymbol{\lambda}} \circ\overline{\Psi}\circ \mathfrak{i}(\mathfrak{z}) - \overline{\Phi}'\circ\overline{\boldsymbol{\lambda}} \circ K_\mathfrak{i}\circ  \wh {\mathfrak{d}} (\mathfrak{z}) \label{eq:morphismfilteredspeccyc}
\end{align}
where we have used that 
\begin{equation} \label{K-i-identity}
	\overline\Psi \circ \mathfrak{i} - {\bf i}\circ \wh \Psi = \overline {d} K_{\mathfrak{i}} + K_{\mathfrak{i}}\wh {\mathfrak{d}}, 
\end{equation}
with $K_{\mathfrak{i}}:\wh {\mathfrak{C}}\to \overline{\bf{C}}$ being defined by
\[
  K_{\mathfrak i}(\alpha,\sum_{i=0}^{N}a_ix^i)=(-\sum_{i=0}^\infty v^i(\alpha)x^{-i-1},0,0).
\]
The first term in \eqref{eq:morphismfilteredspeccyc} has $x$-degree $-k-i$ with leading term $f c_i$ and I-grading equal to zero. The second term has I-grading at most $s+\kappa<0$. Thus $\text{deg}_I(\mathfrak{i}'(\mathfrak{z}'))<0$. Furthermore, the second term in \eqref{eq:morphismfilteredspeccyc} can be written as follows, where each $a_j\in R$:
\[
	 \sum_{j=-\infty }^N a_j U^{l_j}x^j
\]
Since the I-grading is negative, each $l_j<0$. Since the $\Z$-grading is $2k+2i$, we have $j=-2i-2k-4l_j$. Thus this power series has $x$-degree less than $-i-k$. Consequently, $\mathfrak{i}'(\mathfrak{z}')$ is dominated by the first term in \eqref{eq:morphismfilteredspeccyc}, and has $x$-degree $-k-i$ with leading term $f c_i$.
\end{proof}

\begin{cor}\label{cor:idealfiltspec}
	Let $\wt{\lambda}:\wt C\to \wt C'$ be a height $i$ morphism of level $\kappa$ between I-graded $\mathcal{S}$-complexes.
	Then for any $s< -\kappa$, and with $c_i$ as defined in \eqref{eq:cjdefn}, we have
	\[
	  c_iJ_k ^{\leq s} (\wt{C})\subseteq J_{k+i} ^{\leq s+\kappa} (\wt{C}').
	\]
	
\end{cor}

Next, we consider tensor products of filtered special cycles. 
\begin{lem}\label{tensor formula of filtered Fr}
Let $z \in \lhc_{2k}$ (resp.\ $z' \in \lhc'_{2k'}$)
be a special $(k,f, s)$-cycle (resp. $(k',f', s')$-cycle) such that
\begin{equation} \label{deg-I-sum}
  \deg_I(z) + s',\,\deg_I(z') + s<0.
\end{equation}
Then, the chain defined by
\[
z^\otimes:=\hPsi^{\otimes} \circ \hPhi^{\otimes} \circ \wh{T}
(z \otimes_{R[x]} z') \in \lhc^{\otimes}_{2k + 2k'}
\]
is a special $\big(k+k', ff', \max  \{ \deg_I(z) + s', \deg_I(z') +s \} \big)$-cycle. Moreover, we have 
\begin{align}\label{func tensor filtered special}
\deg_I( \hPsi^{\otimes} \circ \hPhi^{\otimes} \circ \wh{T}(z \otimes_{R[x]} z')) \leq  \deg_I(z) + \deg_I(z') . 
\end{align}
\end{lem}
\begin{proof}
This is a filtered version of \cref{special cycle tensor}. The inequality \eqref{func tensor filtered special} follows since all of maps $\hPsi^{\otimes}$, $\hPhi^{\otimes}$ and $\wh{T}$ are filtered. 
Since $z \in \lhc_{2k}$ (resp.\ $z' \in \lhc'_{2k'}$)
is a special $(k,f, s)$-cycle (resp. $(k',f', s')$-cycle), we can take $\mathfrak{z} \in \shc_{2k}$ (resp. $\mathfrak{z}' \in \shc_{2k'}$ ) such that $\deg_I ( \hat{\mathfrak{d}} (\mathfrak{z} )) \leq s $, $\wh{\Psi}(\mathfrak{z})=z$ and the leading term of $\mathfrak{i}(\mathfrak{z})$ is $fx^{-k}$ (resp. $\deg_I ( \hat{\mathfrak{d}}' (\mathfrak{z} ' )) \leq s' $,$\wh{\Psi}'(\mathfrak{z}')=z'$ and the leading term of $\mathfrak{i'}(\mathfrak{z'})$ is $f'x^{-k'}$). Since $\wh{\Psi}$ and $\wh{\Psi}'$ are filtered chain maps, we also have $\deg_I ( \widehat{d} (z)) \leq s $ (resp. $\deg_I ( \widehat{d'} (z')) \leq s'$). Using this, we obtain
 \begin{align}
 	\deg_I( \wh{d}^\otimes (z \otimes_{R[x]} z') ) & = \deg_I( \wh{d}(z) \otimes z' + z \otimes \wh{d} '(z')) \nonumber\\[2mm]
 	& \leq \max \{\deg_I( \wh{d}(z))+ \deg (z'), \deg_I( \wh{d}'(z'))+ \deg_I z  \} \nonumber \\[2mm]
	 & \leq  \max  \{ \deg_I(z) + s', \deg_I(z') +s \}. \label{bound-tensor-deg-I}
 \end{align} 
Since $\hPsi^{\otimes}$, $\hPhi^{\otimes}$ and $\wh{T}$ are filtered chain maps, we conclude that $\deg_I(\wh{d}^\otimes(z^\otimes))$ is bounded by the same term as in \eqref{bound-tensor-deg-I}.

To complete the proof, we need to analyze $\mathfrak{i}^\otimes (\mathfrak{z}^\otimes)$ with $\mathfrak{z}^\otimes= \hPhi^{\otimes} \wh{T}(z \otimes_{R[x]} z')$. Using \Cref{phi-psi-further-props}, relations \eqref{i-T-commute} and \eqref{K-i-identity}, we have the following:
\begin{align}
	\mathfrak{i}^\otimes (\mathfrak{z}^\otimes)
	&= \bPhi^{\otimes}{\bf i}^\otimes \wh{T}(z \otimes_{R[x]} z') \nonumber \\[2mm]
	&= \bPhi^{\otimes} \overline {T}({\bf i}\wh{\Psi} ({\mathfrak z}) \otimes_{R[x]} {\bf i}'\wh{\Psi}'({\mathfrak z}')) \nonumber\\[2mm]
	&= \bPhi^{\otimes} \overline {T}\left((\overline\Psi \circ \mathfrak{i}-\overline {d} K_{\mathfrak{i}}-K_{\mathfrak{i}}\wh {\mathfrak{d}})({\mathfrak z})) \otimes_{R[x]} 
	(\overline\Psi' \circ \mathfrak{i}'-\overline {d}' K_{\mathfrak{i}}'-K_{\mathfrak{i}}'\wh {\mathfrak{d}}')({\mathfrak z}')\right) \label{simplified-i-z-tensor}
\end{align}
Among the nine terms obtained by expanding the above expression, first consider the following term: 
\begin{align*}
	\bPhi^{\otimes} \overline {T}(\overline\Psi \circ \mathfrak{i}({\mathfrak z}) \otimes_{R[x]} \overline\Psi' \circ \mathfrak{i}'({\mathfrak z}')).
\end{align*}	
\Cref{tensor product lem} implies that the above expression is equal to $\mathfrak{i}({\mathfrak z})\cdot \mathfrak{i'}({\mathfrak z'})$, which has the leading term $ff'x^{-k-k'}$. Using the assumption in \eqref{deg-I-sum} and an argument as in the proof of \Cref{filtered Fr under morphism}, we can show the five terms in \eqref{simplified-i-z-tensor} involving either $\mathfrak d$ or $\mathfrak d'$ have $x$-degree less than $-k-k'$. It is also easy to see that the remaining three terms involving $\overline d$ or $\overline d'$ vanish. In summary $\mathfrak{i}^\otimes (\mathfrak{z}^\otimes)$ has the leading term $ff'x^{-k-k'}$.
\end{proof}

Next, we adapt the construction of the ideals in \eqref{ideal-filtered} to the case of enriched $\cS$-complexes.

\begin{lem}
Let $\mathfrak{E}$ be an enriched $\mathcal{S}$-complex over $R[U^{\pm 1}]$ consisting of $\{{\wt C}^i\}, \{\wt \phi_i^j\}, \mathfrak{K}$. If $s\notin \mathfrak{K}$, then the ideal $J_k ^{\leq s} (\wt{C}^n)\subset  R$ does not depend on $n$ for sufficiently large $n$.
\end{lem}

\begin{proof}
For an interval $[s_1,s_2]\subset \R \setminus\mathfrak{K}$, there is some $N>0$ such that for all $n>N$, the ideals $J_k^{\leq s_1}(\wt C^n)$ and $J_k^{\leq s_2}(\wt C^n)$ are equal. The proof follows from this observation and \cref{cor:idealfiltspec}.
\end{proof}

From this lemma it follows that for $s \notin \mathfrak{K}$, we can define $J_k^{\leq s} (\mathfrak{E}) := J_k ^{\leq s} (\wt{C}^n)$ for a sufficiently large $n$. For $s \in \mathfrak{K}$, we define $J_k^{\leq s} (\mathfrak{E}) := J_k ^{\leq s-\epsilon} (\mathfrak{E})$ for a sufficiently small $\epsilon>0$. Moreover, from \cref{cor:idealfiltspec} we obtain:

\begin{cor}\label{cor:idealfiltspecenriched}
	Let $\mathfrak{L}:\mathfrak{E}\to \mathfrak{E}'$ be a height $i$ morphism of level $\kappa$ between enriched $\mathcal{S}$-complexes.
	Then for any $s< -\kappa$, and with $c_i$ as defined in \eqref{eq:cjdefn} (which are the same for all $\wt\lambda_i^j$ involved in $\mathfrak{L}$), we have
	\[
	  c_iJ_k ^{\leq s} (\mathfrak{E})\subseteq J_{k+i} ^{\leq s+\kappa} (\mathfrak{E}').
	\]
	In particular, $J_k ^{\leq s} (\mathfrak{E})$ is a local equivalence invariant of the enriched $\mathcal{S}$-complex $\mathfrak{E}$.
\end{cor}

\begin{rem}
	In fact, the arguments show that $J_k ^{\leq s} (\mathfrak{E})$ is a weak local equivalence invariant.
\end{rem}

Finally, we give conditions for when an enriched $\mathcal{S}$-complex is (weakly) locally equivalent to the trivial enriched complex. First, we need the following.

\begin{lem}\label{triviality via level 0}Let $\wt{C}$ be a I-graded $\mathcal{S}$-complex. 
If there is a special $(0, 1, -\infty)$-cycle $z$ so that 
\begin{equation}\label{eq:degizeq}
\deg_I (z) = 0  ,
\end{equation}
then there is a level-$0$ local map from the trivial I-graded $\mathcal{S}$-complex $R[U^{\pm 1}]_{(0)}$ to $\wt{C}$.  
\end{lem}

 \begin{proof}
 A filtered special $(0,1,-\infty)$-cycle $z=\wh\Psi(\mathfrak{z})$, where $\mathfrak{z}=\left(\al, 1\right)\in {\shc}_{-2}$, sastisfies $\text{deg}_I( \mathfrak{d}\mathfrak{z})\leq -\infty$. This implies $d\al-\delta_{2}(1)=0$. 
Define a local morphism of $\mathcal{S}$-complexes $\wt{\lambda} : R[U^{\pm 1}]_{(0)} \to \wt{C}$ by 
$\wt{\lambda}(0,0,1)= (0,\alpha, 1)$. If $z$ is as in the statement of the lemma, by \eqref{eq:degizeq}, $\deg_I (\alpha) = 0$ holds, which implies this is a level $0$ morphism.
\end{proof}

\begin{cor}\label{triviality via Fr cycles}
An enriched complex $\mathfrak{E}$ consisting of $\{{\wt C}^i\}, \{\wt \phi_i^j\}, \mathfrak{K}$ is weakly local equivalent to the trivial enriched complex if and only if the following two conditions hold: 
\begin{itemize}
    \item[(i)] There is a sequence of special $(0, 1, -\infty)$-cycles $\{z_j\}$ for $\{\wt{C}^j\}$ so that $ \lim_{j\to \infty}  \deg_I (z_j ) \leq 0$.
    \item[(ii)]
    There is a sequence of special $(0, 1, -\infty)$-cycles $\{z_j\}$ for $\{(\wt{C}^j)^\dagger\}$ so that $\lim_{j\to \infty}  \deg_I (z_j) \leq 0$.  
\end{itemize}
\end{cor}
\begin{proof}
Using (i), as in the proof of \cref{triviality via level 0}, we define local morphisms $\wt{\lambda}_i : R[U^{\pm 1}]_{(0)} \to \wt{C}^i $
by $\wt{\lambda}_i (1) := (0, \alpha_i, 0)$, where $\alpha_i$ is an element in $\wt{C}^i$ such that $d\alpha_i=\delta_2(1)$ and $\lim_i \deg_I(\alpha_i) \leq 0$. The sequence $\mathfrak{L} := \{\wt{\lambda}_i\}$ defines a weak local morphism from the trivial enriched $\mathcal{S}$-complex to $\mathfrak{E}$. Using (ii) and dualizing, we obtain a weak local morphism in the opposite direction.
\end{proof}

\subsection{Numerical invariants for local equivalence classes of enriched complexes}

Using filtered special cycles, we define local a equivalence invariant of enriched $\mathcal{S}$-complexes called $\mathcal{J}_{\mathfrak{E}} (k,s)$. This definition is motivated by a desire to generalize the $\Gamma$ invariants from \cite{D18,DS19} and the $r_s$-type invariants from \cite{NST19}.

\begin{defn}\label{N-I-s-comp}
For an I-graded $\mathcal{S}$-complex $\wt{C}$, define 
\begin{align}
\newinv_{\wt{C}} (k,s) :=\inf \left\{ \deg_I ( z ) \mid  z \text{ is filtered special $(k,1, s)$-cycle }\right\}\in [0, \infty]
\end{align}
for $(k,s)\in \Z \times [-\infty, 0)$.
\end{defn}

\begin{lem}\label{N-split-des}
	For an I-graded $\mathcal{S}$-complex $\wt{C}$ and a positive integer $k$ and $s\in  [-\infty, 0)$, we have
	\begin{equation}\label{N-k-p}
	  \newinv_{\wt{C}} (k,s)=\inf\left\{\deg_I(\alpha) \;\; \Big|\;\; \begin{array}{c}
	  \alpha\in C_{2k-1},\, \, \delta_1 v^{j}\alpha=0\,\, \text{ for }\,\, 0\leq j\leq k-2,\\[2mm] \delta_1 v^{k-1}\alpha=1,\, \deg_I(d\alpha)\leq s\end{array} \right\} 
	\end{equation}
	and if $k$ is a non-positive integer we have
	\begin{equation}\label{N-k-np}
	  \newinv_{\wt{C}} (k,s)=\inf\left\{\max\{\deg_I(\alpha),0\} \;\; \Big|\;\; 
	  \begin{array}{c}\alpha\in C_{2k-1},\, \, a_i\in R[U^{\pm 1}]\,\, (0\leq i\leq -k),\, a_{-k}=1,\, \\[2mm]
	  \deg_I(d\alpha-\sum_{i=0}^{-k}v^i\delta_2(a_i))\leq s \end{array} \right\} 
	\end{equation}
	In \eqref{N-k-np}, we may assume $a_i=0$ if $i\not\equiv k \pmod{2}$, and otherwise $a_i=f_iU^{(k+i)/2}$ for some $f_i\in R$.
\end{lem}
\begin{proof}
	This is a straightforward consequence of the definitions. For any $\alpha$ satisfying the condition in \eqref{N-k-p}, we obtain 
	the filtered special $(k,1,s)$ cycle $\hPsi(\alpha,0)$ for a positive integer $k$, and for any $\alpha$ and $a_i$ satisfying \eqref{N-k-np}
	we obtain the filtered special $(k,1,s)$ cycle $\hPsi(\alpha,\sum_{i=0}^{-k}a_ix^i)$ for a non-positive integer $k$.
	For the last part, see Remark \ref{rem:degioffscyc}.
\end{proof}

\begin{rem}\label{rem:newinvvanish}
	Note that $\newinv_{\wt C}(k,s)=0$ when $k\leq 0$ and $s\in [\text{deg}_I(\delta_2(1)), 0)$. To see this, one considers $\alpha=0$ in the conditions appearing in \eqref{N-k-np}.
\end{rem}

The following monotonicity property is a consequence of \Cref{remarks special cycles}.
\begin{lem}\label{mon-N}
	For any I-graded $\mathcal{S}$-complex $\wt{C}$, $\newinv_{\wt{C}}$ is increasing with respect to $k$ and decreasing with respect to $s$. That is to say, for any $(k,s), \, (k',s')\in \Z \times [-\infty, 0)$ 
	with $k'\leq k$ and $s'\geq s$, we have
	\[
	  \newinv_{\wt{C}} (k',s')\leq \newinv_{\wt{C}} (k,s).
	\]
\end{lem}

\begin{lem}\label{beh-mor-N}
	Let $\wt{\lambda}:\wt{C}\to \wt{C}'$ be a strong height $i$ morphism of level $\kappa\geq 0$ between I-graded $\mathcal{S}$-complexes. Then, for $k\in \Z$ and $s\in [-\infty,0)$ satisfying $s+\kappa<0$, we have
	\[
		{\newinv}_{\wt{C}'}(k+i,s+\kappa) \leq {\newinv}_{\wt{C}}(k,s) + \kappa
	\]
\end{lem}
\begin{proof}
	If $z$ is a filtered special $(k,1, s)$-cycle for $\wt{C}$, then Lemma \ref{filtered Fr under morphism} implies that 
	\begin{equation}\label{map-fil-spec-cyc}
	  \hPsi' \circ \hPhi' \circ \lhl (c_i^{-1}z)
	\end{equation}
	is a filtered special $(k+i,1, s+\kappa)$-cycle for $\wt{C}'$, where $c_i$ is defined as in \eqref{eq:cjdefn}. Moreover, $\deg_I$ of the special cycle in \eqref{map-fil-spec-cyc} is at most $\deg_I(z)+\kappa$. The claim
	follows from this observation.
\end{proof}

\begin{lem}\label{key to define J}
For an enriched complex $\mathfrak{E} = (\{{\wt C}^i\}, \{\wt \phi_i^j\}, \mathfrak{K})$, an integer $k$ and $s\in [-\infty, 0)\setminus \mathfrak{K}$, the limit of $\{\newinv_{\wt{C}^i} (k,s)\}_{i \in \Z_{>0}} $ exists in $[0, \infty]$.
\end{lem}

\begin{proof}
	If $[s_1,s_2]$ is an interval in $\R \setminus \mathfrak{K}$, then there is some $N>0$ such that for all $n>N$ any special $(k,f,s_2)$-cycle for ${\wt C}^n$ is also a special $(k,f,s_1)$-cycle.
	In particular, for any such $s_1$, $s_2$ and $N$, we have 
	\begin{equation}\label{pbs-loc-const}
		\newinv_{\wt{C}}(k,s_1)=\newinv_{\wt{C}}(k,s_2).
	\end{equation}
	This observation together with
	Lemma \ref{beh-mor-N} implies that $\{\newinv_{\wt{C}^i} (k,s)\}_{i \in \Z_{>0}} $ is a Cauchy sequence. Thus the limit exists.
\end{proof}

Lemma \ref{key to define J} allows us to extend Definition \ref{N-I-s-comp} to enriched complexes.
\begin{defn}\label{defn:newinv}
	For an enriched $\mathcal{S}$-complex $\mathfrak{E}$ given by the data $(\{{\wt C}^i\}, \{\wt \phi_i^j\}, \mathfrak{K})$, define the invariant $\newinv_{\mathfrak{E}}:\Z \times \left([-\infty, 0)\setminus \mathfrak K\right) \to [0,\infty]$ as follows:
	\begin{equation}\label{enriched-N}
		\newinv_{\mathfrak{E}} (k,s)   :=\lim_{i \to \infty} \newinv_{\wt{C}^i} (k,s).
	\end{equation}
	If $s\in  \mathfrak K\cap [-\infty, 0)$, then we define $\newinv_{\mathfrak{E}} (k,s)=\lim_{s'\to s^-}\newinv_{\mathfrak{E}} (k,s')$.
\end{defn}

\begin{ex}
	For the trivial enriched complex $\mathfrak{E}_0$, we have 
	\begin{align}
		\newinv_{\mathfrak{E}_0} (k,s)   = 
		\begin{cases}
			\infty, & k>0 \\ 
			0, & k \leq 0
		\end{cases},
	\end{align}
	and for the enriched complex $\mathfrak{E}\{r\}$ of Remark \ref{example-enriched} and its dual $\mathfrak{E}\{r\}^\dagger$, we have
	\begin{align}
		\newinv_{\mathfrak{E}\{r\}} (k,s)   = 
		\begin{cases}
			\infty, & k>1 \\ 
			r, & k=1 \\ 
			0, & k \leq 0
		\end{cases},
		\hspace{1cm}		
		\newinv_{\mathfrak{E}\{r\}^\dagger} (k,s)   = 
		\begin{cases}
			\infty, & k\geq 1 \\ 
			\infty, & k=0 \text{ and } s<r\\ 
			0, & \text{otherwise }
		\end{cases}.
	\end{align}
\end{ex}

\begin{rem}
	For an enriched $\mathcal{S}$-complex $\mathfrak{E}$, it is not necessarily true, as was the case for $\mathcal{S}$-complexes in \Cref{rem:newinvvanish}, that $\newinv_{\mathfrak{E}}(k,s)=0$ for $k\leq 0$ and $s$ close to zero. To see this, one can take $\mathfrak{E}$ to contain a sequence of $\mathcal{S}$-complexes $\wt C\{r_i\}^\dagger$ from \Cref{example-enriched} where $r_i $ are positive numbers approaching $0$. See, however, \Cref{rem:newinvvanishtop}.
\end{rem}

We may define a slight variation of $\newinv_{\mathfrak{E}}$ in the following way. 
\begin{defn}\label{N-I-s-comp-var}
	For an I-graded $\mathcal{S}$-complex $\wt{C}$ over $R[U^{\pm 1}]$, define 
	\begin{align}
		\underline{\newinv}_{\wt{C}} (k,s) :=\inf \left\{ \deg_I ( z ) \mid  z \text{ is filtered special $(k,f, s)$-cycle with $f\neq 0$}\right\}\in [0, \infty]
	\end{align}
	for $(k,s)\in \Z \times [-\infty, 0)$.
	For an enriched $\cS$-complex $\mathfrak{E} = (\{{\wt C}^i\}, \{\wt \phi_i^j\}, \mathfrak{K})$ over $R[U^{\pm 1}]$, define
	\begin{align}
		\underline{\newinv}_{\mathfrak{E}} (k,s) :=\lim_{i \to \infty} \underline{\newinv}_{\wt{C}^i} (k,s).
	\end{align}
	if $\Z \times \left([-\infty, 0)\setminus \mathfrak K\right)$. Extend this definition to the case that $s\in \mathfrak K$ by requiring that $\underline{\newinv}_{\wt{C}}$ is continuous from the left with respect to the variable $s$.
\end{defn}

\begin{rem}\label{N-und-simpl-def}
	We have an analogue of Lemma \ref{N-split-des} for $\underline{\newinv}_{\wt{C}} (k,s)$ by replacing the condition $\delta_1v^{k-1}\alpha=1$ in \eqref{N-k-p} and the condition $a_{-k}=1$ in 
	\eqref{N-k-np} respectively with $\delta_1v^{k-1}\alpha\neq 0$ and $a_{-k}\neq 0$.
\end{rem}

\vspace{1mm}

\begin{rem}
        In the case that $R$ is a field, the invariants $\newinv_{\mathfrak{E}}$ and $\underline{\newinv}_{\mathfrak{E}}$ agree. In general we have 
        ${\newinv}_{\mathfrak{E}}(k,s)\geq \underline{\newinv}_{\mathfrak{E}}(k,s)$. For instance, let $\mathfrak{E}'\{r\}$ be the enriched $\mathcal{S}$-complex with the same chain group as $\mathfrak{E}\{r\}$ 
        and the differential that is a multiple of the differential of $\mathfrak{E}\{r\}$ by a non-unit element.
        Then ${\newinv}_{\mathfrak{E}}(1,s)=\infty$ and $\underline{\newinv}_{\mathfrak{E}}(1,s)=r$.
\end{rem}

Next, we develop some basic properties for the invariants $\newinv_{\mathfrak{E}}$ and $\underline{\newinv}_{\mathfrak{E}}$.

\begin{lem} \label{values of J}
	For an enriched $\mathcal{S}$-complex $\mathfrak{E} = (\{{\wt C}^i\}, \{\wt \phi_i^j\}, \mathfrak{K})$, the map $\newinv_{\mathfrak{E}}$ is locally constant on the complement of $\Z\times {\mathfrak K}$ in $\Z \times [-\infty, 0)$.
	Moreover, $\newinv_{\mathfrak{E}}$ takes values in $\mathfrak{K} \cup \{ \infty\}$. Similar properties hold for the map $\underline{\newinv}_{\mathfrak{E}}$.
\end{lem}
\begin{proof}
	The first claim follows from \eqref{pbs-loc-const}. The second part is straightforward, and similar proofs can be used to address analogous claims for $\underline{\newinv}_{\mathfrak{E}}$. 
 \end{proof}
\begin{lem}\label{mon-N-enriched}
	For any enriched $\mathcal{S}$-complex $\mathfrak{E}$, $\newinv_{\mathfrak{E}}$ and $\underline{\newinv}_{\mathfrak{E}}$ are increasing with respect to $k$ and decreasing with respect to $s$. 
	That is to say, for any $(k,s), \, (k',s')\in \Z \times [-\infty, 0)$ with $k'\leq k$ and $s'\geq s$, we have
	\[
	  \newinv_{\mathfrak{E}} (k',s')\leq \newinv_{\mathfrak{E}} (k,s),\hspace{1cm}\underline {\newinv}_{\mathfrak{E}} (k',s')\leq \underline {\newinv}_{\mathfrak{E}} (k,s).
	\]
\end{lem}
\begin{proof}
	In the case of $\newinv_{\mathfrak{E}} $, the claim can be verified using Lemma \ref{mon-N}. The proof for $\underline{\newinv}_{\mathfrak{E}}$ is similar.
\end{proof}

\begin{prop}\label{Fr ineq}
	Let $\mathfrak{L}:\mathfrak{E}\to \mathfrak{E}'$ be a strong height $i$ morphism of level $\kappa$ between enriched $\mathcal{S}$-complexes.
	Then for $k\in \Z$ and $s\in [-\infty,0)$ we have
	\[
		{\newinv}_{\mathfrak{E}'}(k+i,s) \leq {\newinv}_{\mathfrak{E}}(k,s-\kappa) + \kappa.
	\]
	If $\mathfrak{L}:\mathfrak{E}\to \mathfrak{E}'$ is a height $i$ morphism of level $\kappa$ and $c_i$, defined as in \eqref{eq:cjdefn}, is non-zero, we have 
	\[
		\underline {\newinv}_{\mathfrak{E}'}(k+i,s) \leq \underline {\newinv}_{\mathfrak{E}}(k,s-\kappa) + \kappa.
	\]
\end{prop}
\begin{proof}
	The first part follows from \Cref{beh-mor-N}, and the second part can be proved in a similar way.
\end{proof}

\begin{rem}\label{Fr-ineq-strong}
	An analysis of the proof shows that Proposition \ref{Fr ineq} holds with the weaker assumption that the morphisms $\mathfrak{L}$ in the statement do not necessarily satisfy 
	condition (ii) of Definition \ref{enriched morphism definition}.
\end{rem}

\begin{cor}\label{loc-inv-N}
	The invariants ${\newinv}_{\mathfrak{E}}$ and $\underline {\newinv}_{\mathfrak{E}}$ depend only on the weak local equivariance type of $\mathfrak{E}$.
\end{cor}
\begin{proof}
	Proposition \ref{Fr ineq} implies that ${\newinv}_{\mathfrak{E}}$ and $\underline {\newinv}_{\mathfrak{E}}$ depends only the local equivariance type of $\mathfrak{E}$.
	We can upgrade this fact to our desired claim using Remark \ref{Fr-ineq-strong}.
\end{proof}

The following is a consequence of Lemma \ref{tensor formula of filtered Fr}.
\begin{thm}\label{conn sum for N}
	For two enriched complexes $\mathfrak{E}$ and $\mathfrak{E}'$, if $\newinv_\mathfrak{E} (k,s)$, $\newinv_\mathfrak{E'} (k',s')$ are finite and 
	\[s^\otimes := \max \{ \newinv_\mathfrak{E} (k,s) +s', \;\; \newinv_{\mathfrak{E}'} (k',s') +s\} <0,\]
	then the following inequalities hold:
	\[
	  \newinv_{\mathfrak{E}\otimes \mathfrak{E}' }(k+k', s^\otimes) \leq \newinv_\mathfrak{E} (k,s) + \newinv_{\mathfrak{E}'} (k',s'),\hspace{1cm}
	  \underline \newinv_{\mathfrak{E}\otimes \mathfrak{E}' }(k+k', s^\otimes) \leq \underline \newinv_\mathfrak{E} (k,s) + \underline \newinv_{\mathfrak{E}'} (k',s'). 
	\]
\end{thm}

In \cite[\S 7.5]{DS19}, an invariant $\Gamma_\mathfrak{E}$ is defined for any enriched $\mathcal{S}$-complex $\Gamma_\mathfrak{E}$. The following is a straightforward consequence of the definition of $\Gamma$, Lemma \ref{N-split-des} and Remark \ref{N-und-simpl-def}. 
\begin{lem}\label{Gamma-rel}
	For any enriched $\cS$-complex $\mathfrak{E}$ over $R[U^{\pm 1}]$, we have  
	\begin{align}
		\Gamma_{\mathfrak{E}} (k)   = \newinv_{\mathfrak{E} \otimes \Frac{R}} (k,-\infty) = \underline{\newinv}_{\mathfrak{E}} (k,-\infty),
	\end{align}
	where $\mathfrak{E} \otimes \Frac{R}$ is the enriched $\mathcal{S}$-complex obtained using the base change with respect to the field of fractions of $R$.
\end{lem}

The {\it transpose} of $\newinv_{\mathfrak{E}}$ for an enriched $\cS$-complex $\mathfrak{E}$ is the map $\newinv^\intercal_{\mathfrak{E}}:\Z\times [0,\infty]\to [-\infty,0]$ defined by
\begin{equation}\label{transpose-N}
	\newinv^\intercal_{\mathfrak{E}}(k,r) := \min\{\inf \{ s \in [-\infty, 0) \mid \newinv_{\mathfrak{E}}(k,s) \leq r \},0\} \qquad (r < \infty)
\end{equation}	
and $\newinv^\intercal_{\mathfrak{E}}(k,\infty) :=\lim_{r\to \infty} \newinv^\intercal_{\mathfrak{E}}(k,r)$. Note that $\newinv^\intercal_{\mathfrak{E}}(k,r)=0$ if and only if $\newinv_{\mathfrak{E}}(k,0)$ (and hence any $\newinv_{\mathfrak{E}}(k,s) $)	 is greater than $r$. It also follows immediately from the definition that
\begin{equation}\label{N-infty}
  \newinv _{\mathfrak{E}}(k,-\infty )=0 \hspace{.2cm}\Longleftrightarrow \hspace{.2cm} \newinv_{\mathfrak{E}}^\intercal(k,0 )=-\infty.
\end{equation}
This function is again increasing with respect to $k$ and decreasing with respect to $r$. Lemma \ref{values of J} implies that $\newinv^\intercal_{\mathfrak{E}}$ for any enriched complex $\mathfrak{E} = (\{{\wt C}^i\}, \{\wt \phi_i^j\}, \mathfrak{K})$ takes values in $\mathfrak{K} \cup \{-\infty\}$, and it is locally constant on the complement of $\Z\times {\mathfrak K}$ in $\Z \times  [0,\infty]$. Moreover, for any fixed value of $k$, the function $\newinv^\intercal_{\mathfrak{E}}(k,\cdot)$ is continuous from the right. We may recover $\newinv_{\mathfrak{E}}$ from $\newinv^\intercal_{\mathfrak{E}}$ using the following identity:
\[
   \hspace{4cm}\newinv_{\mathfrak{E}}(k,s)=\min \{ r \mid \newinv^\intercal_\mathfrak{E} (k,r)  \leq s\}\hspace{2cm} \text{if $s \notin \mathfrak{K}$.}
\]
The following is a consequence of the definition of $\newinv^\intercal_{\mathfrak{E}}(k,r)$ and \Cref{conn sum for N}:
\begin{equation}\label{eq:connsumntranspose}
  \newinv^\intercal_{\mathfrak{E}\otimes \mathfrak{E}' }(k+k', r+r') \leq \max \{ \newinv^\intercal_{\mathfrak{E}}(k,r)+r',\;\; \newinv^\intercal_{\mathfrak{E}'}(k',r')+r \} . 
\end{equation}
Finally, we mention that the analogue of \Cref{Fr ineq} holds for $\newinv^\intercal$: if there is a strong height $i$ morphism $\mathfrak{E}\to \mathfrak{E}'$ of level $\kappa\geq 0$ between I-graded $\mathcal{S}$-complexes, and $k\in \Z$ and $r\in [0,\infty]$, we have
	\begin{equation}\label{eq:newinvtransposemorphism}
		{\newinv}^\intercal_{\mathfrak{E}'}(k+i,r+\kappa) \leq {\newinv}^\intercal_{\mathfrak{E}}(k,r) + \kappa.
	\end{equation}
The following is the analogue of \Cref{N-split-des}, and the proof is straightforward.

\begin{lem}\label{N-transpose-split-des}
	For an I-graded $\mathcal{S}$-complex $\wt{C}$ and a positive integer $k$ and $r\in  [0,\infty]$, we have
	\begin{equation}\label{N-k-p}
	  \newinv^\intercal_{\wt{C}} (k,r)=\min\left\{\inf\left\{\deg_I(d \alpha) \;\; \Big|\;\; \begin{array}{c}
	  \alpha\in C_{2k-1},\, \, \delta_1 v^{j}\alpha=0\,\, \text{ for }\,\, 0\leq j\leq k-2,\\[2mm] \delta_1 v^{k-1}\alpha=1,\, \deg_I(\alpha)\leq r\end{array} \right\} , \;\; 0 \right\}
	\end{equation}
	and if $k$ is a non-positive integer we have
	\begin{equation}\label{N-k-np}
	  \newinv^\intercal_{\wt{C}} (k,r)=\inf\left\{ \deg_I(d\alpha-\sum_{i=0}^{-k}v^i\delta_2(a_i)) \;\; \Big|\;\; 
	  \begin{array}{c}\alpha\in C_{2k-1},\, \, a_i\in R[U^{\pm 1}]\,\, (0\leq i\leq -k), \\[2mm]
	 \, a_{-k}=1,\,  {\rm{deg}}_I(\alpha)\leq r \end{array} \right\} 
	\end{equation}
	In \eqref{N-k-np}, we may assume $a_i=0$ if $i\not\equiv k \pmod{2}$, and otherwise $a_i=f_iU^{(k+i)/2}$ for some $f_i\in R$.
\end{lem}

Motivated by \cite{NST19}, we define another numerical invariant for enriched complexes.
\begin{defn}
	For any enriched $\cS$-complex $\mathfrak{E}$ over $R[U^{\pm 1}]$ and any $s\in [-\infty,0]$, we define 
	\begin{align}
		r_s (\mathfrak{E}) :=-\newinv^\intercal_{\mathfrak{E}} (0,-s)\in [0, \infty]. 
	\end{align}
\end{defn}

The following is a corollary of Theorem \ref{conn sum for N} and \eqref{eq:connsumntranspose}.
\begin{cor}\label{conn sum key} 
	The invariants $\Gamma_{\mathfrak{E}}$ and $r_s (\mathfrak{E})$ satisfy the following inequalities:
	\begin{align*}
		&\Gamma_{\mathfrak{E}\otimes \mathfrak{E}' }(k+k') \leq \Gamma_\mathfrak{E} (k) + \Gamma_{\mathfrak{E}'} (k'), \\[2mm]
		& r_{s+s'}(\mathfrak{E}\otimes \mathfrak{E}')- s-s' \geq \min \{ r_{s}(\mathfrak{E})-s , \;\; r_{s'}(\mathfrak{E}') -s'\}.
	\end{align*}
\end{cor}

Finally, the invariants that we have defined above can detect the (weak) local equivalence class of the trivial enriched $\mathcal{S}$-complex.

\begin{cor}\label{local eq gamma}
An enriched $\mathcal{S}$-complex $\mathfrak{E}$ is weakly locally equivalent to the trivial enriched $\mathcal{S}$-complex if and only if the following conditions hold:
\begin{equation}
\newinv_{\mathfrak{E}}(0,-\infty)=0 \text{ and }\newinv_{{\mathfrak{E}}^\dagger}(0,-\infty)=0.
\end{equation}
These conditions are equivalent to $r_0(\mathfrak{E})=\infty$ and $r_0(\mathfrak{E}^\dagger)=\infty$.
\end{cor}

\begin{proof}
The first statement follows from \cref{triviality via Fr cycles}, and the last statement from \eqref{N-infty}. 
\end{proof}

\subsection{Homology concordance invariants}

We now use the homomorphism $\Om: \Theta^{3,1}_\Z \to \Theta^\mathfrak{E}_R$ constructed in \cref{Omega}, together with the numerical invariants of $\Theta^\mathfrak{E}_R$ defined above, to construct homology concordance invariants for knots in integer homology 3-spheres. Unless otherwise stated, $R$ is an integral domain algebra over $\Z[T^{\pm 1}]$.

Let $K\subset Y$ be a knot in an integer homology $3$-sphere, and write $\mathfrak{E}(Y,K)$ for the associated enriched $\mathcal{S}$-complex. Applying the construction of \Cref{defn:newinv}, we obtain:
\begin{equation*}
		\newinv_{(Y, K)}(k,s) :=  \newinv_{\mathfrak{E}(Y,K)}(k,s)
\end{equation*}
Here $k\in \Z$, $s\in [-\infty,0)$, and $\newinv_{(Y, K)}(k,s)\in [0,\infty]$. We similarly define the homology concordance invariant $\underline{\newinv}_{(Y, K)}(k,s) $ using \Cref{N-I-s-comp-var}. We also apply the construction of \eqref{transpose-N} to define
\begin{equation*}
	\newinv^\intercal_{(Y, K)}(k,r) :=  \newinv^\intercal_{\mathfrak{E}(Y,K)}(k,r)
\end{equation*}
where $k\in \Z$, $r\in [0,\infty]$, and $\newinv^\intercal_{(Y, K)}(k,r)$ takes values in $[-\infty,0]$. By construction, all of these invariants factor through the homomorphism $\Omega$ and define homology concordance invariants. In the case that $Y$ is the $3$-sphere, we omit it from notation and write $\newinv_K$, and similarly for the other invariants.

\vspace{1mm}

\begin{rem}\label{rem:newinvvanishtop}
	For any sequence of perturbations $\pi_i$ (with auxiliary choices) approaching zero, we have $\lim\sup \text{deg}_I(\delta_2^i(1))\leq m<0$, where $m:=\max\{\mathfrak{K}\cap (-\infty,0)\}$, $\delta_2^i$ is the $\delta_2$-map defined using $\pi_i$, and $\mathfrak{K}$ is the set of critical values of the unperturbed Chern--Simons functional. This follows from non-degeneracy of the reducible connection and standard compactness properties. Using this and \Cref{rem:newinvvanish}, we obtain that $\newinv_{\mathfrak{E}(Y,K)}(k,s)=0$ for $k\leq 0$ and $s\in [m,0)$. In particular, $\newinv_{\mathfrak{E}}^\intercal(k,r)\leq m$ when $k\leq 0$.
\end{rem}

As special values, we have the following homology concordance invariants: 
	\begin{align*}
		&\Gamma_{(Y, K)}(k):= \underline{\newinv}_{(Y, K)}(k,-\infty)\\[2mm]
		& r_s (Y, K) := -\newinv^\intercal_{(Y, K)}(0,-s)
	\end{align*}
	The invariant $\Gamma_{(Y,K)}$ is the same as the invariant $\Gamma^R_{(Y,K)}$ studied in \cite{DS19,DS20}. The invariant $r_s (Y, K)$ is an analogue of the integer homology 3-sphere invariant invariant defined in \cite{NST19}. Later in this section we describe an explicit relationship between this latter invariant and $r_s(Y,U_1)$, where $U_1$ is an unknot in a small ball inside $Y$; see \eqref{inequality of rs}.
	
	The following connected sum inequalities follow from the inequalities of \Cref{conn sum for N} and \eqref{eq:connsumntranspose}, and the fact that $\Om: \Theta^{3,1}_\Z \to \Theta^\mathfrak{E}_R$ is a homomorphism. 

\begin{thm}\label{conn sum for JYK}
Given knots in integer homology $3$-spheres $(Y, K)$ and $(Y', K')$, suppose $s^\otimes<0$, where $s^\otimes := \max \{ \newinv_{(Y,K)} (k,s) +s', \newinv_{(Y',K')}(k',s') +s\}$. Then
\[
\newinv_{(Y \# Y' , K \# K') }(k+k', s^\otimes) \leq \newinv_{(Y, K)} (k,s) + \newinv_{(Y', K')} (k',s'),
\]
with the same inequality holding for the $\underline{\newinv}$-invariants. Furthermore, we have:
\[
\newinv^\intercal_{(Y \# Y' , K \# K') }(k+k', r + r' ) \leq \max \{ \newinv^\intercal_{(Y, K)}(k,r)+r', \;\newinv^\intercal_{(Y,' K')}(k',r')+r \} . 
\]
\end{thm}

Specializing to the case of the invariants $\Gamma_{(Y,K)}$ and $r_s(Y,K)$, we obtain the following. 

\begin{cor} 
\label{thm: connected sum rs}
	The invariants $\Gamma$ and $r_s$ satisfy the following inequalities:
	\begin{align*}
		&\Gamma_{ (Y\# Y', K \# K') }(k+k') \leq \Gamma_{(Y,K)}(k) + \Gamma_{(Y', K')} (k'),  \\[2mm]
		&r_{s+s'}(Y \# Y' , K \# K')- s-s' \geq \min \{ r_{s}(Y, K)-s , \; r_{s'}(Y', K') -s'\}. 
	\end{align*}
\end{cor}

\noindent Note that \Cref{conn sum for JYK} and \Cref{thm: connected sum rs} prove \Cref{thm:connsumineqintro} from the introduction.

We next consider the behavior of these invariants under cobordisms. What follows is a straightforward generalization of some material from \cite[\S 4.4]{DS19}. Recall from the discussion surrounding \Cref{defn:negdefcob} that to cobordism data $(W,S,c)$ with $b_1(W)=b^+(W)=0$ there is an associated integer
\begin{equation}\label{eq:indexformulaenrichedsetting}
	i:= 4\kappa_{\operatorname{min}}(W,S, c)+\frac{1}{4}S\cdot S+ \frac{1}{2}\chi(S)+\frac{1}{2}\sigma(Y,K)-\frac{1}{2}\sigma(Y',K')
\end{equation}
such that if $i\geq 0$, then there is an associated height $i$ morphism of $\mathcal{S}$-complexes. This construction can be used to construct a height $i$ morphism of the associated enriched $\mathcal{S}$-complexes. 

\begin{thm}\label{Fr ineq}
Let $(W,S):(Y,K)\to (Y',K')$ be a cobordism, as in \Cref{defn:negdefcob}, which is negative definite of strong height $i\geq 0$ over $R$, where $i$ can be computed from \eqref{eq:indexformulaenrichedsetting}. Then we have
\begin{equation}\label{eq:newinvineq}
\newinv_{(Y',K')}(k+  i ,s ) \leq \newinv_{(Y,K)}(k ,s- 2 \kappa_{\operatorname{min}}(W,S) ) + 2 \kappa_{\operatorname{min}}(W,S).
\end{equation}
Moreover, if equality is achieved in \eqref{eq:newinvineq} for some $k\in\Z$ and $s\in [-\infty,0)$, with both sides finite and positive, then there exists an irreducible traceless $SU(2)$ representation of $\pi_1(W\setminus S)$. The same conclusion holds for the $\underline{\newinv}$-invariants, under the weaker assumption that the cobordism is not necessarily strong, but satisfies $\eta(W,S)\neq 0$.
\end{thm} 

\begin{proof}
Inequality \eqref{eq:newinvineq} follows from \Cref{nagative definte cobordism induces} and \Cref{Fr ineq}.

Now we prove the second statement, regarding when equality is achieved in \eqref{eq:newinvineq}. For now, assume our enriched $\mathcal{S}$-complexes are simply I-graded $\mathcal{S}$-complexes, and morphisms are level $0$. In particular, the singular instanton $\mathcal{S}$-complexes, and the relevant cobordism maps below, can be defined without perturbations, and the I-gradings are determined by the unperturbed Chern--Simons functional.

Define $\kappa:=\kappa_\text{min}(W,S)$. Let $z=\wh \Psi(\mathfrak{z})$ be a filtered special $(k,1,s-2\kappa)$-cycle with $\mathfrak{z}=(\alpha,\sum_{i=0}^N a_i x^i)$ in $\shc_{2k}(Y,K)$. Recall from \Cref{rem:degioffscyc} that
\[
	\text{deg}_I(z) = \text{deg}_I(\mathfrak{z}) = \max\{ \text{deg}_I(\alpha), 0\},
\]
and $\text{deg}_I(a_i)\leq 0$. Assume $z$ is a filtered special cycle that realizes the value of $\newinv_{(Y,K)}(k ,s- 2 \kappa )$. This is possible because the I-gradings take values in the image of the Chern--Simons functional, a discrete subset of $\R$. Thus we have the relation
\begin{equation}\label{eq:nykmax}
	\newinv_{(Y,K)}(k ,s- 2 \kappa ) = \max\{ \text{deg}_I(\alpha), 0\}.
\end{equation}
Let $\wh\lambda:\shc(Y,K)\to \shc(Y',K')$ be induced by $(W,S,c)$ on small equivariant (filtered) complexes. By \Cref{filtered Fr under morphism}, $z'=\wh \Psi'(\mathfrak{z}')$, where $\mathfrak{z}'=\wh \lambda(\mathfrak{z})$, is a filtered special $(k+i,1,s)$-cycle. We obtain
\begin{equation}\label{eq:ineqnewinvineq1}
	\newinv_{(Y',K')}(k+  i ,s ) \leq  \text{deg}_I(\wh \lambda(\mathfrak{z})) \leq \text{deg}_I(\mathfrak{z}) + 2\kappa = \newinv_{(Y,K)}(k ,s- 2 \kappa ) + 2\kappa.
\end{equation}
where we have used in the second inequality that $\wt \lambda$ has level $2\kappa$, and we have also used $\text{deg}_I(\mathfrak{z}')=\text{deg}_I(z')$. By our assumption that \eqref{eq:newinvineq} is an equality, the inequalities in \eqref{eq:ineqnewinvineq1} are equalities. Next, $ \mathfrak{z}'= (\alpha', \sum_{i=0}^{N'} a_i' x^i)$ where
\begin{equation}\label{eq:alphaprimethreeterms}
	\alpha' = \lambda(\alpha) + \sum_{i=1}^N \sum_{j=0}^{i-1} \mu v^{i-j-1}\delta_2(a_i) + \sum_{i=0}^N (v')^i\Delta_2(a_i).
\end{equation}
This follows from direct computation of $\mathfrak{z}' = \wh\lambda(\mathfrak{z}) = \wh\Phi'\circ\wh{\boldsymbol{\lambda}}\circ\wh\Psi'(\mathfrak{z})$ using the definitions of the maps from \Cref{subsection: equivariant}. As $\newinv_{(Y',K')}(k+  i ,s )>0$, we have by \Cref{rem:degioffscyc} that
\[
	\newinv_{(Y',K')}(k+  i ,s ) = \text{deg}_I(\mathfrak{z}') = \text{deg}_I(\alpha').
\]
Now $\text{deg}_I(\alpha')$ is the maximum of the I-degrees of the three terms appearing in \eqref{eq:alphaprimethreeterms}. 

First, consider the case in which this maximum is realized by the first term:
\[
	\text{deg}_I(\alpha') = \text{deg}_I(\lambda(\alpha)).
\]
From \eqref{eq:nykmax}, and the string of equalities achieved in \eqref{eq:ineqnewinvineq1}, we obtain
\begin{equation}\label{eq:anotherineqdegi}
	\text{deg}_I(\alpha') \geq \text{deg}_I(\alpha) + 2\kappa.
\end{equation}
Furthermore, $\text{deg}_I(\alpha') \leq \text{deg}_I(\alpha) + 2\kappa$ because $\wt \lambda$ has level $2\kappa$. Thus  \eqref{eq:anotherineqdegi} is an equality. Recall from \cite{DS20} that $\alpha = \sum r_i\wt \alpha_i$ and $\alpha'=\sum r'_i\wt \alpha_i'$ are linear combinations over $R$ of irreducible flat connections (in fact, paths of flat connections to the reducible), and $\lambda(\wt \alpha_i) = \sum d_{ij} \wt \alpha_j'$ where $d_{ij}\in \Z$ is the signed count of elements in a moduli space of singular instantons on the cobordism $(W,S)$ with cylindrical ends attached. For such an instanton $A$ in this count we have
\[
	\text{deg}_I(\wt \alpha_i) + 2\kappa =  \text{deg}_I(\wt \alpha'_j) + 2\kappa(A) 
\]
Suppose $\wt\alpha_j'$ is chosen such that $\deg_I(\alpha') = \text{deg}_I(\wt \alpha_j')$ and $r_j'\neq 0$. We obtain
\[
	\text{deg}_I(\alpha) + 2\kappa \geq \text{deg}_I(\wt \alpha_i) + 2\kappa  =\text{deg}_I(\wt \alpha'_j) + 2\kappa(A) = \text{deg}_I(\alpha') + 2\kappa(A) 
\]
Using \eqref{eq:anotherineqdegi} we obtain $\kappa(A)\leq 0$. The energy of an instanton is necessarily non-negative, and thus $\kappa(A)=0$, and $A$ is flat. This flat singular connection corresponds to an $SU(2)$-representation of $\pi_1(W\setminus S)$ that extends irreducible representations of $\pi_1(Y\setminus K)$ and $\pi_1(Y'\setminus K')$ which send meridians to traceless elements in $SU(2)$.

Now suppose $\text{deg}_I(\alpha')$ is realized by the second term in \eqref{eq:alphaprimethreeterms}. Here, $\text{deg}_I(\alpha')=\text{deg}_I(\mu(\beta))$ where
\[
	\beta := \sum_{i=1}^N \sum_{j=0}^{i-1}  v^{i-j-1}\delta_2(a_i)
\]
Since each $\text{deg}_I(a_i)\leq 0$, we have $\text{deg}_I(\beta)\leq 0$. Since $\wt\lambda$ is of level $2\kappa$, $\text{deg}_I(\alpha') \geq \text{deg}_I(\beta) + 2\kappa$, analogous to \eqref{eq:anotherineqdegi}. The argument now proceeds as in the previous case, and we obtain a representation of $\pi_1(W\setminus S)$ that extends irreducible representations at $(Y,K)$ and $(Y',K')$, just as before. In the last case, in which $\text{deg}_I(\alpha')$ is realized by the third term in \eqref{eq:alphaprimethreeterms}, we have some $r\in R[U^{\pm 1}]$ such that $\text{deg}_I(r)\leq 0$ and $\text{deg}_I(\Delta_2(r))\geq \text{deg}_I(r)+2\kappa$. The argument gives a representation of $\pi_1(W\setminus S)$ as before, except that it is only irreducible at the end of $(Y',K')$.

The case involving non-trivial perturbations is similar, but also uses limiting arguments similar to those in \cite[\S 3.2]{D18}, \cite[\S 3]{NST19}.
\end{proof}

Substituting $s=-\infty$, \cref{Fr ineq} recovers one of the inequalities in \cite[Proposition 4.33]{DS20}. 
\vspace{1mm}
\begin{rem}
	The proof of \Cref{Fr ineq} given above shows the following. If equality occurs in \eqref{eq:newinvineq} and $\newinv_{(Y,K)}(k ,s- 2 \kappa_{\operatorname{min}}(W,S) )$ is finite and positive, then there exists an irreducible $SU(2)$ representation of $\pi_1(W\setminus S)$ extending irreducible traceless representations of $\pi_1(Y\setminus K)$ and $\pi_1(Y'\setminus K')$. However, if $\newinv_{(Y,K)}(k ,s- 2 \kappa_{\operatorname{min}}(W,S) )=0$ and $\newinv_{(Y',K')}(k+  i ,s )>0$, then the representation obtained is irreducible at $(Y',K')$, but possibly reducible at $(Y,K)$.
\end{rem}

\vspace{1mm}
\begin{rem}
	A version of \Cref{Fr ineq} also holds for non-trivial bundles. If $(W,S,c)$ is cobordism data as in \Cref{defn:negdefcob}, which is negative definite of strong height $i\geq 0$ over $R$, then 
	\begin{equation*}
\newinv_{(Y',K')}(k+  i ,s ) \leq \newinv_{(Y,K)}(k ,s- 2 \kappa_{\operatorname{min}}(W,S, c) ) + 2 \kappa_{\operatorname{min}}(W,S, c).
\end{equation*}
In this case, if both sides are finite and positive, there exists an irreducible $SU(2)$ representation of $\pi_1(W\setminus (S\cup F))$ which is traceless around meridians of $S$ and equal to $-1$ around meridians of $F$, where $F$ is an embedded surface in $W$ such that $[F]$ is Poincar\'{e} dual to $c\in H^2(W;\Z)$.
\end{rem}

The following is an analogue of \Cref{Fr ineq} for $\newinv^\intercal$, and the proof is similar (see also \eqref{eq:newinvtransposemorphism}).

\begin{thm}\label{newinvinqtranspose}
Let $(W,S):(Y,K)\to (Y',K')$ be a cobordism, as in \Cref{defn:negdefcob}, which is negative definite of strong height $i\geq 0$ over $R$, where $i$ can be computed from \eqref{eq:indexformulaenrichedsetting}. Then we have
\begin{equation}\label{eq:newinvintransposeeq}
\newinv^\intercal_{(Y',K')}(k+i,r+2 \kappa_{\operatorname{min}}(W,S) ) \leq \newinv^\intercal_{(Y,K)}(k ,r) + 2 \kappa_{\operatorname{min}}(W,S)
\end{equation}
Moreover, if equality is achieved in \eqref{eq:newinvintransposeeq} for some $k\in\Z$ and $r\in [0,\infty]$, with both sides finite and negative, then there exists an irreducible traceless $SU(2)$ representation of $\pi_1(W\setminus S)$.
\end{thm} 

\begin{rem}\label{rem:automaticallynegative}
Note that if $k+i\leq 0$, then both sides of \eqref{eq:newinvineq} are automatically negative by \Cref{rem:newinvvanishtop}.
\end{rem}

We now apply \cref{Fr ineq} to relate our invariants to knot surgeries.

\begin{cor}\label{surgery formula of J} Let $K$ be a knot in an integer homology 3-sphere $Y$ satisfying $\sigma  (Y,K) \leq 0$. Then, we have the following inequality, where $Y_1(K)$ denotes $1$-surgery on $K$:
\[
\newinv_{(Y,K)}\left(k- \frac{1}{2} \sigma (Y,K)  ,s \right)
 \leq \newinv_{(Y_{1} (K) ,U_1)}\left(k ,s-\frac{1}{8}\right) +  \frac{1}{8}  
\]
A similar inequality holds for the $\underline{\newinv}$-invariants.
\end{cor}

\begin{proof}
Consider the $4$-manifold with boundary  $W$ obtained by adding a $1$-framed 2-handle to $[0,1]\times Y$ along $\{1\}\times K$, and reversing orientation. We view $W$ as a cobordism $Y_1(K)\to Y$. In this cobordism there is an embedded annulus $S$ formed by taking the core of the 2-handle and connect-summing with a small 2-disk whose boundary is an unknot $U_1\subset Y_1(K)$. We obtain a cobordism of pairs $(W,S):(Y_1(K),U_1)\to (Y,K)$. Note $H^2(W;\Z)$ is generated by $[S]$, $S\cdot S=-1$, and $S$ has genus zero, and $W$ is simply-connected. Choosing bundle data $c=0$, in this case \eqref{eq:kappareducible} is minimized by the unique choice $c_1(L)=0$, and we compute
\[
	\kappa_{\rm min}(W, S,c)= - \left( \frac{1}{4} S \right)^2 = \frac{1}{16}.
\]
Furthermore, \eqref{eq:indexformulaenrichedsetting} is computed to be $i=-\sigma(Y,K)/2$. Note also $\eta(W,S,c)=1$, so that the induced morphism is strong. The result now follows from \Cref{Fr ineq}.
 \end{proof}

Substituting $s=-\infty$, we obtain the following.

\begin{cor}\label{gamma and rs surgery}
Let $K$ be a knot in an integer homology 3-sphere $Y$ satisfying $\sigma  (Y,K) \leq 0$. Then
\[
\Gamma_{(Y,K)}\left(k- \frac{1}{2} \sigma (Y,K) \right)
 \leq \Gamma_{(Y_{1} (K) ,U_1)}(k ) +  \frac{1}{8} . 
\]
\end{cor}

A consequence of \Cref{newinvinqtranspose} and \Cref{rem:automaticallynegative} is the following.

\begin{cor}\label{definite ineq for rs}

Let $(W,S):(Y,K)\to (Y',K')$ be a cobordism of pairs, as in \Cref{defn:negdefcob}, which is negative definite of strong height $0$ over $R$. Let $\kappa:=\kappa_{{\rm{min}}}(W,S)$. Then for all $s\in [-\infty,0]$ we have
\begin{align}\label{definite inequality for r_s}
r _{s-2\kappa}  (Y, K)  \leq r_{s}  (Y', K') + 2\kappa.
\end{align}
Furthermore, if $r_{s-2\kappa}(Y, K) $ is finite, and equality is achieved in \eqref{definite inequality for r_s}, then there exists an irreducible traceless $SU(2)$ representation of $\pi_1(W\setminus S)$. 
\end{cor}

In the sequel, the above  is typically applied to the case in which $(W,S):(Y,K)\to (Y',K')$ is a cobordism of pairs satisfying $b_1(W)=b^+(W)=0$, with $S\subset W$ is a null-homologous annulus, and $\sigma(Y,K)=\sigma(Y',K')$. In this case we have the following inequality, for any $s\in [-\infty,0]$:
\begin{align}\label{definite inequality for r_s 2}
r _{s}  (Y, K)  \leq r_{s}  (Y', K').
\end{align}
If furthermore $\pi_1(W\setminus S)\cong \Z$, then there cannot be any irreducible traceless $SU(2)$ representation of $\pi_1(W\setminus S)$. In this case, for $s\in (-\infty,0]$, \eqref{definite inequality for r_s 2} is a strict inequality.

In proving linear independence of a given sequence of elements in $\Theta^{3,1}_\Z$, we will use the following.
\begin{prop}\label{linear independence r0}
Let $\{(Y_i, K_i )\}_{i \in \Z_{>0}}$ be a sequence of representatives in $\Theta^{3,1}_\Z$ satisfying:
\begin{itemize}
    \item[{\rm{(i)}}] $\infty > r _0(Y_1, K_1) > r _0(Y_2, K_2) > \cdots$
    \item[{\rm{(ii)}}]  $r _0(-Y_i, -K_i) = \infty $. 
\end{itemize}
Then for any linear combination $[(Y, K)] := \sum_{i=0}^{N}m_{i}[(Y_i, K_{i})]$ in $\Theta^{3,1}_\Z$ with $m_N>0$, we have $r_0 (Y, K)<\infty$. In particular, $\{(Y_i, K_i )\}_{i \in \Z_{>0}}$ is a linearly independent subset of $\Theta^{3,1}_\Z$.  
\end{prop}
\begin{proof}
	We first claim that for any positive integer $m$, we have
	\begin{equation}\label{eq:romiclaim}
		r_{0}(\#_{m}(Y_i, K_{i}))=r_{0}(Y_i, K_{i}), \hspace{1cm} r_{0}(\#_{m}(-Y_i, -K_{i}))=\infty.
	\end{equation}
	Applying \cref{thm: connected sum rs} to $(Y_i, K_{i})\#(Y_i, K_{i})$ yields the inequality $r_{0}(\#_2(Y_i, K_{i}))\geq r_{0}(Y_i, K_{i})$. On the other 	hand, we have $r_{0}(Y_i, K_{i})\geq r_{0}(\#_2(Y_i, K_{i}))$ by applying \cref{thm: connected sum rs} to the decomposition 
	$[(Y_i, K_{i})]=2[(Y_i, K_{i})]-[(Y_i, K_{i})]$ in $\Theta^{3,1}_\Z$. This yields the first part of \eqref{eq:romiclaim} for $m=2$. 
	The case for larger values of $m$ follows inductively. 
	The second claim in \eqref{eq:romiclaim}  also follows from  \cref{thm: connected sum rs}.

	Now let $[(Y, K)] := \sum_{i=0}^{N}m_{i}[(Y_i, K_{i})]$ be as in the statement of the proposition. 
	A similar argument as above using \cref{thm: connected sum rs} gives
	\begin{equation}\label{eq:r0maxlincomb}
		r_{0}\left(Y,K\right) = {\rm min}\{r_{0}(Y_i, K_{i})\mid 1\leq i \leq N, \;\; m_i >0\}=r _0(Y_N, K_N)<\infty.
	\end{equation}
	To obtain the last part of the proposition, note that if we have a linear combination $\sum_{i=0}^{N}m_{i}[(Y_i, K_{i})]$ which is zero in $\Theta^{3,1}_\Z$, 
	we may assume 
	that $m_N>0$ without loss of generality, which is a contradiction by the first part of the proposition. 
\end{proof}

The following is a relation between $r_0$ and $\wt{s}$.

\begin{thm}\label{sandr0}
For a knot $K$ in $S^3$ with $\wt{s}(K) \neq 0$, we have $\min  \{r_0(K), r_0(K^*)\} < \infty$.
\end{thm}
\begin{proof}
By construction, $\wt{s}(K)$ depends only on the weak local equivalence class of the enriched $\mathcal{S}$-complex $\mathfrak{E}(Y,K)$. Thus if $\wt{s}(K)$ is non-zero, $\mathfrak{E}(Y,K)$ is not weakly locally equivalent to the trivial enriched $\mathcal{S}$-complex. By \Cref{local eq gamma} we obtain that one of $r_0(K)$ or $r_0(K^*)$ is finite.
\end{proof}

We next make some remarks on the relationship between invariants for unknots and the 3-manifold invariants of \cite{D18, NST19}. First, we give another expression for $r_s(Y,K)$. Assume for simplicity that no perturbations are required, and thus the enriched $\mathcal{S}$-complex $\mathfrak{E}(Y,K)$ is an I-graded $\mathcal{S}$-complex $\wt C=\wt C(Y,K;\Delta_R)$. Then we compute
\begin{align}
	r_s(Y,K) & = -\min \left\{ \text{inf}\{ \text{deg}_I(d \alpha -\delta_2(1) ) \mid \alpha\in C_{-1}, \text{deg}_I(\alpha)\leq -s \},\; 0 \right\}\nonumber \\[2mm]
	& =  \max\left\{\text{sup}\{ \text{deg}_I(f) \mid f\in C^\dagger_{1}, \;\text{deg}_I(\alpha)\leq -s, \;\alpha \in C_{-1},\;  (d^\dagger f)(\alpha)-\delta_1^\dagger(f) \in R\setminus 0  \},\; 0 \right\}\nonumber \\[2mm]
	& = \inf\left\{ r\in (0,\infty] \mid [\delta_1^\dagger] \neq 0 \in H(C^{\leq -s}/C^{\leq r}) \right\} \label{eq:rsaltdesc}
\end{align}
where $C^{\leq r}$ denotes the $R$-chain complex consisting of elements with $\text{deg}_I\leq r$. The invariant $r^R_s(Y)$ of \cite{NST19} is defined similarly as in the last line, but using the complex $C_\ast(Y;R)$ of the integer homology sphere $Y$, with Chern--Simons filtration, and the map $D_1:C_\ast(Y;R)\to R$ (called $\theta_Y$ in \cite{NST19}) in place of $\delta_1$. (There are also some convention differences, reflected in the discussion below; see also \cite[\S 4.1]{NST19}.) The chain complex $C_\ast(Y;R)$ is defined with respect to the coefficient ring $R$, forgetting the $\Z[T^{\pm 1}]$-algebra structure.

There is a chain map $\phi : C_* (Y, U_1;\Delta_R) \to C_*(Y;R)$ and a map $\mathfrak{D}:C_1(Y,U_1;\Delta_R)\to R$ satisfying
\begin{align}\label{phi equivalence between YU and Y1}
\mathfrak{D} d + \delta_1 + D_1 \phi=0.    
\end{align}
For details see \cite[Proposition 5.3]{DS20}. The map $\phi$ preserves the Chern--Simons filtrations, and we obtain, much like in the argument that proves \Cref{Fr ineq}, using the dual version of \eqref{phi equivalence between YU and Y1}, an inequality
\begin{equation}\label{inequality of rs}
	r_s(Y,U_1) \leq 2 r^R_s(Y).
\end{equation}
The factor of $2$ that appears is because of the difference of conventions for the Chern--Simons functional. A similar argument using \eqref{phi equivalence between YU and Y1} and also the dual version yields the inequalities
\begin{align}
	\Gamma_{Y}(1) \leq \Gamma_{(Y,U_1)}(1)\label{eq:gammahom3sphereineq} \\[2mm]
	\Gamma_{Y}(0) \geq \Gamma_{(Y,U_1)}(0)\label{eq:gammahom3sphereineq2}
\end{align}
where $\Gamma_Y$ is as defined in \cite{D18} over $\Q$, and $\Gamma_{(Y,U_1)}$ is defined with the coefficient ring $R$ being any integral domain algebra over $\Q[T^{\pm 1}]$.

In fact, the inequalities \eqref{inequality of rs}--\eqref{eq:gammahom3sphereineq2} are equalities. A proof follows along the same lines of the one sketched in \cite[\S 5.3]{DS20}, where it is explained that $C_\ast(Y,U_1;R)$ is chain homotopy equivalent to the $\Z/4$-graded chain complex $(C_\ast(Y;R),d_Y)\oplus (C_{\ast-2}(Y;R),d_Y)$, and that under this equivalence, $\delta_1$ corresponds to $D_1\oplus 0$. The equivalence, which is roughly induced by certain cobordism maps, is filtration-preserving (adjusting for differences in convention). Given this explicit identification, equality in \eqref{inequality of rs}--\eqref{eq:gammahom3sphereineq2} follows. As we do not use this result elsewhere in the paper, we omit the details.

We have the following result regarding homology cobordisms of integer homology $3$-spheres. 

\begin{cor}\label{irrep for 3-manifolds}
If $r_0(Y, U_1)$ is finite, then for any homology cobordism $W$ from $Y$ to itself, there is an irreducible $SU(2)$ representation of $\pi_1(W)$.
\end{cor}

\begin{proof}
Let $W:Y\to Y$ be a homology cobordism. Choose a submanifold of $W$, diffeomorphic to $[0,1]\times D^3$, which is a regular neighborhood of a path from the incoming copy of $Y$ to the outgoing copy of $Y$. We obtain a cobordism of pairs $(W,S):(Y,U_1)\to (Y,U_1)$ where $S$ is an unknotted annulus inserted in the submanifold identified with $[0,1]\times D^3$. By the Seifert--Van-Kampen Theorem, we have
\[
	\pi_1(W\setminus S)\cong \pi_1(W)\ast \Z,
\]
where the copy of $\Z$ is generated by any circle fiber of the normal bundle of $S$. Now our assumption on $r_0(W,S)$, together with \Cref{definite ineq for rs}, imply that there is an irreducible traceless representation $\rho:\pi_1(W)\ast \Z\to SU(2)$. The traceless condition means that after possibly conjugating, $\rho(1)=\mathbf{i}=\text{diag}(i,-i)\in SU(2)$, where $1\in \Z$. Consider the representation $\rho':=\rho \circ \iota$ where $\iota:\pi_1(W)\to \pi_1(W)\ast \Z$ is the inclusion. If $\rho'$ is reducible, then, because $\text{Hom}(\pi_1(W),\Z)=H^1(W;\Z)=0$, it is trivial. But then $\rho$ has image $\mathbf{i}$ and is reducible, a contradiction. Thus $\rho'$ is the desired irreducible representation.
\end{proof}

\begin{rem}
An alternative proof of \Cref{irrep for 3-manifolds} may be obtained once equality in \eqref{inequality of rs} is established. For then finiteness of $r_0(Y, U_1)$ implies finiteness of $r_0^R(Y)$, and the result then follows from \cite[Theorem 1.1(1)]{NST19}. 
\end{rem}

We end this section with some results involving the invariant $s^\sharp(Y,K)$. In what follows, for an integer homology 3-sphere $Y$, we write $h(Y)\in \Z$ for Fr\o yshov's instanton invariant defined in \cite{Fr02}, which is constructed in the setting of $SO(3)$-equivariant instanton Floer theory for $Y$ with rational coefficients. We study the relationship between $\wt s(Y,K)$ and $h(Y)$. For the remainder of this section, $R=\Q[\![\Lambda]\!]$.
\begin{prop}\label{1}
If $K$ is a knot in an integer homology $3$-sphere $Y$ satisfying $h(Y)= h(Y_{1}(K))$, then 
\[
\wt s(Y, K)\leq \wt s(Y_{1}(K), U_1). 
\]
\end{prop}
\begin{proof}
Let $i:=-\sigma(Y,K)/2$. First assume $i\geq 0$. Consider the cobordism $(W,S):(Y_1(K),U_1)\to (Y,K)$ from the proof of \Cref{surgery formula of J}. There, it was shown that $(W,S)$ is strong and negative definite over $R$ of height $i$. Then we have
\begin{align*}
	i= -\sigma(Y,K)/2 &= h(Y,K)-4h(Y)\\[2mm]
	& = h(Y,K) - 4h(Y_1(K))\\[2mm]
			& = h(Y,K) - h(Y_1(K),U_1) 
\end{align*}
where we have used \Cref{thm:froyshovinvofknot} in the first and third lines, and the assumption $h(Y)= h(Y_{1}(K))$ in the middle. The result now follows from an application of \Cref{definite ineq for ssharp}. If $i\leq 0$, using the trick employed in \Cref{subsec:ideals}, we form a cobordism $(W',S'):(Y_1(K),\#_{-i}T_{2,3}^\ast)\to (Y,K)$ by splicing on $-i$ copies of the blown-up version of the cobordism from \Cref{lem:crossingchangemaptrefoil} onto $(W,S)$. Then $(W',S')$ has height $0$ with $c_0=\eta(W',S')$ equal to $\Lambda^{-i}$ up to a unit. An application of \Cref{definite ineq for stilde} gives the result.
\end{proof}

\begin{prop}\label{2}
If $h(Y)=0$, then $\wt s(Y, U_1)=0$. 
\end{prop}
\begin{proof}
Consider the map $\psi : C(Y;R)\to C(Y, U_1;\Delta_R) $ which is the dual of $\phi:C(-Y, -U_1;\Delta_R)\to C(-Y;R)$ from \eqref{phi equivalence between YU and Y1}. Then \eqref{phi equivalence between YU and Y1} gives a dual relation of the form
\[
	d\mathfrak{D}' + \delta_2 + \psi D_2 = 0.
\]
Under the standing assumption that $R=\Q[\![\Lambda]\!]$ we have $C(Y;R)= C(Y;\Q)\otimes_\Q  \Q[\![\Lambda]\!]$. Further, $h(Y)=0$ implies that $D_2(1)=d_Y\alpha$ for some $\alpha\in C(Y;\Q)$. Now 
\[
	\delta_2(1) = -\psi D_2(1) - d \mathfrak{D}'(1) =  -\psi d_Y(\alpha) - d \mathfrak{D}'(1) = d \beta
\]
where $\beta=-\psi(\alpha)-\mathfrak{D}'(1)$. Then $(0,\beta,1)\in \lhc$ defines a special $(0,1)$-cycle for $\wt C(Y,U_1;\Delta_R)$. From the definition of $\wt s$ we obtain $\wt s(Y,U_1)\leq 0$. Applying the same argument to the orientation-reversal yields the inequality $\wt s(-Y,-U_1)\leq 0$. We then have
\[
	0 \leq -\wt s(-Y,-U_1) \leq \wt s(Y,U_1)
\]
by \eqref{dual-ineq}. This completes the proof.
\end{proof}

The following result implies \Cref{s-tilde-froy} from the introduction.

\begin{prop}\label{relation h and s-tilde}
For a knot $K$ in $Y$ with $h(Y)=0$, if $\wt s(Y,K)>0$, then $h(Y_{1}(K))<0$.
\end{prop}
\begin{proof}
Note that we always have $h(Y_{1}(K))\leq h(Y) =0$, see \cite{Fr02}. Assume $h(Y_{1}(K))=0$. Then
\[
0= \wt s(Y_{1}(K), U_1) \geq \wt s(Y,K)
\]
from \cref{1} and \cref{2}. This contradicts $\wt s(Y,K)>0$. 
\end{proof}

\begin{rem}
	Following the outline given in \cite[\S 5.3]{DS20}, we expect that $\wt C(Y,U_1;\Delta_{\Q[\![\Lambda]\!]})$ is homotopy equivalent to $\wt C\otimes_\Q \Q[\![\Lambda]\!]$ where $\wt C$ is an $\mathcal{S}$-complex over $\Q$. This would imply $\wt s(Y,U_1)=0$ for any integer homology $3$-sphere $Y$. We would then obtain that $\wt s(Y,K)>0$ implies $h(Y_1(K))<h(Y)$.
\end{rem}


\section{Applications} \label{sec:applications}

We now use the invariants introduced in this paper to prove various topological applications. In the first section below, we provide slice genus bounds for the invariants $s_\pm^\sharp(K)$. We then turn to prove the applications described in the introduction involving knot concordance, homology cobordism, and non-abelian traceless $SU(2)$-representations on concordance complements.

\subsection{Genus bounds from $s^\sharp_\pm$}
In \cite{KM13}, Kronheimer and Mrowka prove that $s^\sharp(K)$ gives a lower bound for the twice of the slice genus of $K$. From the definitions given by Gong \cite{Gong21}, one obtains slice genus bounds for $s^\sharp_\pm(K)$ as well. Using our description of these invariants, we improve these bounds. 
\begin{thm}\label{more genus bound}
	For any knot $K$, we have the following genus bounds.
        \begin{itemize}
            \item[(i)] $|s_+^\sharp(K)| , |s_-^\sharp(K)+1| \leq g_4(K)$ hold if $\sigma(K)<0$;
            \item[(ii)] $|s_\pm ^\sharp(K)| \leq g_4(K)$ hold if $\sigma(K)=0$;
            \item[(iii)]$|s_+^\sharp(K)-1| , |s_-^\sharp(K)| \leq g_4(K)$ hold if $\sigma(K)>0$. 
        \end{itemize}
\end{thm}
Before proving this theorem, we observe that above genus bounds can be used to improve a result from \cite{Gong21} about the values of $s^\sharp_\pm$ for quasi-positive knots.
\begin{cor}\label{precise comp of s}Let $K$ be a quasi-positive knot.
If $\sigma(K) \leq 0$, then $s^\sharp_+(K) =  g_4(K)$. If $\sigma(K) < 0$, then $s^\sharp_-(K) =  g_4(K)-1$.
\end{cor}

This corollary determines $s^\sharp_\pm$ for any algebraic knot or more generally any positive knot because any such non-trivial knot has negative signature \cite{Jo89}. 

\begin{proof}
	It is shown in \cite{Gong21} that for any quasi-positive knot $K$, if $g_4(K)$ is even, then we have
	\[
	  s^{\sharp}_+(K) = g_4(K),\hspace{1cm} g_4(K)-1 \leq s^{\sharp}_-(K) \leq g_4(K) ,
	\]
	and if $g_4(K)$ is odd, then we have
	\[
	  g_4(K)\leq  s^\sharp_+ (K) \leq g_4(K)+1,\hspace{1cm}s^\sharp_-(K)=g_4(K)-1. 
	\]
	Now the claim follows from combining the above bounds with \cref{more genus bound}. 
\end{proof}

In order to prove \cref{more genus bound}, we use the following lemma from \cite{Sato-19}.
\begin{lem}[\text{\cite[\S 3]{Sato-19}}]\label{sato}
	For any knot $K$, there exist a non-negative integer $k$, positive integers $m_i$ ,$n_i$ for any $1\leq i \leq g_4(K)$ and a cobordism of pairs
	\[
	  (W, S): (S^3, K) \to \left(S^3, \#_{i=1}^{g_4(K)} K_{m_i, n_i}^*\right)
	\]
	such that $K_{m,n}$ is the doubly twist knot with parameters $m, n\in \Z_{>0}$, $W$ is the connected sum of the product cobordism with $k$ copies of $\overline{\mathbb{C}P}^2$, $g(S)=0$ and the homology class of 
	$S$ is equal to $e_1+\dots+e_l$ where $0\leq l\leq k$ and $e_i$ is a generator of the $H_2$ of the $i^\text{\rm th}$ copy of $\overline{\mathbb{C}P}^2$ in $W$.
\end{lem}
\begin{proof}[Proof of \cref{more genus bound}]
	For the knot $K$, take the cobordism $(W, S)$ provided by \cref{sato}. There is a unique minimal reducible on $(W,S)$ of index $2i-1$ with $i=\sigma (K)/2  - g_4(K)$. In particular, $\eta(W,S)=1$.
	Since $\sigma (\#_{i=1}^{g_4(K)} K_{m_i, n_i}^*) = 2g_4(K) \geq \sigma (K)$, the integer $i$ is non-positive, and if it is negative then $(W,S)$ is not a negative definite cobordism of pairs with a non-negative height.
	To remedy this, take the blown up version of the cobordism from \Cref{lem:crossingchangemaptrefoil} and view it as a cobordism from the unknot to $T_{2,3}$.
	Then splice $-i$ copies of this cobordism along $(W,S)$ to obtain a cobordism as follows:
	\[
	  (W', S') : (S^3, K) \to (S^3, \#_{i=1}^{g_4(K)} K_{m_i, n_i}^*\#_{-i} T_{2,3} ).
	\]
	It is straightforward to check that $(W', S') $ is a negative definite cobordism of pairs with vanishing height and $\eta(W',S')=(T^2 - T^{-2})^{-i}$ up to a unit. 
	Thus, from \cref{cobordism map} and \cref{definite ineq for ssharp}, we have
	\begin{equation}\label{s-pm-ineq}
	  {s}^\sharp_\pm  (\#_{i=1}^{g_4(K)} K_{m_i, n_i}^*\#_{-i} T_{2,3}  ) \leq {s}^\sharp_\pm(K) -i. 
	\end{equation}

	\cref{prop:twobridgescomplexes} implies that the $\mathcal{S}$-complex of $\#_{i=1}^{g_4(K)} K_{m_i, n_i}^*\#_{-i} T_{2,3}$ is locally equivalent to the connected sum of $-g_4(K)-i=-\sigma (K)/2$ copies of 
	$T_{2,3}$ over $\Q[\![\Lambda]\!]$. 
	Therefore, we have 
	 \[
	 {s}^\sharp_\pm  (\#_{i=1}^{g_4(K)} K_{m_i, n_i}^*\#_{-i} T^*_{2,3}  ) =  {s}^\sharp_\pm \left(\#_{-\sigma(K)/2} \; T_{2,3} \right). 
	 \]
	This together with \eqref{s-pm-ineq} gives
	\[
	 {s}^\sharp_\pm (\#_{-\sigma(K)/2} \; T_{2,3} )+ \tfrac{1}{2} \sigma (K)- g_4(K) \leq  {s}^\sharp_\pm(K)   . 
	\]
	By substituting $K^*$ for $K$, we obtain 
	\[
	  -{s}^\sharp_\mp (\#_{-\sigma(K)/2} \; T_{2,3} )- \tfrac{1}{2} \sigma (K)- g_4(K) \leq  -{s}^\sharp_\mp(K). 
	\]
	Combining the last two inequalities, we obtain the desired claim
	\[
	  \left| {s}^\sharp_\pm(K) - \frac{\sigma (K)}{2}- {s}^\sharp_\pm (\#_{-\sigma(K)/2} \; T_{2,3} )\right| \leq g_4(K).\qedhere 
	\]
\end{proof}

\subsection{Applications to knot concordance and H-sliceness} \label{section: knot conc group}

Here we prove the following, which is a restatement of \Cref{two bridge H-slice} from the introduction.

\begin{thm}\label{alternating main2}
Let $m$ be a positive integer and
$\{K_{m, n}\}_{n \in \Z_{\geq 0}}$ be the sequence of two-bridge knots defined by $K_{m,n}:=  K(212mn-68n+53,106m-34 )$. 
Let $K$ be a knot whose concordance class is of the form
\[
[K] = a_1[ K_{m,0}]+  a_2[ K_{m,1}]+  \cdots +a_N [K_{m,N}],
\]
where $a_i\in \Z$ and $a_N$ is positive. Then the knot $K$ is not smoothly H-slice in any positive definite closed 4-manifold with $b_1=0$.
\end{thm}
 Note that $K_{m,0} = 10_{28}$ for any $m\in \Z_{>0 }$. The diagram of $K_{m,n}$ is depicted in Figure \ref{fig:K_m,n}.
\begin{figure}
    \labellist
	\Large\hair 2pt
	\pinlabel -$m$ at 513 37
	\pinlabel $n$ at 612 128
	\endlabellist
	\centering
    \includegraphics[scale=0.6]{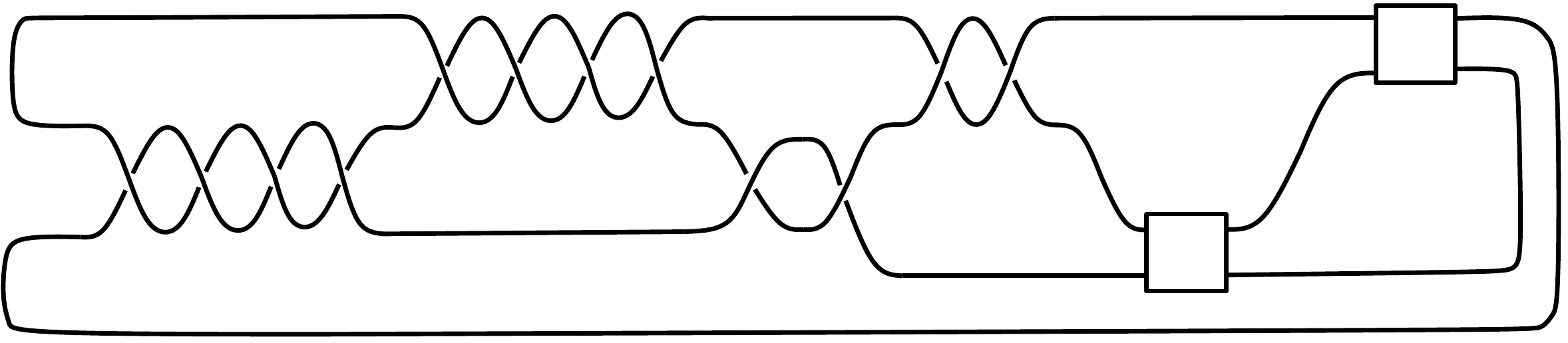}
    \caption{Two-bridge knot $K_{m, n}=K(212mn-68n+53, 106m-34)$. In the box labelled ``$-m$'' there are $m$ full negative twists, and in the box labelled ``$n$'' there are $n$ full positive twists.}
    \label{fig:K_m,n}
\end{figure}
Firstly, we observe that almost all of $K_{m, n}$ are algebraically slice.
Recall the Tristram--Levine signature function:
\begin{equation}\label{LT-sig}
	\sigma_{\omega}(Y,K):={\rm sgn}\left[(1-e^{4\pi i \omega})A_K+(1-e^{-4\pi i\omega})A^{\intercal}_{K}\right]
\end{equation}	
Here $A_K$ is a Seifert matrix for the knot $K\subset Y$, $\omega\in (0, \frac{1}{2})$, and $\text{sgn}(B)$ is the signature of the Hermitian matrix $B$, which is the number of positive eigenvalues minus the number of negative eigenvalues of $B$.
\begin{lem}\label{lem:Kmn}
We have $\sigma(K_{m,n})=\sigma_{1/4}(K_{m,n})=0$ for $m\geq 1$ and $n \geq 0$. Moreover, if $m \geq 7$, then $\sigma_{\omega}(K_{m,n})= 0$ for $\omega \in (0, \frac{1}{2})$. Thus $K_{m,n}$ is torsion in the algebraic concordance group if $m\geq 7$, $n\geq 0$. 
\end{lem}
\begin{proof}

Taking a Seifert surface for the knot $K_{m, n}$, together with a choice of symplectic basis, as shown in Figure \ref{fig:seifert}. Then the corresponding Seifert matrix $A_{m, n}$ is computed as follows:
\[
A_{m, n}=\left[\begin{array}{cccccc}
2&&&&&\\
-1&2&&&&\\
&-1&-1&&&\\
&&-1&-1&&\\
&&&-1&-m&\\
&&&&-1&n
\end{array}\right]
\]
\begin{figure}
    \labellist
	\Large\hair 2pt
	\pinlabel -$m$ at 514 37
	\pinlabel $n$ at 614 128
	\endlabellist
	\centering
    \includegraphics[scale=0.6]{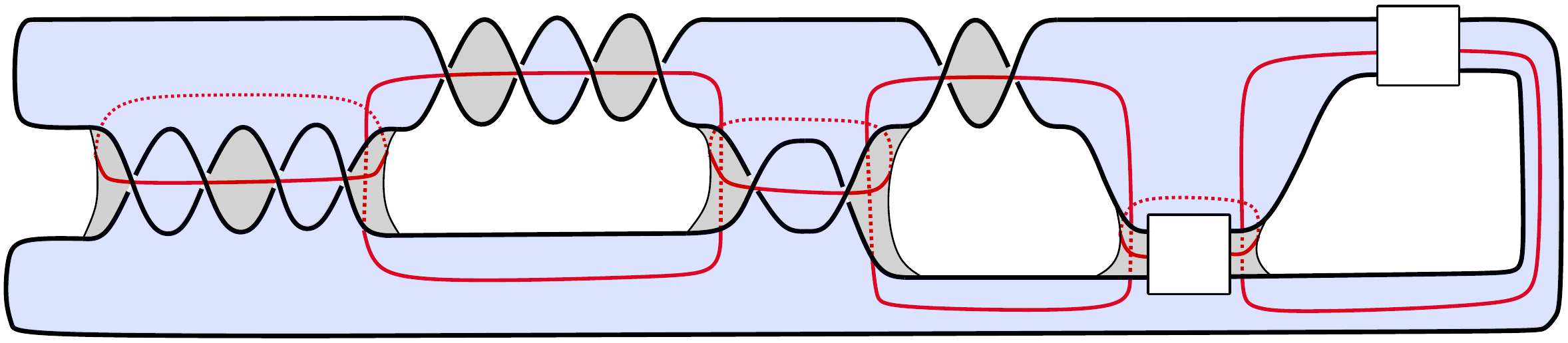}
    \caption{A Seifert surface for $K_{m, n}$ and its symplectic basis. }
    \label{fig:seifert}
\end{figure}
Now $\sigma_\omega(K_{m,n})$ is the signature of the matrix
$A_{m,n,\omega}:=(1-e^{4\pi i \omega})A_{m, n}+(1-e^{-4\pi i \omega})A^{\intercal}_{m, n}$. We first show that the number of positive and negative eigenvalues of this Hermitian matrix does not change in the case that $\omega=1/4$ and $m\geq 1$, $n\geq 0$, and also in the case that $\omega$ is arbitrary and $m\geq 7$, $n\geq 0$.

To this end, a direct computation gives the following expression for the determinant:
\begin{align*}
	f(m,n,\omega):= \det(A_{m,n,\omega})=64 \sin ^6(2 \pi  \omega) \big(8 mn \cos (12 \pi  \omega)+\cos (4 \pi  \omega) ((98
   m-34) n+26)\\[2mm]
   +\cos (8 \pi  \omega) (-42 mn+8 n-8)-64 mn+26 n-19\big)
\end{align*}
Now we consider $m$, $n$ as continuous variables. Setting $x=\cos(4\pi \omega)$, we compute
\begin{align*}
	\frac{d f}{dm} &= -256 n  \sin ^8(2 \pi  \omega) p(x), \qquad p(x) = 16 x^2 - 26 x + 11,\\[2mm]
	\frac{d f}{dn} &= -256  \sin ^8(2 \pi  \omega)q(x), \qquad q(x) = 16 m x^2 - 26 m x + 11 m + 8 x  -9 .
\end{align*}
The quadratic $p(x)$ is always positive, and thus $df/dm\leq 0$. On the other hand, $q(x)$ is always positive when $m\geq 7$, and in this case we thus have $df/dn<0$. In summary, the function $f(m,n,\omega)$ is non-increasing in $n,m$ when $n\geq 0$ and $m\geq 7$. We also have, with $x$ as before:
\[
	f(m,0,\omega) = -64  \sin ^6(2 \pi  \omega) p(x) < 0.
\]
We conclude that $f(m,n,\omega)$ is negative for all $m\geq 7$, $n\geq 0$, $\omega\in (0,1/2)$. Also
\[
	f(m,n,1/4) = -64 (53 + 4 (-17 + 53 m) n)
\]
which is negative for all $m\geq 1$, $n\geq 0$. Thus the number of positive and negative eigenvalues of $A_{m,n,\omega}$ does not change under the stated conditions. Thus $\sigma_{\omega}(K_{m,n})=\sigma(K_{1,0})$ for $m,n,\omega$ as in the statement of the lemma. The proof is completed by showing that $\sigma(K_{1,0})=\sigma(10_{28})=0$, which is straightforward.
\end{proof}
In order to prove \cref{alternating main2}, we use the following property of the invariant $r_0(K)$:
\begin{prop}\label{h-slice obstruction}
	For a knot $K$ in $S^3$, if the invariant $r_0(K)$ is finite and $\sigma(K)=0$, then $K$ is not smoothly H-slice in any negative definite closed 
	$4$-manifold with $b_1=0$.
\end{prop}
\begin{proof}
	If $K$ is H-slice in a negative definite $4$-manifold with $b_1=0$, then we obtain a negative definite cobordism of pairs 
	$(W,S):(S^3,U_1)\to (S^3,K)$ of height $0$ with $\eta(W,S)=1$ and $\kappa_{\rm min}(W,S)=0$. 
	Now applying \cref{definite ineq for rs} gives a contradiction. 
\end{proof}

\begin{proof}[Proof of \cref{alternating main2}]
It suffices to verify the hypotheses (i) and (ii) of \Cref{linear independence r0} for the family of knots $\{K^\ast_{m,n}\}_{n\geq 0}$, where $m>0$ is fixed. Indeed, the result then follows from \Cref{linear independence r0} and \Cref{h-slice obstruction}.

\begin{figure}
\labellist
	\large\hair 2pt
	\pinlabel -$m$ at 544 422
	\pinlabel $n$ at 644 513
	\pinlabel -$m$ at 540 233
	\endlabellist
\centering
\includegraphics[scale=0.5]{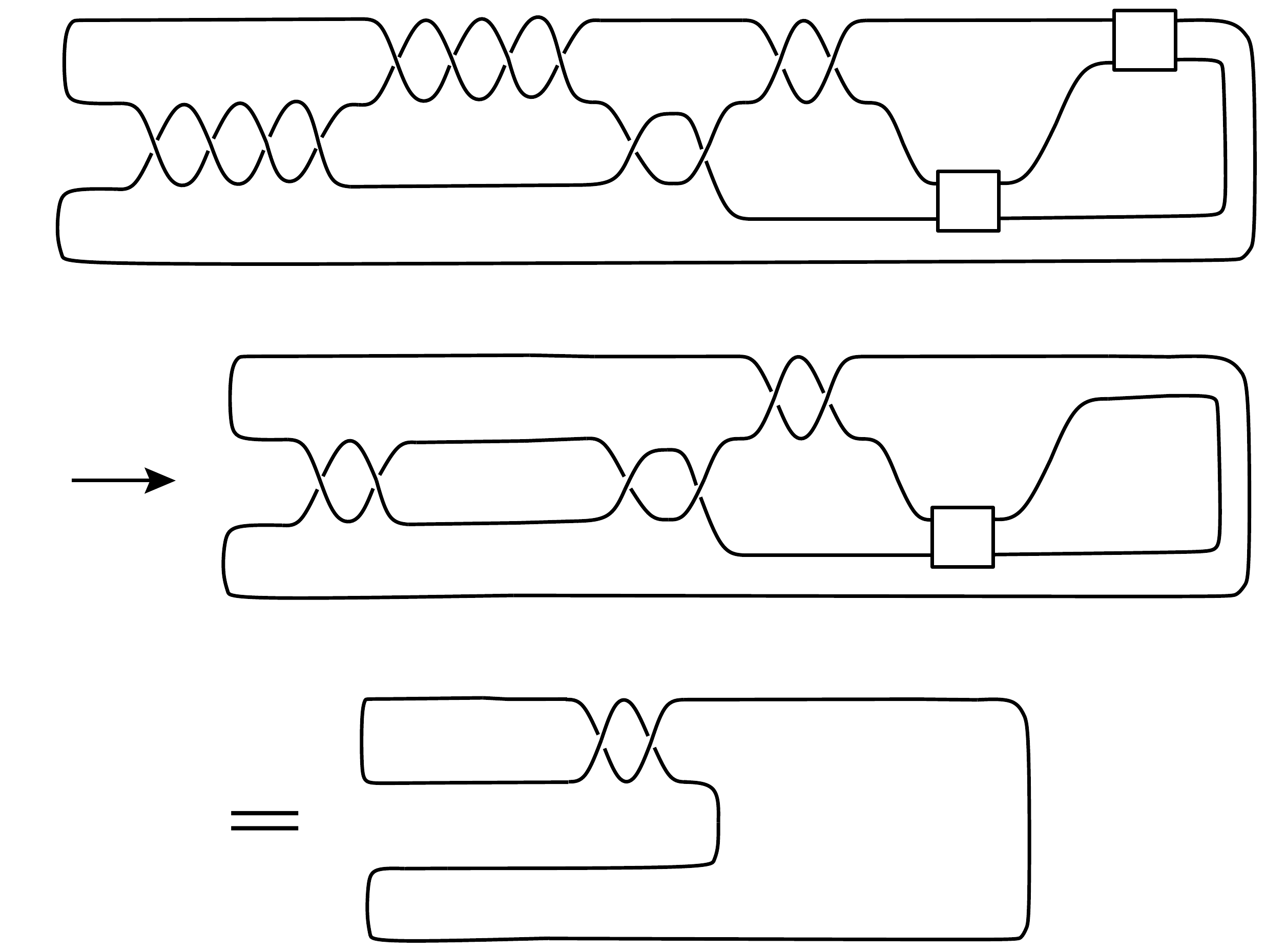}
\caption{Changing $n+3$ negative crossings in the diagram for $K_{m,n}$ produces an unknot.}
\label{fig;twobride-unknot}
\end{figure}

Consider the negative to positive crossing changes from $K_{m,n}$ to the unknot shown in \Cref{fig;twobride-unknot}. These give rise to an immersed cobordism from the unknot to $K_{m,n}$ with genus 0, no positive double points, and $(n+3)$ negative double points. By blowing up and capping off the unknot, we obtain a null-homologous disk with boundary $K_{m,n}$ in $(\#_{n+3}\overline{\mathbb{C}P}^2) \setminus B^4$.  Now, \cref{lem:Kmn} and \cref{h-slice obstruction} implies 
\begin{equation}\label{eq:r0ofkmninfinity}
	r_0 (K_{m,n}) = \infty.
\end{equation}

Similar to the previous construction, a positive to negative crossing change from $K_{m,n}$ to $K_{m,n+1}$ induces, after blowing up, a cobordism $(W_{m,n}, S_{m,n}) : (S^3, K_{m,n}) \to (S^3, K_{m,n+1})$ where $S_{m,n}$ is a null-homologous annulus and $W_{m,n}$ is a twice punctured $\overline{\mathbb{C}P}^2$. Then \Cref{definite ineq for rs} gives
\begin{equation}\label{eq:someinequalityinaproof}
r_0 (K^\ast_{m,n+1}) \leq r_0 (K^\ast_{m,n}).
\end{equation}
Here we have used that $\sigma(K_{m,n})=\sigma(K_{m,n+1})=0$, which implies that the cobordism $(W_{m,n}, S_{m,n}) $ is negative definite over $R$ of strong height $0$.

The cobordism $(W_{m,n}, S_{m,n}) $ is given by the Kirby diagram in \Cref{fig;twobridge-cc}. Note that the two-bridge presentation of $K_{m,n}$ implies that $\pi_1(S^3 \setminus K_{m,n})$ is generated by two meridional loops, one for each of the two strands that pass through the $(-1)$-framed $2$-handle. This added $2$-handle induces a relation on $\pi_1(W_{m,n} \setminus S_{m,n} )$ that equates these two elements. Thus $\pi_1(W_{m,n} \setminus S_{m,n} ) \cong \Z $. As such a cobordism does not admit any irreducible $SU(2)$-representations, \Cref{definite ineq for rs} implies that \eqref{eq:someinequalityinaproof} is a strict inequality:
\begin{equation}\label{eq:someinequalityinaproof2}
r_0 (K^\ast_{m,n+1}) < r_0 (K^\ast_{m,n}).
\end{equation}

\begin{figure}
    \labellist
	\Large\hair 2pt
	\pinlabel -$m$ at 513 38
	\pinlabel $n$ at 612 129
	\pinlabel -$1$ at 660 166
	\endlabellist
\centering
\includegraphics[scale=0.6]{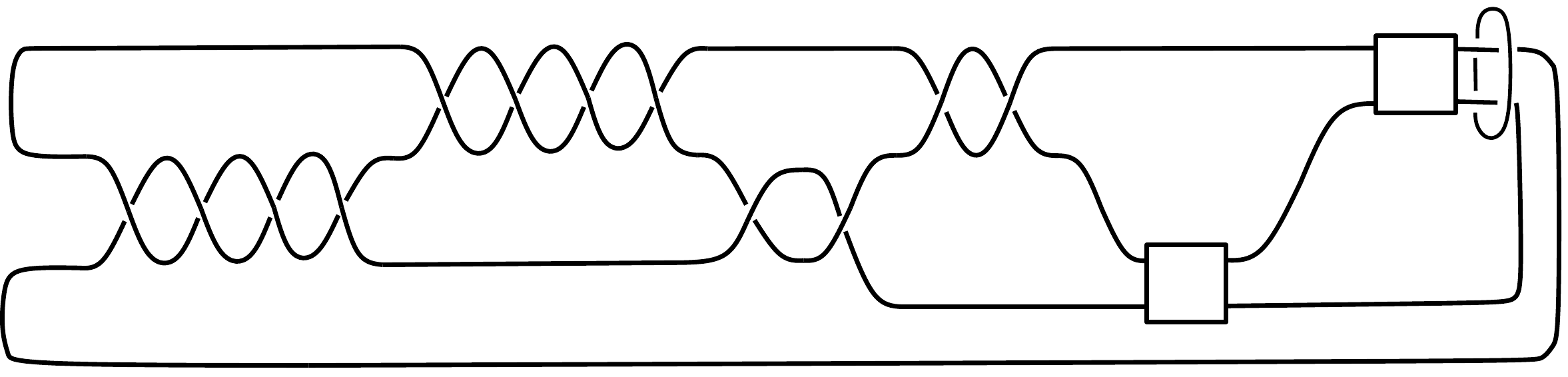}
\caption{A Kirby diagram for a cobordism $(W,S_{m,n}) :(S^3, K_{m, n}) \to (S^3, K_{m, n+1})$. The cobordism is obtained by attaching the $(-1)$-framed $2$-handle to $(S^3,K_{m,n})$.}
\label{fig;twobridge-cc}
\end{figure}

Now consider the case $n=0$. \Cref{fig;7-4knot}  describes a positive to negative crossing change from 
$K_{m,0}=10_{28}$ to $7_4^*$.
This gives an immersed cobordism from $10_{28}^*$ to $7_4$ with genus 0 and one positive double point. Therefore, it follows from \cite[Corollary 3.24]{DS20} 
and \cite[Proposition 4.33]{DS20} that
\[
\frac{3}{5} = \Gamma_{7_4}(1) \leq \frac{1}{2} + \Gamma_{10_{28}^*}(0).
\]
This shows $\Gamma_{10_{28}^*}(0)>0$, and property \eqref{N-infty}
 implies 
$r_0 (K^*_{m,0}) = r_0(10_{28}^*) < \infty$. Together with \eqref{eq:r0ofkmninfinity} and \eqref{eq:someinequalityinaproof2}, this verifies the hypotheses in \Cref{linear independence r0}, and the proof is complete.
\end{proof}

\begin{figure}
\centering
\includegraphics[scale=0.6]{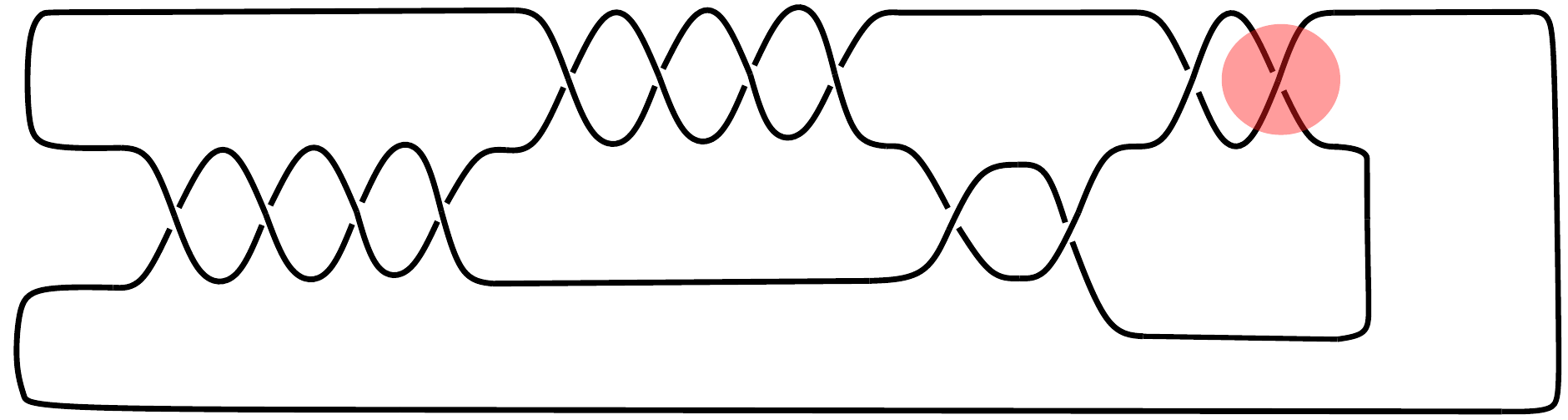}
\caption{Changing a positive crossing in $10_{28}$ gives $7_{4}^*$. (The knot $7_4$ is depicted in \Cref{fig:double-twist}.)}
\label{fig;7-4knot}
\end{figure}

\subsubsection{The result on satellite operations}

We next prove \Cref{main thm satellite} and \Cref{Cor satellite}, the results on satellite operations described in the introduction. The following implies \Cref{main thm satellite}, as we will see below.

\begin{thm}\label{knot concordance main thm2} 
Suppose that a pattern $P\subset S^1\times D^2$ satisfies:
\begin{itemize}
    \item[(i)] There is some knot $K$ such that $r_0(P(K))<\infty$ and $r_0(P(K)^\ast)=\infty$.
    \item[(ii)] $P(U_1)$ is the unknot.
\end{itemize}
Then, the image of the induced map 
$P \colon \mathcal{C} \to \mathcal{C}$
generates an infinite rank subgroup of $\mathcal{C}$.
\end{thm}

Suppose that $(W,S):(S^3,K)\to (S^3,K')$ is a cobordism of pairs satisfying $b_1(W)=b^+(W)=0$, with $S$ a null-homologous annulus. Cut out a regular neighborhood of $S$ and glue in the pair $(S^1\times D^2\times I, P\times I)$ along its boundary in the standard way to obtain a cobordism of pairs
\begin{equation}\label{eq:satellitecobconst}
	(W, S_P):(S^3,P(K))\to (S^3,P(K'))
\end{equation}
where $S_P$ is the image of $P\times I$ after gluing. Note that $S_P$ is null-homologous in $W$.

\begin{lem}\label{lem:pi1satellitecob}
Suppose $P(U_1)$ is the unknot. If $\pi_1(W\setminus S)\cong \Z$, then $\pi_1(W\setminus S_P) \cong \Z$.
\end{lem}

\begin{proof}
	Write $W\setminus P_S = W_S \cup X_P$ where $W_S=W\setminus N(S)$ is the complement of a regular neighborhood of $S$ in $W$, and $X_P=(S^1\times D^2\setminus P)\times I$. Thus $W\setminus S_P$ is the gluing of $W_S$ and $X_P$ along their common boundary $Z\cong T^2\times I$. By the Seifert--Van-Kampen Theorem,
	\begin{equation}
		\pi_1(W\setminus S_P) \cong \pi_1(W_S) \ast_{\pi_1(Z)} \pi_1(X_P).\label{eq:seifertvksatellitecob}
	\end{equation}
	Note $\pi(Z)\cong \Z\oplus \Z$ and, by assumption, $\pi_1(W_S)\cong \Z$. In particular, the amalgated product in \eqref{eq:seifertvksatellitecob} is unchanged if $\pi_1(W_S)$ is replaced by $\pi_1((S^3\setminus U_1)\times I)\cong \Z$. Another application of Seifert--Van-Kampen says that this latter amalgated product is isomorphic to $\pi_1((S^3\setminus P(U_1))\times I)$. By our assumption that $P(U_1)=U_1$, this group is isomorphic to $\Z$, and the result follows.
\end{proof}

\begin{lem}\label{lem:satellitecobpi1}
	Let $K\subset S^3$ be a knot. Then there is a knot $K'\subset S^3$ and a cobordism $(W,S):(S^3,K')\to (S^3,K)$ such that $b_1(W)=b^+(W)=0$, $S$ is null-homologous in $W$, and $\pi_1(W\setminus S)\cong \Z$.
\end{lem}

\begin{proof}
	Let $K'$ be obtained from $K$ by changing $n$ negative crossings to positive crossings in some diagram. Then, in the standard manner using blow-ups, there is an associated cobordism $(W,S):(S^3,K')\to (S^3,K)$ such that $W$ is the cylinder $S^3\times I$ blown up $n$ times. To arrange that $\pi_1(W\setminus S)\cong \Z$, we proceed as follows. Let $K$ have a diagram which includes the crossings that are altered to obtain $K'$. The group $\pi_1(W\setminus S)$ has presentation given by a Wirtinger presentation for $\pi_1(S^3\setminus K)$ using this diagram, but with additional relations: for each Wirtinger relation $x_j x_i = x_k x_j$ corresponding to a crossing that is changed, we obtain relations $x_j=x_i$ and $x_j=x_k$. To guarantee $\pi_1(W\setminus S)\cong \Z$, one must simply have enough crossing changes so that these additional relations identify all the Wirtinger generators.
	
If $K$ is given as the closure of a braid $\beta$ with $p$ strands, $K'$ can be chosen explicitly as follows. This is similar to what is done in the proof of \cite[Proof of Theorem 5.17]{NST19}. Let 
\[
\Delta_H=(\sigma_1 \sigma_2 \cdots \sigma_{p-1})(\sigma_1 \sigma_2 \cdots \sigma_{p-2})\cdots (\sigma_1\sigma_2) \sigma_1
\]
where $\sigma_i$ are standard generators of the braid group on $p$ strands. Then let $K'$ be the closure of $\Delta_H^2 \beta$. Note that $K'$ is obtained from $K$ by viewing $K$ as the closure of $\Delta_H^{-1} \Delta_H \beta$ and changing all of the negative crossings in $\Delta_H^{-1}$ to positive crossings. As $\pi_1(S^3\setminus K)$ is generated by a collection of $p$ meridians around the braid strands at the top or bottom of the braid closure, and the crossings of $\Delta_H$ relate all of these strands, these $p$ generators are identified in $\pi_1(W\setminus S)$, as desired.
\end{proof}

In what follows, we use the signature formula for satellites due to Litherland \cite{litherland}. With our conventions, this reads as follows. For a pattern $P$ with winding number $p$, and a knot $K$ in $S^3$, we have:
\begin{equation}\label{eq:signatureofsatellites}
	\sigma_\omega( P(K) ) = \sigma_{p\omega}( K ) + \sigma_{\omega}(P(U_1))
\end{equation}
where $p\omega$ is taken modulo $\frac{1}{2}\Z$ to lie in $[0,1/2)$. Here we assume $e^{4\pi i\omega}$ is not a root of the Alexander polynomial of the satellite knot $P(K)$.

\begin{proof}[Proof of \Cref{knot concordance main thm2}]
The result is known for non-zero winding number patterns by \cite[Proposition 8]{HC21}. Thus we assume the winding number of $P$ is zero.	By assumption, we have a knot $K_1:=K$ that satisfies $r_0(P(K_1))<\infty$ and $r_0(P(K_1)^\ast)=\infty$. Let $K_2:=K'$ be a knot as given by \Cref{lem:satellitecobpi1}. Then, since $b_1(W)=b^+(W)=0$, $S$ is null-homologous, and $\pi_1(W\setminus S)\cong \Z$, we may form the associated cobordism as in \eqref{eq:satellitecobconst} to obtain $(W,S_P):(S^3,P(K_2))\to (S^3,P(K_1))$. By \eqref{eq:signatureofsatellites} with $\omega=1/4$, and the assumption that $P$ has winding number zero, we have
	\[
		\sigma(P(K_1))= \sigma(P(U_1)) = 0,
	\]
	and similarly $\sigma(P(K_2))=0$. Thus $(W,S_P)$ is a cobordism which is negative definite of strong height $0$. Then \Cref{definite ineq for rs} and the discussion following it applied to the cobordism $(W,S_P)$ gives us
	\[
		r_0(P(K_2)) < r_0(P(K_1)).
	\]
	That this inequality is strict uses $\pi_1(W\setminus S_P)\cong \Z$, which follows from \Cref{lem:pi1satellitecob}, $\pi_1(W\setminus S)\cong \Z$, and our assumption $P(U_1)=U_1$.	Similarly, viewing $(W,S)$ as a cobordism $(S^3,P(K_1)^\ast)\to (S^3,P(K_2)^\ast)$, \Cref{definite ineq for rs} yields $\infty = r_0(P(K_1)^\ast) = r_0(P(K_2)^\ast)$. We continue in this fashion, inductively defining $K_i$ from $K_{i-1}$ using \Cref{lem:satellitecobpi1}. Then $\{K_i\}_{i=1}^\infty$ satisfies the hypotheses of \Cref{linear independence r0}, and thus gives a linearly independent set in the concordance group.
\end{proof}

\begin{proof}[Proof of \Cref{main thm satellite}]
Let $K$ be a knot that can be unknotted by a sequence a positive to negative crossing changes. As in the proof of \Cref{knot concordance main thm2}, we may assume the winding number of $P$ is zero, so that $\sigma(P(K))=\sigma(P(U_1))=0$. 

We obtain from the crossing changes of $K$ a cobordism $(W,S):(S^3,K)\to (S^3,U_1)$ with $b_1(W)=b^+(W)=0$ such that $S$ is a null-homologous annulus. Applying the construction of \eqref{eq:satellitecobconst}, we obtain $(W,S_P):(S^3,P(K))\to (S^3,P(U_1))$. As $P(U_1)=U_1$, we can cap this off, and conclude that $P(K)^\ast$ is $H$-slice in a negative definite $4$-manifold with $b_1=0$. Then \Cref{h-slice obstruction} implies $r_0(P(K)^\ast)=\infty$.

Next, since $P(K)$ is by assumption quasi-positive and non-slice, \Cref{s-tilde-values} implies $\wt s(P(K))=g_4(P(K))>0$. In particular, the enriched $\mathcal{S}$-complex associated to $P(K)$ is not weakly locally equivalent to the trivial enriched $\mathcal{S}$-complex. By \Cref{local eq gamma}, we then have either $r_0(P(K))<\infty$ or $r_0(P(K)^\ast)<\infty$. Having shown $r_0(P(K)^\ast)=\infty$ above, it must be that $r_0(P(K))<\infty$. As the pattern $P$ and the knot $K$ satisfy the hypotheses of \Cref{knot concordance main thm2}, the result follows.
\end{proof}

\begin{prop}\label{prop:tpqsatellitespec}
If $P$ satisfies the conditions of \Cref{knot concordance main thm2} with $K=T_{p,q}$ in condition (i), then
\[
	\{ P(T_{p,q+pn})\}_{n=0}^\infty
\]
is a linearly independent set in the homology concordance group.
\end{prop}

\begin{proof}
 This follows from the proof of \Cref{lem:satellitecobpi1}, which gives the construction for the sequence of knots in the proof of \Cref{knot concordance main thm2}, and the following observation: if $T_{p,q+np}$ is obtained in a standard way by taking the closure of a $p$ strand braid $\beta$, then $T_{p,q+(n+1)p}$ is the braid closure of $\Delta_H^2 \beta$.
\end{proof}

\begin{proof}[Proof of \Cref{Cor satellite}]
	This follows from \Cref{prop:tpqsatellitespec} with $P=\text{Wh}^r$.
\end{proof}

We next show that the patterns of \Cref{fig:satellite} satisfy the hypotheses of \Cref{main thm satellite}. In what follows, we allow $m$ or $n$ to be zero, in which case the corresponding sequences are empty.

\begin{prop}\label{quasipositivity}
Let $P$ be a pattern as in Figure \ref{fig:satellite}, for $\{a_i\}_{i=1}^m$ and $\{b_i\}_{i=1}^n$ sequences of negative integers with $m\geq n-1$ and $\max\{m,n\}>0$. Then the image of the induced concordance map $P \colon \mathcal{C} \to \mathcal{C}$
generates an infinite rank subgroup of $\mathcal{C}$.
\end{prop}

Write $A:=\{a_i\}_{i=1}^m$ and $B:=\{b_i\}_{i=1}^n$. Note that when $A=\emptyset$ and $B=\{-1\}$, the pattern $P$ is the Whitehead double. When $A=\{-1,\ldots,-1\}$ and $B=\emptyset$, we obtain the $(m+1,1)$-cable. The case of $A=\{-k-1\}$ and $B=\{-1\}$ is Yasui's pattern $P_{0,k}$ from \cite[Figure 10]{Yas15}.

\begin{lem}\label{quasipositivitylem}
A pattern $P$ as in \Cref{quasipositivity} preserves strong quasi-positivity, in the following sense: if $K$ is a non-trivial strongly quasi-positive knot, then so too is $P(K)$.
\end{lem}

It is proved in \cite[\S 2]{Rud93} that the Whitehead double of strongly quasi-positive knot is also strongly quasi-positive. \Cref{quasipositivitylem} gives a generalization of this fact. We follow the method in \cite[\S 2]{Rud93} to prove \cref{quasipositivitylem}. In what follows we write $g(K)$ for the Seifert genus of a knot, and use that quasi-positive Seifert surfaces realize $g(K)$, see \cite{Rud93}.

\begin{figure}[t]
    \centering
    \labellist
	\Large\hair 2pt
	\pinlabel ${\color{red}a_1}$ at 551 618
	\pinlabel ${\color{red}a_2}$ at 493 495
	\pinlabel ${\color{red}a_n}$ at 274 248
	\pinlabel ${\color{red}a_m}$ at 55 56
	\pinlabel ${\color{red}a_{n+1}}$ at 213 177
	\endlabellist
    \includegraphics[scale=0.47]{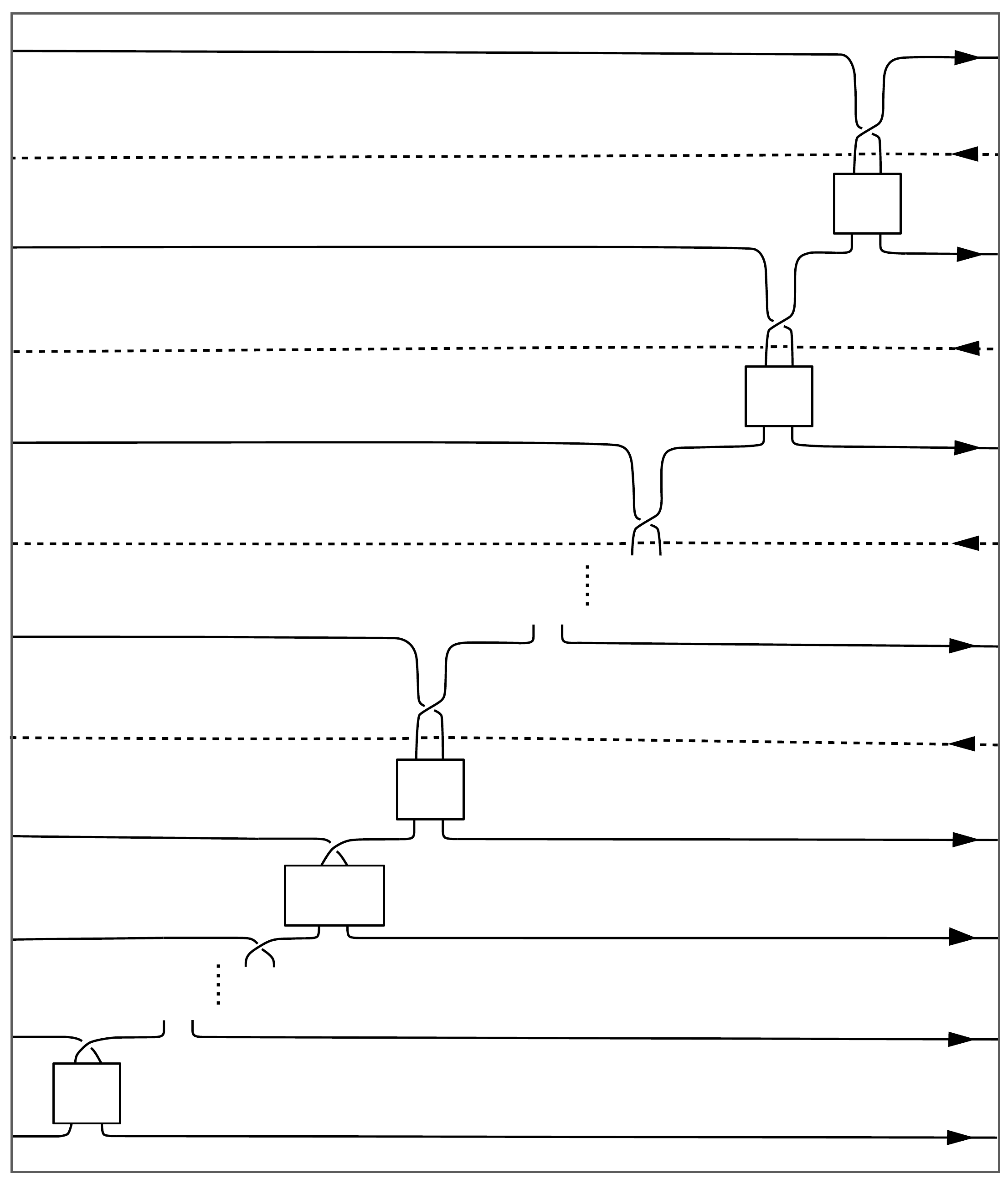}
    \caption{The pattern $P'$ (omit horizontal dotted lines) and $P''$ (include horizontal dotted lines).}
    \label{fig:patternprimes}
  \end{figure}
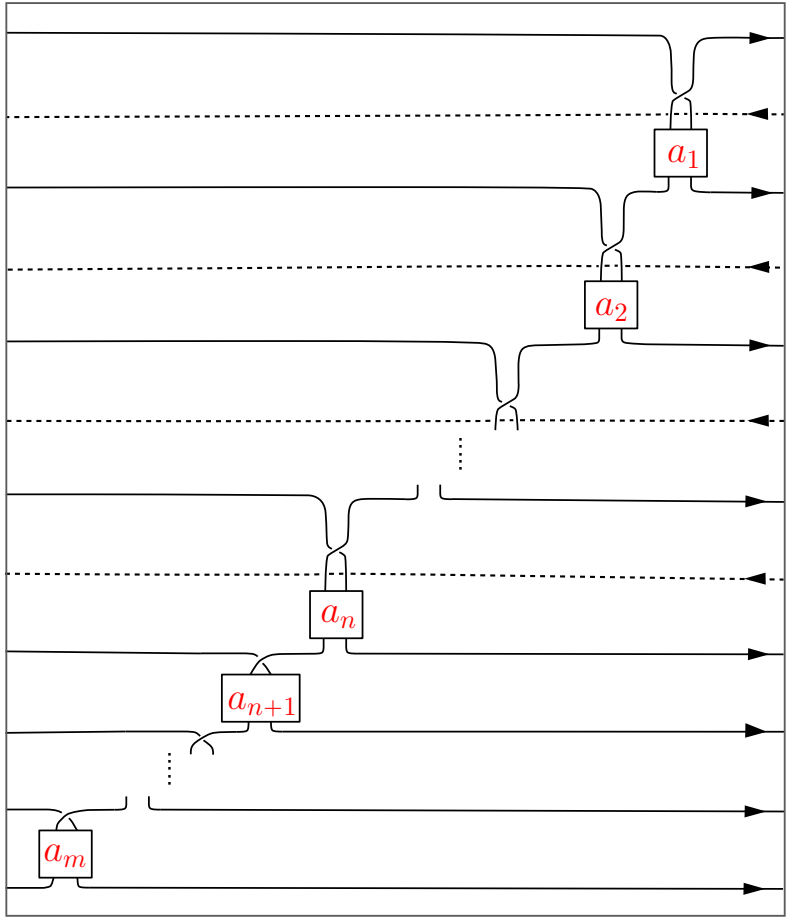

\begin{figure}[t]
    \centering
    \labellist
	\Large\hair 2pt
	\pinlabel ${\color{red}a_1}$ at 337 221
		\pinlabel ${\color{red}a_{2}}$ at 282 141
	\pinlabel ${\color{red}a_m}$ at 144 41
	\endlabellist
    \includegraphics[scale=0.6]{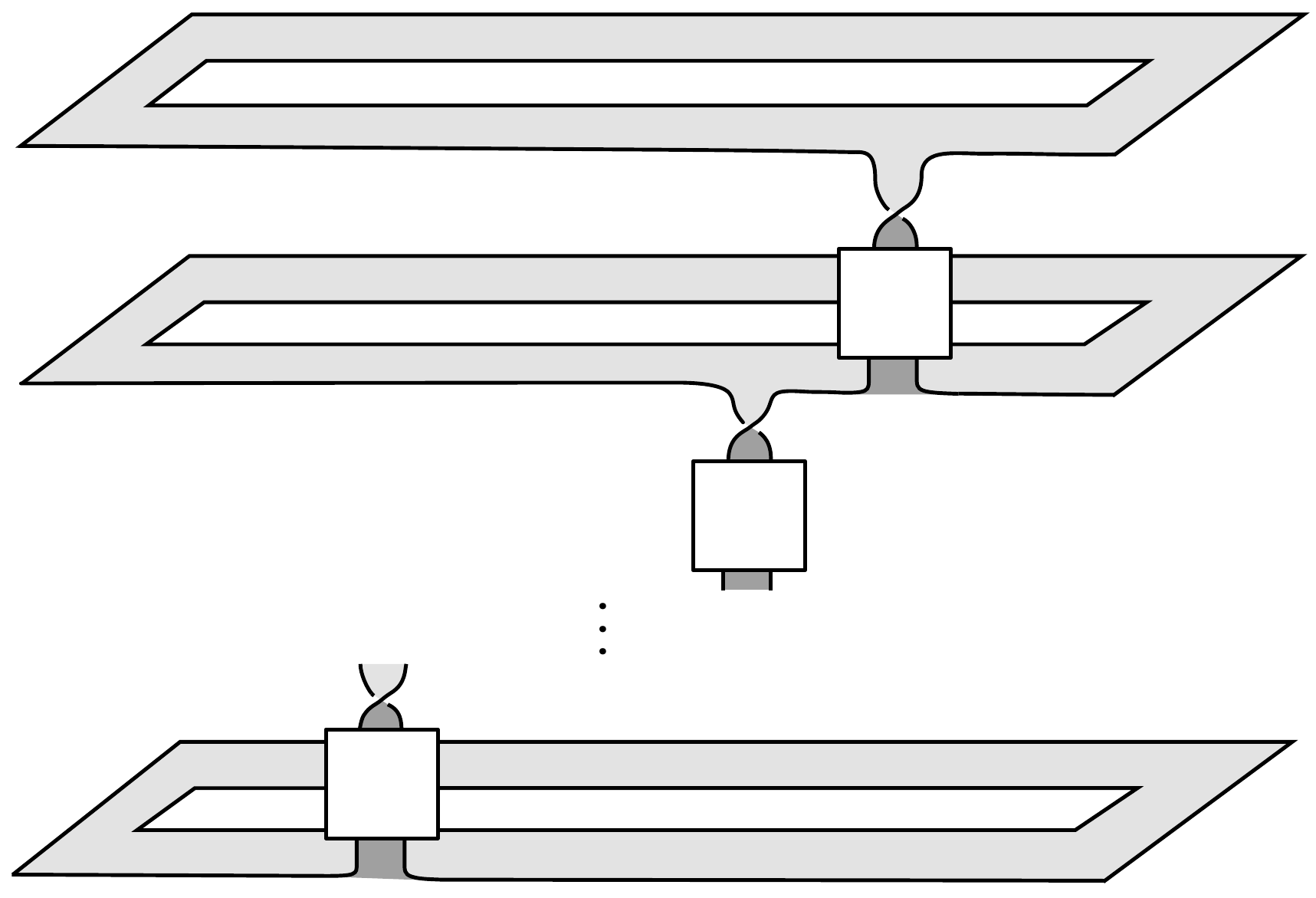}
    \caption{The surface $\overline{S}$.}
    \label{fig:floors}
  \end{figure}

\begin{proof}
We first show that the pattern $P'$ shown in \Cref{fig:patternprimes}, the case where $B=\emptyset$, preserves strong quasi-positivity in this sense. Let $S$ be a quasi-positive Seifert surface for a non-trivial knot $K$.
Then, the union of $(m+1)$ parallel copies of $S$ is also quasi-positive, which we denote by $\widetilde{S}$. Consider a surface $\overline{S}$ with $(m+1)$ annuli banded together as in \Cref{fig:floors}. Recall that $P'(K)$ is formed by cutting out from $(S^3,K)$ a regular neighborhood pair $(S^2\times D^2, S^1)$ of the knot $K$, and gluing in the pair $(S^1\times D^2,P')$. Upon forming $P'(K)$, the surface $\overline{S}$ is transferred to $S^1\times D^2$ in such a way that the interior boundaries of the annuli in $\overline{S}$ glue to the parallel copies of $S$ in $\wt S$. The resulting surface $S'$ is a quasi-positive Seifert surface for $P'(K)$, and thus $P'$ preserves quasi-positivity. Note that the Euler characteristic of $S'$ is
\[
\chi(S')= (m+1)\chi(S)-m.
\]

We next use Rudolph's result \cite{Ru92} that any full subsurface of a quasi-positive surface is also quasi-positive. (Recall that a {\it full subsurface} is a subsurface whose inclusion induces an injection on fundamental groups.)  Let $\widetilde{S}'$ be the union of $n$ parallel copies of a regular neighborhood of $\partial S$ in $S$ and $(m-n+1)$ parallel copies of $S$. Then $\widetilde{S}'$ is regarded as a subsurface of $\widetilde{S}$. Fitting this into the above construction, where we glue $\widetilde{S}'$ and $\overline{S}$,
we obtain a corresponding subsurface $S''$ of $S'$. Since $g(S)>0$, the boundary of $S$ is not null-homotopic in $S$, and hence $S''$ is a full subsurface of $S'$. Consequently, $S''$ is a quasi-positive Seifert surface for $P''(K)$ with Euler characteristic 
\[
\chi(S'')  = (m-n+1)\chi(S)-m,
\]
where $P''$ is the pattern depicted in \Cref{fig:patternprimes}, with dotted lines included.

Next, we consider plumbings of $S''$ with twisted annuli $B_i$ ($1 \leq i \leq n$), each which has $b_i$ full twists, so that the resulting surface $S'''$ has boundary $P(K)$.
Then, by \cite[Theorem]{Ru98},
the quasi-positivity of $S''$ and $B_i$ implies
that $S'''$ is quasi-positive. Moreover, we have
\[
\chi(S''') = \chi(S'') -n = (m-n+1) \chi(S)-m-n.
\]
This implies $g(P(K))=g(S''')= (m-n+1) g(K) +n$. Under our assumptions on $m$ and $n$, we obtain $g(P(K))>0$. In particular, $P(K)$ is non-trivial.
\end{proof}

\begin{proof}[Proof of \Cref{quasipositivity}]
Let $K$ be any non-trivial positive knot, such as a positive torus knot $T_{p,q}$. Then by \Cref{quasipositivitylem}, $P(K)$ is a non-trivial strongly quasi-positive knot, and hence $P$ satisfies condition (i) of \Cref{main thm satellite}. Condition (ii), that $P(U_1)=U_1$, is straightforward. This completes the proof.
\end{proof}

The property of preserving strongly quasi-positive knots is closed under composition of patterns, and so too is the property $P(U_1)=U_1$. Thus if $P_1,\ldots,P_l$ are patterns as in \Cref{quasipositivity}, then $P_l\circ \cdots \circ P_1$ also satisfies \Cref{main thm satellite}. This remark, together with \Cref{prop:tpqsatellitespec}, gives the following.

\begin{cor}\label{cor:satellitecompositions}
Let $P_1,\ldots,P_l$ be patterns as in \Cref{quasipositivity}, and $P:=P_l\circ \cdots \circ P_1$. Then
\[
	\{ P(T_{p,q+pn})\}_{n=0}^\infty
\]
is a linearly independent set in the homology concordance group.
\end{cor}

\subsection{Applications to the homology cobordism group}

We now give several applications to the homology cobordism group of homology 3-spheres. We begin by proving \Cref{hom cob main} from the introduction. To this end, we have:

\begin{thm}\label{linearly indep stilde}
	For a knot $K$ in $S^3$ with $\wt{s}(K)>0$, the set $\{ S^3_{1/n}(K)\}_{n=1}^\infty $ is linearly independent in the homology cobordism group. 
\end{thm}
\begin{proof}
	\cref{relation h and s-tilde} implies $h(S_{1}(K))<0$. Now \cite[Theorem 1.8]{NST19} completes the proof.
\end{proof}

\begin{proof}[Proof of \Cref{hom cob main}]
By \cref{linearly indep stilde}, we only need to check $\wt{s}(K)>0$ for knots listed in (i) and (ii), which follows
from \cref{sslice-torus} and \cref{alternating stilde}. 
\end{proof}

The following is a restatement of \Cref{intro non-triviality of r0 for surgery} from the introduction.

\begin{thm}\label{non-triviality of r0 for surgery}
	Let $K$ be a knot in an integer homology 3-sphere $Y$ satisfying $\sigma  (Y,K) \leq 0$. Suppose 
	$\frac{1}{8} < \Gamma_{(Y,K)} \left( -\frac{1}{2} \sigma (Y,K) \right)$.
	Then, $\Gamma_{Y_1(K)} (0) >0$ and $r_0 (Y_{1}(K))< \infty $.
\end{thm}

\begin{proof}
	By \cref{gamma and rs surgery} with $k=0$, we have
	\[
	  0<\Gamma_{(Y,K)} \left( -\frac{1}{2} \sigma (Y,K) \right) -\frac{1}{8}\leq \Gamma_{(Y_{1} (K) ,U_1)}(0). 
	 \]
	Using \eqref{eq:gammahom3sphereineq2}, we conclude that $\Gamma_{Y_1(K)} (0) >0$. 
	The second claim follows from the analogue of \eqref{N-infty} for integer homology spheres.
\end{proof}

The following is useful criterion for determining when the surgeries of a knot determines a linearly independent set in homology cobordism. It will be used to prove \Cref{1028} from the introduction.

\begin{thm}\label{applications to homology cobordism group: general}
	Let $K$ be a knot in $S^3$ satisfying $\sigma  (K) \leq 0$. Suppose 
	$\frac{1}{8} < \Gamma_K \left(-\frac{1}{2} \sigma (K) \right).$ Then the set
	$\{ S^3_{1/n}(K)\}_{n=1}^\infty $ is linearly independent in the homology cobordism group. 
\end{thm}
\begin{proof}
By \cref{non-triviality of r0 for surgery}, we have $r_0(S_{1}(K)) < \infty$. Then,
the proof of \cite[Theorem 5.12]{NST19} implies
\begin{align*}
\infty > r_0(S_{1}(K)) > r_0(S_{1/2}(K)) > \cdots \qquad \text{and}\\[2mm]
\infty = r_0(-S_{1}(K)) = r_0(-S_{1/2}(K)) = \cdots.\qquad \phantom{\text{and}}
\end{align*}
The proof now follows from \cite[Corollary 5.6]{NST19}, the 3-manifold analogue of \cref{linear independence r0}.
\end{proof}

The following result, with \Cref{lem:Kmn}, implies \Cref{1028}.

\begin{thm}\label{thm:kmnhomcobresult}
		For any of the two-bridge knots $K_{m,n}$ defined in \eqref{eq:twobridgefamily} with $m\geq 1$ and $n\geq 0$, the set $\{ S^3_{1/k}(K_{m,n})\}_{k=1}^\infty $ is linearly independent in the homology cobordism group. 
\end{thm}

\begin{proof}
	For $K_{m,0}=10_{28}$ (where $m$ is arbitrary), we show below in \Cref{gamma-and-r-for-10_28} that 
	\[
		\Gamma_{10_{28}^\ast}(0) = \frac{8}{53} > \frac{1}{8}.
	\]
	Next, suppose $m\geq 1$ and $n\geq 0$. Recall the crossing change cobordism $(S^3, K_{m,n+1}^\ast) \to (S^3, K_{m,n}^\ast)$ from the proof of \cref{alternating main2}, which is negative definite of strong height zero (and level zero), since $\sigma(K_{m,n}^\ast)=(K_{m,n+1}^\ast)=0$. From \Cref{Fr ineq} (see also \cite[Proposition 4.33]{DS20}) we obtain
	\[
		\Gamma_{K_{m,n}^\ast}(0) \leq \Gamma_{K_{m,n+1}^\ast}(0). 
	\]
	We inductively obtain $\Gamma_{K_{m,n+1}^\ast}(0) >1/8$. The result now follows from \Cref{applications to homology cobordism group: general}.
\end{proof}

We provide some more examples using \Cref{applications to homology cobordism group: general}, where the knot is a two-bridge knot. Note that for a two-bridge knot $K$ with non-zero signature, \Cref{hom cob main} provides a linearly independent set of the surgeries on $K$. For this reason we focus on examples with zero signature. We will use the following.

\begin{figure}
\labellist
	\Large\hair 2pt
	\pinlabel $m$ at 121 82
	\pinlabel $n$ at 257 36
	\endlabellist
\centering
\includegraphics[scale=0.6]{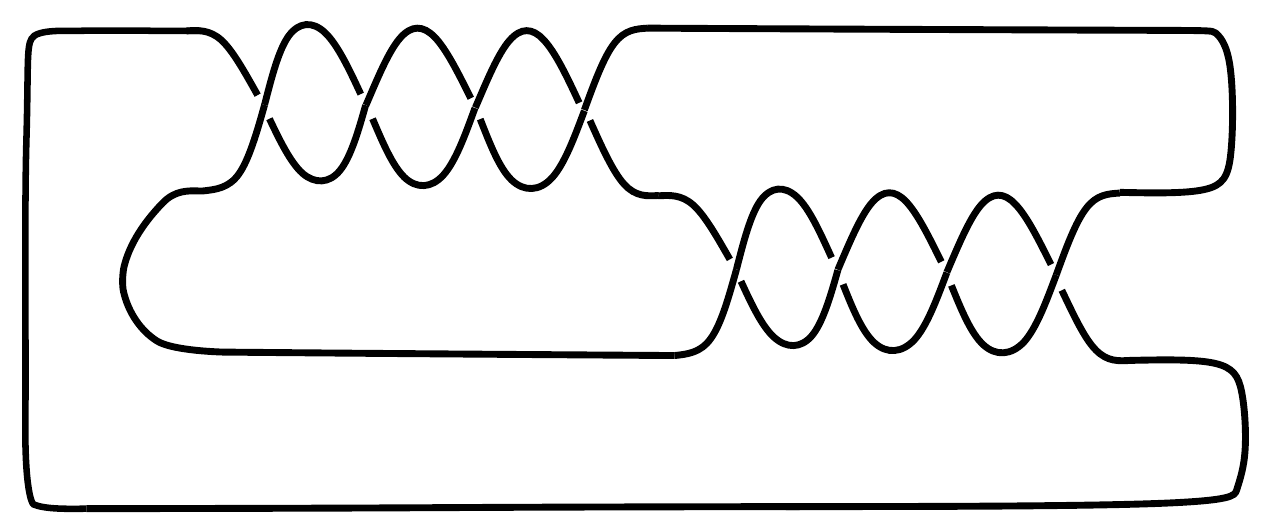}
\caption{The double twist knot $D_{m,n}$ is a two-bridge knot involving a strand of $m$ full twists and one of $n$ full twists. The particular example shown is $D_{2,2}$, which is the knot $7_4$ in Rolfsen notation.}
\label{fig:double-twist}
\end{figure}

\begin{lem}\label{lem:doubletwistcrossingchange}
	For $m,n \geq 2$ with $\max\{m,n\}\geq 3$, let $D_{m,n}$ be a double twist knot as described in \Cref{fig:double-twist}.
	Suppose $K$ is a knot with $\sigma(K)=0$ obtained from $D_{m,n}$ by changing a positive crossing to a negative crossing.
	Then $\{ S^3_{1/k}(K)\}_{k=1}^{\infty} $ is linearly independent in the homology cobordism group.  
\end{lem}
\begin{proof}
	Since a single negative crossing change gives rise to a negative definite cobordism 
of height $1$ and $\kappa_{\rm min}=1/4$, \cite[Proposition 4.33]{DS20} implies 
\[
  \Gamma_{D_{m,n}}(1)-\frac{1}{2} \leq  \Gamma_{K}(0).   
\]
On the other hand,  \cite[Corollary 3.24]{DS20}  computes
\[
\Gamma_{D_{m,n}}(1) = \frac{(2m-1)(2n-1)}{4mn-1}.
\]	
It is then straightforward to verify using these two relations that $ \Gamma_{K}(0)>1/8$ for $m\geq 2$, $n\geq 3$. Now, the assertion follows from \cref{applications to homology cobordism group: general}. 
 \end{proof}

With our conventions, the double twist knot $D_{m,n}$ has the two-bridge knot description
\[
	D_{m,n} = K(4mn-1, -2m)
\]
Given an integer $l\geq 1$ we consider the two-bridge knot
\[
	D_{l,m,n} := K(16 lmn - 4l -4mn +4m+1, -8ln+2n-2).
\]
This has Hirzebruch--Jung continued fraction $(-2m,-2n,2l,2)$, and by changing a negative crossing in the twist corresponding to ``$2$'', we obtain $D_{m,n}$. The mirror of the case of $m=n=2$ and $l=1$ is depicted in \Cref{fig;7-4knot}. Indeed, in this case $D_{1,2,2}=K(53,-14)=10_{28}^\ast$.

\begin{thm}\label{thm:doubletwistcrossingchangefamily}
For integers $l\geq 1$, $m,n \geq 2,$ with $\max\{m,n\}\geq 3$, the two-bridge knot $D_{l,m,n}$ has $\sigma(D_{l,m,n})=0$, and $\{ S^3_{1/k}(D_{l,m,n})\}_{k =1}^\infty $ is linearly independent in the homology cobordism group. 
\end{thm}

\begin{proof}
By the above description of $D_{l,m,n}$ and \Cref{lem:doubletwistcrossingchange}, it suffices to show $\sigma(D_{l,m,n})=0$.
Similar to the setup in the proof of \Cref{lem:Kmn}, $\sigma(D_{l,m,n})$ is given by the signature of the matrix
\[
\left[\begin{array}{cccc}
-2m&-1&&\\
-1&-2n&-1&\\
&-1&2l&-1\\
&&-1&2
\end{array}\right]
\]
which, given that $l,m,n$ are positive, is clearly zero.
\end{proof}

Note that $D_{1,2,2}=10_{28}^\ast$ is not included in this statement, although the result still holds (in fact, it is contained in \Cref{thm:kmnhomcobresult}). Special cases of \Cref{thm:doubletwistcrossingchangefamily} include
\begin{gather}
	D_{1,2,3} = K(77,50)=12a_{380}^\ast, \hspace{0.75cm} D_{1,3,2} = K(81,52) = 12a_{596}^\ast, \label{eq:moretwobridgeexs1} \\[2mm]
	 D_{1,4,2} = K(109,70),\hspace{0.75cm} D_{1,2,4} = K(101,66). \label{eq:moretwobridgeexs2}
\end{gather}
In the next section, a more explicit way of proving these cases is explained. We remark that $D_{l,m,n}$ is not always algebraically slice: for example, $D_{1,2,3}$ has non-trivial Tristram--Levine signature function.

\subsection{More two-bridge examples using the ADHM construction}

In \cite{DS20}, the $\mathcal{S}$-complex of a two-bridge knot $K(p,q)$ is studied using information derived from an equivariant ADHM correspondence. The data of $d,\delta_1,\delta_2$ are completely determined, and constraints are given for which components of the $v$-map can possibly be non-zero. All of this information is represented entirely by arithmetic conditions in terms of $(p,q)$. Here, we use this information to compute the local equivalence class of the enriched $\mathcal{S}$-complex for $10_{28}$. We also provide some further examples. We refer to \cite{DS20} for the details on the structure of $\mathcal{S}$-complexes used below.

\vspace{1mm}

\begin{rem}
	The conventions of this paper differ from those in \cite{DS19,DS20}: the two-bridge knot $K(p,q)$ in this paper corresponds to the two-bridge knot $K(p,-q)$ in those references.
\end{rem}

For two-bridge knots, one can avoid perturbations and the generality of enriched $\mathcal{S}$-complexes, and work exclusively with $I$-graded $\mathcal{S}$-complexes. Throughout this section, we work with $I$-graded $\mathcal{S}$-complexes over $R[U^{\pm 1}]$, where $R=\Z[T^{\pm 1}]$. We begin by giving a simple expression for the local equivalence class of the $I$-graded $\mathcal{S}$-complex for the two-bridge knot $10^\ast_{28}$. Define an $I$-graded $\mathcal{S}$-complex
\[
	\widetilde C_\ast' = C_\ast'\oplus C_{\ast-1}'\oplus R[U^{\pm 1}]
\]
where $C_\ast'$ is freely generated by $\alpha$, $\beta$. The differential $\widetilde d'$ satisfies $d'(\beta)=(T^2-T^{-2})\alpha$ and $\delta_2'(1)=(T^2-T^{-2})\alpha$, and is otherwise zero. The $\Z\times \R$-gradings for the generators $\alpha$ and $\beta$ are respectively defined to be $(-2,-20/53)$ and $(-1,8/53)$.

\begin{figure}[t]
\centering
\begin{tikzpicture}[xscale=1.65,yscale=1.15]
\draw (-3.0,1.0) node {$-2$};
\draw (-1.5,1.0) node {$-1$};
\draw (0.0,1.0) node {$0$};
\draw (1.5,1.0) node {$1$};
\draw (3.0,1.0) node {$2$};
\draw (4.5,1.0) node {$3$};
\draw (0,0) node {$\zeta^0$};
\node [below, green] at (0,-0.25) {$0$};
\draw (1.5,0.0) node {$\zeta^{3}$};
\node [below, green] at (1.5,-0.25) {$\frac{33}{53}$};
\draw (3.0,0.0) node {$\zeta^{1}$};
\node [below, green] at (3.0,-0.25) {$\frac{39}{53}$};
\draw (0.0,-1.5) node {$\zeta^{16}$};
\node [below, green] at (0,-1.75) {$\frac{20}{53}$};
\draw (4.5,0.0) node {$\zeta^{2}$};
\node [below, green] at (4.5,-0.25) {$\frac{50}{53}$};
\draw (-3.0,-0.0) node {$\zeta^{19}$};
\node [below, green] at (-3.15,-0.25) {$-\frac{20}{53}$};
\draw (1.5,-4.5) node {$\zeta^{6}$};
\node [below, green] at (1.5,-4.75) {$\frac{26}{53}$};
\draw (3.0,-3.0) node {$\zeta^{4}$};
\node [below, green] at (3.0,-3.25) {$\frac{41}{53}$};
\draw (0.0,-4.5) node {$\zeta^{13}$};
\node [below, green] at (0.0,-4.75) {$\frac{19}{53}$};
\draw (3.0,-1.5) node {$\zeta^{17}$};
\node [below, green] at (3.0,-1.75) {$\frac{35}{53}$};
\draw (0.0,-3.0) node {$\zeta^{5}$};
\node [below, green] at (0.0,-3.25) {$\frac{21}{53}$};
\draw (1.5,-3.0) node {$\zeta^{15}$};
\node [below, green] at (1.5,-3.25) {$\frac{30}{53}$};
\draw (3.0,-6.0) node {$\zeta^{9}$};
\node [below, green] at (3.0,-6.25) {$\frac{32}{53}$};
\draw (1.5,-1.5) node {$\zeta^{18}$};
\node [below, green] at (1.5,-1.75) {$\frac{22}{53}$};
\draw (-1.5,-1.5) node {$\zeta^{22}$};
\node [below, green] at (-1.5,-1.75) {$\frac{8}{53}$};
\draw (4.5,-3.0) node {$\zeta^{7}$};
\node [below, green] at (4.5,-3.25) {$\frac{56}{53}$};
\draw (1.5,-6.0) node {$\zeta^{10}$};
\node [below, green] at (1.5,-6.25) {$\frac{31}{53}$};
\draw (-1.5,-3.0) node {$\zeta^{14}$};
\node [below, green] at (-1.5,-3.25) {$\frac{12}{53}$};
\draw (0.0,-7.5) node {$\zeta^{20}$};
\node [below, green] at (0.0,-7.75) {$\frac{18}{53}$};
\draw (0.0,-6.0) node {$\zeta^{8}$};
\node [below, green] at (0.0,-6.25) {$\frac{5}{53}$};
\draw (3.0,-4.5) node {$\zeta^{12}$};
\node [below, green] at (3.0,-4.75) {$\frac{51}{53}$};
\draw (1.5,-7.5) node {$\zeta^{21}$};
\node [below, green] at (1.5,-7.75) {$\frac{27}{53}$};
\draw (-1.5,-7.5) node {$\zeta^{25}$};
\node [below, green] at (-1.65,-7.75) {$-\frac{5}{53}$};
\draw (-1.5,-6.0) node {$\zeta^{11}$};
\node [below, green] at (-1.5,-6.25) {$\frac{2}{53}$};
\draw (0.0,-9.0) node {$\zeta^{23}$};
\node [below, green] at (0.0,-9.25) {$\frac{14}{53}$};
\draw (-3.0,-7.5) node {$\zeta^{24}$};
\node [below, green] at (-3.15,-7.75) {$-\frac{8}{53}$};
\draw (1.5,-9.0) node {$\zeta^{26}$};
\node [below, green] at (1.5,-9.25) {$\frac{23}{53}$};
\draw [->, line width=0.22mm, blue] (1.15,0.0) -- (0.35,0.0); 
\draw [->, line width=0.22mm, red] (-0.35,0.0) -- (-2.5,0.0); 
\draw [->, line width=0.22mm, purple] (4.15,0.0) -- (3.35,0.0); 
\draw [->, line width=0.22mm, purple] (4.15,-0.25) -- (3.5,-1.25); 
\draw [->, line width=0.22mm, purple] (2.65,-0.25) -- (2.0,-1.25); 
\draw [->, line width=0.22mm, purple] (2.65,-1.5) -- (1.85,-1.5); 
\draw [->, line width=0.22mm, purple] (-1.85,-1.25) -- (-2.65,-0.25);
\draw [->, line width=0.22mm, purple] (1.15,-0.25) -- (0.5,-1.25); 
\draw [->, line width=0.22mm, purple] (-0.35,-3.0) -- (-1.15,-3.0); 
\draw [->, line width=0.22mm, purple] (-0.35,-6.0) -- (-1.15,-6.0); 
\draw [->, line width=0.22mm, purple] (1.15,-4.5) -- (0.35,-4.5); 
\draw [->, line width=0.22mm, purple] (1.15,-7.5) -- (0.35,-7.5); 
\draw [->, line width=0.22mm, purple] (1.15,-9.0) -- (0.35,-9.0); 
\draw [->, line width=0.22mm, purple] (2.65,-3.0) -- (1.85,-3.0); 
\draw [->, line width=0.22mm, purple] (2.65,-6.0) -- (1.85,-6.0); 
\draw [->, line width=0.22mm, purple] (4.15,-3.0) -- (3.35,-3.0); 
\draw [->, line width=0.22mm, purple] (4.15,-3.25) -- (3.5,-4.25); 
\draw [->, line width=0.22mm, purple] (-0.35,-4.25) -- (-1.15,-3.25);
\draw [->, line width=0.22mm, purple] (1.15,-4.25) -- (0.35,-3.25);
\draw [->, line width=0.22mm, purple] (2.65,-4.25) -- (1.85,-3.25);
\draw [->, line width=0.22mm, purple] (-1.9,-7.5) -- (-2.6,-7.5);
\end{tikzpicture}
\caption{This diagram shows generators for $C_\ast$ and the reducible $\zeta^0$ for the knot $10^\ast_{28}=K(53,34)$. All arrows are multiplication by $\pm (T^2-T^{-2})$. The displayed arrows represent $d$ except for the two which are incident to $\zeta^0$, which are $\delta_1$ and $\delta_2$. The $v$-map might have non-zero components (not displayed).}
\label{fig:5314complex}
\end{figure}
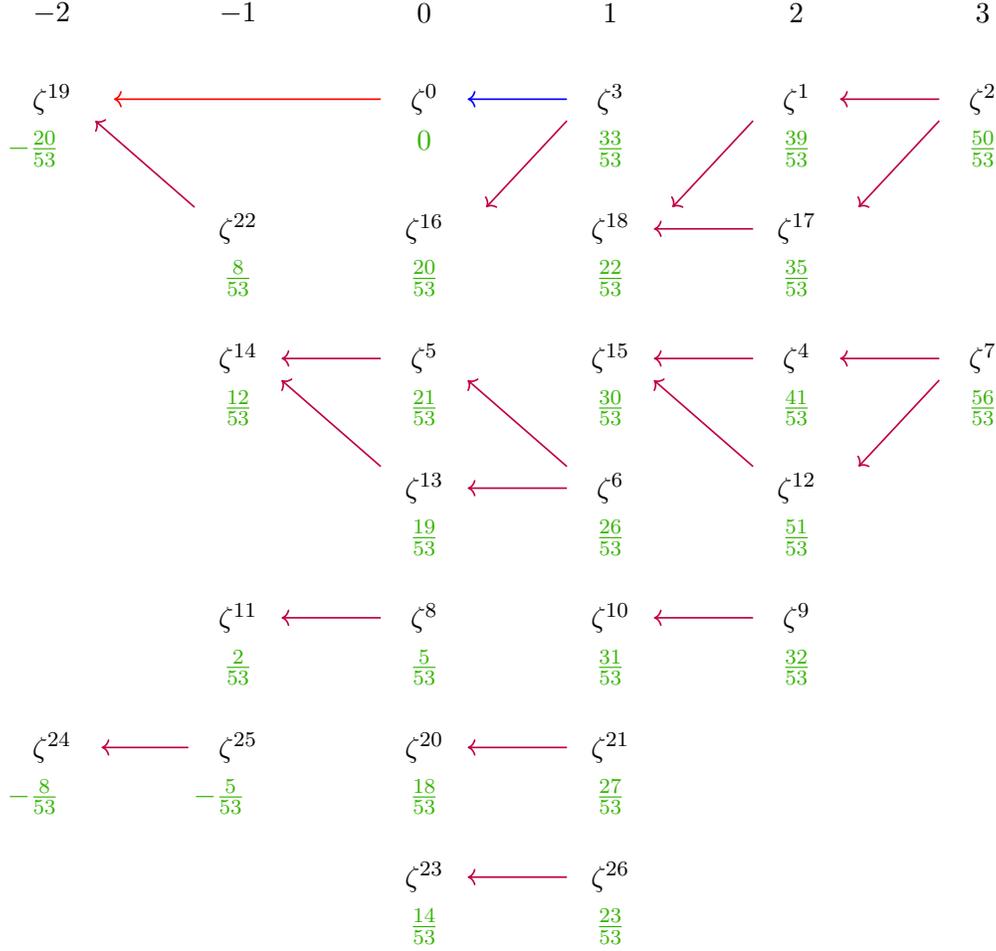

\begin{prop}
	The $I$-graded $\mathcal{S}$-complex of the knot $10^\ast_{28}$ over $R[U^{\pm 1}]$, where $R=\Z[T^{\pm 1}]$, is locally equivalent to the $I$-graded $\mathcal{S}$-complex $\widetilde C'$ which is defined above.
\end{prop}

\begin{proof}
	Recall that $10^\ast_{28}$ is the two-bridge knot $K(53,34)$. Let $\widetilde C=\wt C(10^\ast_{28};\Delta_R)$ be the associated $I$-graded $\mathcal{S}$-complex over $R[U^{\pm 1}]$. We define morphisms $\widetilde {\lambda}:\widetilde C\to \widetilde C'$ and $\widetilde {\lambda}:\widetilde C'\to \widetilde C$. First we define $\widetilde {\lambda}'$. We write the components of this morphism in the usual way:
	\[
		\widetilde{\lambda}'  = \left[\begin{array}{ccc}{\lambda}' & 0 & 0 \\ \mu' & {\lambda}' & \Delta'_2 \\ \Delta_1' & 0 & 1  \end{array} \right]
	\]
	We declare that the components of $\widetilde{\lambda}'$ as indicated above are given by:
	\[
		{\lambda}'(\alpha) = \zeta^{19}, \qquad {\lambda}'(\beta) = \zeta^{22}, \qquad \mu'(\alpha) =0 , \qquad \mu'(\beta) = cU^{-1}\zeta^1
	\]
	and $\Delta_1'=\Delta_2'=0$. Here $c$ is an integer such that $v(\zeta^{22}) = cU^{-1}(T^2-T^{-2})\zeta^{18}$. We are using the constraint from \cite{DS20} that the only possible non-zero component of the $v$-map acting on $\zeta^{22}$ is some multiple of $\zeta^{18}$. It is straightforward to check that $\widetilde{\lambda}'$ is a local map of $\mathcal{S}$-complexes. Furthermore, because $\text{deg}_I(\beta)=8/53 > -14/53 = \text{deg}_I(U^{-1}\zeta^1)$, it is a level $0$ local morphism of $I$-graded $\mathcal{S}$-complexes.
 Now we turn to the construction of $\widetilde{\lambda}$. For this, we declare the nonzero components:
 \[
 	{\lambda}(\zeta^{19}) = \alpha, \qquad {\lambda}(\zeta^{22})=\beta
 \]
 \[
 	\mu(\zeta^{16}) = -c_1\zeta^{22}, \qquad \mu(\zeta^{20}) = -c_2\zeta^{22}, \qquad \mu(\zeta^{23}) = -c_3\zeta^22
 \]
 Here $c_1,c_2,c_3$ are constants defined by the following relations:
 \begin{gather*}
 		v(\zeta^3)=c_1(T^2-T^{-2})\zeta^{22},\\ v(\zeta^{21}) = c_2 (T^2-T^{-2})\zeta^{22},\\  v(\zeta^{26}) = c_3(T^2-T^{-2})\zeta^{22}
 \end{gather*}
 In writing these expressions, we are again using the constraint on the form of the $v$-map from \cite{DS20}, which is a numerical computation. We also set $\Delta_1=\Delta_2=0$. It is straightforward to verify that $\widetilde {\lambda}$ as defined is a local map of $\mathcal{S}$-complexes. Furthermore, because
\[
 	\text{deg}_I(\zeta^{16} ) = \frac{20}{53} , \qquad \qquad	\text{deg}_I(\zeta^{20} ) = \frac{18}{53} ,\qquad \qquad \text{deg}_I(\zeta^{23} ) = \frac{14}{53}
\]
are all greater than $\text{deg}_I(\beta)=8/53$, this is a level $0$ local morphism of $I$-graded $\mathcal{S}$-complexes.

Having constructed level $0$ local morphisms in each direction, we have proved the result.
\end{proof}

\definecolor{faintgreen}{RGB}{248, 248, 248}

\renewcommand{\arraystretch}{1.5}
\begin{table}[]
\centering
\begin{tabular}{ccccc}
\toprule
$(p,q)$ & {Other name}    & $\Gamma_K(0)$ &  $\Gamma_K(0)>\frac{1}{8}$? & $r'_{s^\star}(K^\ast)$ \\
\midrule
\phantom{mm} $(37,17)$   \phantom{mm}  &  \phantom{mm} $10^\ast_{34}$    \phantom{mm}    & \phantom{mm}  $2/37$       \phantom{mm}    &  \phantom{mm} N   \phantom{mm} &   \phantom{mm} 13/37   \phantom{mm}   \\
$(49,23)$     & $12a_{169}$        & $4/49$           &  N&   $17/49$ \\
\rowcolor{faintgreen} $(53,34)$     & $10^\ast_{28}$        & $8/53$           &  Y & $19/53$\\
$(61,29)$     &         & $6/61$           &  N & $21/61$\\
\rowcolor{faintgreen}$(65,51)$     & $11a_{333}^\ast$        & $14/65$           &  Y &   $25/65$ \\
$(69,29)$     & $10_{32}$        & $6/69$           &  N & $19/69$\\
$(73,23)$     & $12a_{1148}$        & $2/73$           &  N & $19/73$\\
$(73,35)$     &          & $8/73$           &  N & $25/73$  \\
\rowcolor{faintgreen} $(77,50)$     & $12a_{380}^\ast$        & $18/77$           &  Y & $27/77$  \\
\rowcolor{faintgreen} $(81,52)$     & $12a_{596}^\ast$        & $16/81$           &  Y & $29/81$ \\
$(85,41)$     &          & $10/85$          &  N & $29/85$  \\
\rowcolor{faintgreen} $(93,73)$     &          & $22/93$           &  Y  & $35/93$\\
$(93,41)$     &   $11a_{93}^\ast$      & $8/93$           &  N & $27/93$\\
$(97,31)$     &          & $4/97$           &  N  & $25/97$\\
$(97,47)$     &          & $12/97$           &  N  & $33/97$ \\
\rowcolor{faintgreen} $(101,66)$     &          & $28/101$           &  Y & $35/101$\\
$(105,41)$ & $11a_{175}$ & $6/105$ & N  & $25/105$ \\
\rowcolor{faintgreen} $(109,70)$     &          & $24/109$           &  Y & $39/109$\\
\rowcolor{faintgreen} $(109,53)$     &          & $14/109$           &  Y & $37/109$ \\
$(109,51)$     &          & $2/109$           &  N & $47/109$ \\
\bottomrule
\end{tabular}
\vspace{1cm}
\caption{\label{table:twobridgecomps} Every two-bridge knot $K=K(p,q)$ with (i) $|p|\leqslant 109$, (ii) $\sigma(K)=0$, and (iii) $\Gamma_{K}(0),\Gamma_{K^\ast}(0)$ not both zero, is equivalent to one of the above, after possibly taking the mirror. The identifiers for the knots in the second column are from Knotinfo \cite{knotinfo}. In each case, $\Gamma_{K^\ast}=\Gamma_{U_1}$ and $r'_s(K)=r_s'(U_1)$, while for the mirror, $r'_s(K^\ast)=\infty$ for $s<-\Gamma_K(0)$ and $r'_s(K^\ast)$ is finite for $s\in (-\Gamma_K(0),0]$. In the last column above, $s^\star$ is any real number in $(-\Gamma_K(0),0]$. }
\end{table}

In what follows, we denote by
\[
	r_s'(K) = -\inf \{ r \in [-\infty, 0) \mid \underline{\newinv}_{K}(0,r) \leq -s \}
\]
the version of the $r_s$-invariant for a knot $K$ which only depends on the enriched $\mathcal{S}$-complex defined over the field of fractions of $R=\Z[T^{\pm 1}]$.

\begin{cor}\label{gamma-and-r-for-10_28}
The non-trivial $\Gamma$ and $r_s'$ invariants for $10_{28}$ and its mirror are give by:
\[
	\Gamma_{10^\ast_{28}}(i) = \begin{cases} 0 & \qquad  i<0 \\ 8/53   & \qquad i=0\\ \infty & \qquad  i>0 \end{cases}
\]
\vspace{.1cm}
\[
	r'_s({10_{28}}) = \begin{cases} \infty & \qquad s \in(-\infty,-8/53) \\  19/53  & \qquad s\in( -8/53, 0] \end{cases}
\]

\end{cor}

Similar computations can be carried out for other two-bridge knots. In Table \ref{table:twobridgecomps}, we list some such examples, focusing on the case of two-bridge knots with zero signature. The computations were partially done by computer, following the algorithm given in \cite{DS20} for $\mathcal{S}$-complexes of two-bridge knots. The knots which have a shaded cell in the table are ones to which Theorem \ref{applications to homology cobordism group: general} applies. Note that the examples \eqref{eq:moretwobridgeexs1}--\eqref{eq:moretwobridgeexs1} appear in this table.

\subsection{Irreducible $SU(2)$-representations}\label{Extendability of irreducible $SU(2)$-representation}

We now consider results on the existence of non-abelian traceless $SU(2)$ representations for the fundamental groups of homology concordance complements, and a related result for homology cobordisms. (Note that \Cref{alg conc irrep}, which involves arbitrary holonomy parameters, is proved in the next section.) 

\begin{thm}\label{existence main}
Let $K$ be a knot in an integer homology $3$-sphere $Y$. If the enriched $\mathcal{S}$-complex of $(Y,K)$ is not weakly locally equivalent to the trivial enriched $\mathcal{S}$-complex, then, for any homology concordance $(W,S):(Y,K)\to (Y',K')$, there exists an irreducible traceless representation $\pi_1 (W \setminus S) \to SU(2)$. 
\end{thm}
A similar result involving a non-vanishing condition on the signature was proven in \cite{DS19}. 
\begin{proof}
If $\mathfrak{E}(Y,K)$ is not locally equivalent to the trivial enriched $\mathcal{S}$-complex, \Cref{local eq gamma} says that either $r_0(K)< \infty$ or $r_0(K^*)< \infty$. Then \cref{definite ineq for rs} completes the proof.  
\end{proof}

As a corollary, we have the following, which implies \Cref{thm:stilderepresentationsintro} from the introduction.
 
\begin{cor}
Let $K$ be a knot in $S^3$. Suppose one of the following invariants is non-trivial (i.e. not equal to the value of an invariant for the unknot): 
\begin{itemize}
    \item[(i)] $\sigma(K)$, $s^\sharp(K)$, $s^\sharp_\pm(K)$, $\wt s(K)$, or $\wt\varepsilon(K)$;
    \item[(ii)] one of Kronheimer--Mrowka's invariants from \cite{KM19b}, such as $f_\sigma(K)$ or $z^\natural(K)$. 
\end{itemize}
Then, for any knot concordance $S \subset I \times S^3$ from $K$ to $K'$, there exists an irreducible traceless representation $\pi_1 (I \times S^3 \setminus S) \to SU(2)$. 
\end{cor}
\begin{proof}
If one of the above invariants is non-trivial, \Cref{thm:kminvtsarelocequiv} implies that the enriched $\mathcal{S}$-complex of $K$ is not weakly locally equivalent to that of the unknot. Thus, the result follows from \cref{existence main}. 
\end{proof}

We also provide a result on the existence of irreducible representations on homology cobordisms. 

\begin{thm}
Let $K$ be a knot in an integer homology $3$-sphere $Y$ satisfying $\sigma  (Y,K) \leq 0$. Suppose 
$\frac{1}{8} < \Gamma_{(Y,K)} \left( -\frac{1}{2} \sigma (Y,K) \right).$  Fix a positive integer $n$.
Then, for any homology cobordism $W$ from $Y_{1/n}(K)$ to itself, there exists an irreducible $SU(2)$-representation on $\pi_1(W)$.
\end{thm}
\begin{proof}
 \Cref{non-triviality of r0 for surgery} and \eqref{inequality of rs} imply $r_0(Y_1(K),U_1)<\infty$. There is a standard negative definite cobordism with $b_1=0$ from $Y_{1/n}(K)$ to $Y_1(K)$ by attaching a chain of $(-2)$-framed two-handles of length $n-1$ to a meridian of $K\subset Y_{1/n}(K)$, and this gives $r_0(Y_{1/n}(K),U_1)\leq r_0(Y_{1}(K),U_1)<\infty$. The result then follows from \Cref{irrep for 3-manifolds}.
\end{proof}


\section{General holonomy parameters}\label{Theory for general holonomy parameters}

Singular connections are the key geometrical player in the definition of the $\cS$-complex of a knot. Up to this point, we have only considered $\cS$-complexes defined by singular connections whose holonomies along meridians of a codimension two submanifold are asymptotic to a conjugate of the element 
\begin{equation}\label{i-conj-2}
\left[\begin{array}{cc}
i &0\\
0&-i \end{array}
\right] \in SU(2).
\end{equation}
One can consider more generally singular connections where \eqref{i-conj-2} is replaced with 
 \begin{equation}\label{omega-conj-2}
\left[\begin{array}{cc}
e^{2\pi i \omega} &0\\
0&e^{-2\pi i \omega} \end{array}
\right] \in SU(2)
\end{equation}
for some fixed $\omega\in (0, \frac{1}{2})$ \cite{KM93i,KM11}. In \Cref{S-comp-gen-hol}, we review how such singular connections can also be used to produce $\cS$-complexes for knots. In this section, we mainly follow \cite{Ha21}, see also \cite{Ech19}. However, there are some minor differences in our conventions that we discuss in detail below. We also study functoriality of $\cS$-complexes of knots with holonomy parameter $\omega$ with respect to cobordisms of pairs. In Subsection \ref{fil-spec-omega}, we study filtered special cycles with holonomy parameter $\omega$ and generalize the construction of the numerical invariants of Section \ref{Section: Concordance invariants from filtered special cycles} to the case of arbitrary holonomy parameter $\omega$. In particular, the proof of \Cref{alg conc irrep} will be given in this section.

\subsection{$\cS$-complexes from singular instantons with general holonomy parameters}\label{S-comp-gen-hol}
A singular connection with holonomy parameter $\omega$ for a pair $(Y,K)$ of a knot in an integer homology 3-sphere is, roughly speaking, a connection on the trivial $SU(2)$-bundle over $Y$ that is singular along $K$ and whose holonomy along any family of shrinking meridians of $K$ is asymptotic to a conjugate of the matrix \eqref{omega-conj-2}. This definition can be made precise by picking a model singular connection $A_0^\omega$ and then considering connections of the form $A_0^\omega+a$ where $a$ is a 1-form with coefficients in $\su$ that lives in a weighted Sobolev space and has appropriate decay behavior in a neighborhood of the knot $K$. We refer the reader to \cite{KM93i,KM11} for the details of the definition of singular connections, and we only review some of the features of singular connections that are essential for our purposes. We write $\mathcal A^\omega(Y,K)$ for the space of all singular connections for the pair $(Y,K)$. There is an action of a gauge group $\mathcal G^\omega(Y,K)$ on $\mathcal A^\omega(Y,K)$ and the quotient space is denoted by $\mathcal B^\omega(Y,K)$.

The definition of singular connections can be adapted to pairs $(W,S)$ of an embedded surface $S$ in a 4-manifold $W$ (or more generally submanifolds of codimension two in a smooth manifold). We will be interested in the case that $(W,S):(Y,K)\to (Y',K')$ is a cobordism of pairs with cylindrical ends, and consider singular $U(2)$-connections with holonomy parameter $\omega$ that are asymptotic to fixed representatives of $\alpha\in \mathcal B^\omega(Y,K)$ and $\alpha'\in \mathcal B^\omega(Y',K')$ such that the induced connection on the determinant bundle is some fixed (non-singular) $U(1)$ connection. The configuration spaces of such connections modulo the gauge group is denoted by $\mathcal B^\omega(W,S,c;\alpha,\alpha')$ where $c\in H^2(W;\Z)$ denotes the first Chern class of the determinant bundle. The set of connected components of $\mathcal B^\omega(W,S,c;\alpha,\alpha')$ can be characterized as a torsor over $\Z\oplus \Z$, where $(1,0)\in \Z\oplus \Z$ acts by {\it adding an instanton} and $(0,1)\in \Z\oplus \Z$ acts by {\it adding a monopole}. Any element $z$ of this torsor is called a {\it path from $\alpha$ to $\alpha'$ along $(W,S)$} and the corresponding connected component of $\mathcal B^\omega(W,S,c;\alpha,\alpha')$ is denoted by $\mathcal B^\omega_z(W,S,c;\alpha,\alpha')$. 

Associated to any singular connection $A$ representing an element of $\mathcal B^\omega(W,S,c;\alpha,\alpha')$, there are two quantities that are called {\it topological energy} and {\it monopole number}. The topological energy of $A$ is defined with the same Chern-Weil integral as before:
\[
 \kappa (A) = \frac{1}{8\pi^2} \int_W {\rm tr}{ ( F_{\ad(A)}\wedge F_{\ad(A)} ) },
\]
where $F_{\ad(A)}$ denotes the trace-free part of the curvature of $A$ corresponding to the splitting ${\mathfrak u}(2)=\su\oplus \R$ of the Lie algebra of $U(2)$. The curvature of $A$ extends over the singular locus and has the form 
\[
  \left[\begin{array}{cc}
	\Omega &0\\
	0&\Omega' 
\end{array}
\right]
\]
with respect to the splitting that gives the presentation in \eqref{omega-conj-2}. The monopole number of $A$ is defined as 
\[
  \nu (A) =-\frac{i}{2\pi}\int_S \left(\Omega-\Omega'\right)+2\omega S\cdot S.
\]
The topological energy and the monopole number of $A$ are locally constant and depend only on the path $z$ represented by $A$. Adding an instanton to $A$ changes $(\kappa(A),\nu(A))$ to $(\kappa(A)+1,\nu(A))$, and adding a monopole changes $(\kappa(A),\nu(A))$ to $(\kappa(A)+2\omega,\nu(A)+2)$.  Moreover, the connected components of $\mathcal B^\omega(W,S,c;\alpha,\alpha')$ are determined by their topological energies and monopole numbers.

The subspace of $\mathcal B^{\omega}(Y,K)$ consisting of connections with trivial curvature has a topological description. By taking holonomy along closed curves, the space of such flat connections, denoted by $\fC(Y,K,\omega)$, can be identified with the character variety 
\begin{equation}\label{char-var-omega}
	\{\varphi\in {\rm Hom}(\pi_{1}(Y\setminus K), SU(2))\mid \text{ $\varphi(\mu)$ is conjugate to \eqref{omega-conj-2} }\}/SU(2),
\end{equation}	
where $\mu$ is any meridian of $K$.
In particular, there is a distinguished flat singular connection $\theta^\omega$ in $\mathcal B^{\omega}(Y,K)$ that is reducible and it corresponds to the class of the abelian representation in \eqref{char-var-omega}. A flat singular connection is non-degenerate if it is cut down transversely by the flat curvature equation. The reducible connection $\theta^\omega$ is non-degenerate if and only if $e^{4\pi i \omega}$ is not a root of the Alexander polynomial  \cite[Lemma 15]{Ech19}. (See also \cite{KM11}.)

The ASD equation is naturally defined for singular connections on a 4-manifold, and we write $M^\omega_z(W,S,c;\alpha,\alpha')$ for the subspace of $\mathcal B^\omega_z(W,S,c;\alpha,\alpha')$ given by the solutions of the ASD equation. The local behavior of this moduli space around the class of a singular connection is controlled by an ASD operator $D_A$. Assuming that $\alpha$ and $\alpha'$ are non-degenerate, $D_A$ is a Fredholm operator defined on appropriate weighted Sobolev spaces that in particular give exponential decay at the ends. In this case, the moduli space $M^\omega_z(W,S,c;\alpha,\alpha')$ is a smooth manifold of dimension $\ind(D_A)$ in a neighborhood of the class of $A$ if $D_A$ is surjective. Analogous to the topological energy and monopole number, the index of $D_A$ depends only on $z$. In fact, the dimension formulas of \cite{KM93i} together with excision imply that for any singular connection $A$ representing an element of $\mathcal B^\omega(W,S,c;\alpha,\alpha')$ the expression 
\[
  \ind(D_A)-\left(8\kappa(A)+2(1-4\omega) \nu(A)-\frac{3}{2} (\chi(W)+\sigma(W))+\chi(S)+8\omega^2S\cdot S\right)
\]
depends only on $\alpha$ and $\alpha'$.
Motivated by this discussion, we write $\ind(z)$ for the index of the ASD operator of any element (and hence all elements) of $\mathcal B^\omega_z(W,S,c;\alpha,\alpha')$.

\begin{ex}\label{cob-red}
	Suppose $(W,S):(Y,K)\to (Y',K')$ is a cobordism of pairs such that $b^1(W)=b^+(W)=0$. We also assume that $e^{4\pi i \omega}$ is not a root
	of the Alexander polynomials of $(Y,K)$ and $(Y',K')$ and thus the reducible singular flat connections $\theta^\omega$ and $\theta'^\omega$ for
	$(Y,K)$ and $(Y',K')$ are non-degenerate. Let also $c\in H^2(W;\Z)$ and $B_0$ be a $U(1)$ connection on the line bundle with $c_1=c$. If $L$ is a $U(1)$-bundle on $W$, then there is a reducible instanton $A_L^\omega$ that has the form $B\oplus B^*\otimes B_0$ with 
	$B$ being a singular $U(1)$ connection such that $\frac{i}{2\pi} F_B$ represents the cohomology class $c_1(L)+\omega S$. Then we have 
	\[
	  \kappa(A_L^\omega)=-(c_1(L)+\omega S-\frac{1}{2}c)^2,\hspace{1cm}\nu(A_L^\omega)=(c-2c_1(L))\cdot S
	\]
	and the index formula
	\begin{align}\label{ind-formula}
	  \ind(D_{A_L^\omega})=&8\kappa(A_L^\omega)+2(1-4\omega) \nu(A_L^\omega)-\frac{3}{2} (\chi(W)+\sigma(W))+\chi(S)+8\omega^2S\cdot S+\nonumber\\[2mm]
	  \hspace{1cm}&
	  +\sigma_{\omega}(Y,K)-\sigma_{\omega}(Y',K')-1,
	\end{align}
	where $\sigma_{\omega}(Y,K)$, $\sigma_{\omega}(Y', K')$ are the  
	Tristram-Levine signatures of $(Y, K)$, $(Y', K')$ as given in \eqref{LT-sig}. This formula, in the case that $\omega$ is a rational number, is proved in \cite{Ha21}. 
	To extend the formula to the irrational case, note that if $[\omega_0-\epsilon,\omega_0+\epsilon]$ is an interval such that $e^{4\pi i \omega}$ is not a root of Alexander polynomial
	for any $\omega$ in this interval, then the expression on the left hand side of \eqref{ind-formula} is constant for all values of $\omega\in [\omega_0-\epsilon,\omega_0+\epsilon]$.
	The left hand side is an index of an elliptic operator, and it is clear from the Fredholm theory developed in \cite{KM93i,KM11} that the ASD operators $D_{A_L^\omega}$
	for $\omega\in [\omega_0-\epsilon,\omega_0+\epsilon]$ form a continuous family of Fredholm operators with a continuously varying domains and codomains. In particular,
	standard property of the indices of Fredholm operators imply that the left hand side of \eqref{ind-formula} is also constant. Now we can use this observation to obtain 
	\eqref{ind-formula} in the case that $\omega$ is irrational.
\end{ex}

A {\it lifted singular connection} $\widetilde \alpha$ on $(Y,K)$ with holonomy parameter $\omega$ is an element $\alpha$ of $\mathcal B^\omega(Y,K)$ together with a choice of a homotopy class of a path from $\alpha$ to $\theta^\omega$ in $\mathcal B^\omega(Y,K)$. The path from $\alpha$ to $\theta^\omega$ determines a connected component $z$ of $\mathcal B^\omega(\R\times Y,\R\times K;\alpha,\theta^\omega)$. Using this path $z$, we can define\footnote{Our convention of the Chern--Simons functional $\cs$ differs from the convention of \cite{Ha21} by a factor of $2$.}
\[
  \cs(\widetilde \alpha)=2\kappa(z),\hspace{1cm}{\rm hol}_K(\widetilde \alpha)=\nu(z),\hspace{1cm}{\rm deg}_{\Z}(\widetilde \alpha)=\ind(z).
\]
There is a forgetful map that sends the lifted singular connection $\widetilde \alpha$ to the underlying singular connection $\alpha$. This map identifies lifted singular connections as the universal cover of $\mathcal B^\omega(Y,K)$. Then $\cs$ and ${\rm hol}_K$ define two real-valued functions on the universal cover. The critical points of the Chern--Simons functional $\cs$ are given by lifted flat singular connections. For a pair of integers $k$, $l$, we can add $k$ instantons and $l$ monopoles to the path component of a lifted connected $\widetilde \alpha$ to obtain a new lifted connection with the same underlying singular connection. If we denote this lifted connection by $U^{2k+l}T^{2l}\widetilde \alpha$ then we have 
 \begin{align}
  \cs(U^iT^j\widetilde \alpha)=\cs(\widetilde \alpha)&+i+j(2\omega-\frac{1}{2}),\hspace{1cm}{\rm hol}_K(U^iT^j\widetilde \alpha)={\rm hol}_K(\widetilde \alpha)+j,\label{lifts-cs-hol}\\[2mm]
  &{\rm deg}_{\Z}(U^iT^j\widetilde \alpha)=4i+{\rm deg}_{\Z}(\widetilde \alpha).\label{lifts-degZ}
\end{align}
for any pair of integers $i$ and $j$ satisfying the property that $j$ is an even integer $2i-j$ is divisible by $4$. We formally extend the definition of lifted flat connections to the include case that $i$, $j$ are arbitrary integers by requiring \eqref{lifts-cs-hol}. 

The computation in \eqref{lifts-cs-hol} motivates the definition of a ring with two gradings. From now on we assume that $\omega\in (0,\frac{1}{4})$. Let $R$ be a ring, and consider the ring $R[T^{-1},T]\!][U^{\pm 1}]$ of Laurent polynomials in the variable $U$ with coefficients in the ring of Laurent power series in the variable $T$. Thus, for any element $q$ of this ring there is a finite positive integer $N$ such that
\[
  q=\sum_{i=-N}^Nr_i(T)U^i
\]
with $r_i(T)\in R[T^{-1},T]\!]$. We may regard $R[T^{-1},T]\!][U^{\pm 1}]$ as an algebra over $R[U^{\pm 1}, T^{\pm 1}]$ in the obvious way. We define the I-grading of a monomial $rU^iT^j$ for a non-zero $r\in R$ to be $i+j(2\omega-1/2)$. Now the I-grading of an arbitrary non-zero element $q$ of $R[T^{-1},T]\!][U^{\pm 1}]$ is defined to be the maximum of the I-gradings of all monomials in $q$. In particular, the assumption that $\omega<1/4$ implies that the I-grading of any non-zero element is a finite number. We extend the I-grading to the zero element of $R[T^{-1},T]\!][U^{\pm 1}]$ by $-\infty$. We also define a $\Z$-grading $\deg_\Z$ on $R[T^{-1},T]\!][U^{\pm 1}]$ by setting elements of the form $r(T)U^i$ to be homogenous of degree $4i$.

Next, we recall the definition of the $\cS$-complex of $(Y,K)$ for a holonomy parameter $\omega$ such that $e^{4\pi i \omega}$ is not a root of the Alexander polynomial of $(Y,K)$. The assumption on $\omega$ implies that the reducible $\theta^\omega$ for the pair $(Y,K)$ is non-degenerate. After a small perturbation of $\cs$, we may assume that its critical points in the space of lifted singular connections are non-degenerate. (We assume that the perturbation is invariant with respect to the action of $U$ and $T$ on the space of lifted singular connections.) Let $C^\omega(Y,K)^\circ$ be the $R$-module freely generated by the critical points of the perturbed $\cs$. Then there is an action of $R[U^{\pm 1},T^{\pm 1}]$ on this module, and $C^\omega(Y,K)^\circ$ is in fact a finitely generated free module over $R[U^{\pm 1},T^{\pm 1}]$. For a reason that becomes clear momentarily, we take the completion 
\begin{equation}\label{comp-cx}
  C^\omega(Y,K):=C^\omega(Y,K)^\circ\otimes_{R[U^{\pm 1},T^{\pm 1}]} R[T^{-1},T]\!][U^{\pm 1}]
\end{equation}
to obtain a finitely generated free module over $R[T^{-1},T]\!][U^{\pm 1}]$. We use the value of the perturbed $\cs$ to define $\deg_I$ for any element $r \widetilde \alpha\in C^\omega(Y,K)$ where $\widetilde \alpha$ is a lifted connection and $r\in R$ is non-zero. We extend $\deg_I$ to all elements of $C^\omega(Y,K)$ using \eqref{deg-I-non-hg}. This grading, called I-grading, has the property that
\begin{equation}\label{mult-I-gr}
  \deg_I(q\cdot \zeta)=\deg_I(q)+\deg_I(\zeta)
\end{equation}
for any $\zeta\in C^\omega(Y,K)$ and $q\in R[T^{-1},T]\!][U^{\pm 1}]$. We also use $\deg_\Z$ to define a $\Z$-grading on $C^\omega(Y,K)$.

The moduli spaces of singular ASD connections determine a differential $d$ on $C^\omega(Y,K)$. To be more specific, let $\widetilde \alpha$ be a lifted (perturbed) flat connection with the underlying singular flat connection $\alpha$. Suppose $z$ is a path from $\alpha$ to another irreducible perturbed flat connection $\alpha'$ such that $\ind(z)=1$. The path $z$ determines a lift of $\widetilde \alpha'$ of $\alpha'$ by requiring that 
\[
  \cs(\widetilde \alpha)=2\kappa(z)+\cs(\widetilde \alpha'),\hspace{1cm}
  {\rm hol}_K(\widetilde \alpha)=\nu(z)+{\rm hol}_K(\widetilde \alpha').
\]
In particular, $\deg_\Z(\widetilde \alpha)$ is equal to $\ind(z)+\deg_\Z(\widetilde \alpha')$, and hence the difference between the $\Z$-gradings of $\widetilde \alpha$ and $\widetilde \alpha'$ is $1$. The downward gradient flow line equation for the perturbed $\cs$ can be identified as a perturbation of the ASD equation over $\R\times Y$, and with a slight abuse of notation, we still write $M_z^\omega(\R\times Y,\R\times K;\alpha,\alpha')$ for the solutions of this perturbed ASD equation in $\mathcal B_z^\omega(\R\times Y,\R\times K;\alpha,\alpha')$. By choosing the perturbation of $\cs$ generically, we may assume that $M_z^\omega(\R\times Y,\R\times K;\alpha,\alpha')$ is cut down transversely and hence it is a smooth $1$-dimensional manifold. This moduli space admits a natural orientation and a free $\R$-action given by translation. Therefore, we obtain a well-defined integer $m_z$ by the signed count of the elements of the quotient space  $M_z^\omega(\R\times Y,\R\times K;\alpha,\alpha')/\R$. Now we define 
\begin{equation}\label{diff-omega}
  d(\widetilde \alpha)=\sum_{z} m_z \cdot \widetilde \alpha'
\end{equation}
where the sum is over paths $z$ as above. It is clear from the construction that $\deg_I(\widetilde \alpha')$ is smaller than $\deg_I(\widetilde \alpha)$. However, we cannot guarantee that there are finitely many terms in \eqref{diff-omega} and this is the reason that we need to work with the completion in \eqref{comp-cx}.

The map $d$ extends to an $R[T^{-1},T]\!][U^{\pm 1}]$-module homomorphism. This defines a differential on $C^\omega(Y,K)$. This differential decreases $\deg_\Z$ by $1$ and decreases the I-grading. We similarly define maps 
\[
  \delta_1:C^\omega(Y,K) \to R[T^{-1},T]\!][U^{\pm 1}],\hspace{1cm}
  \delta_2:R[T^{-1},T]\!][U^{\pm 1}] \to C^\omega(Y,K), 
\]
using the singular ASD connections that are asymptotic to $\theta^\omega$, and a homomorphism
\[
  v:C^\omega(Y,K)\to C^\omega(Y,K)
\]
by cutting down the moduli spaces by holonomy along the path $\R\times \{p\}$ where $p$ is a basepoint on $K$. 

Using these homomorphisms, we can form $\widetilde C^\omega(Y,K)$ together with a differential $\widetilde d$ and a map $\chi$ as in \Cref{review-S-comp}. The $\Z$-grading $\deg_\Z$ and the I-grading $\deg_I$ define the structure of a variant of I-graded $\cS$-complexes: the differential $\widetilde d$ and the homomorphism $\chi$ respectively decrease and increase $\Z$-grading by $1$. The map $\widetilde d$ decreases the I-grading and $\chi$ preserves the I-grading. Multiplication of an element of $\widetilde C^\omega(Y,K)$ by an element of $R[T^{-1},T]\!][U^{\pm 1}] $ changes the I-grading as in \eqref{mult-I-gr}. Moreover, multiplying a homogenous element (with respect to $\deg_\Z$) in $\widetilde C^\omega(Y,K)$ by $r(T)U^i$ increases $\deg_\Z$ by $4i$. This algebraic structure of $\widetilde C^\omega(Y,K)$ motivates a generalization of the notion of I-graded $\cS$-complexes. We continue to call any such algebraic structure an I-graded $\cS$-complex; if we want to specify the role of $\omega$, we say an I-graded $\mathcal{S}$-complex {\it{with holonomy parameter $\omega$}}. 

Most of the algebraic notions introduced in \Cref{enriched-s-comp} can be adapted to the present setup. We define, just as before, height $i$ morphisms of I-graded $\cS$-complexes with holonomy parameter $\omega$ and with level $\delta$. There is a small modification that we need to make in the definition of strong morphisms because the variable $T$ has non-trivial I-grading.  A height $i$ morphism of I-graded $\cS$-complexes with level $\delta$ is strong if the element $c_i$, defined as in \eqref{eq:cjdefn}, has the form 
\[
  c_i=r_0+r_1T+r_2T^2+\dots\in R[T^{-1},T]\!]
\]
where $r_0$ is an invertible element of the ring $R$. This is equivalent to saying that $c_i$ is invertible and $\deg_I(c_i)=0$. We define enriched complexes with holonomy parameter $\omega$ as in \Cref{def of enriched cpx} with the change that instead of one discrete subset $\frak K$ of $\R$, we have one such discrete set $\frak K_m$ for any integer $m$, and we update item (iii) by replacing \eqref{cond-k} with 
\begin{equation}\label{cond-k-gen}
\deg_\Z(\zeta)=m \hspace{0.3cm} \Longrightarrow \hspace{.3cm} \deg_I(\zeta)\in B_\delta(\frak K_m) .
\end{equation}
Finally, we follow \Cref{enriched morphism definition} without any change to define (weak) local equivalence of enriched complexes with holonomy parameter $\omega$.

Now we return to our topological setup. The following lemma justifies our requirement in \eqref{cond-k-gen} on the algebraic side.

\begin{lem}\label{cs-val-degZ}
	For any $(Y,K)$ and any $\omega$, there is a discrete subset $\mathfrak K_m^0\subset \R$ for any integer $m$ such that for any lifted singular flat connection $\widetilde \alpha$ on 
	$(Y,K)$ with holonomy parameter $\omega$ and $\deg_\Z(\alpha)=m$, we have $\cs(\widetilde \alpha)\in \mathfrak K_m^0$. 
	Moreover, there is a constant $M_0$ such that the following
	holds: for any positive real number $\delta$, there is a positive constant $\epsilon$ such that if we perturb $\cs$ by a perturbation term whose norm is less than $\delta$
	and $\widetilde \alpha$ is a lift of a critical point of the perturbed $\cs$ with $\deg_\Z(\alpha)=m$, then 
	\begin{equation}\label{ineq-cs}
	  \cs(\widetilde \alpha)\in \bigcup_{i=m-M_0}^{m+M_0}B_\delta(\mathfrak K_i^0).
	\end{equation}
\end{lem}
In the statement of this lemma, we do not assume that perturbed singular flat connections on $(Y,K)$ are necessarily non-degenerate. In this case, the ASD operator associated to a connection $A$ from some singular flat connection $\alpha$ to $\theta^\omega$ is not necessarily Fredholm. To define $\deg_\Z$ in this case, we modify the function spaces in the definition of $D_A$ to weighted Sobolev spaces. That is to say, we replace $L^2_k$ function spaces with $L^2_{k,\delta}$ for some sufficiently small positive value of $\delta$. An element $f$ of some function space $L^2_{k,\delta}$ over the cylinder $\R\times Y$ is characterized by the property that $e^{\delta|t|}f$ is an element of $L^2_k$ where $t$ denotes the parameter on the $\R$ factor.

\begin{proof}
	The moduli space of singular flat connections $\fC(Y,K,\omega)$ is compact by the Uhlenbeck compactness theorem. Thus $\fC(Y,K,\omega)$ has finitely many path 
	connected components, and we pick a basepoint from each path connected component to obtain a finite set of 
	singular flat connections $\{\alpha_i\}_{i=1}^N$. Any lifted flat connection $\widetilde \alpha$ can be written as the concatenation of a lift $\widetilde \alpha_i$ of 
	a basepoint singular flat connection $\alpha_i$ and a path $z$ within one of the connected components of $\fC(Y,K,\omega)$. The path $z$ is represented by a singular
	flat connection $A$ on $\R\times Y$ with holonomy parameter $\omega$ such that the restriction of $A$ to $\{t\}\times Y$ for any $t\in \R$ is flat. In particular, the topological energy
	of $A$ vanishes, and we have $\cs(\widetilde \alpha)=\cs(\widetilde \alpha_i)$. 
	
	Although $\deg_\Z(\widetilde \alpha)$ is not necessarily equal to $\deg_\Z(\widetilde \alpha_i)$, 
	we claim that the difference between these quantities is uniformly bounded. First note that $h^1(\alpha)$, the dimension of the 
	Zariski tangent space of $\fC(Y,K,\omega)$ at $\alpha\in \fC(Y,K,\omega)$ (or equivalently the kernel of the gradient of $\cs$ at $\alpha$), is bounded by some finite constant $M$. 
	This follows from the compactness of $\fC(Y,K,\omega)$ and the fact that $h^1(\alpha)$ defines a lower semi-continuous function on $\fC(Y,K,\omega)$. For any 
	$\alpha \in \fC(Y,K,\omega)$, there is a small neighborhood of $\alpha$ in $\mathcal B^\omega(Y,K)$ such that for any path $z$ of singular flat connections within this neighborhood, 
	the ASD index of the path is at most $h^1(\alpha)$. This can be seen from the spectral flow interpretation of the ASD index over a cylinder.
	From this observation, we see that the difference between $\deg_\Z(\widetilde \alpha)$ and $\deg_\Z(\widetilde \alpha_i)$
	is bounded by a constant $M_0$ that is independent of $\widetilde \alpha$. As a consequence of the relationship between $(\cs(\widetilde \alpha),\deg_\Z(\widetilde \alpha))$
	and $(\cs(\widetilde \alpha_i),\deg_\Z(\widetilde \alpha_i))$, it suffices to prove the first claim in the case
	that $\widetilde \alpha$ is a lift of one of the basepoints. Now the claim follows because there are finitely many basepoints and the topological energy and the index of the lifts of these 
	basepoints satisfy \eqref{lifts-cs-hol} and \eqref{lifts-degZ}.
	
	For the second claim of the lemma, the key point is that any critical point $\alpha$ of perturbed $\cs$ for a small enough perturbation belongs
	to a small neighborhood of $\alpha_0\in \fC(Y,K,\omega)$. We may assume that this neighborhood is small enough such that a path $z$ from $\alpha_0$ to $\alpha$
	in this neighborhood has the property that 
	\[
	  \kappa(z)<\frac{\delta}{2},\hspace{1cm} \ind(z)<M_0,
	\]
	where $M_0$ is the constant obtained in the previous paragraph. Since any lift $\widetilde \alpha$ of $\alpha$ is obtained from concatenating a lift of $\widetilde \alpha_0$ and 
	a path $z$ as above, the property in \eqref{ineq-cs} holds.
\end{proof}

 For two perturbations of the Chern--Simons functional in the definition of $\widetilde C^\omega(Y,K)$, we obtain morphisms of some level $\delta$ between the corresponding I-graded $\cS$-complexes. Therefore, taking a sequence of perturbations going to zero gives an enriched complex $\mathfrak E^\omega(Y,K)$ in the above generalized sense. Although \eqref{cond-k} does not hold in the case that $\omega$ is irrational, we have \eqref{cond-k-gen} for arbitrary $\omega$, as a consequence of Lemma \ref{cs-val-degZ}, by taking $\mathfrak K_m$ to be the union of the discrete sets $\mathfrak K_i^0$ for $i$ between $m-M_0$ and $m+M_0$. The homotopy-type of this enriched complex is a topological invariant of $(Y,K)$.
 
 \vspace{1mm}

\begin{rem}
	Suppose $l$ is the $\Z/2$-bundle over $Y\setminus K$ that restricts to a non-trivial bundle on any meridian of $K$. Since the structure group of $l$ 
	is discrete, it comes with a canonical connection $\iota$. Given a singular connection $A$ on $(Y,K)$ with holonomy parameter $\omega$, the 
	connection $A\otimes \iota$ determines a  singular connection on $(Y,K)$ with holonomy parameter $1/2-\omega$. The above construction can be 
	modified in a straightforward way to define an enriched complex $\mathfrak E^\omega(Y,K)$ for $\omega\in (1/4,1/2)$, which is now defined over the 
	ring $R[\![T^{-1},T][U^{\pm 1}]$. The operation of taking tensor product with $\iota$ allows us to obtain this enriched complex from 
	$\mathfrak E^{\frac{1}{2}-\omega}(Y,K)$. In particular, we do not get any new information by considering the values of $\omega$ that are greater than $1/4$.
\end{rem}

This construction of enriched complexes of knots with arbitrary holonomy parameter $\omega$ is functorial with respect to cobordisms of pairs. First we need the following generalization of Definition \ref{defn:negdefcob}.
\begin{defn} \label{defn:negdefcob-omega}
	Let $(W,S):(Y,K)\to (Y',K')$ be a cobordism of pairs and $c\in H^2(W;\Z)$. For a non-negative integer 
	$i$, we say $(W, S,c )$ is a {\it height $i$ negative definite cobordism} with respect to the holonomy 
	parameter $\omega$ if it satisfies the following properties.
	\begin{itemize}
		\item[(i)]$b^{1}(W)=b^{+}(W)=0$ and 
		$e^{4\pi i \omega}$ is not the root of the Alexander polynomials of $(Y,K)$, $(Y',K')$.
		\item[(ii)] For any $L$, the index of the reducible ASD connection $A^\omega_{L}$ 
		from Example \ref{cob-red} is at least
		 $2i-1$.
	\end{itemize}
	We say $(W, S,c )$ is a {\it strong height $i$ negative definite cobordism (over $R[\![T]\!]$)} if 
	\begin{equation}\label{eta-omega}
		  \eta^{\omega}(W,S, c):=
		  \sum_{\substack{c_{1}(L)\in H^{2}(W; \Z)\\\ind(A_L)=2i-1} } (-1)^{c_{1}(L)^{2}}T^{\nu(A_0)-\nu(A)}\in R[\![T]\!]
		\end{equation}
		is invertible and has vanishing $\deg_I$. Here $A_0$ is a reducible ASD connection of index $2i-1$ on $(W,S,c)$
		whose topological energy is equal to 
		\begin{equation}\label{kappa-min}
		  \kappa_{\rm min}^\omega(W, S, c):=\min \left\{\kappa(A_L^\omega)\mid \ind(A_L^\omega)=2i-1
		   \right\}.
		\end{equation}
		In what follows, if $c$ is trivial, then we drop it from the notation in \eqref{eta-omega} and \eqref{kappa-min}.
\end{defn}

Note that the index formula \eqref{ind-formula} and the assumption on $A_0$ implies that the terms in \eqref{eta-omega} have non-negative powers of $T$ and hence it belongs to $R[\![T]\!]$. Moreover, $\eta^{\omega}(W,S, c)$ is independent of the choice of $A_0$ because $\nu(A_0)$ is determined by $\kappa_{\rm min}^\omega(W, S, c)$. We also remark that the subset of $H^2(W;\Z)$ given by $c_1$ of connections $A^\omega_{L}$ with minimal index is independent of $\omega$. In particular, $\eta^{\omega}(W,S, c)$ (defined using the value of $i$ that minimizes the index of the ASD operator) is independent of $\omega$.\\

\begin{ex}\label{croos-chng-omega}
	Let $K'$ is obtained from a knot $K\subset Y$ by changing a negative crossing to a positive crossing, and $(W, S):(Y, K)\rightarrow (Y, K')$ the blow-up of the crossing change cobordism.
	There are two reducibles $A_0$ and $A_1$, corresponding to the trivial cohomology class and the the exceptional class, that minimize the index. The index of these reducibles is 
	$\sigma_{\omega}(Y,K)-\sigma_{\omega}(Y',K)-1\in \{\pm 1\}$. We have
	\[
	  (\kappa(A_0),\nu(A_0))=(4\omega^2,0),\hspace{1cm}  (\kappa(A_1),\nu(A_1))=(1-4\omega+4\omega^2,-4).
	\]	
	which implies that
	\[
	  \eta^{\omega}(W, S)=1-T^4.
	\]
	This is an invertible element of $\Z[\![T]\!]$ with trivial $\deg_I$. Therefore, $(W,S)$ gives a strong height $0$ or $1$ negative definite cobordism.
	In the case that $K'$ is obtained from $K$ by changing a positive crossing to a negative crossing, the cobordism $(W, S)$ has a unique reducible with minimal index 
	$\sigma_{\omega}(Y,K)-\sigma_{\omega}(Y',K)-1\in \{-3,-1\}$ and vanishing values of $\kappa$, $\nu$. In particular, $\eta^{\omega}(W, S)=1$. Thus $(W,S)$ is 
	a strong height $0$ negative definite cobordism if $\sigma_{\omega}(Y,K)=\sigma_{\omega}(Y',K)$.\\
\end{ex}

\begin{ex}\label{immersed-cob-omega}
	As a generalization of the previous example, let $S:K\to K'$ be an immersed knot cobordism in $[0,1]\times Y$ with transverse double points. We write $s_+$ and $s_-$ for the number of positive and negative double points
	of $S$. Then blowing up the double points of $S$ 
	determines a cobordism $(W,\overline S):(Y,K)\to (Y,K')$. If $i=\frac{1}{2}(\sigma_{\omega}(K) -\sigma_{\omega}(K'))-g(S)\geq 0$, then $(W,\overline S)$ gives a height $i$ negative definite cobordism
	with $ \eta^{\omega}(W, S)=(1-T^4)^{s_+}$ and $ \kappa_{\rm min}^\omega(W, S, c)=4s_+\omega^2$.
\end{ex}

For a negative definite cobordism $(W,S,c):(Y,K)\to (Y',K')$ of height $i$, we define maps
\begin{align}\label{cob-comp-omega}
  \lambda,\, \mu: C^\omega(Y,&K)\to C^\omega(Y',K'),\hspace{1cm}\Delta_1:C^\omega(Y,K) \to R[T^{-1},T]\!][U^{\pm 1}],\nonumber\\[2mm] &\hspace{1cm} \Delta_2: R[T^{-1},T]\!][U^{\pm 1}] \to C^\omega(Y',K'),
\end{align}
essentially in the same way as in the case $\omega=1/4$, using moduli spaces $M^\omega_z(W,S,c;\alpha,\alpha')$. We only need clarify our convention on relating lifted flat connections. Fix a reducible connection $A_0$ as in Definition \ref{defn:negdefcob-omega}. Let $\widetilde \alpha$ be a lifted singular connection on $(Y,K)$ with the underlying singular connection $\alpha$ and $z$ be a path from $\alpha$ to $\alpha'$ along $(W,S,c)$. (In practice, and for the definition of the cobordism maps above, the path $z$ is determined by a singular instanton.) Then we fix a lift $\widetilde \alpha'$ of $\alpha'$ by requiring:
\begin{equation}\label{convetion-lifts}
	 \cs (\wt{\alpha'})+2\kappa(z)=\cs (\wt{\alpha}) + 2\kappa^{\omega}_{\min}(W, S, c),\hspace{1cm}
	 {\rm hol}_K (\wt{\alpha'})+\nu(z)={\rm hol}_K (\wt{\alpha}) + \nu(A_0).
\end{equation}
As mentioned above, $\nu(A_0)$ is independent of $A_0$ as long as its energy is equal to $\kappa_{\rm min}^\omega(W, S, c)$.

The maps in \eqref{cob-comp-omega} combine, in the way described in \Cref{review-S-comp}, to define a height $i$ morphism 
\[
	\widetilde \lambda^\omega_{(W,S,c)}:\widetilde C^\omega(Y,K) \to \widetilde C^\omega(Y',K').
\]
That is to say, $\widetilde \lambda_{(W,S,c)}$ is a chain map, and if we define $c_j$ as in Definition \ref{height i morphism}, then $c_i=\eta^\omega(W,S,c)$ and $c_j=0$ for $j<i$. For $i=0$ this is proved in \cite{Ha21} and for higher values of $i$ the argument of \cite{DS20} adapts without any essential change. In the absence of perturbation terms, the first identity in \eqref{convetion-lifts} and non-negativity of the topological energy of instantons imply that $\widetilde \lambda_{(W,S,c)}$ is a morphism of I-graded complexes with level $2\kappa^{\omega}_{\min}(W, S, c)$. In general, we obtain a morphism of enriched complexes 
\[
  \mathfrak L^\omega_{(W,S,c)}:\mathfrak E^\omega(Y,K)\to \mathfrak E^\omega(Y',K')
\]
with height $i$ and level $2\kappa^{\omega}_{\min}(W, S, c)$ using sequences of perturbations terms that go to zero.

We may forget the I-grading in the above construction and obtain an $\cS$-complex in the sense of Definition \ref{S-comp} over the ring $R[T^{-1},T]\!]$ with a $\Z/4$ grading. Here we use the isomorphism of degree $4$ provided by $U$ to roll the $\Z$-graded complex $\widetilde C^\omega(Y,K)$ into a $\Z/4$ graded complex which is still denoted by the same notation.  Unlike the case of holonomy parameter $\omega=1/4$, the $\cS$-complex $\widetilde C^\omega(Y,K)$ does not see much about the knottedness of $K$ in $Y$. To be more precise, we have the following.

\begin{prop}[\cite{Ha21}]\label{omega-sig-det}
	If $K$ and $K'$ are two homotopic knots in an integer homology sphere $Y$ with $\sigma_\omega(Y,K)=\sigma_\omega(Y,K')$ and non-vanishing values of Alexander polynomial at $e^{4\pi i \omega}$, 
	then the $\cS$-complexes  $\widetilde C^\omega(Y,K)$ and  $\widetilde C^\omega(Y,K')$ are chain homotopy equivalent.
\end{prop}
\begin{proof}
	We may obtain $K'$ from $K$ by a sequence of crossing changes. This gives a cylinder cobordism $S:K\to K'$ immersed into $[0,1]\times Y$.
	Use the construction of $(W,\overline S):(Y,K) \to (Y',K')$ in Example \ref{immersed-cob-omega} to obtain a height $0$ negative definite cobordism.
	The above functoriality discussion gives a morphism $\widetilde \lambda:\widetilde C^\omega(Y,K) \to \widetilde C^\omega(Y,K')$. We obtain a morphism $\widetilde \lambda'$ of $\cS$-complexes in the reverse 
	direction by flipping $(W,\overline S)$. The composition of $\widetilde \lambda$ and $\widetilde \lambda'$ in either order is the identity multiplied by a term of the form $(1-T^4)^n$ \cite{Ha21}. This gives
	the desired claim, because $1-T^4$ is a unit in $R[T^{-1},T]\!]$.
\end{proof}

In the case of a knot $K$ in $S^3$, we can completely characterize $\widetilde C^\omega(K)$. For any $\omega$, there is a two-bridge torus knot $T_{2, 2n+1}$ such that $\sigma_\omega(T_{2, 2n+1})=-2$ and $\Delta_{T_{2, 2n+1}}(e^{4\pi i \omega})\neq 0$ \cite{Ha21}. 
For any such knot, the character variety \eqref{char-var-omega} is non-degenerate and has one reducible element $\theta^\omega$ and one irreducible element $\alpha$. In particular, we can use a perturbation in the definition of $\widetilde C^\omega(T_{2, 2n+1})$ such that $C^\omega(T_{2, 2n+1})=R[T^{-1},T]\!]$. There is an iterated crossing change cobordism from the unknot to $T_{2, 2n+1}$ which gives rise to a strong height $1$ negative definite cobordism of pairs by following the construction of Example \ref{immersed-cob-omega}. In particular, this implies that the map $\delta_1$ is multiplication by a unit. Now degree consideration implies that the remaining maps involved in the $\cS$-complex $\widetilde C^\omega(T_{2, 2n+1})$ have to be trivial. After a change of basis, we may also assume that $\delta_1:R[T^{-1},T]\!] \to R[T^{-1},T]\!]$ is the identity map. We denote this $\cS$ complex by $\widetilde C\langle1\rangle $. More generally, for any integer $k$, let $\widetilde C\langle k\rangle $ be given by tensoring $n$ copies of $\widetilde C\langle 1\rangle $. (For negative $k$, this means that $\widetilde C\langle k\rangle $ is the tensor product of $-k$ copies of the dual of $\widetilde C\langle 1\rangle $.) The connected sum theorem of \cite{Ha21} implies that  $\widetilde C^\omega(\#_kT_{2, 2n+1})$ is chain homotopy equivalent to $\widetilde C\langle 1\rangle $. Therefore, the following is a corollary of Proposition \ref{omega-sig-det}.

\begin{cor}[\cite{Ha21}]\label{s-comp-omega-S3}
	For any knot $K$ in $S^3$ and $\omega\in (0,\frac{1}{4})$ that $e^{4 \pi i \omega}$ is not a root of the Alexander polynomial of $K$, the $\cS$-complex $\widetilde C^\omega(K)$ is chain homotopy equivalent to $\widetilde C\langle -\sigma_\omega(K)/2\rangle $. 
\end{cor}

In the more general case of arbitrary integer homology spheres $Y$, it is reasonable to expect that $\widetilde C^\omega(Y,K)$ is determined by $Y$ and the integer $\sigma_\omega(Y,K)$. At least in the case that $K$ is null-homotopic, we can verify this by following the argument as in the proof of Corollary \ref{s-comp-omega-S3}. Applying the connected sum theorem of \cite{Ha21}, we conclude that $\wtC^{\omega}(Y, K)$ is chain homotopy equivalent to 
\begin{equation}\label{s-comp-omega-null}
	\wtC^{\omega}(Y, U_1)\otimes \widetilde C\langle -\sigma_{\omega}(Y, K)/2 \rangle .
\end{equation}	
Now if $R$ is an integral domain, \eqref{s-comp-omega-null} implies that the following identity holds for the Fr\o yshov invariant $h^\omega(Y,K):=h(\widetilde C^\omega(Y,K))$ of the $\cS$-complex $\widetilde C^\omega(Y,K)$:
\[
  h^\omega(Y,K):=h^\omega(Y,U_1)-\frac{1}{2}\sigma_{\omega}(Y, K)
\] 
Following the same argument as in the proof of \cite[Proposition 5.3]{DS20}, we can see that $h^\omega(Y,U_1)=4h(Y)$ if $2\in R$ is non-zero. Here $h(Y)$ is the Fr{\o}yshov invariant of the integral homology $3$-sphere $Y$ defined using coefficient ring $R$ \cite{Fr02}. In summary, we have the following.

\begin{prop}
	Let $R$ be an integral domain with $2\in R$ non-zero. 
	Let $K$ be a null-homotopic knot in an integer homology sphere $Y$ such that $e^{4 \pi i \omega}$ is not a root of the Alexander polynomial of $K$. Then
	\[h^{\omega}(Y, K)=4h(Y)-\frac{1}{2}\sigma_{\omega}(Y, K).\]
\end{prop}

\subsection{Filtered special cycles for general holonomy parameters}\label{fil-spec-omega}

In the previous section, we learned that $\cS$ complexes of knots in the case $\omega\in (0,\frac{1}{4})$, and after forgetting the I-gradings, do not contain any interesting information. This is in contrast to the case $\omega=\frac{1}{4}$, where we have obtained non-trivial topological information from the $\cS$-complex of a knot $K$, such as the invariant $\widetilde s(K)$. Nevertheless, the I-grading of $\mathfrak E^\omega(Y,K)$ is expected to say more about the topological type of $K$. This I-grading was already used in \cite{Ha21} to study concordances between torus knots. In this section, we initiate the study of some of the concordance invariants that one can obtain from applying the filtered constructions of the previous sections.

For any $\omega$, our concordance invariant in a raw form is the local equivalence class of the enriched complex $\mathfrak E^\omega(Y,K)$. For any ring $R$, we define $\Theta^{\mathfrak{E},\omega} _{R}$ to be the set of local equivalence classes of enriched $\cS$-complexes over the ring $R[T^{-1},T]\!][U^{\pm 1}]$ with holonomy parameter $\omega$. 
As in the previous instances, tensor product gives an additive structure on this set and we would like to define a homomorphism 
\[
  \Om^{\mathfrak{E}, \omega}:\Theta^{3,1}_\Z \to  \Theta^{\mathfrak{E},\omega} _{R}
\]
by mapping $[(Y,K)]$ to the local equivalence class of $\mathfrak E^\omega(Y,K)$. There is an issue with this definition because $e^{4\pi i\omega}$ might be a root of Alexander polynomial of $K$, and hence $\mathfrak E^\omega(Y,K)$ might not be well-defined. One obvious fix for this issue is to limit to the values of $\omega$ such that $e^{4\pi i\omega}$ is not the root of Alexander polynomial of any $K$ in an integer homology sphere. For instance, we can pick $\omega$ to be any rational number if the form $\frac{m}{4p^k}$ where $p$ is a prime number. For any such $\omega$, $e^{4\pi i\omega}$ is not a root of the Alexander polynomial of any knot (See, for example, \cite[Corollary 3.2]{Liv02}.)

The construction of filtered special cycles in Section \ref{Section: Concordance invariants from filtered special cycles} can be adapted to enriched complexes of knots with general holonomy parameters. If $\widetilde C^\omega$ is an $\cS$-complex over the ring $R[T^{-1},T]\!][U^{\pm 1}]$ with parameter $\omega$, then we can form the equivariant complexes and the maps between them as in \Cref{subsection: equivariant}. For any $f\in R$, $s \in \R\cup \{-\infty\}$, $k \in \Z$, we define a filtered special $(k, f, s)$-cycle to be an element $z\in  \lhc_{2k}^\omega$ such that $z=\wh{\Psi}(\mathfrak{z})$ for some $\mathfrak{z} \in \shc_{2k}^\omega$ and 
\[
	  \mathfrak{i}(\mathfrak{z}) = \sum_{i=-\infty}^{-k} b_i x^{i},  \hspace{.5cm } \deg_I (\shd \mathfrak{z}) \leq s,
\]
where $b_{-k}=f+\sum_{i=1}^\infty r_iT^i\in R[\![T]\!]$ with $r_i\in R$. In particular, if $f\neq 0$, the $\deg_I(b_{-k})=0$. Using filtered special cycles, we have the following counterparts of Definitions \ref{N-I-s-comp} and \ref{N-I-s-comp-var}:
\begin{align*}
	\newinv_{\wt{C}^\omega} (k,s) &:=\inf \left\{ \deg_I ( z ) \mid  z \text{ is filtered special $(k,1, s)$-cycle }\right\}\in [0, \infty]\\[2mm]
	\underline{\newinv}_{\wt{C}^\omega} (k,s) &:=\inf \left\{ \deg_I ( z ) \mid  z \text{ is filtered special $(k,f, s)$-cycle with $f\neq 0$}\right\}\in [0, \infty]
\end{align*}
Property \eqref{cond-k-gen} implies that the above definition extends to enriched complexes of holonomy parameter $\omega$ by a limiting process as in \Cref{filt-spec-cyc}. For an enriched $\cS$-complex $\mathfrak E^\omega$, the transpose $\newinv_{\mathfrak E^\omega}^\intercal$ of $\newinv_{\mathfrak E^\omega}$ is defined as in \eqref{transpose-N}. The value of $\newinv_{\mathfrak E^\omega} (k,s)$ belongs to $\mathfrak K_{2k-1}$, and similarly for $\underline{\newinv}_{\mathfrak E^\omega}$ and  $\newinv_{\mathfrak E^\omega}^\intercal$. As a counterpart of \ref{mon-N}, $\newinv_{\mathfrak E^\omega} (k,s)$ and $\underline{\newinv}_{\mathfrak E^\omega} (k,s)$ are increasing with respect to $k$ and decreasing with respect to $s$. Thus,  $\underline{\newinv}_{\mathfrak E^\omega}(k,r)$ is increasing with respect to $k$ and decreasing with respect to $r$. Define
\[
  \Gamma_{\mathfrak E^\omega}(k):=\underline{\newinv}_{\mathfrak E^\omega} (k,-\infty),\hspace{1cm}r_s (\mathfrak{E}^\omega) :=-\newinv^\intercal_{\mathfrak E^\omega} (0,-s).
\]
Applying these to the enriched complex $\mathfrak E^\omega(Y,K)$ of a knot $K$ gives the topological invariants $\newinv_{(Y,K)}^\omega$, $\underline{\newinv}_{(Y,K)}^\omega$, $\newinv_{(Y,K)}^{\omega,\intercal}$, $\Gamma_{(Y,K)}^\omega$ and $r_s^\omega(Y,K)$.

The connected sum theorem of \cite{Ha21} gives rise to the analogues of the connected sum inequalities in \Cref{conn sum for JYK}. For the pairs $(Y,K)$ and $(Y',K')$, let $k$, $k'$ be integers and $s$, $s'$ be negative real numbers such that $s^\otimes := \max \{ \newinv^\omega_{(Y,K)} (k,s) +s', \newinv^\omega_{(Y',K')} (k',s') +s\}$ is negative.  Then
\[
  \newinv_{(Y\#Y',K\#K')}^\omega(k+k', s^\otimes) \leq 
  \newinv_{(Y,K)}^\omega (k,s) + \newinv_{(Y',K')}^\omega(k',s'). 
\]
A similar inequality holds for $\underline{\newinv}_{(Y,K)}^\omega$. If $r$ and $r'$ are positive real numbers, then 
\[
\newinv^{\omega,\intercal}_{(Y \# Y' , K \# K') }(k+k', r + r' ) \leq \max \{ \newinv^{\omega,\intercal}_{(Y, K)}(k,r)+r', \;\newinv^{\omega,\intercal}_{(Y,' K')}(k',r')+r \} . 
\]

The following result generalizes \Cref{Fr ineq} and \Cref{newinvinqtranspose} to other holonomy parameters.

\begin{thm}\label{newinvinqtranspose-omega}
	Let $(Y,K)$ and $(Y',K')$ be given such that $e^{4\pi i \omega}$ is not a root of 
	the Alexander polynomials of these knots. Let $(W,S):(Y,K)\to (Y',K')$ 
	be a cobordism which is negative definite of strong height $i\geq 0$
	with respect to the holonomy parameter $\omega$. Then we have
	\begin{align}
		\newinv^\omega_{(Y',K')}(k+  i ,s ) &\leq 
		\newinv^\omega_{(Y,K)}(k ,s- 2 \kappa^\omega_{\operatorname{min}}(W,S) ) + 
		2 \kappa^\omega_{\operatorname{min}}(W,S),\label{eq:newinvineq-omega}\\[2mm]
		\newinv^{\omega,\intercal}_{(Y',K')}(k+i,r+2 \kappa^\omega_{\operatorname{min}}(W,S) ) &\leq \newinv^{\omega,\intercal}_{(Y,K)}(k ,r) 
		+ 2 \kappa^\omega_{\operatorname{min}}(W,S)\label{eq:newinvineq-trasnpose-omega}
	\end{align}
	Moreover, if equality is achieved in \eqref{eq:newinvineq-omega} 
	(resp. \eqref{eq:newinvineq-trasnpose-omega}) for some $k\in\Z$ and $s\in [-\infty,0)$ (resp. $r\in [0,\infty]$), with both sides finite
	 and positive (resp. negative), then there exists an irreducible $SU(2)$ representation of 
	 $\pi_1(W\setminus S)$ that maps a meridian of $S$ to \eqref{omega-conj-2}.
	 The same conclusion holds for $\underline{\newinv}_{(Y,K)}^\omega$, under the weaker assumption that 
	 the cobordism is not necessarily strong, but satisfies $\deg_I(\eta^\omega(W,S))=0$.
\end{thm}

\begin{ex}
	The roots of the Alexanader polynomial of the trefoil $T_{2,3}$ are $e^{\pi i/3}$, $e^{5\pi i/3}$, and $\sigma_\omega(T_{2,3})$ is equal to $-2$ for $1/12<\omega<5/12$ and is
	equal to $0$ for $0<\omega<1/12$ or $5/12<\omega<1/2$. This computation of Tristram-Levine signature is reflected in the fact that the moduli space of flat connections 
	$\fC(T_{2,3},\omega)$ for $1/12<\omega<5/12$ has a unique non-degenerate irreducible $\alpha_\omega$ and a non-degenerate reducible. For any such $\omega$ that has 
	the form of a rational number $\frac{m}{2p}$ with $p=6n-1$, we have
	\begin{equation}\label{possible cs}
	  \cs(\wt \alpha_\omega)\in \left\{\frac{(5p+6m)^2}{12p^2}+\frac{k}{p}\right\}_{k\in \Z},
	\end{equation}
	where $\wt \alpha_\omega$ is a lift of $\alpha_\omega$. To see this, note that $\alpha_\omega$ can be lifted to a flat connection $\widehat \alpha_\omega$ on the Brieskorn homology
	sphere $\Sigma(2,3,p)$. Following the arguments in \cite{CoSa,Ha21}, one first determines the flat connection $\widehat \alpha_\omega$ on $\Sigma(2,3,p)$ and then uses 
	\cite{FS90} to show that 
	the (ordinary) Chern-Simons functional of $\widehat \alpha_\omega$  is equal to $(5p+6m)^2/24p\in \R/\Z$. Now \eqref{possible cs} follows from this computation.

	For the rest of this example, assume that the above value of $\omega$ is less than $1/4$.
	Since $\alpha_\omega$ is non-degenerate, we may form the enriched complex $\mathfrak E^\omega(T_{2,3})$ of 
	$T_{2,3}$ using a sequence of perturbations the Chern-Simons functional that are trivial in a neighborhood of the elements of $\fC(T_{2,3},\omega)$. In particular, 
	$\Gamma ^{\omega}_{T_{2,3}}(1)$ is a positive real number in the set \eqref{possible cs}. Moreover, applying Example \ref{croos-chng-omega} to the crossing change 
	cobordism from the unknot to the trefoil and then applying Theorem \ref{newinvinqtranspose-omega} implies that 
	\[
	  \Gamma ^{\omega}_{T_{2,3}}(1 ) \leq  \frac{2m^2}{p^2}. 
	\]	
	In particular, if $2m^2<p$, then the above two constraints uniquely determine $\Gamma ^{\omega}_{T_{2,3}}(1)$. We can use this to show the following identities:
	\[
	  \Gamma^{1/10}_{T_{2,3}}(1)=\frac{1}{300},\hspace{1cm} \Gamma^{1/11}_{T_{2,3}}(1)=\frac{1}{1452}.
	\]
	In general, we expect that
	\[
	  \Gamma^{\omega}_{T_{2, 3}}(1)=12\left(\omega-\frac{1}{12}\right)^2
	\]
	for $\omega\in (1/12,1/4]$.	
\end{ex}

\vspace{1mm}

\begin{cor}\label{rep-existence}
	Suppose $K$ is a knot in integer homology sphere $Y$ such that $e^{4\pi i \omega}$ is not a root of 
	the Alexander polynomial of $K$ and the weak local equivalence class of $\mathfrak E^\omega(Y,K)$ is 
	non-trivial.
	Then for any homology concordance 
	$(W,S):(Y,K) \to (Y',K')$, there is a representation of $\pi_1(W\setminus S)$ into $SU(2)$ 
	that restricts to irreducible representations of $Y\setminus K$ and $Y'\setminus K'$ and the conjugacy class 
	of a meridian of $S$ is mapped to the conjugacy class of \eqref{omega-conj-2} in $SU(2)$.
\end{cor}

This corollary folllows from Theorem \ref{newinvinqtranspose-omega}. A homology concordance $(W,S)$ determines a negative definite cobordism of strong height $0$ with $\kappa^\omega_{\operatorname{min}}(W,S)=0$ and $\eta^\omega(W,S)=1$. The cobordism from $(Y',K')$ to $(Y,K)$ obtained by flipping $(W,S)$ satisfies a similar property. Thus \eqref{eq:newinvineq-omega} and \eqref{eq:newinvineq-trasnpose-omega} imply 
\[
  \newinv^\omega_{(Y,K)}(k,s)=\newinv^\omega_{(Y',K')}(k,s),\hspace{1cm} 
  \newinv^{\omega,\intercal}_{(Y,K)}(k,r)=\newinv^{\omega,\intercal}_{(Y',K')}(k,r)
\] 
for all values of $k$, $s$ and $r$. In particular, we obtain the claim in Corollary \ref{rep-existence} if we can show that either $\newinv^\omega_{(Y,K)}$ or $\newinv^{\omega,\intercal}_{(Y,K)}$ is finite and non-zero. Arguing as in Remark \ref{rem:newinvvanishtop}, the value of $\newinv^{\omega,\intercal}_{(Y,K)}(k,s)$ is non-zero for $k\leq 0$. Since $\mathfrak E^\omega(Y,K)$ is not weakly locally equivalent to the trivial enriched complex, the following generalization of Proposition \ref{local eq gamma} shows that $\newinv^{\omega,\intercal}_{(Y,K)}(0,0)$ is finite.

\begin{prop}\label{local eq gamma-omega}
	An enriched $\mathcal{S}$-complex $\mathfrak{E}^\omega$ with holonomy parameter $\omega$ is weakly locally equivalent to the trivial enriched $\mathcal{S}$-complex if and only if the following conditions hold:
	\begin{equation}
		\newinv_{\mathfrak{E}^\omega}(0,-\infty)=0 \text{ and }\newinv_{({\mathfrak{E}^\omega)}^\dagger}(0,-\infty)=0.
	\end{equation}
	These conditions are equivalent to $r_0(\mathfrak{E}^\omega)=\infty$ and $r_0((\mathfrak{E}^\omega)^\dagger)=\infty$.
\end{prop}

\begin{cor}\label{rep-existence-special}
	Suppose $K$ is a knot in an integer homology sphere $Y$ such that $e^{4\pi i \omega}$ is not a root of 
	the Alexander polynomial of $K$ and it satisfies one of the following assumptions:
	\begin{itemize}
		\item[(i)] $4h(Y)+\sigma_\omega(Y,K)\neq 0$;
		\item[(ii)] $\newinv_{(Y,K)}^\omega\neq \newinv_{(S^3,U_1)}^\omega$;
		\item[(iii)] $\underline{\newinv}_{(Y,K)}^\omega\neq \underline{\newinv}_{(S^3,U_1)}^\omega$.
	\end{itemize}
	Then the same claim as in Corollary \ref{rep-existence} holds.
\end{cor}
\begin{proof}
	All of the invariants $\newinv_{(Y,K)}^\omega$, $\underline{\newinv}_{(Y,K)}^\omega$ and $4h(Y)+\sigma_\omega(Y,K)$ depend on the weak local equivalence class of $(Y,K)$. So, if any one of them is non-trivial, then  
	the weak local equivalence class of $\mathfrak E^\omega(Y,K)$ is non-trivial, too. Now the claim follows from Corollary \ref{rep-existence}.
\end{proof}

\begin{proof}[Proof of \Cref{alg conc irrep}]
	This follows from \Cref{rep-existence-special} (i) applied to knots in the $3$-sphere.
\end{proof}

We already studied the behavior of $\newinv^\omega_{(Y,K)}(k,s)$ when we vary the variables $k$ and $s$. The following proposition asserts that $\newinv^\omega_{(Y,K)}(k,s)$ is also well-behaved with respect to varying $\omega$.

\begin{prop}
	Suppose $K$ is a knot in an integer homology sphere $Y$ such that $e^{4\pi i \omega}$ is not a root of 
	the Alexander polynomial of $K$. Then for any $\epsilon>0$, there exists $\delta>0$ such that 
	\[
	|\newinv^\omega_{(Y,K)}(k,s) - \newinv^{\omega'}_{(Y,K)}(k,s) |< \epsilon
	\]
	whenever $|\omega - \omega'|<\delta$. Similar claims hold for $\underline{\newinv}_{(Y,K)}^\omega$ and $\newinv_{(Y,K)}^{\omega,\intercal}$.
\end{prop}
\begin{proof}
	For $(Y,K)$, $\omega$ as in the statement, and any $\epsilon>0$, there is $\delta>0$ such that the following holds \cite{DS:prep}: for any $\omega'$ with 
	$|\omega - \omega'|<\delta$, there are strong height $0$ morphisms of enriched $\cS$-complexes 
	\[
	  \Phi_\omega^{\omega'}:{\mathfrak{E}}^\omega(Y,K) \to {\mathfrak{E}}^{\omega'}(Y,K),\hspace{1cm}
	  \Psi_{\omega'}^{\omega}:{\mathfrak{E}}^{\omega'}(Y,K) \to \mathfrak{E}^{\omega}(Y,K)
	\]
	with level $\epsilon$  such that $\Phi_\omega^{\omega'}\circ \Psi_{\omega'}^{\omega}$ and $\Psi_{\omega'}^{\omega} \circ \Phi_\omega^{\omega'}$ are chain 
	homotopic to the identity morphisms using chain 
	homotopies of level $\epsilon$. Now the claim is an immediate consequence of this property.
\end{proof}

\bibliographystyle{plain}
\bibliography{tex}
\Addresses

\end{document}